\documentclass[12pt]{article}

\usepackage[utf8]{inputenc}
\usepackage[T1]{fontenc}

\usepackage[a4paper, vmargin=2cm, hmargin=2cm]{geometry}
\usepackage{ragged2e}
\usepackage[pagebackref, hidelinks]{hyperref}
\hypersetup{
	colorlinks=true, 
	linkcolor = magenta, 
}

\usepackage{amsmath, amsthm, amsfonts}
\usepackage{amscd, amssymb, latexsym}
\usepackage{upgreek} 
\usepackage{mathtools} 
\usepackage[bb=libus]{mathalpha} 

\usepackage{comment, color}
\usepackage{enumitem}

\usepackage{tikz}
\usepackage{cleveref}

\theoremstyle{plain}
\newtheorem{theorem}{Theorem}[section]
\newtheorem{corollary}[theorem]{Corollary}
\newtheorem{lemma}[theorem]{Lemma}
\newtheorem{proposition}[theorem]{Proposition}

\newtheorem{introtheorem}{Theorem}

\newtheorem{introproposition}[introtheorem]{Proposition}

\newtheorem{introcorollary}[introtheorem]{Corollary}

\theoremstyle{definition}
\newtheorem{definition}[theorem]{Definition}
\newtheorem{example}[theorem]{Example}
\newtheorem{introdefinition}[introtheorem]{Definition}

\theoremstyle{remark}
\newtheorem{remark}[theorem]{Remark}
\newtheorem{introremark}[introtheorem]{Remark}

\newtheorem{question}[theorem]{Question}
\newtheorem{introquestion}[introtheorem]{Question}

\newcommand{\N}{\mathbb{N}}
\newcommand{\Z}{\mathbb{Z}}
\newcommand{\Q}{\mathbb{Q}}
\newcommand{\R}{\mathbb{R}}
\newcommand{\C}{\mathbb{C}}
\newcommand{\K}{\mathbb{K}}

\newcommand{\Kres}{\mathbb{k}}
\newcommand{\Hyp}{\mathbf{H}}
\newcommand{\Sph}{\mathbf{S}}
\newcommand{\Proj}{\mathbf{P}}

\newcommand{\EE}{\mathcal{E}}

\newcommand{\OO}{\mathcal{O}}

\newcommand{\ufi}{\mathfrak{u}}

\DeclareMathOperator{\Hom}{Hom}
\DeclareMathOperator{\Aut}{Aut}

\DeclareMathOperator{\Isom}{Isom}

\DeclareMathOperator{\GL}{GL}
\DeclareMathOperator{\SL}{SL}
\DeclareMathOperator{\PGL}{PGL}
\DeclareMathOperator{\PSL}{PSL}
\DeclareMathOperator{\PO}{PO}

\DeclareMathOperator{\Mod}{Mod}

\DeclareMathOperator{\Tr}{Tr}
\DeclareMathOperator{\sign}{sign}
\DeclareMathOperator{\Bus}{Bus}
\DeclareMathOperator{\It}{\mathfrak{I}}

\DeclareMathOperator{\Card}{Card}
\DeclareMathOperator{\Span}{Span}
\DeclareMathOperator{\Conv}{Conv}

\DeclareMathOperator{\Grp}{Grp}
\DeclareMathOperator{\Mon}{Mon}
\DeclareMathOperator{\Stab}{Stab}

\DeclareMathOperator{\exoH}{\mathfrak{c}}
\DeclareMathOperator{\exoT}{\mathfrak{p}}

\DeclareMathOperator{\dH}{\mathbf{d}}

\newcommand{\Cr}[1]{\ensuremath{\left[ {#1} \right]}}
\newcommand{\class}[1]{\ensuremath{\left[ {#1} \right]}}

\title{The variety of group actions on \\ all algebraic real hyperbolic spaces}

\author{Bruno Duchesne, Christopher-Lloyd Simon}
\date{\today}

\begin{document}
	
	\maketitle
	
	\begin{abstract}
		For a cardinal $\kappa$, denote by $\mathbf{H}^\kappa$ the algebraic real hyperbolic space of dimension $\kappa$.
		For a topological group $\Gamma$, we study the set of continuous representations $\Gamma \to \operatorname{Isom}(\mathbf{H}^\kappa)$ up to continuous self-representations $\operatorname{Isom}(\mathbf{H}^\kappa)\to \operatorname{Isom}(\mathbf{H}^\kappa)$.
		
		The novelty of this work relies in considering simultaneously all cardinals, finite or infinite.
		We will endow this set of classes of representations with a natural topology, and show that this character variety is compact.
		This will also enable us to recover all previous compactifications of actions on $\mathbf{H}^n$ by certain actions on real trees for the equivariant Gromov-Hausdorff topology.
		
		A class of representations recovers in particular the homothety class of its marked length spectrum.
		We will define the notion of algebraic cross-ratio and prove a GNS-embedding result, enabling us to generalize some rigidity properties of the marked length spectrum. 
		
		We will also introduce a notion of abstract cross-ratio, and use it to show that a wide class of groups $\Gamma$ (characterized by the existence of what we call a $3$-full action on a $\operatorname{CAT}(-1)$-space) admit at most one class of irreducible representations into $\operatorname{Isom}(\mathbf{H}^\kappa)$ whose boundedness properties are controlled by those of $(X,d)$.
		We will apply this to topological groups $\Gamma$ such as the isometry group $\operatorname{Isom}(\mathbf{H}^\kappa)$ itself, the automorphism group $\operatorname{Aut}(T_\omega)$ of the simplicial tree with countably infinite valency, and the automorphism group $\operatorname{PGL}_2(\mathbb{K}, \lvert\cdot \rvert)$ of the projective line over a non-Archimedean field.
		
	
	This work of foundational nature opens many directions of research: the size of the character variety for other special examples of groups (mapping class groups, Cremona groups), the definition and study of geometric structures on the character variety (symmetric Thurston metric, Goldman symplectic form), the description of the subset associated to convex cocompact representations.
	
	The background sections are detailed enough to make the work self-contained, and contain various new observations of independent interest (about strongly hyperbolic spaces, cross-ratios, kernels of hyperbolic type).
	
\end{abstract}

\begin{centering}
	\subsubsection*{Acknowledgements}
\end{centering}

We thank Nicolas Monod for discussions concerning algebraic cross-ratios, Federico Viola for discussing some points on the representation theory of $\SL_2(\K)$, Gilles Courtois and Antonin Guilloux for comments on their work about Hausdorff dimension of limiting representations, Pierre Py and David Xu for their comments on a first version of this work as well as Eduardo Reyes and Vincent Guirardel for exchanging ideas neighbouring this work.
We thank the organizers of the webcasted \href{https://sites.google.com/site/annaerschler/grseminar}{grseminar}, through which this collaboration was born.\\

Bruno DUCHESNE is supported by the \href{https://anr.fr/Projet-ANR-24-CE40-3137}{PLAGE project ANR-24-CE40-3137}.

Christopher-Lloyd SIMON is supported by the \href{https://perso.univ-rennes1.fr/serge.cantat/ERCGroupsOfAlgebraicTransformations.html}{ERC GOAT 101053021}.

\renewcommand{\contentsname}{Plan of the paper}
\setcounter{tocdepth}{2}
\tableofcontents

\section*{Introduction}
\addcontentsline{toc}{section}{Introduction}

This work concerns continuous actions of topological groups by isometries on algebraic hyperbolic spaces of arbitrary dimensions.
For a cardinal $\kappa$, denote by $\Hyp^\kappa$ the algebraic hyperbolic space of dimension $\kappa$.
For a topological group $\Gamma$, we wish to study the set of continuous representations $\Gamma \to \Isom(\Hyp^\kappa)$ up to continuous representations $\Isom(\Hyp^\kappa)\to \Isom(\Hyp^{\kappa'})$.
The main questions are: 
\begin{enumerate}[noitemsep]
	\item 
	How to distinguish actions, measure their proximity, and move between one another?
	\item
	What is a reasonable topology on this space and its properties (compact, connected)?
	\item
	When is that space empty (obstruction) or reduced to a singleton (rigidity) ?
\end{enumerate}
This work addresses some aspects of these questions, as explained in this introduction.

\subsection*{Variety of actions up to exotic deformations}
\addcontentsline{toc}{subsection}{Variety of actions up to exotic deformations}

An element $g\in\Isom(\Hyp^\kappa)$ has translation length $\ell(g)=\inf\{d(gx,x) \colon x\in\Hyp^\kappa\}$.
A group representation $\rho\colon\Gamma\to\Isom(\Hyp^\kappa)$ has a \emph{length function} $\ell_\rho \colon \Gamma \to \R_{\ge 0}$ defined by $\ell_\rho \colon \gamma \mapsto \ell(\rho(\gamma))$.
If two representations are conjugate (related by inner automorphism of $\Isom(\Hyp^\kappa)$) then they have the same length functions.

When the dimension $\kappa$ is a finite cardinal $n\in \N$, the only self-representations of $\Isom(\Hyp^n)$ are inner automorphisms, hence for a discrete group $\Gamma$ we recover its well-studied character varieties of actions on $\Hyp^n$.
In particular, its subset corresponding to irreducible representations is well understood: they are characterized by their length functions \cite{Kim_marked-length-rigidity-symmetric-space_2001} and they form a smooth variety compactified by actions on real trees with small edge stabilizers \cite{Morgan_actions-trees-compactification-SOn1-representation_1986, Paulin_1988}.

The main novelty of this work is to consider simultaneously all cardinals $\kappa$. Observe that denoting $\omega=\Card(\N)$ the infinite countable cardinal, we have $\kappa+\omega = \max\{\kappa,\omega\}$.
A crucial difference with finite dimension is the existence of a family parametrized by $t\in (0,1]$ of exotic deformations $\exoH_t\colon\Hyp^\kappa\to\Hyp^{\kappa+\omega}$ characterized by the metric relations (where $\dH$ is the distance on a algebraic hyperbolic space):
\begin{equation*}
	\forall x,y\in\Hyp^\kappa \colon \cosh \dH(\exoH_t(x),\exoH_t(y)) = (\cosh \dH(x,y))^t.
\end{equation*}
These exotic embeddings come with exotic representations $\chi_t\colon\Isom(\Hyp^\kappa)\to\Isom(\Hyp^{\kappa+\omega})$. 
It turns out that by \cite{Monod-Py_self-representations-Mobius_2019}, all irreducible representations $\Isom(\Hyp^{\kappa}) \to \Isom(\Hyp^{\kappa'})$ are induced by isometric embeddings or exotic inclusions of $\Hyp^{\kappa}$ into $\Hyp^{\kappa'}$. 

Exotic deformations induce homotheties of length functions 
\begin{equation*}
	\forall g\in\Isom(\Hyp^\kappa) \colon  \ell(\chi_t(g))=t\ell(g)
\end{equation*}
so a representation $\rho\colon\Gamma\to\Isom(\Hyp^\kappa)$ and its exotic deformation $\chi_t\circ\rho$ have homothetic length functions. 
In particular the marked length spectrum rigidity for irreducible representations \cite[Theorem A]{Kim_marked-length-rigidity-symmetric-space_2001} breaks down for infinite dimensions, and we will see how to restore it.
Let us emphasize that we wish to consider all possible representations $\Gamma\to \Isom(\Hyp^\kappa)$, not only irreducible ones, and for that (the homothety class of) the length function is insufficient as it cannot distinguish elliptic or parabolic representations (with fixed points in $\Hyp^\kappa$ or $\partial\Hyp^\kappa$, and no hyperbolic elements), since they all have trivial length function. 

Thus, to consider representations $\Gamma \to \Isom(\Hyp^\kappa)$ up to self-representations, we will use a tool introduced by Monod and Py in \cite{Monod-Py_self-representations-Mobius_2019} called \emph{functions of hyperbolic type}.
To a representation $\rho\colon\Gamma\to\Isom(\Hyp^\kappa)$ and a point $o\in\Hyp^\kappa$, is attached the function of hyperbolic type $F_{\rho,o} \colon \Gamma \to \R$ defined by $F_{\rho,o} \colon \gamma\mapsto \cosh \dH(\rho(\gamma) x,x)$, which captures all distances between points in the orbit $\rho(\Gamma)\cdot o$.
Let us explain how such functions of hyperbolic type can be characterized with only one inequality.

On a topological space $X$, a \emph{kernel} is a continuous function $K\colon X^2 \to \R$ which is symmetric ($\forall x,y\in X \colon K(x,y)=K(y,x)$). We say that it has diagonal $s\in \R$ when $\forall z\in X \colon K(z,z)=s$. We say that it is non-negative when $\forall x,y\in X \colon K(x,y)\ge 0$.

\begin{introdefinition}[kernel of hyperbolic type. \Cref{def:kernel-hyperbolic-type}]
	A kernel $K\colon X\times X\to \R$ is of \emph{hyperbolic type} when it is nonnegative with diagonal $1$, and for all $x_0\in X$ the \emph{visual kernel} $K_0 \colon X\times X \to \R_{\ge0}$ defined by $K_0\colon (x_i,x_j)\mapsto K(x_0, x_i)K(x_0,x_j)-K(x_i,x_j)$ is of positive type, namely we have the following \emph{hyperbolic Cauchy-Schwartz inequalities}:
	\begin{equation}
		\label{introeq:kernel-hyperbolic-CS}
		\tag{HCS}
		\textstyle 
		\forall n\in \N,\;
		\forall x\in X^n,\;
		\forall \lambda \in \R^n \colon 
		\quad
		\sum_{i,j=1}^n\lambda_i\lambda_jK(x_i,x_j)\leq \left( \sum_{k=1}^n \lambda_k K(x_0,x_k)\right)^2.
	\end{equation}
\end{introdefinition}

Similarly to the GNS construction for Hilbert spaces, from such a kernel of hyperbolic type, one constructs a hyperbolic space $\Hyp^\kappa$ and a continuous map $f \colon X\to \Hyp^\kappa$ such that for all $x,y\in X$, $K(x,y)=\cosh \dH(f(x),f(y))$. 
Moreover, if the kernel is invariant under some group action $\Gamma \to \mathfrak{S}(X)$ then there is a representation $\Gamma \to\Isom(\Hyp^\kappa)$ such that $f$ is equivariant.

\begin{introremark}[Cayley-Menger derterminants]
	The idea to characterize when a metric space $(X,d)$ embeds isometrically in some $\Hyp^\kappa$ is not recent and goes back (at least) to Blumenthal's work \cite{Blumenthal_distance_geometry_1950} (first published in 1950).
	An isometric embedding exists if and only if all its hyperbolic Cayley-Menger determinants are non-positive:
	\begin{equation*}
		\forall n\in \N,\, \forall x_0,\dots,x_n\in X\colon  \det(-\cosh(d(x_i,x_j))\leq0.
	\end{equation*}
	We will explain the relation with kernels of hyperbolic type in \Cref{subsec:kernel_hyperbolic_type}.
	One advantage of Cayley-Menger determinants is to provide, when $X$ is finite, a simple characterization of the minimal dimension $\kappa$ for which an isometric embedding exists.
\end{introremark}

To encode group representations in some $\Isom(\Hyp^\kappa)$ we will use the following notion.

\begin{introdefinition}[function of hyperbolic type]
	On a topological group $\Gamma$, a function $F\colon \Gamma \to\R$ is of \emph{hyperbolic type} when the kernel $\Gamma \times \Gamma \ni (\alpha,\beta) \mapsto F(\alpha^{-1}\beta) \in \R$ is of hyperbolic type.
	
	We denote by $\mathcal{F}(\Gamma)$ the space of functions of hyperbolic type on $\Gamma$ endowed with the pointwise convergence topology. 
	When $\Gamma$ is discrete, it is a closed subspace of the product $\R_{\ge 1}^\Gamma$.
\end{introdefinition}

For a continuous representation $\rho\colon \Gamma \to \Isom(\Hyp^\kappa)$ and $o\in\Hyp^\kappa$, the function
\begin{equation*}
	F_{\rho,o}\colon \Gamma\to\R_{\ge 1}
	\qquad
	F_{\rho,o}\colon \gamma \to \cosh \dH(\rho(\gamma)o,o)
\end{equation*} is of hyperbolic type. 
Conversely, for every function of hyperbolic type $F\colon \Gamma \to \R$, there is a continuous representation $\rho\colon \Gamma \to \Isom(\Hyp^\kappa)$ and a point $o\in\Hyp^\kappa$ such that $F=F_{\rho,o}$.
Note that one may recover the length function $\ell_F \colon \Gamma \to \R_{\ge 0}$ of any associated representation by the formula $\ell_F(\gamma)= \lim_{+\infty} \tfrac{1}{n}\log F(\gamma^n)$.

For two representations $\rho,\rho'$ that are conjugated by $\Isom(\Hyp^\kappa)$ and two points $o,o'\in \Hyp^\kappa$, the associated functions of hyperbolic type are \emph{comparable}:
$\| \log(F_{\rho,o})-\log(F_{\rho',o'})\|_\infty<\infty$.
Moreover for $t\in (0,1]$ the exotic deformation and representation $(\chi_t, \exoH_t)$ exponentiates functions of hyperbolic type: $F_{\chi_t\circ\rho, \exoH_t(o)}=F_{\rho,o}^t$.
To encompass all self-representations of $\Isom(\Hyp^\kappa)$, we define two functions of hyperbolic type $F,F'\colon \Gamma \to \R$ as \emph{homothetic} when there exists $t>0$ such that $F^t, F'$ are comparable, namely $\log(F^t/F')$ is bounded on $\Gamma$. 
We thus define the \emph{variety of actions} $\Proj\mathcal{C}(\Gamma)$ as the topological quotient space of $\mathcal{F}(\Gamma)$ by this homothety equivalence. 

The space $\Proj\mathcal{C}(\Gamma)$ has some pathological properties. For example, the class of the function $\mathbf{1}$ constant to $1$ lies in the closure of any other class. 
But classes of representation can develop other (less dramatic) degeneracies leading to singular points of the character variety (as in the classical Teichm\"uller theory) and this leads to considering the following subsets of representations (for instance analogue to quasi-Fuchsian loci).

A representation $\Gamma \to\Isom(\Hyp^\kappa)$ (or a function of hyperbolic type) is called \emph{neutral} when the length function is trivial and \emph{elementary} when there is an invariant non-empty subset with at most $2$ points in the boundary $\partial\Hyp^\kappa$ or a fixed point in $\Hyp^\kappa$. The notions of non-neutrality and non-elementarity are preserved by homothety of functions of hyperbolic type (\Cref{cor:homothetic_length_spectrum}), and we denote the corresponding subspaces of $\Proj\mathcal{C}(\Gamma)$ by $\Proj\mathcal{C}_{nn}(\Gamma)\supset\Proj\mathcal{C}_{ne}(\Gamma)$.

If two functions of hyperbolic type are comparable or homothetic then they have equal or homothetic length functions, so we obtain a function $\Proj\ell \colon \Proj\mathcal{C}(\Gamma)\to \Proj(\R_{\ge 0}^\Gamma)$.
A function of hyperbolic type which is not trivial (not in the class of the constant function $\mathbf{1}$) may still have trivial length function which is trivial (equal to $0$). 

Our main Theorems \ref{thm:length-non-neutral-classes} and \ref{thm:topo-PC-PCnn-PCne} on the topology of $\Proj\mathcal{C}(\Gamma)\supset \Proj\mathcal{C}_{nn}(\Gamma)\supset \Proj\mathcal{C}_{ne}(\Gamma)$ and of the projectivized length function $\ell \colon \Proj\mathcal{C}(\Gamma)\to \Proj(\R_{\ge 0}^\Gamma)$ are the following.

\begin{introtheorem}[length function on $\Proj\mathcal{C}_{nn}(\Gamma)\supset \Proj\mathcal{C}_{ne}(\Gamma)$]
	Let $\Gamma$ be a topological group. 
	
	The length function $\ell \colon \mathcal{F}(\Gamma) \to \R^\Gamma$ defined by $F \mapsto \ell_F$ quotients to a function $\ell \colon \Proj\mathcal{C}(\Gamma)\to\Proj\R^\Gamma$ which is continuous and injective in restriction to $\Proj\mathcal{C}_{nn}(\Gamma)$; in particular the sets $\Proj\mathcal{C}_{nn}(\Gamma)$ and $\Proj\mathcal{C}_{ne}(\Gamma)$ are Hausdorff and open in $\Proj\mathcal{C}(\Gamma)$.
\end{introtheorem}

The following statement shows how the topology of $\Proj\mathcal{C}(\Gamma), \Proj\mathcal{C}_{nn}(G)$ and $\Proj\mathcal{C}_{ne}(G)$ behaves when $\Gamma$ is finitely generated. The subspaces $\mathcal{A}_{nn}(S; 2,\varepsilon)$ and $\mathcal{A}_{nn}(S; 2,\varepsilon)$ of $\mathcal{F}$, defined in \Cref{subsec:topology_PC}, are parametrized by a generating set $S$ and a positive number $\varepsilon$.

\begin{introtheorem}[topology of $\Proj\mathcal{C}(\Gamma)\supset \Proj\mathcal{C}_{nn}(G) \supset \Proj\mathcal{C}_{ne}(G)$]
	Let $\Gamma$ be a finitely generated group. 
	\begin{enumerate}[noitemsep]
		\item The variety of actions $\Proj\mathcal{C}(\Gamma)$ is compact.
		\item 
		The subspace $\Proj\mathcal{C}_{nn}(\Gamma)$ is a filtering union of compact metrizable subspaces $\mathcal{A}_{nn}(S; 2,\varepsilon)$ so the restriction $\Proj\ell \colon \mathcal{A}_{nn}(S; 2,\varepsilon)\to\Proj\R^\Gamma\setminus\{0\}$ is a homeomorphism onto its image.
		
		\item
		The subspace $\Proj\mathcal{C}_{ne}(\Gamma)$ is a filtering union of compact metrizable subspaces $\mathcal{A}_{ne}(S; 2,\varepsilon)$ so the restriction $\Proj\ell \colon \mathcal{A}_{ne}(S; 2,\varepsilon)\to\Proj\R^\Gamma\setminus\{0\}$ is a homeomorphism onto its image.
	\end{enumerate}
\end{introtheorem}

The injectivity of $\Proj \ell\colon \Proj\mathcal{C}_{nn}(\Gamma) \to \R_{\ge 0}$ means that the homothety class of the length function determines the class of the function of hyperbolic type: this a length spectrum rigidity which we obtained using the methods developed in the next subsection. 

\subsection*{Cross-ratios and length spectrum rigidity}
\addcontentsline{toc}{subsection}{Cross-ratios and length spectrum rigidity}

We will be interested in character varieties of groups admitting faithful actions on certain hyperbolic spaces of non-positive curvature: such actions are often better understood on the boundary at infinity, where the cross-ratio becomes an important tool encoding the asymptotic properties of the metric.

Fix a metric space $(X,d)$, which may not be proper or complete.
At a base point $o\in X$ are defined the \emph{Gromov product} $(x \mid y)_o = \tfrac{1}{2}\left(d(o,x)+d(o,y)-d(x,y)\right)$ and the \emph{visual metametric} $d_o(x, y)=\exp\left(-(x\mid y)_o\right)$. 
Hence the \emph{cross-ratio} of $(x,x'),(y,y')\in X^2\times X^2$ is defined (for any $o\in X$) as the positive real number:
\begin{equation*}
	\Cr{x,x';y,y'}
	= \frac{d_o(x,x')\times d_o(y,y')}{d_o(x,y') \times d_o(y,x')}
	= \exp \tfrac{1}{2}\left(d(x,x')+d(y,y')-d(x,y')-d(y,x')\right)
\end{equation*}
The cross-ratio satisfies the \eqref{eq:Cr-Invariance}, \eqref{eq:Cr-Inversion} and \eqref{eq:Cr-Cocycle} identities (see \Cref{subsec:strong-hyp_cross-ratio}).

For $\varepsilon>0$, $(X,d)$ is strongly-$\varepsilon$-hyperbolic when every tetrahedron $(x,x',y,y') \in X^4$ satisfies $1\le \Cr{x,x';y,y'}^{-\varepsilon}+ \Cr{x,x';y',y}^{-\varepsilon}$, that is the Ptolemy inequality for $d_o^{\varepsilon}$:
\begin{equation*}
	d_o^{\varepsilon}(x,x') d_o^{\varepsilon}(y, y')
	\le 
	d_o^{\varepsilon}(x,y') d_o^{\varepsilon}(y, x')
	+ 
	d_o^{\varepsilon}(x,y) d_o^{\varepsilon}(x', y')
\end{equation*}
This implies Gromov-$\delta$-hyperbolicity for $\delta = \log(2)/\varepsilon$; indeed, it is an exponential strengthening of the Gromov-hyperbolicity, see \Cref{rem:NS-strongly-hyperbolic-characterisation}).
The class of strongly-$1$-hyperbolic spaces contains the class of $\operatorname{CAT}(-1)$-spaces, hence of trees and algebraic hyperbolic spaces of any dimension.

The \Cref{sec:strong_hyperbolicity} will provide background (and new results) about strongly hyperbolic spaces and their isometry groups. 
Our main motivation for this setting stems from the fact that strong-hyperbolicity ensures a continuous extension of the cross-ratio to subsets of quadruples in the bordification $(X\cup \partial X)$ of which at least $3$ are distinct, which further enjoys separation properties in restriction to such quadruples in the boundary $\partial X$ (see \Cref{eq:Cr=0} and \Cref{eq:Cr=infty}).

\begin{introremark}[coarse strong hyperbolicity]
	We will see that to have this continuous extension property, it suffices to assume that $X$ is coarsely strongly hyperbolic, a notion equivalent to that of asymptotic $PT(-1)$-space introduced in \cite{Miao-Schroeder_Hyperbolic-Ptolemey-Mobius_2012} in relation to the Ptolemaic inequality.    
\end{introremark}

\begin{introremark}[why strongly-hyperbolic?]
	Even though all strongly hyperbolic spaces for the applications in this paper will be $\operatorname{CAT}(-1)$, this level of generality is natural because of the following remark and question. 
	
	All ``exotic'' embeddings metric spaces  (such as trees \cite{Burger-Iozzi-Monod_embedding-trees-hyperbolic-spaces_2005} or algebraic hyperbolic spaces \cite[Proposition 8.2]{Monod_notes-functions-hyperbolic-groups_2020}) into $\Hyp^\kappa$ derive from the fact that for some $\lambda>1$ the function $\lambda^d$ is a kernel of hyperbolic type on $X$. 
	Thus we may wonder: for which metric spaces $(X,d)$ is there some $\lambda>1$ for which $\lambda^d$ is a kernel of hyperbolic type on $X$? 
	
	Our \Cref{prop:coarsely-strongly-hyperbolic} will show that such a space must be coarsely strongly hyperbolic.
\end{introremark}

For a topological space $Y$, denote $Y^{(4)}\subset Y^4$ the quadruples of which at least $3$ are distinct. 
We will define the notion of \emph{abstract cross-ratio} on a topological space $Y$ as a continuous function $B\colon Y^{(4)}\to [0,\infty]$ satisfying the \eqref{eq:Cr-Invariance}, \eqref{eq:Cr-Inversion} and \eqref{eq:Cr-Cocycle} relations as well as the characterizations \eqref{eq:Cr=0} and \eqref{eq:Cr=infty} of the fibres above $0$ and $\infty$.
We say that an abstract cross-ratio \emph{generates the topology} when for every triple of distinct points $x,y,z \in Y$, the collection of subsets $\{w\in Y\colon B(w,x,y,z)<\varepsilon\}$ for $\varepsilon>0$ is a basis of neighbourhoods of $x$.

For example, the cross-ratio in a strongly hyperbolic space is an abstract cross-ratio that generates the topology.
We will show the following \Cref{prop:power_cross_ratio} and \Cref{cor:power_cross_ratio}.

\begin{introproposition}[abstract cross-ratios on $3$-full strongly hyperbolic spaces]\label{introprop:abstract_cross_ratio}
	Consider a strongly hyperbolic space $(X,d)$ and a subgroup $\Gamma<\Isom(X)$ acting $3$-transitively on $\partial X$.
	If an abstract cross-ratio $B$ on $\partial X$ generating the topology is $\Gamma$-invariant, then there exists $t>0$ such that for all $(w,x,y,z) \in (\partial X)^{(4)}$ we have  $B(w,x,y,z)=\Cr{w,x,y,z}^t$.
	
	Consequently if an embedding of strongly hyperbolic spaces $f\colon \partial X_1\to \partial X_2$ is a homeomorphism on its image, and is equivariant under the action of a subgroup $\Gamma\subset \Isom(X_1)$ acting $3$-transitively on $\partial X_1$, then there is $t>0$ such that for all $(w,x,y,z)\in (\partial X_1)^{(4)}$  we have
	\begin{equation*}
		\Cr{f(w),f(x),f(y),f(z)}_{\partial{X_2}}=\Cr{w,x,y,z}_{\partial X_1}^t
	\end{equation*}
\end{introproposition}

Such notions of abstract cross-ratios have been used to show many rigidity results for representations (see for example \cite{Otal_geometrie-symplectique-geodesiques_1992, Bourdon_cross-ratio-CAT(-1)_1996, Labourie_Cross-ratios-Anosov-rep-energy_2008}).

Here, we will characterize abstract cross-ratios coming from an embedding into some $\Hyp^\kappa$, which we call \emph{algebraic cross-ratios}. 
Let us first mention that bijections of $\partial \Hyp^\omega$ preserving the cross-ratio are induced by a unique isometry of $X$, namely the M\"obius groups of $\partial\Hyp^\kappa$ coincides with the isometry group of $\Hyp^\kappa$.
Next, observe that in the upper half-space model for $\Hyp^\kappa$, for any three distinct boundary points $0,1,\infty\in\partial\Hyp^\kappa$, there is a Hilbert space structure on $\partial \Hyp^\kappa\setminus\{\infty\}$ such that for all $x,y\neq\infty$:
\begin{equation*}
	\Cr{x,y,0,\infty}\Cr{y,0,1,\infty}=\tfrac{\|y-x\|}{\|0-1\|}
\end{equation*}
Moreover isometries of $\Hyp^\kappa$ fixing $\infty$ acts by similarities on $\partial \Hyp^\kappa\setminus\{\infty\}$.
This suggests the following \Cref{def:algebraic-cross-ratio}.

\begin{introdefinition}[algebraic cross-ratio]
	\label{introdef:algebraic-cross-ratio}
	On a topological space $Y$, we define an \emph{algebraic cross-ratio} on $Y$ 
	as an abstract cross-ratio
such that for any triple of distinct points denoted $0,1,\infty \in Y$ the function on $(Y\setminus \{\infty\})^2 \to \R_{\ge 0}$ defined by 
\begin{equation*}
	(x,y) \mapsto \left(R(x,y,0,\infty)R(y,0,1,\infty)\right)^2
\end{equation*}
is a kernel conditionally of negative type on $Y\setminus \{\infty\}$.
\end{introdefinition}

For a topological space $Y$, the classical GNS construction for kernel conditionally of negative type $K\colon Y^2\to\R$ gives the existence of a Hilbert space $\R^\kappa$ and a continuous map $f\colon Y\to\R^\kappa$ such that $K(x,y)=\|f(x)-f(y)\|^2$. Moreover, if there is a group $\Gamma$ acting on $Y$ preserving $K$ then the construction also yields a representation $\rho\colon G\to\Isom(\R^\kappa)$ such that $f$ is $\rho$-equivariant. 
The following \Cref{thm:algebraic-cross-ratio} is an analogous statement for algebraic cross-ratios. Recall that a subset of Hilbert space is total when it is not included in a strict closed affine subspace.

\begin{introtheorem}[GNS for algebraic cross-ratios]
\label{introthm:algebraic-cross-ratio}
Consider a topological space $Y$ having at least four points with an algebraic cross-ratio $R\colon Y^{(4)} \to [0,+\infty)$.

There is a cardinal $\kappa$ and a continuous function $f\colon Y \to \partial \Hyp^{\kappa}$ with total image such that for all $(x^-,x^+,y^-, y^+) \in Y^{(4)}$ we have $\Cr{f(x^-),f(x^+);f(y^-),f(y^+)}=R(x^-, x^+, y^-, y^+)$.

Moreover, for any two such maps $f_1\colon Y \to \partial \Hyp^{\kappa_1}$ and $f_2\colon Y \to\partial \Hyp^{\kappa_2}$, there exists a unique isometry $\varphi\colon\Hyp^{\kappa_1}\to\Hyp^{\kappa_2}$ such that $f_2=\varphi \circ f_1$, in particular $\kappa_1 = \kappa_2$.

Consequently, if a group $\Gamma$ acts on $Y$ preserving $R$ then there is a representation $\rho\colon \Gamma\to \Isom(\Hyp^\kappa)$ such that $f$ is $\rho$-equivariant.
\end{introtheorem}

We finally turn to results of (global) marked length spectrum rigidity, namely distinguishing representations $\rho\colon \Gamma\to\Isom(X,d)$ from their length function $\ell \colon \Gamma \to \R_{\ge 0}$. 
The following \Cref{prop:same-length-spectrum-means-conjugate} provides rigidity of the marked length spectrum, generalizing the finite dimensional case \cite{Kim_marked-length-rigidity-symmetric-space_2001}. A representation $\rho\colon \Gamma \to\Isom(\Hyp^\kappa)$ is \emph{minimal} when there is no non-trivial closed invariant totally geodesic subspaces $\emptyset \subsetneq \Hyp^{\kappa'} \subsetneq \Hyp^\kappa$).

\begin{introproposition}[rigidity of marked length spectrum]\label{introprop:rigidity_length_spectrum}
Consider a group $\Gamma$ with minimal representations $\rho_i\colon G\to \Isom(\Hyp^{\kappa_i})$ for $i=1,2$. 
If $\rho_1,\rho_2$ have non-trivial and equal length functions $\ell_{\rho_1}=\ell_{\rho_2}$, then $\kappa_1=\kappa_2$ and $\rho_1,\rho_2$ are conjugated by an isometry $\Hyp^{\kappa_1}\to \Hyp^{\kappa_2}$.
\end{introproposition}

In finite dimension $n$, \cite[Theorem A]{Kim_marked-length-rigidity-symmetric-space_2001} proves that two Zariski-dense representations $\rho_i\colon \Gamma \to\Isom(\Hyp^n)$ for $i=1,2$ with homothetic length functions (for which there is $\lambda>0$ such that $\ell_{\rho_1} =\lambda \ell_{\rho_2}$) must be conjugated. 
In infinite dimensions this statement fails, since for every representation $\rho\colon \Gamma \to\Isom(\Hyp^\kappa)$ with non-trivial length function, for every $t\in (0,1)$ the exotic deformations $\chi_t\circ\rho$ has homothetic length function but is not conjugated to $\rho$. The following \Cref{cor:homothetic_length_spectrum} provides the appropriate generalization to arbitrary dimensions characterizing representations and functions of hyperbolic type corresponding to homothetic marked length spectrum.

\begin{introcorollary}[homothetic marked length spectra]
Consider functions of hyperbolic type $F_i\colon \Gamma\to \R_{\ge 1}$ yielding non-zero length functions $\ell_i \colon \Gamma\to \R_{\ge 0}$, and any minimal representations $\rho_i\colon \Gamma \to \Isom(\Hyp^{\kappa_i})$ realizing $F_i$. 
For every $t\in (0,1]$ the following are equivalent:
\begin{itemize}[noitemsep]
	\item[$F$:] The $F_i$ are \emph{$t$-equivalent}, namely $\log(F_1/F_2^t)$ is bounded.
	\item[$\ell_F$:] The $\ell_i$ are \emph{$t$-homothetic}, namely $\ell_1=t\ell_2$.
	\item[$\rho_F$:] The representation $\rho_1$ is conjugate to the $t$-deformed representation $\chi_t \circ \rho_2$.
\end{itemize}
\end{introcorollary}

\subsection*{Geometric convergence and comparison with previous results}
\addcontentsline{toc}{subsection}{Geometric convergence and comparison with previous results}

Fix a finitely generated torsion-free group  $\Gamma$. 
In our compact space $\Proj\mathcal{C}(\Gamma)$, we may consider the closure of various subsets of classes of representations to obtain compactifications.
For example, for each finite cardinal $n\in \N_{\ge 2}$, we may consider the subset $\Proj \mathcal{C}^{n}_{fd}(\Gamma)$ of homothety classes associated to faithful and discrete representations $\Gamma\to \Isom(\Hyp^n)$.

We should compare this compactification with various previously known compactifications obtained for instance in \cite{Morgan-Shalen_valuations-trees-degernations-hyperbolic_1984, Morgan_actions-trees-compactification-SOn1-representation_1986, Bestvina_degenerations-hyperbolic-space_1988,Paulin_1988} of certain character varieties by certain actions on real trees.
More precisely, Paulin introduced the equivariant Gromov-Hausdorff topology on spaces of actions to show the following.
Let $\mathcal{H}^n_{fd}(\Gamma)$ be the Gromov-Hausdorff space of faithful and discrete $\Gamma$-actions on $\Hyp^n$ and $\Proj\mathcal{TS}(\Gamma)$ the quotient by homothety of the Gromov-Hausdorff space of $\Gamma$-actions on real trees having small edge stabilizers (stabilizers of non-trivial geodesic segments do not contain free subgroups of rank $\ge 2$).
Note that for $n_1,n_2\in \N$, if $n_1<n_2$ then $\mathcal{H}^{n_1}_{fd} \subset \mathcal{H}^{n_2}_{fd}$.
Paulin proved \cite{Paulin_1988} that if $\Gamma$ is a finitely generated torsion-free group containing a free group of rank $2$, then $\mathcal{H}^n_{fd}(\Gamma)$ is compactified by a closed subspace of $\Proj\mathcal{TS}(\Gamma)$. When $\Gamma$ is the fundamental group of a closed surface of genus $\ge 2$, it follows from \cite[Theorem 3.2]{Skora_splittings_1996} that for all $n\ge 2$, the closure of $\mathcal{H}^n_{fd}(\Gamma)$ in this compactification is all of $\Proj\mathcal{TS}(\Gamma)$, and this coincides with Thurston's compactification of the Teichm\"uller space by measured laminations.

\Cref{eg:strict-inclusions-closure-PCn-fd}
exhibits many finitely generated groups $\Gamma$ for which the inclusion $\bigcup_n \overline{\Proj\mathcal{C}^n_{fd}(\Gamma)} \subset  \bigcup_n \Proj\mathcal{C}^n_{fd}(\Gamma) \sqcup \Proj\mathcal{TS}(\Gamma) $ is strict.

Our \Cref{thm:identification_compactification} confirms that our compactification recovers that of Paulin, namely that convergence for the equivariant Gromov-Hausdorff topology is equivalent to convergence of homothety classes of length functions.

\begin{introtheorem}[recovering compactification of faithful and discrete representations]
Let $\Gamma$ be a finitely generated torsion-free group containing a free group of rank $2$. 
The Gromov-Hausdorff space $\Proj\mathcal{H}_{fd}^n(\Gamma)\sqcup \Proj\mathcal{TS}(\Gamma)$ is homeomorphic to its image $\Proj\mathcal{C}_{fd}^n(\Gamma)\sqcup \Proj\mathcal{TS}(\Gamma)$ in $\Proj\mathcal{C}(\Gamma)$.
\end{introtheorem}

Our compact space $\Proj\mathcal{C}(\Gamma)$ may contain more actions on real trees than $\Proj\mathcal{TS}(\Gamma)$, and those may be limits of actions on $\Hyp^{\kappa_m}$ for varying $\kappa_m\in \N\sqcup\{\omega\}$ as we now explain.

On a real tree $(T,d)$, the kernel $\lambda^d$ is of hyperbolic type, thus every $\Gamma$-action on a real tree yields a homothety class of functions of hyperbolic type: those define a subset $\Proj \mathcal{T}(\Gamma) \subset \Proj \mathcal{C}(\Gamma)$.
Note that $\Proj \mathcal{TS}(\Gamma)$ is closed in $\Proj \mathcal{T}(\Gamma)$, so points in $\Proj \mathcal{T}(\Gamma)\setminus \Proj \mathcal{TS}(\Gamma)$ correspond to new classes of representations, that are only encompassed by our compactification.
Examples of $\Gamma$ for which $\Proj \mathcal{T}(\Gamma)\setminus \Proj \mathcal{TS}(\Gamma)\ne \emptyset$ include free groups on $\ge 2$ generators and automorphism groups of real trees $(T,d)$ such that $\partial T$ is infinite.
The following \Cref{thm:geometric-convergence-tree} shows that the convergence of a rescaled hyperbolic space to a real tree can be obtained equivariantly inside $\Hyp^\omega$.

\begin{introtheorem}[geometric convergence to actions on trees]
\label{thm:geometric-convergence-tree-intro}

Consider a group $\Gamma$ finitely generated by $S$, and a sequence of representations $\rho_m\colon \Gamma\to\Isom(\Hyp^{\kappa_m})$ for cardinals $\kappa_m\in \N\cup\{\omega\}$.

If no $\rho_m$ has a fixed point at infinity, but the infimum displacement length over $S$ is unbounded as $m\to \infty$, then up to extracting a subsequence of $(\rho_m)$ there are:
\begin{enumerate}[noitemsep]
	\item a real tree $(T,d)$ and an isometric action $\rho_T\colon\Gamma\to\Isom(T)$
	\item an embedding $\exoT_\lambda\colon T\to\Hyp^\omega$ equivariant for $\pi_\lambda\colon \Isom(T)\to\Isom(\Hyp^\omega)$ with $\lambda=e$,
	\item an isometric action $\rho_\omega \colon \Gamma \to \Isom(\Hyp^\omega)$ preserving $\exoT_\lambda(T)$,
	\item a sequence of representations $\rho_m'\colon\Gamma\to\Isom(\Hyp^\omega)$ such that $F_{\rho_m'}\sim F_{\rho_m}$,
	\item a sequence of $(\rho_m,\rho_m')$-equivariant exotic embeddings $\exoH_{t_m}\colon\Hyp^{\kappa_m}\to\Hyp^\omega$.
\end{enumerate}
such that $(\rho'_m)_m \colon \Gamma \to \Isom(\Hyp^\omega)$ converges to $\rho_\omega \colon \Gamma \to \Isom(\Hyp^\omega)$ and $\rho_\omega=\pi_\lambda\circ\rho_T$.

Moreover, there exist points $o_m\in \operatorname{Core}(\rho_m)$ and $o\in T$ such that for every finite set $K\subset \Gamma$, we have Hausdorff convergence of $\exoH_{t_m}(\Conv(\rho_m(K) \cdot o_m))$ to $\exoT_\lambda(\Conv(\rho_\omega(K) \cdot o))$.
\end{introtheorem}

\begin{introremark}[the use of Gromov-Hausdorff convergence]
A similar result has been obtained by Courtois--Guilloux \cite{Courtois-Guilloux_Hausdorff-infinite-dim-hyp_2024}, also using the rescaling method by exotic deformations representations $\chi_t$, in order to obtain continuity results for the Hausdorff dimension of limit sets.
\end{introremark}

In our statement above the sequence of dimensions can be arbitrary (of various unbounded cardinals), so to obtain our Gromov-Hausdorff convergence we need a metric estimate for the convex hull of finitely many points which is independent on the dimension, which we proved in  \cite{Duchesne-Simon_simplices_2025} and recall here. 

\begin{introproposition}[Hausdorff distance from convex hull to $1$-skeleton]
For a set of points $X \subset \Hyp^\kappa\cup \partial \Hyp^\kappa$, every point of its convex hull $\Conv(X)$ is at distance at most $\sinh^{-1}(1)=\log(1+\sqrt{2})$ from its $1$-skeleton $X^{(1)}=\bigcup_{x_1,x_2 \in Y}[x_1, x_2]$.
\end{introproposition}

Our \Cref{thm:geometric-convergence-tree} constructs the limiting representation, the action on the limiting tree, and also describes the geometric convergence towards this limit. The \Cref{sec:ultraproducts} will show that this action on the tree realized as a subspace of $\Hyp^\omega$ can be constructed immediately from ultraproducts of hyperbolic spaces $\Hyp^{\kappa_m}$. 

\Cref{eg:bendings} shows that for many groups (which split as an amalgam over $\Z$), the space $\Proj \mathcal{C}(\Gamma)$ may be much larger than $\bigcup_n \Proj \mathcal{C}^n(\Gamma) \sqcup \Proj \mathcal{T}(\Gamma)$.

\subsection*{Groups with a unique homothety class of actions on \texorpdfstring{$\Hyp^\kappa$}{Hkappa}}
\addcontentsline{toc}{subsection}{Groups with a unique homothety class of actions on \texorpdfstring{$\Hyp^\kappa$}{Hkappa}}

In the final part of this paper, we prove that some groups of special interests have exactly one non-elementary action on some $\Hyp^\kappa$ up to the equivalence introduced above.

\begin{introtheorem}[\Cref{cor:point_character_variety}]

If $G$ is one the following topological groups:
\begin{enumerate}
	\item the isometry group $\Isom(\Hyp^\kappa)$ of the real hyperbolic space $\Hyp^\kappa$ of countable dimension $\kappa$,       
	\item the isometry group $\Aut(T_\kappa)$ of the $\kappa$-regular simplicial tree $T_\kappa$ where $\kappa\ge 3$ is countable,
	\item the group $\PGL_2(\K)$ where $\K$ is a complete non-Archimedean valued field
\end{enumerate}
then the space $\Proj\mathcal{C}_{ne}(G)$ has exactly one point. 
\end{introtheorem}

Some cases of this statement were already known, for example the first item \cite{Monod-Py_self-representations-Mobius_2019} or the two other items when the group is locally compact \cite{Burger-Iozzi-Monod_embedding-trees-hyperbolic-spaces_2005}. We group them in one result since we derive (the unicity) from a more general result: the main point is that these groups admit an action on a $\operatorname{CAT}(-1)$ space $(X,d)$ (a hyperbolic space $\Hyp^\kappa$ or a real tree), that is sufficiently rich to reflect the structure of the group. More precisely, we will prove the following \Cref{{thm:point_character_variety}}, where boundedness properties and $3$-full actions are defined in \Cref{sec:rigidity}.

\begin{introtheorem}[rigidity of groups acting fully on $\operatorname{CAT}(-1)$ spaces]
\label{thm:point_character_variety_intro}
Consider a topological group $G$ with a $3$-full action on a $\operatorname{CAT}(-1)$-space $(X,d)$.
If every non-elementary continuous action  of $G$ on $\Hyp^\kappa$ is $\EE_d$-bornologous, then $\Proj\mathcal{C}_{ne}(G)$ has at most one point.  
\end{introtheorem}

\begin{introremark}[no local compactness, no Gelfand pairs]
The main novelty in this statement is the fact that $G$ is not necessarily locally compact. In particular, the notion of Gelfand pairs, on which previous proofs rely, is not available in this context. Thus we use a different strategy relying on our unicity of cross-ratios up to powers (\Cref{introprop:abstract_cross_ratio}) and length spectrum rigidity (\Cref{introprop:rigidity_length_spectrum}). 
\end{introremark}

\begin{proof}[Ideas in the proof of \Cref{thm:point_character_variety_intro}]

We first construct a $G$-equivariant boundary map $\partial X\to\partial \Hyp^\kappa$ that is a homeomorphism on its image. Using this map, we can pullback the cross-ratio from $\partial \Hyp^\kappa$ to $\partial X$. We then show bi-continuity of this map (which would be automatic in a locally compact setting) using boundedness properties. We conclude the proof by using uniqueness of the cross-ratio up to powers on $\partial X$.
\end{proof}

\subsection*{Further results, Some questions, Relations to other works}
\addcontentsline{toc}{subsection}{Further results, Some questions, Relations to other works}

\subsubsection*{Geometry of $\Proj \mathcal{C}(\Gamma)$: dimension grading, connectivity, metric}

\begin{introremark}[critical exponent and dimension graduation of $\Proj \mathcal{C}(\Gamma)$]
For a function of hyperbolic type $F\in \mathcal{F}(\Gamma)$, we define its critical exponent as the supremum of $t>0$ such that $F^t$ is of hyperbolic type (\Cref{def:critical-exponent}).
We show that the critical exponent is upper-semi-continuous on $\mathcal{F}(\Gamma)$ (\Cref{prop:semicontinuity_critical_exponent}).
We characterize actions on real trees as those with infinite critical exponent (\Cref{prop:characterisation_kernel_hyperbolic_type_tree}).

For homothety classes of functions of hyperbolic type with finite critical exponent, we obtain a preferred lift $\Proj \mathcal{C}(\Gamma)\to \mathcal{C}(\Gamma)$ by choosing those with critical exponent $1$: this yields a critical dimension $\kappa$ and critical representation $\rho \colon \Gamma \to \Isom(\Hyp^\kappa)$.
We show that where the critical exponent is finite, the critical dimension is upper-semi-continuous on $\Proj \mathcal{C}(\Gamma)$ (\Cref{prop:semicontinuity_critical_dimension}).

Thus, the critical dimension yields a graduation on $\Proj\mathcal{C}(\Gamma)$, whose infinite locus is the boundary at infinity consisting of actions on trees.
\end{introremark}

\begin{introquestion}[connectedness and bendings]
For which groups $\Gamma$ is $\Proj\mathcal{C}_{ne}(\Gamma)$ or $\Proj\mathcal{C}_{fd}(\Gamma)$ connected? When so, does there exist $n\in \N$ such that the corresponding space $\mathcal{C}^n_{ne}(\Gamma)$ or $\mathcal{C}^n_{fd}(\Gamma)$ is connected? 
Can we find obstructions to such connectedness in terms of the (bounded) homology invariants or the homotopy type of the group $\Gamma$?

On a topological space, the kernels of positive type form a convex $\R_{\ge 0}$ cone stable by product, and kernels of conditionally negative type form a convex $\R_{\ge 0}$-cone).
In a similar vein, we wonder whether kernels of hyperbolic type are stable under certain weighted averages (which only depend on the kernels and not on the choice of points in the space).

For $\Gamma=\pi_1(\Sigma_{g})$ the fundamental group of a closed orientable surface of genus $g\ge 2$, \Cref{eg:bendings} shows how the bending procedure enables to move between distinct classes of representations, yielding a family of distinct points in $\Proj \mathcal{C}(\Gamma)$ parametrized by an infinite dimensional Lie group.
This raises the question:
If two points in $\Proj \mathcal{C}_{ne}(\Gamma)$ or $\Proj \mathcal{C}_{fd}(\Gamma)$ are connected, can they be connected through bendings and limits?
Let us note that $\Proj \mathcal{H}^2_{fd}(\Gamma)$ is connected through bendings, since any two points in Teichm\"uller space are connected through a sequence of Goldman twists \cite{Goldman_Invariant-functions-Lie-groups_1986}, which one may write explicitly using their Fenchel-Nielsen coordinates, hence the compactification $\Proj \mathcal{H}^2_{fd}(\Gamma) \sqcup \Proj\mathcal{TS}(\Gamma)$ is connectd through bendings and limits.
This question could be generalized to $\Gamma$ the fundamental groups of a graph-of-groups (in the sense of Bass-Serre) such that all edges with a non-trivial group separate the graph. 
\end{introquestion}

\begin{introquestion}[type structures and metrics on strata]
On a group $\Gamma$, define a \emph{type structure} as a partition of its conjugacy classes into $3$ types (elliptic, parabolic, loxodromic) coming from an action on a Gromov-hyperbolic metric space.
Many groups have natural type structures: hyperbolic groups acting on their Cayley-Graphs, mapping class groups of surfaces acting on curve complexes (see \Cref{introquest:Mod(S)}), or the Cremona group acting on $\Hyp^\omega$ (see \Cref{eg:Cremona_cross-ratio}). 

Each type structure $\iota$ on $\Gamma$ yields a \emph{stratum} $\Proj\mathcal{C}_{\iota}(\Gamma)\subset \Proj\mathcal{C}(\Gamma)$ of type-preserving classes, and this yields a partition into all possible types $\Proj\mathcal{C}(\Gamma) = \sqcup_{\iota} \Proj\mathcal{C}(\Gamma)$.
One may wonder for a group $\Gamma$, how many strata, how do they wrap around one another, what is the topology of a special stratum?
Such questions on the size and topology of the partition into types $\Proj \mathcal{C}(\Gamma)=\sqcup_\iota \Proj \mathcal{C}_{\iota}(\Gamma)$ would shed brighter light onto rigidity properties.

Finally, one may define a metric on each stratum $D\colon \Proj\mathcal{C}_{\iota}(\Gamma) \to [0,+\infty]$, inspired by \cite{Thurston_minimal-stretch-maps_1986} and \cite{Reyes_space-metric-structures-hyperbolic-groups_2023}, as follows.
For $\class{F_1},\class{F_2} \in \Proj\mathcal{C}_{\iota}(\Gamma)$, consider the infimum over representatives $F_1, F_2 \in \mathcal{F}_{\iota}(\Gamma)$ of
\begin{equation*}
	D(F_1,F_2)= \inf \left\{ \log \lvert \ell_{F_1}(\gamma)/\ell_{F_2}(\gamma) \rvert \colon \gamma\in \Gamma \; \text{is $\iota$-loxodromic}\right\}
\end{equation*}
One may study the properties of these metrics on strata, together with the geometry and dynamics of the isometric actions by $\operatorname{Out}(\Gamma)$.
\end{introquestion}

\subsubsection*{Geometric actions on strongly hyperbolic and $\operatorname{CAT}(-\varepsilon)$ spaces}

\begin{introquestion}[Strongly hyperbolic \& $\operatorname{CAT}(0)\implies \operatorname{CAT}(-1)$]
Gromov's Jungentraum (\cite[p. 193]{Gromov_Asymptotic-invariants-infinite-groups_1993} is to show that every (finitely generated) hyperbolic group $\Gamma$ admits geometric action (proper cocompact) on a $\operatorname{CAT}(-1)$ space.
While this is wide open, it is expected to fail, but no counterexamples are known.
A weaker but equally open request is to ask whether it admits a $\operatorname{CAT}(0)$ model.

The introduction of strongly hyperbolic spaces, a nice class of Gromov-hyperbolic spaces containing $\operatorname{CAT}(-1)$ spaces, is partly motivated by this question \cite{Nica_applications-strong-hyperbolicity_2019}.
Indeed \cite{Nica-Spakula_strong-hyperbolicity_2016} shows that on a hyperbolic group $\Gamma$, the Green metric arising from random walks are strongly hyperbolic, thus $\Gamma$ acts properly cocompactly on a strongly hyperbolic space.

While working in the general context of strongly hyperbolic spaces, we were led to obtain a couple of new properties relating the cross-ratio to the metric, such as \Cref{lem:Buseman-length-function} and \Cref{lem:translation-length-from-cross-ratio}. 
Such properties would have been much easier to prove under an additional $\operatorname{CAT}(0)$ assumption (see \Cref{rem:infimum-translation-length}).

This leads us to the following \cref{quest:stron-hyp+CAT0->CAT(-1)}.
Let $(X,d)$ be strongly-$1$-hyperbolic and $\operatorname{CAT(0)}$: when is it $\operatorname{CAT}(-1)$ ? (a converse to \cite[Theorem 4.2]{Nica-Spakula_strong-hyperbolicity_2016}).
\end{introquestion}

\begin{introremark}[convex cocompact]
Let $\Gamma$ be a hyperbolic group.
Following ideas of Furman \cite{Furman_coarse-geometry_2002}, Reyes \cite{Reyes_space-metric-structures-hyperbolic-groups_2023} considers the space $\mathcal{D}(\Gamma)$ of classes of pseudo-metrics on $\Gamma$ that are quasi-isometric to a length metric modulo rough similarity.

If $F\in \mathcal{F}(\Gamma)$ is a function of hyperbolic type on a group $\Gamma$ then $d(\gamma_1,\gamma_2)=\cosh^{-1}(F(\gamma^{-1}_1\gamma_2))$ is a pseudo-metric on $\Gamma$ (since there is a representation $\rho\colon\Gamma\to\Isom(\Hyp^\kappa)$ and $o\in\Hyp^\kappa$ such that $\cosh^{-1}(F(\gamma^{-1}_1\gamma_2))=\dH(\rho(\gamma_1)o,\rho(\gamma_2)o)$. Moreover, two functions of hyperbolic type are homothetic if and only if their associated pseudo-metrics are roughly similar in the sense of \cite{Reyes_space-metric-structures-hyperbolic-groups_2023}. 
Finally, Xu \cite{Xu_convex-cocmpact-rep-infinite-dim-hyp_2024} shows that for a finitely generated group $\Gamma$ acting on $\Hyp^\kappa$, an orbit map $\Gamma \to \Hyp^\kappa$ is a quasi-isometry if and only if the action is convex compact (this is an open condition in $\Proj\mathcal{C}(\Gamma)$).

This yields an injective map from the space $\Proj\mathcal{C}_{cc}(\Gamma)$ of such classes of convex cocompact actions to the space $\mathcal{D}(\Gamma)$.
Reyes conjectures that for a finitely generated free group $\Gamma$, the set of classes corresponding to finite dimensional convex-cocompact actions $\bigcup_n \Proj\mathcal{C}^n_{cc}(\Gamma)$ has dense image in $\mathcal{D}(\Gamma)$.
\end{introremark}

\subsubsection*{Which groups act faithfully on $\Hyp^\kappa$ ?}

\begin{introquestion}[automorphism groups of simple graphs]
\label{introquest:graphs-kernel-hyperbolic}
A simple graph is given by an adjacency matrix $A$ on a set of vertices $X$, that is an $X\times X$ matrix with values in the boolean semi-field $\{0,1\}$ having diagonal $0$. 
This yields a metric space $(X,d)$, where the distance between the vertices $u,v\in X$ is given by the path metric: $d(u,v)=\min\{n\in \N_{>0} \colon A^n_{uv}>0\}$.
The automorphism group of the graph $\Aut(A)$ coincides with the isometry group $\Isom(X,d)$.

A general question is: which graphs admit a function of the distance $d$ (such as a polynomial in $d$ and its exponential $\exp(\varepsilon d)$) that is a kernel of hyperbolic type $C\colon X\times X\to \R$, which is separating (in the sense of \Cref{rem:separation=injectivity}).
This would yield an embedding $X\to \Hyp^\kappa$ that is equivariant under some representation of their automorphism group $\Isom(X,d) \to \Isom(\Hyp^\kappa)$.


In particular, we may define the critical exponent of $(X,d)$ as the supremum of $\varepsilon \in \R_{\ge 0}$ for which the kernel $C=\exp(\varepsilon d)$ is of hyperbolic type (note that $\varepsilon=0$ always yields the trivial function of hyperbolic type).
The critical exponent is $+\infty$ if and only if $(X,d)$ embeds isometrically in a real tree (see \Cref{prop:embedding_0_hyperbolic_space_tree} and the surrounding remarks), and such graphs admit a combinatorial characterization (known as block graphs, see \cite[Theorem 4.3]{Howorka_metric-properties-clique-graphs_1979}).

Is there a combinatorial characterization of infinite simplicial graphs with positive critical exponent, namely for which there is $\varepsilon>0$ such that $C=\exp(\varepsilon d)$ is a kernel of hyperbolic type?
\end{introquestion}

\begin{introquestion}[Mapping class groups]
\label{introquest:Mod(S)}
Consider a surface $\Sigma_{g,n}$ with genus $g$ and $n$ punctures, of negative Euler characteristic. 
We wonder which mapping class groups $\Mod(\Sigma_{g,n})$ admit faithful representation into $\Isom(\Hyp^\kappa)$ for some $\kappa$.
If such a mapping class group is non-linear over $\R$ then $\kappa$ must be infinite.
If there exists such a representation without fixed points then $\Mod(\Sigma_{g,n})$ would not have Kazhdan's property (T), and if there exists such a representation that is metrically proper then $\Mod(\Sigma_{g,n})$ would have the Haagerup's property.

The Teichm\"uller spaces of complex dimension $3g-3+n\ge 2$ are not Gromov-hyperbolic for all the classical metrics (Teichm\"uller \cite{Mazur-Wolf_Teichmuller-not-delta-hyp_1995}, Weil-Petersson \cite{Brock-Farb_Teichmuller-not-delta-hyp_2006}, Thurston \cite{Huang-Papadopoulos_Teichmuller-not-delta-hyp_2021}).
However Bonahon \cite{Bonahon_Teichmuller-via-geodesic-currents_1988} showed (for $n=0, g\ge 2$) that the space of geodesic currents on $\mathcal{GC}(\Sigma_{g,n})$ on $\pi_1(\Sigma_{g,n})$ with its intersection form $\It\colon \mathcal{GC}(\Sigma_{g,n}) \times \mathcal{GC}(\Sigma_{g,n}) \to \R_{\ge 0}$, contains a copy of Teichm\"uller space in a positive-level set $\{x \in \mathcal{GC}(\Sigma_{g,n}) \colon \It(x,x )=\pi \lvert \chi\rvert\}$, while the space of measured laminations identifies with the isotropic cone $\mathcal{ML}(\Sigma_{g,n})=\{ \xi \in \mathcal{GC}(\Sigma_{g,n}) \colon \It(\xi,\xi)=0\}$, which enables him to recover Thurston's compactification.
This is reminiscent of the hyperboloid model in hyperbolic geometry, except that the intersection form has signature $\ge 3g-3+n$ (since equality holds in restriction to $\mathcal{ML}(\Sigma_{g,n})$); in particular the intersection form does not define a separated pairing on $\mathcal{ML}(\Sigma_{g,n})$ in the sense of \Cref{eg:abstract-Cr-from-It}.

However if we restrict to the subset of projective measured laminations that are filling $\Proj\mathcal{MLF}(\Sigma_{g,n})\subset \Proj\mathcal{ML}(\Sigma_{g,n})$, the intersection function is separating and generates the topology in the sense of \Cref{eg:abstract-Cr-from-It}, which implies that we obtain an abstract cross-ratio on $\Proj\mathcal{MLF}(\Sigma_{g,n})$ that generates the topology by the formula:
\begin{equation*}
	\forall (u,v,x,y)\in \Proj \mathcal{MLF}(\Sigma_g)^{(4)} 
	\quad \colon \quad  
	B(u,v,x,y)=\tfrac{\It(u,v)\It(x,y)}{\It(u,y)\It(x,v)}.
\end{equation*}
We may further restrict this to the space of projective measured laminations that are fixed by pseudo-Anosov elements $\Proj \mathcal{MLA}(\Sigma_{g,n})$.
Is this abstract cross-ratio an algebraic cross-ratio in the sense of \Cref{introdef:algebraic-cross-ratio}?  
A positive answer would yield (by \Cref{introthm:algebraic-cross-ratio}) an embedding $\mathcal{MLA}(\Sigma_g) \to \partial \Hyp^\kappa$ which is equivariant under a faithful representation $\Mod(\Sigma_g)\to \Isom(\Hyp^\kappa)$. 

A weaker question is whether $\Mod(\Sigma_{g,n})$ admits a faithful action (maybe fixed point-free or metrically proper) on some strongly hyperbolic space.
When $3g-3+n\ge 2$ and $(g,n)\ne (1,2)$, the extended mapping class group $\Mod^\pm(\Sigma_{g,n})$ acts faithfully by automorphisms of the curve graph $\mathcal{SC}(\Sigma_{g,n})$ (it fact $\Mod^\pm(\Sigma_{g,n})$ coincides with $\Aut \mathcal{SC}(\Sigma_{g,n})$ by \cite{Ivanov_curve-complex_1997, Luo_automorphisms-curve-complex_2000}).
By \cite{Masur-Minsky_curve-graph-hyperbolic_1999}, the curve graph is Gromov-hyperbolic and its boundary identifies with the space of ending laminations $\partial \mathcal{SC}(\Sigma_{g,n}) = \mathcal{MLE}(\Sigma_{g,n})$, which are in particular filling.
Moreover, the loxodromic elements for the action of $\Mod(\Sigma_{g,n})$ on $\mathcal{SC}(\Sigma_{g,n})$ are precisely the pseudo-Anosov elements, whose fixed points identify with $\mathcal{MLA}(\Sigma_{g,n}) \subset \partial \mathcal{SC}(\Sigma_{g,n})$.
Thus we wonder which curves graphs $\mathcal{SC}_{g,n}$ are strongly hyperbolic. And for those which are, does their metric cross-ratio on the boundary coincide with the intersection cross-ratio defined above?

In relation to the previous \Cref{introquest:graphs-kernel-hyperbolic}, we also wonder whether curve graphs admit a function of the distance yielding a separating kernel of hyperbolic type.
\end{introquestion}

\subsection*{Conventions and notations}

The expression ``for any'' is systematically used in the precise sense, to mean ``for some or equivalently for all''.
The set of natural integers $\N$ contains $0$, whereas $\N^*=\N\setminus\{0\}$. For a set of numbers, we denote in index some restrictions on certain sets of numbers, such as $\N_{\ge 2}$, $\R_{\ge 1}$.

\newpage
\section{Geometry of strongly hyperbolic spaces}

We will recall from \cite{Nica-Spakula_strong-hyperbolicity_2016, Das-Simmons-Urbanski_gromov-hyperbolic-spaces_2017} the notion and properties of strongly hyperbolic spaces, which are metric spaces (maybe not proper or complete) whose distance satisfies a strengthening of Gromov's hyperbolicity condition, ensuring that their boundary is endowed with 
a continuous cross-ratio satisfying appropriate separation properties.
The class of strongly hyperbolic spaces contains the class of $\operatorname{CAT}(-1)$-spaces, hence of trees and algebraic hyperbolic spaces of any dimension, which will be at the centre of our interest.

Let us briefly recall from \cite[Chapter II.1]{Bridson-Haefliger_metric-non-positive-curvature_1999} that for $\varepsilon\ge 0$ a metric space is $\operatorname{CAT}(-\varepsilon)$ when it is geodesic (any two points are connected by some geodesic) and the geodesic triangles have lengths satisfying the $\operatorname{CAT}(-\varepsilon)$ inequality (saying that they are thinner than the comparison triangles with the same edge lengths in the unique complete Hadamard space of curvature $-\varepsilon$).

\subsection{Strongly hyperbolic space and their boundary}\label{sec:strong_hyperbolicity}

Fix a metric space $(X,d)$.
At a base point $o\in X$, the \emph{Gromov product} $(\mid)_o\colon X^2 \to [0,+\infty)$ and the \emph{visual metametric} $d_o\colon X^2 \to (0,+\infty)$ are defined for $x,y\in X$ by:
\begin{equation*}
(x \mid y)_o = \tfrac{1}{2}\left(d(o,x)+d(o,y)-d(x,y)\right)
\qquad 
d_o(x, y)=\exp\left(-(x\mid y)_o\right).
\end{equation*}

For $\delta>0$, the space $(X,d)$ is \emph{Gromov-$\delta$-hyperbolic} when: 
\begin{equation*}
\forall o,x,y,z\in X \colon \quad
(x \mid z)_o \ge \min\{(x \mid y)_o, (y \mid z)_o\} - \delta
\end{equation*}
or equivalently for every tetrahedron $(x,y,x',y') \in X^4$:
\begin{equation*}
d(x,x')+d(y, y')\le \max\{d(x,y')+d(y, x'), d(x,y)+d(x', y')\}+2\delta.
\end{equation*}

For $\varepsilon>0$, the space $(X,d)$ is \emph{strongly-$\varepsilon$-hyperbolic} when: 
\begin{equation}
\forall o,x,y,z \in X \colon \quad
d_o^\varepsilon(x , z) \le d_o^\varepsilon(x , y)+ d_o^\varepsilon(y , z)
\end{equation}
or equivalently every tetrahedron $(x,x',y,y') \in X^4$ satisfies Ptoleme's inequality for $d_o^{\varepsilon}$:
\begin{equation}
\label{eq:Ptolemy-visual_strong-hyperbolic}
d_o^{\varepsilon}(x,x') d_o^{\varepsilon}(y, y')
\le 
d_o^{\varepsilon}(x,y') d_o^{\varepsilon}(y, x')
+ 
d_o^{\varepsilon}(x,y) d_o^{\varepsilon}(x', y')
\end{equation}
which can also be rewritten (independently of $o$) as a Ptoleme inequality for $\exp(\tfrac{1}{2}\varepsilon d)$:
\begin{equation}
\label{eq:Ptolemy_strong-hyperbolic}
\exp\tfrac{\varepsilon}{2}\left(d(x,x')+d(y,y')\right) \leq
\exp\tfrac{\varepsilon}{2}\left(d(x,y')+d(y,x')\right)+
\exp\tfrac{\varepsilon}{2}\left(d(x,y)+d(x',y')\right)
\end{equation}

From now on we assume that $(X,d)$ is strongly-$\varepsilon$-hyperbolic for some parameter $\varepsilon>0$. This implies (\cite[Theorem 4.2]{Nica-Spakula_strong-hyperbolicity_2016}) that $(X,d)$ is Gromov-$\delta$-hyperbolic for $\delta=\ln(2)/\varepsilon$.
\begin{remark}[strong hyperbolicity is an exponential strengthening of the Gromov-hyperbolicity]
\label{rem:NS-strongly-hyperbolic-characterisation}
By \cite[Lemma 6.2]{Nica-Spakula_strong-hyperbolicity_2016}, a Gromov-hyperbolic space $(X,d)$ is strongly hyperbolic if and only if there exist $C,\varepsilon_2, R_0 \in \R_{>0}$ such that for every tetrahedron $(x,y,x',y') \in X^4$ and $R\ge R_0$ with $d(x,x')+d(y, y')\ge R$ we have $\lvert d(x,x')+d(y, y')-d(x,y')-d(y, x')\rvert \le C \exp(-\varepsilon_2 R)$.
\end{remark}

We say that a sequence $(x_n)\in X^\N$ \emph{converges to infinity} when $(x_n \mid x_m)_o \to \infty$ as $m,n\to \infty$, and that two sequences $(x_n), (y_n)\in X^\N$ converging to infinity are \emph{equivalent} when $(x_n \mid y_m)_o\to \infty$ as $m,n\to \infty$.
These notions do not depend on $o\in X$.
The Gromov boundary $\partial X$ is the set of equivalence classes of sequences converging to infinity.

The \emph{bordification} of $(X,d)$ is the disjoint union $X\sqcup \partial X$ endowed with the following topology: a subset $U\subset X\sqcup \partial X$ is open when $U\cap X$ is open in $X$ and for all $u\in U\cap \partial X$ there exists $t>0$ such that $V_o^t(u)=\{v \in X\sqcup \partial X \colon (u\mid v)_o>t \}$ is contained in $U$; which is independent of $o$ and completely metrizable (see \cite[§3.4.2]{Das-Simmons-Urbanski_gromov-hyperbolic-spaces_2017}).

The strong hyperbolicity ensures a continuous extension to the bordification $X\sqcup \partial X$ of the Gromov product $( \mid )_o \colon (X\sqcup \partial X)^2 \to [0,+\infty]$ (by \cite[Lemma 3.4.22]{Das-Simmons-Urbanski_gromov-hyperbolic-spaces_2017}) hence of the Ptolemaic metametric $d_o^\varepsilon \colon (X\sqcup \partial X)^2 \to [0,+\infty)$ where (by \cite[Observation 3.6.7]{Das-Simmons-Urbanski_gromov-hyperbolic-spaces_2017}) it becomes a complete metametric that is compatible with the topology on $X\sqcup \partial X$ and which restricts to a metric precisely on the boundary, and this \emph{boundary visual metric} from $o$ for the parameter $\varepsilon$ denoted $d_o^\varepsilon \colon (\partial X)^2 \to [0,+\infty)
$ thus recovers topology of $\partial X$.

\begin{remark}[Hamenst\"adt visual metric]
Let us recall \cite[Proposition 3.6.19]{Das-Simmons-Urbanski_gromov-hyperbolic-spaces_2017}. For every pair $(o,\infty)\in X\times \partial X$, there is complete metametric $d_{o,\infty}^\varepsilon$ on $(X\sqcup\partial X)\setminus \{\infty\}$ defined by:
\begin{equation}
	\label{eq:Hamenstadt-visual-metric}
	d_{o,\infty}^\varepsilon(x,y) = \exp \left[-\varepsilon \left( (x\mid y)_o-(x\mid \infty)_o-(y\mid \infty)_o \right) \right] = \frac{d_o^\varepsilon(x,y)}{d_o^\varepsilon(x,\infty) d_o^\varepsilon(y,\infty)}
\end{equation}
which is compatible with the topology on $(X\cup\partial X)\setminus \{\infty\}$ induced by that on $(X\sqcup\partial X)$, such that $B\subset (X\sqcup\partial X)\setminus \{\infty\}$ is bounded if and only if $\infty \notin \overline{B}$ 
; and, moreover, the restriction of $d_{o,\infty}^\varepsilon$ to $\partial X\setminus \{\infty\}$ is a metric called the Hamenst\"adt visual metric.

Note that for all $x\in (X\sqcup \partial X)\setminus \{\infty\}$ we have $d_{o,\infty}^\varepsilon(x,o) = 1/d_o^\varepsilon(x,\infty)$.
\end{remark}

\begin{remark}[Buseman function]
\label{rem:Buseman-cocycle}
For $\xi \in X$, the \emph{Buseman function} $\Bus_\xi \colon X\times X \to \R$ is defined for $y,z\in X$ by $\Bus_\xi(y,z)=d(y,\xi)-d(z,\xi)=(z\mid \xi)_y-(y\mid \xi)_z$.
The latter formula extends continuously to $\xi \in X\sqcup \partial X$ by \cite[Lemma 3.4.10]{Das-Simmons-Urbanski_gromov-hyperbolic-spaces_2017}, where it remains invariant under the diagonal action of $\Gamma$ on $(X\cup \partial X)\times X^2$.
It satisfies for all $\xi\in(X\cup \partial X)$ and $x,y,z\in X$ the triangle inequality $\lvert \Bus_\xi(y,z)\rvert \le d(y,z)$ and the additivity $\Bus_\xi (y,z) = \Bus_\xi(y,x)+\Bus_\xi(x,z)$.
\end{remark}

The following class of strongly hyperbolic spaces plays an important role in this work.

\begin{definition}[real tree]
\label{def:real-tree}
A metric space $(T,d)$ is called a \emph{real tree} when any two points are joined by a unique injective segment and such segments are isometries.


\end{definition}

\begin{remark}[connected Gromov-$0$-hyperbolic]
\label{rem:real-trees=connected_Gromov-0-hyperbolic}
A metric space $(T,d)$ is a real tree if and only if it is connected and Gromov-$0$-hyperbolic (see \cite[Chapter 3]{Evans_probability-real-trees_2008}).
The Gromov-$0$-hyperbolicity means that every quadruple $o,x,y,z\in T$ satisfies $(x\mid z)_o = \max\{(x\mid y)_o, (y\mid z)_o\}$, namely for every $o\in T$ the visual metametric $d_o\colon T\times T \to \R_{\ge 0}$ is an ultrametric.    
\end{remark}

The following lemma shows that one can interpolate the points of a $0$-hyperbolic metric space minimally so as to obtain a real tree which is uniquely defined up to equivariant isometry.

\begin{proposition}[embedding $0$-hyperbolic spaces in real trees]
\label{prop:embedding_0_hyperbolic_space_tree}
Let $(X,d)$ be a $0$-hyperbolic metric space.
There is a real tree $T$ and an isometric embedding $\varphi\colon X\to T$ such that $T=\bigcup_{x,y\in X}[\varphi(x),\varphi(y)]$ and a representation $\Isom(X)\to\Isom(T)$ such that $\varphi$ is equivariant.
\end{proposition}
\begin{proof}
Part of this well-known fact are proposed in \cite[Exercise 8]{Ghys-Harpe_groupes-hyperboliques_1990}.
A proof of the full proposition follows from that of the very similar \cite[Theorem 3.38]{Evans_probability-real-trees_2008} (one should only check that their construction does not depend on their initial choice of base point).
\end{proof}

\begin{remark}[connected strongly-$(+\infty)$-hyperbolic]
\label{rem:real-trees=connected_strongly-infty-hyperbolic}
A real tree is $\operatorname{CAT}(-1)$ hence strongly hyperbolic with parameter $1$ \cite[Theorem 5.1]{Nica-Spakula_strong-hyperbolicity_2016}.
Moreover, since the definition of real trees is invariant under rescaling the metric by a positive constant, a real tree is strongly-$\varepsilon$-hyperbolic for all $\varepsilon\in (0,+\infty)$.
Conversely, a connected metric space that is strongly-$\varepsilon$-hyperbolic for all $\varepsilon \in (0,+\infty)$ is Gromov-$0$-hyperbolic hence a real tree.
\end{remark}

\subsection{Cross-ratio for strongly hyperbolic spaces}
\label{subsec:strong-hyp_cross-ratio}

Let $(X,d)$ be a strongly hyperbolic metric space.
The \emph{cross-ratio} of $(x,x'),(y,y')\in X^2\times X^2$ is the positive real number:
\begin{equation}
\label{eq:def-cross-ratio-X}
\Cr{x,x';y,y'}
= \exp \tfrac{1}{2}\left(d(x,x')+d(y,y')-d(x,y')-d(y,x')\right).
\end{equation}
The name cross-ratio is justified by the fact that for all $o\in X$ we have:
\begin{equation}
\label{eq:cross-ratio_visual-d_o}
\Cr{x,x';y,y'} 
= \frac{d_o(x,x')\times d_o(y,y')}{d_o(x,y') \times d_o(y,x')}.
\end{equation}

The cross-ratio is invariant under the $(\Z/2)^2$-action which exchanges the pair or inverts both of their orientations, in formulae: 
\begin{equation}
\label{eq:Cr-Invariance}
\tag{Invariance}
\Cr{x, x'; y, y'} 
= \Cr{y, y'; x, x'}
= \Cr{x', x; y', y}.
\end{equation}
Moreover, it satisfies the inversion rule given by swapping either $(x, y)$ or $(x', y')$:
\begin{equation}
\label{eq:Cr-Inversion}
\tag{Inversion}
\Cr{x, x'; y, y'} 
= \Cr{y, x'; x, y'}^{-1}
= \Cr{x, y'; y, x'}^{-1}.
\end{equation}
More generally it satisfies the cocycle relation, as for every $x''\in X$ we have:
\begin{equation}
\label{eq:Cr-Cocycle}
\tag{Cocycle}
\Cr{x, y; x', y'} = \Cr{x, y; x'', y'} \times 
\Cr{x'', y; x', y'}.
\end{equation}

(Note that this notion of cross-ratio can be defined for any metric space $(X,d)$, and according to Equation \eqref{eq:Ptolemy_strong-hyperbolic} it can be used to characterize the strong-$\varepsilon$-hyperbolicity by saying that all tetrahedra $(x,x',y,y')\in X^4$ satisfy $1\le \Cr{x,x';y,y'}^{-\varepsilon}+ \Cr{x,x';y',y}^{-\varepsilon}$.)

The strong hyperbolicity ensures that the cross-ratio extends continuously to quadruples of distinct points in the bordification $(X\sqcup \partial X)$ with values in $(0,\infty)$ on which it defines a continuous function that is invariant under the diagonal action of $\Isom(X,d)$ on $(\partial X)^{(4)}$, satisfying the \eqref{eq:Cr-Invariance}, \eqref{eq:Cr-Inversion} and \eqref{eq:Cr-Cocycle} relations.

Moreover, the cross-ratio further extends by continuity to the subspace $\partial^4 X\subset \left(\partial X\right)^4$ of quadruples of which at least three are distinct with values in $[0,\infty]$, and by the separation property of the boundary visual metric, it satisfies:
\begin{align}
\label{eq:Cr=0}\tag{$\Cr{\cdot}=0$}
\Cr{a^-,a^+;b^-,b^+} = 0 & \iff (a^-=a^+) \; \mathrm{or} \; (b^-=b^+) 
\\  
\label{eq:Cr=infty}\tag{$\Cr{\cdot}=\infty$}
\Cr{a^-,a^+;b^-,b^+} = \infty & \iff (a^-=b^+) \; \mathrm{or} \; (b^-=a^+).
\end{align}
Note that if $a^-=b^-$ or $a^+=b^+$, we have $[a^-,a^+;b^-,b^+]=1$, but the converse may often fail.

In formula, recalling that $\partial X$ is endowed for each $o\in X$ with the visual distance $d_o^\varepsilon(x,y)=\exp\left(-\varepsilon (x\mid y)_o\right)$, we have for  $(a^-,a^+, b^-,b^+)\in \partial^4 X$:
\begin{equation} \label{eq:Cr-boundary-do}
\Cr{a^-,a^+;b^-,b^+}^\varepsilon=\frac{d_o^\varepsilon(a^-,a^+)\times d_o^\varepsilon(b^-,b^+)}{d_o^\varepsilon(a^-,b^+) \times d_o^\varepsilon(b^-,a^+)}
\end{equation}

\begin{remark}[cross-ratio on $\partial^4 X$]
We introduced the cross-ratio on $X^4$ only to define it conveniently on $\partial^4 X$ and clarify its properties.
We use the same denomination and notation for both, but will mostly work with its restriction to $\partial^4 X$.
\end{remark}

\begin{remark}[Conventions] 
Several other conventions for the cross-ratio are used, both regarding the order of the variables in the brackets and the actual formula.

Our formula and order are compatible with those of Labourie in \cite{Labourie_Cross-ratios-Anosov-rep-energy_2008, Labourie-McShane_cross-ratios-identities-Teichmuller_2009}. Our formula for $\Cr{a^-,a^+; b^-, b^+}$ matches the one defined by Bourdon in \cite{Bourdon_cross-ratio-CAT(-1)_1996} for $[\xi\xi'\eta\eta']$ with the variables in the order $(\xi, \xi', \eta, \eta')=(a^-,b^-,a^+, b^+)$.

The formula introduced by Otal in \cite{Otal_geometrie-symplectique-geodesiques_1992} is the logarithm of that in Bourdon, hence the logarithm of ours after swapping the variables $(a^+ b^-)$ in the bracket.

In \cite{Kim_marked-length-rigidity-symmetric-space_2001}, Kim uses the square of the definition of Bourdon in \cite{Bourdon_cross-ratio-CAT(-1)_1996}, hence the square of ours after swapping the variables $(a^- b^+)$ in the bracket.
\end{remark}

\subsection{Isometries of strongly hyperbolic spaces}

The isometry group $\Isom(X)$ is endowed with the pointwise-convergence topology induced by the product topology on the set of functions $X^X$.

An isometry between strongly hyperbolic spaces preserves the finite cross-ratio, and it extends to a homeomorphism between their boundaries which preserves the boundary cross-ratio.
In particular, the elements and subgroups of $\Isom(X,d)$ act on $X\sqcup \partial X$ by homeomorphisms, and can thus be classified according to their dynamical behaviour. 

For $\gamma\in \Isom(X,d)$ its \emph{stable translation length} is defined for any base point $o\in X$ by:
\begin{equation}
\label{eq:translation-length_limit}
\ell(\gamma) 
= \lim_{n\to \infty} \tfrac{1}{n}d(o, \gamma^n o)
= \inf \{\tfrac{1}{n}d(o, \gamma^n o)\colon n\in \N_{\ge 1}\}
\end{equation}
Note that for all $\alpha \in \Isom(X,d)$ and $n\in \Z$ we have $\ell(\alpha\gamma\alpha^{-1})=\ell(\gamma)$ and $\ell(\gamma^n)=\lvert n \rvert \ell(\gamma)$.

Let us recall \cite[Theorem 1.9]{Reyes_isometries-Gromov-hyperbolic_2018}.

\begin{theorem}[continuity of stable translation-length]
When $\Isom(X)$ is endowed with the pointwise-convergence topology, the function $\ell \colon \Isom(X)\to \R_{\ge0}$ is continuous.
\end{theorem}

The following is an equivalent reformulation of \cite[Theorem 6.1.4]{Das-Simmons-Urbanski_gromov-hyperbolic-spaces_2017}.

\begin{proposition}[classification of isometries]
\label{prop:classification-g-Isom(X)}
A non-trivial $\gamma \in \Isom(X)$ is either:
\begin{enumerate}[noitemsep, align=left]
	\item[\emph{elliptic}:] the $\gamma$-orbit of any point of $X$ has $0$ accumulation points in $\partial X$;
	\item[\emph{parabolic}:] the $\gamma$-orbit of any point of $X$ has $1$ accumulation point in $\partial X$; 
	\item[\emph{loxodromic}:] the $\gamma$-orbit of any point of $X$ has $2$ accumulation points in $\partial X$.
\end{enumerate}
\end{proposition}
It follows that a non-trivial element $\gamma \in \Isom(X)$ is elliptic if and only if the $\gamma$-orbit of any point of $X$ is bounded.
A loxodromic element $\gamma\in \Isom(X)$ has a pair of (attractive, repulsive) fixed points $(\gamma^-,\gamma^+)\in (\partial X) \times (\partial X)$.

The following classification is \cite[Theorem 6.2.3]{Das-Simmons-Urbanski_gromov-hyperbolic-spaces_2017}.

\begin{theorem}[classification of group actions]
\label{thm:classification-G-Isom(X)}
A sub(semi)group $\Gamma \subset \Isom(X)$ is either:
\begin{enumerate}[noitemsep, align=left]
	\item[\emph{elliptic}:] when the $\Gamma$-orbit of some any point of $X$ is bounded; 
	\item[\emph{loxodromic}:] when $\Gamma$ contains a loxodromic element;
	\item[\emph{parabolic}:] when it is neither elliptic nor hyperbolic; 
\end{enumerate}
and parabolic subgroups must have a unique $\Gamma$-fixed point in $\partial X$.
\end{theorem}

Hence the translation length vanishes on $\Gamma$ if and only if it is elliptic or parabolic.
However, beware that some parabolic subgroups contain only elliptic elements \cite[Example 11.2.18]{Das-Simmons-Urbanski_gromov-hyperbolic-spaces_2017}.

\begin{definition}[limit set]
A subgroup $\Gamma \subset \Isom(X)$ has \emph{limit set} defined for any $o\in X$ by:
\begin{equation*}
	\Lambda(\Gamma)=\{c \in \partial X \colon \ \exists (\gamma_n)\in G^{\mathbb{N}},\, \gamma_n o \to c\}.
\end{equation*}
\end{definition}

Thus $\Gamma$ is of non-elliptic type if and only if $\Lambda(\Gamma)\ne \emptyset$ in which case $\Lambda(\Gamma)$ is the minimal non-empty closed $\Gamma$-invariant subset of $\partial X$ (\cite[Proposition 7.4.1]{Das-Simmons-Urbanski_gromov-hyperbolic-spaces_2017}).
Note that $\Card \Lambda(\Gamma)\ge 2$ if and only if $\Gamma$ is loxodromic.

By \cite[Proposition 7.3.1]{Das-Simmons-Urbanski_gromov-hyperbolic-spaces_2017}, a subgroup $\Gamma\subset \Isom(X)$ has limit set satisfying either $\Card \Lambda(\Gamma) \le 2$ or $\Card \Lambda(\Gamma) \ge \Card(\R)$.
The following is \cite[Propositions 7.4.7]{Das-Simmons-Urbanski_gromov-hyperbolic-spaces_2017}.

\begin{proposition}[loxodromic axes are dense]
\label{prop:loxodromic-axes-dense}
For a loxodromic subgroup $\Gamma\subset \Isom(X)$, 
the subset of pairs $(\gamma^-,\gamma^+)\in (\partial X)\times(\partial X)$ of (repulsive, attractive) fixed points of loxodromic $\gamma\in \Gamma$ is contained and dense in $\Lambda(\Gamma)\times \Lambda(\Gamma)$.
\end{proposition}

A representation $\rho\colon \Gamma \to \Isom(X)$ inherits all objects and attributes of its image $\rho(\Gamma)\subset \Isom(X)$, in particular its translation length function denoted $\ell_\rho$, its limit set denoted $\Lambda(\rho)$, its type (elliptic, parabolic, loxodromic), et caetera.

\begin{remark}[infimum displacement]
\label{rem:infimum-translation-length}
When $(X,d)$ is also $\operatorname{CAT}(0)$, the \cite[Exercise II.6.(1)]{Bridson-Haefliger_metric-non-positive-curvature_1999} leads to showing that the translation length of $\gamma\in \Isom(X,d)$ equals its \emph{infimum displacement}:
\begin{equation*}
	\ell(\gamma) = \inf\{d(o, \gamma(o)) \colon o\in X \}.
\end{equation*}
In that case, for a loxodromic $\gamma\in \Isom(X)$, there exists a unique geodesic connecting its fixed points $(\gamma^-,\gamma^+)$, and that is $\{o\in X\colon d(o,\gamma o) = \ell(\gamma)\}$.
This follows from the description of loxodromic-axes in $\operatorname{CAT}(0)$ spaces by \cite[Theorem 6.8]{Bridson-Haefliger_metric-non-positive-curvature_1999}, and the closeness of long geodesic segments with close endpoints in strongly hyperbolic spaces as measured by \Cref{rem:NS-strongly-hyperbolic-characterisation}. 

In particular, denoting the limit set of $\Isom(X)$ by $\Lambda$, this yields a dense subset of $\Lambda \times \Lambda$ that are joined by a unique geodesic. 
\end{remark}

\begin{question}[when strongly-$1$-hyperbolic implies $\operatorname{CAT}(-1)$]
\label{quest:stron-hyp+CAT0->CAT(-1)}
Let $(X,d)$ be a strongly-$1$-hyperbolic space that is $\operatorname{CAT(0)}$: when is it $\operatorname{CAT}(-1)$ ? (a converse to \cite[Theorem 4.2]{Nica-Spakula_strong-hyperbolicity_2016}).
\end{question}

\subsection{Characterizing subgroups by length functions and cross-ratios}

This subsection shows that the translation length function or cross-ratio contains enough information to recognize which subgroups of $\Isom(\Gamma)$ have limit set of cardinal $\ge 3$ (hence uncountably infinite), and when so that they determine one another.

Recall from \Cref{rem:Buseman-cocycle} the definition and properties of Buseman functions for strongly hyperbolic spaces.
We will first recover from Buseman functions the translation lengths (never mentioned in \cite{Das-Simmons-Urbanski_gromov-hyperbolic-spaces_2017}, where only dynamical derivatives are considered \cite[§4.2.3]{Das-Simmons-Urbanski_gromov-hyperbolic-spaces_2017}).

\begin{lemma}[translation length from Buseman]
\label{lem:Buseman-length-function}
Consider a loxodromic $\gamma \in\Isom(X)$ with (attractive, repulsive) fixed points $(\gamma^+, \gamma^-)$.
For every $o\in X$ we have
\[\Bus_{\gamma^+}(o,\gamma o)= \ell(\gamma) =\Bus_{\gamma^-}(\gamma o, o).\]
In particular, the dynamical derivative defined in \cite[§4.2.3]{Das-Simmons-Urbanski_gromov-hyperbolic-spaces_2017} for the parameter $b=e$, of the loxodromic element $\gamma\in \Isom(X)$ at the points $\gamma^\pm \in \partial X$ is $\exp(\pm \ell(\gamma))$.
\end{lemma}

\begin{proof}
By antisymmetry and $\Gamma$-invariance of the Buseman function we have $\Bus_{\gamma^+}(o, \gamma^{-1} o) = -\Bus_{\gamma^+}(\gamma^{-1} o, o) = -\Bus_{\gamma^+}(o, \gamma o)$.
Moreover the second equality follows by applying the first to $\gamma^{-1}$ and $\gamma o$, using $\ell(\gamma^{-1})=\ell(\gamma)$ and $(\gamma^{-1})^+=\gamma^-$.

For all $o\in X$ and $n\in \N$ we have, using additivity and equivariance of the Buseman function $\Bus_{\gamma^+}\colon X\times X\to \R$ that $\Bus_{\gamma^+}(o,\gamma^n o) = \sum_1^n \Bus_{\gamma^+}(\gamma^{k-1} o, \gamma^{k}o)= n \Bus_{\gamma^+}(o, \gamma o)$.

From the \cite[Corollary 3.4.12]{Das-Simmons-Urbanski_gromov-hyperbolic-spaces_2017} applied to the identity \cite[Proposition 3.3.3 (h)]{Das-Simmons-Urbanski_gromov-hyperbolic-spaces_2017} we have
\(\Bus_{\gamma^+}(o,\gamma^{-n} o) = - d(o,\gamma^{-n}o) + 2 (o\mid \gamma^+)_{\gamma^{-n} o} \)
but $(o\mid \gamma^+)_{\gamma^{-n}o}\rightarrow (o\mid \gamma^+)_{\gamma^{-}} \in \R$.

By equivariance and antisymmetry $\Bus_{\gamma^+}(o,\gamma^n o) = -\Bus_{\gamma^+}(\gamma^n o, o) = -\Bus_{\gamma^+}(o,\gamma^{-n}o)$
so
$\Bus_{\gamma^+}(o,\gamma o) = \lim_n \tfrac{1}{n}\Bus_{\gamma^+}(o,\gamma^{n} o) = -\lim_n \tfrac{1}{n}\Bus_{\gamma^+}(o,\gamma^{-n} o) = -\lim_n \frac{-d(o,\gamma^{-n} o)}{n} = \ell(\gamma)$.
\end{proof}

\begin{lemma}[translation-length from cross-ratio]
\label{lem:translation-length-from-cross-ratio}
An element $\gamma \in \Isom(X)$ is loxodromic if and only if $\ell(\gamma)>0$.
In that case, it has (repulsive, attractive) fixed points $(\gamma^-, \gamma^+) \in \partial^2 X$ and:
\begin{equation}
	\label{eq:translation-length_Cr-boundary_log}
	\forall \xi \in X\sqcup \partial X\setminus \{\gamma^-,\gamma^+\} \colon
	\quad
	\ell(\gamma) 
	= \log  \Cr{\gamma \xi, \gamma^-; \xi, \gamma^+}
\end{equation}
\end{lemma}
\begin{proof}
The criterion on $\ell(\gamma)>0$ is \cite[Chapter 10, Proposition 6.3]{Coornaert-Delzant-Papadopoulos_geometrie-groupes-hyperboliques_1990}.
Now assume $\ell(\gamma)>0$.

For all $o\in X$ we compute that $\log  \Cr{\gamma o, \gamma^{-}, o, \gamma^{+}} = \tfrac{1}{2}(\Bus_{\gamma^-}(\gamma o,  o)+\Bus_{\gamma^+}( o,\gamma o))$ and by \cite[Proposition 4.2.16]{Das-Simmons-Urbanski_gromov-hyperbolic-spaces_2017} both terms are equal to the translation length $\ell(\gamma)$ by \Cref{lem:Buseman-length-function}.
(One may also work in the polar coordinates of $(\gamma^-,o,\gamma^+)$ and apply \cite[Lemma 4.6.3]{Das-Simmons-Urbanski_gromov-hyperbolic-spaces_2017}.)
The continuous extension of the cross-ratio to quadruple of distinct points in the bordification shows that the identity holds in the limit where $o\in X$ converges to $\xi \in X \sqcup \partial X \setminus \{\gamma^-,\gamma^+\}$.
\end{proof}

\begin{lemma}[cross-ratio from translation-length]
\label{lem:cross-ratio-from-translation-length}
For loxodromic $\alpha,\beta \in \Isom(X)$ with distinct (attractive, repulsive) fixed points $(\alpha^+,\alpha^-), (\beta^+,\beta^-)$:
\begin{equation}
	\label{eq:translation-length_Cr-boundary_exp_1}
	\Cr{\alpha^-, \alpha^+; \beta^-,\beta^+} 
	= \lim_{m,n\to+\infty}\exp \tfrac{1}{2}\left(\ell(\alpha^m)+\ell(\beta^n)-\ell(\alpha^m\beta^{n})\right)
\end{equation}
\end{lemma}

\begin{remark}[another proof for $\operatorname{CAT}(-1)$]
The \Cref{eq:translation-length_Cr-boundary_exp_1} is shown for rank one symmetric spaces in \cite[Theorem 1]{Kim_marked-length-rigidity-symmetric-space_2001}.
One may adapt its proof whenever a loxodromic $\gamma\in \Isom(X)$ has a unique geodesic translation axis connecting its fixed points given by $\{o\in X\colon d(o,\gamma o) = \ell(\gamma)\}$, for instance when $X$ is strongly hyperbolic and $\operatorname{CAT}(0)$ by \Cref{rem:infimum-translation-length}.

This proof does not generalize so easily to the general setting of strongly hyperbolic spaces.
Our proof exploits the continuity and precise asymptotics for Buseman functions and Gromov products over triples of distinct points in the bordification (\cite[3.4.10 \& 3.4.12]{Das-Simmons-Urbanski_gromov-hyperbolic-spaces_2017})
\end{remark}

\begin{proof}
Observe that for $m,n\in \N_{>0}$ we have $\alpha^m(\beta^n\alpha^m)^\pm = (\alpha^m\beta^n)^\pm$ and $\beta^n(\alpha^m\beta^n)^\pm = (\beta^n\alpha^m)^\pm$.
Let $o\in X$. 
By \Cref{lem:Buseman-length-function} we have $\ell(\alpha^m)+\ell(\beta^n)=\Bus_{\alpha^+}(\alpha^{-m}o,o)+\Bus_{\beta^+}(\beta^{-n}o,o)$ and using the antisymmetry and additivity of the Buseman function we have:
\[-\ell(\alpha^m\beta^n) 
= \Bus_{(\alpha^m\beta^n)^+}(\alpha^m o, \beta^{-n}o) 
= \Bus_{(\alpha^m\beta^n)^+}(\alpha^m o, o) + \Bus_{(\alpha^m\beta^n)^+}(o, \beta^{-n}o).\]
We may rewrite the term $\Bus_{(\alpha^m\beta^n)^+}(\alpha^m o, o) = \Bus_{(\beta^n\alpha^m)^+}(o, \alpha^{-m} o)$ using $\alpha^m$-invariance of the Buseman function and $\alpha^{-n}(\alpha^m\beta^n)^+=(\beta^n\alpha^m)^+$.

We say that two real sequences $s_k,s_k' \in \R$ indexed by $k\in \N$ are asymptotic and denote $s_k \asymp s_k'$ when $\lim_{k\to +\infty}(s_k-s_k') = 0$.
For sequences depending on several variables, we put the variable(s) tending to infinity in index.

As $m \to +\infty$ we have $(\alpha^m\beta^n)^+ \to \alpha^+$ hence
$\Bus_{(\alpha^m\beta^n)^+}(o, \beta^{-n} o) \asymp_m \Bus_{\alpha^+}(o, \beta^{-n}o)$,
and as $n \to +\infty$ we have $(\beta^n\alpha^m)^+ \to \beta^+$ hence
$\Bus_{(\beta^n\alpha^m)^+}(o,\alpha^{-m}o) \asymp_n \Bus_{\beta^+}(o, \alpha^{-m}o)$.
Hence $-\ell(\alpha^m\beta^n)\asymp_{m,n} \Bus_{\alpha^+}(o,\beta^{-n}o)+\Bus_{\beta^+}(o,\alpha^{-m}o)$
so using once again antisymmetry and additivity of the Buseman function we have
\begin{equation*}
	\ell(\alpha^m)+\ell(\beta^n)-\ell(\alpha^m\beta^n) 
	\asymp_{m,n}
	\Bus_{\alpha^+}(\alpha^{-m}o, \beta^{-n}o)+\Bus_{\beta^+}(\beta^{-n}o, \alpha^{-m}o) 
\end{equation*}
and expanding in terms of Gromov-products, using its continuity over triples in the bordification (\cite[Corollary ]{Das-Simmons-Urbanski_gromov-hyperbolic-spaces_2017}), and expressing everything in terms of distances, we have that $\ell(\alpha^m)+\ell(\beta^n)-\ell(\alpha^m\beta^n)$ is asymptotic to:
\begin{align*}
	&\asymp 
	(\alpha^+,\beta^{-n}o)_{\alpha^{-m}o}-(\alpha^+,\alpha^{-m}o)_{\beta^{-n}o}
	+
	(\beta^+,\alpha^{-m}o)_{\beta^{-n}o}-
	(\beta^+,\beta^{-n}o)_{\alpha^{-m}o}
	\\
	& \asymp 
	(\alpha^{m}o,\beta^{-n}o)_{\alpha^{-m}o}-
	(\alpha^{m}o,\alpha^{-m}o)_{\beta^{-n}o}+
	(\beta^{n}o,\alpha^{-m}o)_{\beta^{-n}o}-
	(\beta^{n}o,\beta^{-n}o)_{\alpha^{-m}o}
	\\
	& = d(\alpha^{m}o,\alpha^{-m}o)+
	d(\beta^{n}o,\beta^{-n}o)-
	d(\alpha^{m}o,\beta^{-n}o)-
	d(\beta^{n}o,\alpha^{-m}o)
	\\
	& \asymp 2 \log\Cr{\alpha^+,\alpha^-;\beta^+,\beta^-}
\end{align*}
as desired.
\end{proof}

\begin{remark}
Consider loxodromic $\alpha,\beta\in \Isom(\Hyp^n)$ with distinct fixed points. 
Applying identity \eqref{eq:translation-length_Cr-boundary_exp_1} to $(\alpha,\beta)$ and $(\alpha,\beta^{-1})$, and taking the ratio yields:
\begin{equation}
	\label{eq:translation-length_Cr-boundary_exp_2}
	\tfrac{\Cr{\alpha^-,\alpha^+;\beta^+,\beta^-}}{\Cr{\alpha^-,\alpha^+;\beta^-,\beta^+}}
	= \lim_{m,n\to+\infty}\exp \tfrac{1}{2} \left(\ell(\alpha^m\beta^{n})-\ell(\alpha^m\beta^{-n})\right)
\end{equation}
which measures the \emph{complex distance} between the axes $(\alpha^+,\alpha^-)$ and $(\beta^+,\beta^-)$ (see \Cref{rem:screw-axes-H3}).
\end{remark}

\begin{proposition}[translation length and limit set cross-ratio]
\label{prop:ell(Gamma)_Cr-Lambda(Gamma)}
For a subgroup $\Gamma\subset \Isom(X)$ with infinite limit set $\Lambda(\Gamma)\subset \partial X$, its translation length function $\ell \colon \Gamma \to \R_{\ge 0}$ determines and is determined by the cross-ratio on $\Lambda(\Gamma)^{4} \cap \partial^4 X$.
\end{proposition}
\begin{proof}
The cross-ratio determines the translation length since by assumption, for any loxodromic $\gamma\in \Gamma$ we may choose $\xi \in \Lambda(\Gamma)\setminus \{\gamma^\pm\}$, hence use equation \eqref{eq:translation-length_Cr-boundary_log}.
The translation length determines the cross-ratio on $\Lambda^4\cap \partial^4X$ because by \Cref{prop:loxodromic-axes-dense} the set of quadruples of distinct points of the form $(\alpha^-,\alpha^+;\beta^-, \beta^+)$ for loxodromic $\alpha,\beta\in \Gamma$ is dense in $\Lambda(\Gamma)^4$, so we conclude using equation \eqref{eq:translation-length_Cr-boundary_exp_1} and continuity of the cross-ratio on $\partial^4 X$.
\end{proof}

We will need the following lemma to characterize when a subgroup contains two elements with disjoint fixed points. A function between two topological spaces is called \emph{proper} when the preimage of a compact subset is compact.

\begin{lemma}[equal fixed points from length function]
\label{lem:equal-fixed-points-from-length}
%
Consider loxodromic $\alpha,\beta \in \Isom(X)$, and denote $\Mon(\alpha, \beta)$ and $\Grp(\alpha, \beta)$ the monoid and group that they generate.

We have $\{\alpha^-,\alpha^+\}\cap \{\beta^-,\beta^+\}=\emptyset \iff \exists p\in \N^* 
\colon (m,n)\mapsto \ell(\alpha^{pm}\beta^{pn})$ is proper on $\Z^2$.

We have $(\alpha^+=\beta^+)\vee(\alpha^-=\beta^-) \iff \ell \colon \Mon(\alpha, \beta) \to \R_{\ge 0}$ is a morphism of monoids.

We have $(\alpha^+=\beta^-)\vee(\alpha^-=\beta^+) \iff 
\ell \colon \Mon(\alpha, \beta^{-1}) \to \R_{\ge 0}$ is a morphism of monoids.

In particular $\alpha, \beta$ have disjoint fixed point pairs if and only if $\ell \colon \Grp(\alpha, \beta) \to \R_{\ge 0}$ is not the absolute value of a homomorphism to $\R$.

\end{lemma}

\begin{proof} 
We will first show all the implications $\implies$.

Suppose that $\{\alpha^-,\alpha^+\}\cap \{\beta^-,\beta^+\}=\emptyset$.
Fix a base point $o\in X$ and for $\delta>0$ and consider the open neighbourhood of $\xi\in \partial X$ defined by $V_o^\delta(\eta)=\{\xi \in \partial X \colon (\eta \mid \xi)_o<\delta\}$.
Choose $\delta>0$ so that $V_o^\delta(\alpha^\pm), V_o^\delta(\beta^\pm)$ are all disjoint at positive distance from one another for the visual metric $d^\varepsilon_o$.
There exist $p\in \N^*$ such that $\alpha^{\pm p}$ sends the complement of $V_o^\delta(\alpha^{\mp})$ into $U_o^\delta(\alpha^{\pm})$ and similarly $\beta^{\pm  k}$ sends the complement of $V_o^\delta(\beta^{\mp})$ into $U_o^\delta(\beta^{\pm})$.
In such a Schottky position, we may play ping-pong to deduce that the elements $a=\alpha^{p}$ and $a=\beta^{p}$ freely generate a rank-$2$ free group.
Denoting the $\{a^{\pm 1} ,b^{\pm1}\}$-word length by $\lvert \cdot \rvert \colon \Grp(a,b)\to \N$, there is by \cite[Equation (10.3.6)]{Das-Simmons-Urbanski_gromov-hyperbolic-spaces_2017} a $C\in \R_{>0}$ such that for all reduced $\gamma =a^{m_1}b^{n_1}\cdots a^{m_l}b^{n_l} \in \Grp(a,b)$, we have, denoting $C'=C\min\{\ell(a),\ell(b)\} \in \R_{>0}$, that
\begin{equation*}
	d(\gamma o,o) \geq C\sum_{k=1}^l d(a^{n_k}o,o)+d(b^{m_k}o,o)
	\geq C \sum_{i=1}^k\lvert m_k\rvert \ell(a)+\lvert n_k\rvert \ell(b) \geq C' \lvert \gamma \rvert
\end{equation*}
In particular, for all $m,n\in\Z$ we have that $\gamma=a^m b^n$ satisfies $\lvert \gamma ^l\rvert =l\lvert \gamma \rvert$ hence $\ell(\gamma)\geq C'\lvert \gamma \rvert$ proving properness of the length function on the subset $\{a^nb^m \colon n,m\in\Z\}$.

Suppose that $\alpha^+=\beta^+$.
Every non-trivial element $\gamma\in \Mon(\alpha, \beta)$ is loxodromic with attractive fixed point $\gamma^+=\alpha^+=\beta^+$, and using \cite[Proposition 4.2.16]{Das-Simmons-Urbanski_gromov-hyperbolic-spaces_2017} at that point yields the additivity of translation lengths (the cocycle rule f or $\ell(\gamma)=\log \lvert \Cr{\gamma \xi, \gamma^-, \xi, \gamma^+} \rvert$). 
The implications of $\alpha^-=\beta^-$ and of $(\alpha^+=\beta^-)\vee(\alpha^-=\beta^+)$ follow by applying the previous one to $(\alpha^{-1},\beta^{-1})$ and to $(\alpha^{+1},\beta^{-1})$.

Now we may easily prove the converse implications $\impliedby$. 

If $\alpha, \beta$ do not have disjoint fixed point pairs then any of the previous morphism properties implies that for all $p\in \N^*$ the function $(m,n) \mapsto \ell(\alpha^{pm}, \beta^{pn})$ is not proper on $\Z^2$.

If $\ell\colon \Mon(\alpha,\beta) \to \R_{\ge 0}$ is a morphism, then $\forall p\in \N^*$ the function $(m,n) \mapsto \ell(\alpha^{pm}, \beta^{pn})$ is not proper on $\Z^2$, so the contrapositive of the first implication implies $\{\alpha^\pm\}\cap\{\beta^\pm\}\neq\emptyset$. If $\alpha^+\neq\beta^+$ and $\alpha^-\neq\beta^-$, then $\alpha^+=\beta^-$ or $\alpha^-=\beta^+$. 
By the previous discussion, we deduce that $\ell \colon \Mon(\alpha,\beta^{-1})\to \R_{\ge 0}$ is a morphism of monoids. 
Consequently $\ell \colon \Grp(\alpha, \beta)\to \R_{\ge 0}$ is a group homomorphism, hence it is trivial, namely $\ell(\beta)=0$, a contradiction.
\end{proof}

\section{Geometry of the algebraic hyperbolic space}

For a cardinal $\kappa$, let $\Hyp^\kappa$ be the real algebraic hyperbolic space of dimension $\kappa$. 
Denote by $\infty$ the infinite countable cardinal, thus $\Hyp^\omega$ is the separable infinite dimensional hyperbolic space. 
The metric space $\Hyp^\kappa$ is $\operatorname{CAT}(-1)$, thus strongly hyperbolic for $\varepsilon = 1$ by \cite[Theorem 5.1]{Nica-Spakula_strong-hyperbolicity_2016}.
Consequently all the notions and results in the previous section apply.
In this section will recall the refinements of these results and the additional structures that will be needed in the sequel.

\subsection{Models in other geometries: Cartan and Iwasawa decompositions}
\label{subsec:Hn-models}


\paragraph{Minkowski-Lorentz hyperboloid model}

Denote the Hilbert space of dimension $\kappa$ by $\R^\kappa$.
The Minkowski space of dimension $1+\kappa$ is the oriented vector space $\R^{1+\kappa}$ with the symmetric bilinear form of signature $(1,\kappa)$ defined for $x,y\in\R^{1+\kappa}$ by \(\langle x,y\rangle =x_0y_0-\sum_{i>0}x_iy_i\).
Its isotropic cone $\mathbf{X}_{0} =\{x\in\R^{1+\kappa} \colon \langle x,x\rangle=0\}$ is asymptotic to the $2$-sheet hyperboloid $\mathbf{X}_{+1} =\{x\in\R^{1+\kappa} \colon \langle x,x\rangle=+1\}$ and the $1$-sheet hyperboloid $\mathbf{X}_{-1} =\{x\in\R^{1+\kappa} \colon \langle x,x\rangle=-1\}$.

The \emph{hyperboloid model} of $\Hyp^\kappa$ is obtained by endowing the upper-half component of the $2$-sheeted hyperboloid $\{x\in\mathbf{X}_{+1} \colon x_0>0\}$ with the distance $\dH$ given by
\(\cosh(\dH(x,y))=\langle x \mid y\rangle\).

\paragraph{Cayley-Klein projective model}

The \emph{projective model} for $\Hyp^\kappa$ is the subset $\Proj\mathbf{X}_{+1}$ of lines inside the cone, with the distance $\dH([x],[y])$ given in terms of the hyperbolic angle by:
\begin{equation}
\label{eq:hyperbolic-distance}
\cosh(\dH([x],[y]))=\tfrac{\langle x,y\rangle}{\sqrt{\langle x,x\rangle\langle y,y\rangle}}
\end{equation}
Here, the boundary $\partial \Hyp^\kappa$ identifies with the projectivized isotropic cone $\Proj \mathbf{X}_{0}$.
Note that the upper component $\{x\in\mathbf{X}_{+1} \colon x_0>0\}$ yields a section to the projectivization $\mathbf{X}_{+1}\to \Proj \mathbf{X}_{+1}$.

Hence, the isometry group $\Isom(\Hyp^\kappa)$ identifies with the projective orthogonal group $\PO(1,\kappa)$ (one can mimic the proof in finite dimension \cite[Theorem I.2.24]{Bridson-Haefliger_metric-non-positive-curvature_1999}).
It acts transitively on $\Hyp^\kappa=\Proj \mathbf{X}_{+1}$, and the stabilizer of $o=\Proj (1,0^\kappa)$ is the projective orthogonal group $\operatorname{PO}(\kappa)$, which acts freely transitively on orthonormal frames at $o\in \Hyp^\kappa$.

\paragraph{Euclidean ball model and Cartan decomposition}

The \emph{Euclidean ball model} is obtained from the projective space $\Proj\R^{1+\kappa}$ by choosing the Euclidean chart $\{x\in \R^{1+\kappa} \colon x_0=1\}$ with origin $o=(1,0^\kappa)$ so that the quadric pair $(\Proj\mathbf{X}_{0}, \Proj\mathbf{X}_{+1})$ is identified with the unit sphere and unit open ball $(\mathbf{S}^{\kappa-1}, \mathbf{B}^{\kappa})$.
At the centre $o$, the tangent space of $\Hyp^\kappa$ with its induced inner product is isometric to that of $\R^\kappa$ with its Euclidean inner product.
The geodesic map from the unit tangent sphere at $o\in \Hyp^\kappa$ to the Gromov boundary $\partial \Hyp^\kappa$ is a homeomorphism, and under the identification with $\Sph^{\kappa-1}$ the visual distance is given by $d_o(x,y)=\sin \tfrac{1}{2} (\widehat{xoy})$ (see \cite[§1.2]{Bourdon_cross-ratio-CAT(-1)_1996} or \cite[Lemma 3.5.1]{Das-Simmons-Urbanski_gromov-hyperbolic-spaces_2017}).
This yields an equivariant identification between isometry group of $\Hyp^\kappa$ and the conformal group of $\Sph^{\kappa-1}$.

Let us recall the Cartan decomposition of $\Isom(\Hyp^\kappa)$.
Consider polar coordinates for $\Hyp^\kappa$ given by the bi-infinite geodesic $(0,\infty) \in (\partial \Hyp^\kappa)^{(2)}$ passing through our fixed point $o\in \Hyp^\kappa$.
The stabilizer of such a geodesic contains a one parameter loxodromic group $A \colon \R \to \Isom(\Hyp^\kappa)$ with $A(t)$ acting on this axis by translation of $\lvert \ell(A(t))\rvert=t$.
Consider an orthonormal frame at a point $x\in \Hyp^\kappa$: on may act by $\operatorname{O}(\kappa)$ to send $x$ to the intersection of the geodesic arc $[o,\infty)$ with the sphere $\Sph_o(\dH(o,x))$; then by $A(-d(o,x))$ to bring this new point back to $o$; and finally by $\operatorname{PO}(\kappa)$ to adjust the frame.
Thus, every $\gamma\in \Isom(\Hyp^\kappa)$ has a unique factorization $\gamma=\gamma_1A(t)\gamma_2$ with $\gamma_i \in \operatorname{PO}(\kappa)$.

\paragraph{Affine model and special dimension $\le 3$.}

The \emph{affine model} for the boundary of the hyperbolic space $\partial\Hyp^\kappa$ is obtained in the projective space $\Proj (\R^{1+\kappa})$ by projecting $\Proj\mathbf{X}_0$ from a point $\infty\in \Proj\mathbf{X}_0$ onto the tangent space at another point $0\in \Proj\mathbf{X}_0$, yielding an identification $\partial \Hyp^\kappa=\R^{\kappa-1}\sqcup \{\infty\}$ hence an identification $\Hyp^\kappa = \R^{\kappa-1}\times (0,\infty)$ where the metric is $(dx+dy)/y^2$.
The group $\Isom(\Hyp^\kappa)=\PO(1,\kappa)$ acts transitively on $\partial \Hyp^\kappa = \R^\kappa$ and the stabilizer of $\infty\in \R^\kappa$ is the group of affine similitudes $\Isom(\R^\kappa) \rtimes \R_{>0} = (\R^\kappa \rtimes \operatorname{O}(\R^\kappa)) \rtimes \R_{>0}$. 

For $\kappa = 2, 3$ we find the upper-half-space models $\Hyp^2 \simeq (0,\infty) \times \R$ and $\Hyp^3 \simeq (0,\infty) \times \C$ identifying $\partial \Hyp^2 \simeq \R\Proj^1\simeq \{\infty\}\sqcup \R$ and $\partial \Hyp^3 \simeq \C\Proj^1\simeq \{\infty\}\sqcup  \C$.
The action of the pair $\Isom(\Hyp^2) \subset \Isom(\Hyp^3)$ on the pair of boundaries $\partial \Hyp^2\subset \partial \Hyp^3$ identifies with the action of the pair of $\PGL_2(\R)\subset \PGL_2(\C)$ on the pair of projective lines $\R\Proj^1\subset \C\Proj^1$.
Recall that the action of $\PGL_2(\mathbb{K})$ is simply transitive on triples of distinct points.
From this it follows that $\Isom(\Hyp^\kappa)$ acts triply transitively on $\partial \Hyp^\kappa$, since one may act on a triple of distinct points by moving one point at a time (and reordering the triple) while remaining each time in a copy of $\Hyp^3$, until reaching any other triple of distinct points.
After reordering the triples the initial triple $(x_0,y_0,z_0)$ and target triple $(x_1,y_1,z_1)$ so that their common points appear at the beginning in the same order, we then successively map $(x_0,y_0,z_0) \mapsto (x_0,y_0,z_1) \mapsto (x_0,y_1,z_1) \mapsto (x_1,y_1,z_1)$.

\subsection{Convex hull and Hausdorff distance to the \texorpdfstring{$1$-skeleton}{1-skeleton}}\label{sec:convex-hull}

For a subset $X\subset \Hyp^\kappa \sqcup \partial \Hyp^\kappa$, its \emph{span} $\Span(X)$ is the smallest closed totally geodesic and complete subspace $\Hyp^{\kappa'}$ containing all geodesics between pairs of points in $X$: in the linear and projective linear models, it corresponds to the linear and projective linear spans, intersected with $\mathbf{X}_{+1}$ and $\Proj \mathbf{X}_{+1}$.
Such a subset $X$ is called \emph{total} when it spans the whole space and is called \emph{independent} when any subset of $n\in \N$ points in $X$ span a totally geodesic copy of $\Hyp^{n-1}$.
In particular, two distinct geodesics in $\Hyp^\kappa$ span a unique totally geodesic copy of $\Hyp^2$ or $\Hyp^3$.

For a subset $X\subset \Hyp^\kappa \sqcup \partial \Hyp^\kappa$, its \emph{convex hull} $\Conv(X)\subset \Span(X)\cup \partial \Span(X)$ is the smallest closed geodesically convex subset that contains it: 
in the linear and projective linear models, it corresponds to the linear and projective linear closed convex hulls, intersected with the closed cone $\mathbf{X}_{\ge 0}$ and filled quadric $\Proj \mathbf{X}_{\ge 0}$.
The convex hull of $X$ can be constructed inductively, by considering its skeleta of increasing dimensions: the \emph{$k$-skeleton} $X^{(k)}$ is the union of all geodesic simplices constructed on $1+k$ points of $X$.
For instance, $X^{(1)}=\bigcup_{x_1,x_2 \in X}[x_1, x_2]$.


We obtained the following in \cite[Theorem 0.4]{Duchesne-Simon_simplices_2025} as a corollary of a more precise result concerning hyperbolic simplices, providing optimal bounds for the Hausdorff distances between their skeleta of various dimensions, and characterizing the simplices achieving that bound. 

\begin{proposition}[Hausdorff distance from convex hull to $1$-skeleton]
\label{prop:Hausdist-polytope-1-skeleton}
For a set of points $X \subset \Hyp^\kappa\cup \partial \Hyp^\kappa$, every point of its convex hull $\Conv(X)$ is at distance at most $\sinh^{-1}(1)=\log(1+\sqrt{2})$ from its $1$-skeleton $X^{(1)}=\bigcup_{x_1,x_2 \in Y}[x_1, x_2]$.
\end{proposition}

\begin{remark}[uniform bound]
The fact that in finite dimensions $\kappa$, the distance to the $1$-skeleton is uniformly bounded (independently of $Y$) was used in \cite[Theorem 3]{Bestvina_degenerations-hyperbolic-space_1988}, which refers to \cite{Bonahon-bouts-des-varietes-hyperboliques-de-dimension-3}.
We were motivated to obtain the uniformity of this bound on the dimension, as it will enable the statement about Gromov-Hausdorff convergence of our Theorem \ref{thm:tree-embedding-Hyp-infty} (proved in Step 7) to hold for infinite dimensions $\kappa_m$.

We refer to \cite{Duchesne-Simon_simplices_2025} for the optimal value of the bound and other interesting results concerning the characterization of hyperbolic simplices with maximal inradius.  
\end{remark}

\subsection{Cross-ratio for algebraic hyperbolic spaces}
\label{subsec:Hn-cross-ratio}

Let us express the cross-ratio on $\partial \Hyp^\kappa$ in terms of the Hilbert-Minkowski model $(\R^{1+\kappa}, \langle , \rangle)$. 

\begin{lemma}
For quadruples of distinct points $a^-,a^+,b^-,b^+$ in $\partial\Hyp^\kappa$, and for any of their representatives in the isotropic cone $\mathbf{X}_0$ of the Hilbert-Minkowski space $(\R^{1+\kappa}, \langle\ ,\ \rangle)$, we have
\begin{equation}
	\label{eq:cross-ratio_scalar-product}[a^-,a^+;b^-,b^+]^2=\frac{\langle a^-,a^+\rangle\langle b^-,b^+\rangle}{\langle a^-,b^+\rangle \langle b^-,a^+\rangle}
\end{equation}
\end{lemma}

\begin{proof}
Consider sequences $x_n,x'_n,y_n,y'_n\in  \Proj\mathbf{X}_{+1}$ respectively converging to $a^-,a^+,b^-,b^+\in \Proj\mathbf{X}_{0}$. 
By Equation \eqref{eq:hyperbolic-distance}, we have:
\begin{equation*}
	\frac{\langle x_n,x'_n\rangle\langle y_n,y'_n\rangle}{\langle x_n,y'_n\rangle \langle y_n,x'_n\rangle} 
	= \frac{\cosh(\dH(x_n,x'_n))\cosh(\dH(y_n,y'_n))}{\cosh(\dH(x_n,y'_n))\cosh(\dH(y_n,x'_n))}
\end{equation*}
The left-hand side converges to $\frac{\langle a^-,a^+\rangle\langle b^-,b^+\rangle}{\langle a^-,b^+\rangle \langle b^-,b^+\rangle}$.
By the defining Equation \eqref{eq:def-cross-ratio-X} for the cross-ratio in $\Hyp^\kappa$, and the limit $\cosh(r)/e^r\to 1/2$ as $r\to+\infty$, the right-hand side is asymptotic to $\Cr{x_n,x'_n,y_n,y'_n}^2$ as $n\to \infty$, which converges to $\Cr{a^-, a^+, b^-, b^+}^2$ by the continuous extension to quadruples of distinct points in the boundary $\partial \Hyp^\kappa$.
\end{proof}

\begin{remark}[projective cross-ratio from Veronese embedding]
\label{rem:Euclid-Veronese}
For $\Hyp^2$, the formula \eqref{eq:cross-ratio_scalar-product} recovers the absolute value of the projective cross-ratio of four points in $\R\Proj^1\subset \R\Proj^2$.

On the one had, given four lines $a^-, a^+, b^-, b^+ \in \Proj(\R^2)$, the absolute value of their projective cross-ratio can be obtained (as known to Euclid \cite[Book VI]{Euclide_Elements-VolII_1974}) by lifting them to vectors in $\R^2$ and considering the square root of the cross-ratio of areas: \[\sqrt{\left| \tfrac{\det(a^-,a^+)\det(b^-,b^+)}{\det(a^-,b^+)\det(b^-,a^+)}\right|}.\]

On the other hand, the space of real binary quadratic forms with the discriminant is a model for the Minkowski space $(\R^{1+2}, \langle,\rangle)$.
The Veronese map sending a vector $v \in \R^2$ to the degenerate form $q_v\colon u \mapsto \det(u,v)^2$ yields a quadratic function $\R^2 \to \mathbf{X}_0 \subset \R^{1+2}$ sending the determinant $\det(v,v')$ to the square of the scalar product $\langle q_v,q_{v'}\rangle^2$.
This projectifies to an isomorphism of projective lines $\Proj(\R^2)\to \Proj(\mathbf{X}_0)\subset \Proj(\R^{1+2})$ (see~\cite[Lemma 3.1]{CLS_Conj-PSL2K_2022}).

\end{remark}

\begin{lemma}[unicity of the cross-ratio]
\label{lem:unique_cross-ratio_invariant}
If a continuous function $(\partial\Hyp^\kappa)^{(4)} \to (0, +\infty)$ is invariant under the diagonal action of $\Isom(\Hyp^\kappa)$ and satisfies the \eqref{eq:Cr-Invariance} \& \eqref{eq:Cr-Cocycle} identities, then it is equal to a power of the cross-ratio, namely $\Cr{\cdot}^t$ for some $t\in [0,\infty)$.
\end{lemma}

\begin{proof}
Any four points are contained in a copy of $\partial \Hyp^3$.
In restriction to each copy of $\partial \Hyp^2\subset \partial \Hyp^3$ one can mimic the proof of \cite[Lemma 3.2]{Hamenstadt_Bounded-cohomology-cross-ratios_2009}, using the simple transitivity for the action of $\PGL_2(\R)$ on triples of distinct points in $\R\Proj^1$ and the \eqref{eq:Cr-Invariance} \& \eqref{eq:Cr-Cocycle}.
The $\Isom(\Hyp^3)$-invariance implies that the exponent $t$ does not depend on the copy of $\partial \Hyp^2\subset \partial \Hyp^3$, since $\Isom(\Hyp^3)$ acts transitively on the triples of distinct points hence on circles $\partial \Hyp^2\subset \partial \Hyp^3$.
\end{proof}

\begin{remark}[generalisation to abstract cross-ratios]
\Cref{prop:power_cross_ratio} will extend \Cref{lem:unique_cross-ratio_invariant} to abstract cross-ratios on a more general class of spaces.
\end{remark}

We thus deduce Corollary \ref{cor:cross-ratio-H3}, confirming that our normalization for the cross-ratio coincides with the modulus of the projective cross-ratio.

\begin{corollary}[projective cross-ratio]
\label{cor:cross-ratio-H3}
In the upper-half-space models $\Hyp^2  \simeq \mathbb{R}\times (0,+\infty)$ and $\Hyp^3 \simeq \mathbb{C}\times (0,+\infty)$ identifying $\partial \Hyp^2 \simeq \mathbb{RP}^1\simeq \mathbb{R}\cup\{\infty\}$ and $\partial \Hyp^3 \simeq \mathbb{CP}^1\simeq \mathbb{C}\cup\{\infty\}$, we find that the cross-ratio of $(a^+, a^-, b^+, b^-)$ is the modulus of their projective cross-ratio:
\begin{equation}    
	\Cr{a^-,a^+,b^-,b^+}=\left| \tfrac{(a^+-a^-)(b^+-b^-)}{(a^+-b^-)(b^+-a^-)} \right|.
\end{equation}

\end{corollary}

\begin{proof}
Remark \ref{rem:Euclid-Veronese} implies that the result holds for quadruples of distinct points contained in any conformal copy $\R\Proj^1$ inside $\C\Proj^1$.
Moreover the modulus of the projective cross-ratio on $\C\Proj^1$ satisfies the conditions in \Cref{lem:unique_cross-ratio_invariant}, so it is a power of the metric cross-ratio, which must be $1$ as they agree on $\R\Proj^1$.
\end{proof}

\begin{corollary}[recovering projective cross-ratio]
\label{rem:cross-ratio-argument}
For quadruple of points $(a^-,a^+;b^+,b^-)$ in $\partial \Hyp^3=\C\Proj^1$, the projective cross-ratio $\tfrac{(a^+-a^-)(b^+-b^-)}{(a^+-b^-)(b^+-a^-)}$ has argument $\pm\vartheta\bmod{2\pi}$ given in terms of $X^+=\Cr{a^-,a^+;b^-,b^+}$ and $X^-=\Cr{a^-,a^+;b^+,b^-}$ by
\(2\cos(\vartheta)= X^+(1+1/(X^+)^{2}-1/(X^-)^2)\).

Consequently, a quadruple of distinct points $(a^-,a^+;b^+,b^-)$ in $\partial\Hyp^3$ is uniquely determined up to $\Isom(\Hyp^3)$ by the cross-ratios $X^+=\Cr{a^-,a^+;b^-,b^+}$ and $X^-=\Cr{a^-,a^+;b^+,b^-}$.
\end{corollary}
\begin{proof}
Let $Z^+ = \tfrac{(a^+-a^-)(b^+-b^-)}{(a^+-b^-)(b^+-a^-)}$ and $Z^-=\tfrac{(a^+-b^-)(b^+-a^-)}{(a^+-a^-)(b^+-b^-)}$ so that $X^+=\lvert Z^+\rvert$ and $X^-=\lvert Z^-\rvert$.
The inverses of these complex projective cross-ratios sum up to one:
\begin{equation*}
	\tfrac{1}{Z^+}+\tfrac{1}{Z^-} =
	\tfrac{(a^+-b^-)(b^+-a^-)}{(a^+-a^-)(b^+-b^-)}+\tfrac{(a^+-b^+)(b^--a^-)}{(a^+-a^-)(b^--b^+)}=1.
\end{equation*}
Moreover $\vartheta = \arg(Z^+)$ satisfies $\lvert Z^+ -1\rvert^2 = \lvert Z^+\rvert^2 - 2\cos(\vartheta)\lvert Z^+\rvert + 1$, so after dividing by $\lvert Z^+\rvert$ and rewriting $\lvert Z^+ -1\rvert^2/\lvert Z^+\rvert = \lvert Z^+\rvert/\lvert Z^-\rvert^2$ we have the result.

Recall that $\Isom^+(\Hyp^3)=\PGL_2(\C)\times \Z/2$ acts uniquely transitively on quadruple of distinct points in $\partial\Hyp^3=\C\Proj^1$ with a given cross-ratio up to complex conjugacy, that is with a given modulus and argument $\bmod{\pi}$.
\end{proof}

\begin{corollary}[cocyclic points]
\label{cor:cocyclic-cross-ratio}
A quadruple of points $(a^+,a^-;b^+,b^-)$ in (a copy of) $\partial \Hyp^3=\C\Proj^1$ is cocyclic (belong to a copy of $\partial \Hyp^2=\R\Proj^1$) in this order if and only if we have: 
\begin{equation*}
	\Cr{a^+,a^-;b^+,b^-}^{-1}+\Cr{a^+,a^-;b^-,b^+}^{-1} = 1.
\end{equation*}
This identity characterizes when loxodromic elements in $\Isom(\Hyp^\kappa)$ have coplanar axes.
\end{corollary}
\begin{proof}
For $a^\pm,b^\pm \in \C\Proj^1$, the inverses of the complex projective cross-ratios sum up to one:
\begin{equation*}
	\tfrac{(a^+-b^-)(b^+-a^-)}{(a^+-a^-)(b^+-b^-)}+\tfrac{(a^+-b^+)(b^--a^-)}{(a^+-a^-)(b^--b^+)}=1
\end{equation*}
so the sum of the moduli is greater than one with equality if and only if both complex projective cross ratios were real numbers in $[0,1]$.

The same computation (expressing cross ratios in terms of visual distances) that the four points on a Euclidean sphere satisfy Ptoleme's inequality, with quality if and only if the four points are cocyclic.
(Note that any four points in Euclidean space lie on a common sphere.)
\end{proof}

\begin{remark}[complex distance]
\label{rem:screw-axes-H3}
For distinct $a^\pm,b^\pm \in \partial \Hyp^3$, the geodesics $(a^-,a^+),(b^-,b^+)\subset \Hyp^3$ are joined by a unique orthogeodesic arc of length $\lambda\in [0,\infty)$, along which they are screwed by an angle $\theta$: the complex number $\lambda+i\theta$ is called the \emph{complex distance} between them.
The isometry $C\in \Isom(\Hyp^3)$ sending $a^\pm$ to $b^\pm$ which is a translation-rotation along their common orthogeodesic has $\tfrac{1}{2}\Tr(C)=\cosh \tfrac{1}{2}(\lambda+i\theta)$.

Up to the action of $\Isom^+(\Hyp^3)=\PSL_2(\C)$ on $\partial \Hyp^3=\partial \C\Proj^1$, we may fix $a^\pm=\pm 1$ and find $b^\pm=\pm \exp(\lambda+i\theta)$ so that $C\colon z \mapsto \exp(\lambda+i\theta)z$.

We may recover the complex distance from the cross-ratios since with the previous notations we have the projective cross-ratios $Z^\pm = \pm4e^{\lambda+i\theta}/(1\pm e^{\lambda+i\theta})^2$ so their absolute values $X^\pm$ satisfy:
\begin{equation*}
	\tfrac{1}{\Cr{a^+,a^-;b^+, b^-}} - \tfrac{1}{\Cr{a^+,a^-;b^-, b^+}} 
	= \tfrac{1}{X^+}-\tfrac{1}{X^-}
	= \tfrac{\lvert 1+\exp(\lambda+i\theta)\rvert^2 - \lvert 1-\exp(\lambda+i\theta)\rvert ^2}{4 \exp(\lambda)} 
	= \cos(\theta)
\end{equation*}
\begin{equation*}
	\tfrac{\Cr{a^+,a^-;b^+, b^-}}{\Cr{a^+,a^-;b^-, b^+}}  
	= \tfrac{X^+}{X^-}
	= \tfrac{\lvert 1+\exp(\lambda+i\theta)\rvert^2}{\lvert 1-\exp(\lambda+i\theta)\rvert^2}
	=\left| \tanh \tfrac{1}{2}(\lambda+i\theta)\right|^2.
\end{equation*}

Note that for loxodromic $\alpha,\beta\in \Isom(\Hyp^3)$ with distinct fixed points, \Cref{eq:translation-length_Cr-boundary_exp_2} gives:
\begin{equation*}
	\tfrac{X^+}{X^-}=\tfrac{\Cr{\alpha^-,\alpha^+;\beta^+,\beta^-}}{\Cr{\alpha^-,\alpha^+;\beta^-,\beta^+}}
	= \lim_{m,n\to+\infty}\exp \tfrac{1}{2}\left(\ell(\alpha^m\beta^{n})-\ell(\alpha^m\beta^{-n})\right).
\end{equation*}
\end{remark}

\begin{example}[Cremona group action]
\label{eg:Cremona_cross-ratio}
Let $\Gamma=\operatorname{Bir}(\C\Proj^2)$ be the group of birational functions of the complex projective plane (also known as the Cremona group), and consider the representation $\rho\colon\Gamma\to\Isom(\Hyp^\omega)$ constructed in \cite{Cantat_group-birational-surfaces_2011}. 

An element $g\in \Gamma$ has a dynamical degree denoted $\lambda(g)$ and by \cite[Remark 4.5]{Cantat-Lamy-2013} it satisfies $\ell \rho(g)=\log\lambda(g)$. 
One may compare this with our \Cref{lem:Buseman-length-function} expressing translation lengths in terms of Buseman functions, and the relation between Buseman functions and dynamical derivatives \cite[§4.2.3]{Das-Simmons-Urbanski_gromov-hyperbolic-spaces_2017}.
By \Cref{eq:translation-length_Cr-boundary_exp_1}, for $\alpha,\beta\in\Gamma$ with positive dynamical degree and distinct fixed points, we have:
\begin{equation*}
	\log \Cr{\rho(\alpha)^-,\rho(\alpha)^+;\rho(\beta)^-,\rho(\beta)^+}=\lim_{m,n\to\infty}
	\tfrac{1}{2}\left(\lambda(\alpha^m)+\lambda(\beta^{n})-\lambda(\alpha^m\beta^{n}) \right).
\end{equation*}
\end{example}

\subsection{Isometry groups of algebraic hyperbolic spaces}
\label{subsec:Isom(Hyp)}

For algebraic hyperbolic spaces, the classification of non-trivial isometries \ref{prop:classification-g-Isom(X)} can be made more explicit (see~\cite[Proposition 6.1.11]{Das-Simmons-Urbanski_gromov-hyperbolic-spaces_2017}).

\begin{proposition}[classification of isometries]
If a non-trivial $\gamma \in \Isom(\Hyp^\kappa)$ is:
\begin{enumerate}[noitemsep, align=left]
	\item[elliptic] then it is conjugate to an element in the group $\operatorname{O}(\kappa-1)$ of isometries of $\Sph^{\kappa-1}$. 
	\item[parabolic] then it is conjugate to an element in the group $\Isom(\R^{\kappa-1})$ of affine isometries of $\R^{\kappa-1}$ which has no fixed points.
	\item[loxodromic] then it is conjugate to an element in the group $\operatorname{O}(\kappa-2)\times \R_+^*$ of linear similarities of $\R^{\kappa-1}$ with non-trivial dilatation factor.
\end{enumerate}
\end{proposition}

Consider a loxodromic $\gamma \in \Isom(\Hyp^\kappa)$.
Since $\Hyp^\kappa$ is $\operatorname{CAT}(-1)$, it follows from \Cref{rem:infimum-translation-length} that its stable translation length defined by formulae \eqref{eq:translation-length_limit} is equal to its infimum translation length:
\begin{equation*}
\ell(\gamma) = \inf\{\dH(o, \gamma(o)) \colon o\in \Hyp^\kappa \}.
\end{equation*}
The \emph{translation axis} of $\gamma$ is the set of points $o\in \Hyp^\kappa$ such that $\dH(o,\gamma o)=\ell(\gamma)$: it is a hyperbolic geodesic joining its attractive and repulsive fixed points $\gamma^+,\gamma^-\in \partial \Hyp^\kappa$, on which $\gamma$ acts by translation of $\ell(\gamma)$.
Moreover, for all $\xi \in \partial \Hyp^\kappa\setminus\{\gamma^-,\gamma^+\}$ we have $\ell(\gamma)=|\log  \Cr{\gamma \xi , \gamma^-, \xi, \gamma^+}| $.

\begin{definition}[elementary]
A subgroup $\Gamma\subset \Isom(\Hyp^\kappa)$ is \emph{elementary} when it has invariant subset of cardinal $1$ in $X\sqcup \partial X$ or of cardinal $2$ in $\partial X$. 
\end{definition}

\begin{remark}[on terminology]\label{rem:elementary}
Beware that conflicting terminologies appear in the literature. 

For example, \cite{Das-Simmons-Urbanski_gromov-hyperbolic-spaces_2017} and \cite{Kim_marked-length-rigidity-symmetric-space_2001} are not compatible, and we use a mix of them.
Their notions of elementarity agree, but ours is less restrictive (when there is a global fixed point at infinity but the limit set is infinite, these authors do not consider the action to be elementary, whereas we do).
Their notion of parabolicity is different, and we use the same as in \cite{Das-Simmons-Urbanski_gromov-hyperbolic-spaces_2017}.
\end{remark}

The elementary representations partition into the following classes:
\begin{enumerate}[noitemsep, align=left]
\item[\emph{elliptic}:] there is a global fixed point $o\in \Hyp^\kappa$, so $\Gamma$ can be conjugated in $\Isom(\Sph_o^{\kappa-1})$.
\item[\emph{parabolic}:] there is a global fixed point $\infty\in \partial \Hyp$ and its horospheres are all invariant, so $\Gamma$ can be conjugated in $\Isom(\R^{\kappa-1})$.
\item[\emph{focal}:] there is a global fixed point $\infty\in \partial \Hyp$ but horospheres are not invariant, so $\Gamma$ can be conjugated in $\Isom(\R^{\kappa-1})\rtimes \R_{>0}$ and is projection in $\R_{>0}$ is non-trivial.

\item[\emph{lineal}:] there is an invariant pair of points at infinity but no fixed point in $\Hyp^\kappa$, so $\Gamma$ can be conjugated in $ \Isom(\Sph_o^{\kappa-2})\rtimes \R^*$ and its projection in $\R^*$ is infinite.
\end{enumerate} 

\begin{remark}[elementary and length function]
An elementary representation may have zero length function (in which case it is elliptic or parabolic) or non-zero length function (in which case it is focal or lineal).
\end{remark}

\begin{lemma}[Tits alternative]
\label{lem:Tits-alternative}
A subgroup $\Gamma \subset \Isom(\Hyp^\kappa)$ is non-elementary if and only if it contains loxodromic elements $\alpha,\beta\in \Gamma$ whose fixed points $\alpha^\pm, \beta^\pm$ are all distinct.

In that case, there exists $p\in \N_{\ge 1}$ such that $a=\alpha^p$ and $b=\beta^p$ freely generate a rank-$2$ free group, and the restricted length-function $\Z^2 \ni (m,n)\mapsto \ell(a^mb^n) \in \R_{\ge 0}$ is proper.
\end{lemma}
\begin{proof}
Since $\Gamma$ is non-elementary then it has non-zero length function so it admits loxodromic elements.
If any two loxodromic elements in $\Gamma$ have a common fixed point, then it would be lineal or focal, so we may find $\alpha,\beta$ with distinct fixed points.
The second statement follows from (the proof) of \Cref{lem:equal-fixed-points-from-length}.
\end{proof}

The following \Cref{cor:elem-Hyp_length-homomorphism} will serve in \Cref{thm:length-non-neutral-classes} to separate the length functions associated to elementary and non-elementary representations of groups.

\begin{corollary}[elementary from homomorphisms]
\label{cor:elem-Hyp_length-homomorphism}
A subgroup $\Gamma\subset \Isom(\Hyp^\kappa)$ is elementary if and only if its length function $\ell\colon \Gamma\to \R_+$ is the absolute value of a homomorphism $\Xi \colon \Gamma \to \R$. 
\end{corollary}
\begin{proof}
This follows from \Cref{lem:equal-fixed-points-from-length}, but let us spell out the proof for clarity.

Assume $\Gamma$ is elementary. 
If $\Gamma$ is elliptic, then we set $\Xi$ to be the trivial homomorphism. 
If $\Gamma$ has a fixed point $\xi\in \partial X$, then we set $\Xi$ to be the Busemann homomorphism $g\mapsto \Bus_\xi(gx,x)$ for any $x\in\Hyp^\kappa$ associated to $\xi$ (see \cite[Chapter II.8]{Bridson-Haefliger_metric-non-positive-curvature_1999} for details about Busemann functions).
If $\Gamma$ is lineal then it leaves invariant some geodesic line $(\xi, \xi')\subset X$ and the translation parameter along this geodesic yields the desired homomorphism $\Xi$.
In all cases, we have $\ell(\gamma)=\lvert \Xi(\gamma)\rvert$.

Conversely, assume that $\Gamma$ is non-elementary, then $\Gamma$ contains two loxodromic elements $\alpha,\beta$ with disjoint fixed point sets at infinity by \Cref{lem:Tits-alternative} so by \Cref{lem:equal-fixed-points-from-length} there exists $p\in \N_{\ge 1}$ such that the restricted length-function $\Z^2 \ni (m,n)\mapsto \ell(a^{pm}b^{pn}) \in \R_{\ge 0}$ is proper, so there cannot be a homomorphism $\Xi \colon\Gamma\to\R$ such that $\ell(\gamma)=\lvert\Xi(\gamma)\rvert$.   
\end{proof}

When a subgroup is non-elementary, its limit set $\Lambda(\Gamma)$ is infinite and is the smallest closed $\Gamma$-invariant subset of $\partial\Hyp^\kappa$ (see \cite[§7.2-§7.3]{Das-Simmons-Urbanski_gromov-hyperbolic-spaces_2017}). 

\begin{lemma}[convex core]
Consider a subgroup $\Gamma \subset \Isom(\Hyp^\kappa)$ which is non-neutral.

There is a smallest non-empty (closed) invariant convex subspace of $\Hyp^\kappa$: that is the convex hull of $\Lambda(\Gamma)$.
We call it the \emph{convex core}, and denote it $\operatorname{Core}(\Gamma)$.

There is a smallest non-empty (closed) totally geodesic invariant subspace: that is the span of its limit set $\Span(\Lambda(\Gamma))\subset \Hyp^\kappa$.
\end{lemma}

\begin{proof}
If the representation is not elementary, this is \cite[§7.2-7.3]{Das-Simmons-Urbanski_gromov-hyperbolic-spaces_2017}).

If the representation is not neutral, then there is an loxodromic element and thus the limit has at least two points. Any convex subset $C$ that is invariant contains $\Lambda(\Gamma)$ in its boundary because it does not depend on the chosen base point. 
By the continuity of geodesics with respect to their endpoints (\cite[Proposition 4.4.4]{Das-Simmons-Urbanski_gromov-hyperbolic-spaces_2017}), the convex hull $C$ contains all geodesics between points in $\Lambda(\Gamma)$. 
In particular, the intersection of all invariant closed convex subsets is non-empty and thus this is the unique minimal invariant closed convex subset.

The same argument gives the same result for invariant totally geodesic subspaces.
\end{proof}

Recall that $\Isom(\Hyp^\kappa)=\PO(1,\kappa)$, and an isometric action of a group $\Gamma$ on $\Hyp^\kappa$ may be lifted to a linear action on $\R^{1+\kappa}$ preserving $\R_{>0
}\times\R^\kappa$ (or the upper half of $\mathbf{X}_{+1}$).

\begin{definition}[minimal and irreducible]
A representation $\Gamma\to\Isom(\Hyp^\kappa)$ is \emph{minimal} when its action on $\Hyp^\kappa$ admits no proper closed invariant totally geodesic subspace. 

A representation $\Gamma \to \PO(1,\kappa)$ is \emph{irreducible} when its corresponding linear action on $\R^{1+\kappa}$ does not preserve a non-trivial closed linear subspace.
\end{definition}

\begin{remark}[irreducible = minimal and non-elementary]
\label{rem:irred=>eleminimal}
It follows from the definitions that a representation is irreducible if and only if it is minimal without fixed point at infinity. 
Indeed in $\R^{1+\kappa}$, a non-trivial closed linear subspace admits a non-trivial closed orthogonal complement: at least one of these intersects the isotropic cone or the unit hyperboloid $\mathbf{X}_{0} \cup \mathbf{X}_{+1}$, which projectifies to $\Hyp^\kappa \cup \partial \Hyp^\kappa$.
\end{remark}

\section{From kernels to isometric embeddings}

In this section we will consider topological spaces $X,Y$ and topological groups $\Gamma, G$ acting by homeomorphisms on these spaces, but we will mostly focus on sets and groups with the discrete topology.

\begin{definition}[kernel]
\label{def:algebraic-kernel}
On a topological space $X$, a \emph{kernel} is a continuous function $K\colon X^2 \to \R$ that is symmetric ($\forall x,y\in X \colon K(x,y)=K(y,x)$).
We say that it has diagonal $s\in \R$ when $\forall z\in X \colon K(z,z)=s$. We say that it is non-negative when $\forall x,y\in X \colon K(x,y)\ge 0$.
\end{definition}

We recall here some topological background concerning the topologies that one may endow the spaces of functions over $X$ and the space of representations $\Gamma \to \Isom(X)$.

\begin{definition}[pointwise topology]
For a topological space $X$ and $p\in \N$, the set of continuous functions $X^p\to \R$ is endowed with the pointwise convergence topology inherited by the product topology when viewed as a subset of $\prod_{X^p}\R$. 

It is a locally convex topological vector space.
Its compact convex sets correspond to the products of compact intervals, of the form $\prod_{x\in X^p} [min_x, max_x]$.
\end{definition}

\begin{lemma}[closed subspaces]
\label{lem:pointwise-closed}
On a space $X$, the subset of continuous functions $K\colon X^p\to \R$ cut out by any collection of non-strict equalities between continuous functions of the values of $K$ at finite number of points at a time is closed. 

More precisely, if for all $n\in \N$ we have a set $L_n$ (which may have any cardinal) indexing continuous functions $D^n_l\colon \R^{n^p}\to \R$, then the following formula defines a closed subspace: \[\forall n\in \N,\; \forall l\in L_n,\; \forall x\in X^{n^p} \colon D^n_l(K(x_{i_1},\dots,x_{i_p}))\ge0.\]
\end{lemma}

\begin{example}[group as set]
On a topological group $\Gamma$, the subset of functions $K\colon \Gamma^{1+p}\to \R$ which are left invariant is closed and identifies with the space of all functions $F\colon \Gamma^p\to \R$ by $K(\gamma_0, \dots,\gamma_n)=F(\gamma_{0}^{-1}\gamma_{1}, \gamma_1^{-1}\gamma_2,\dots, \gamma_{n-1}^{-1}\gamma_n)$ and $F(\gamma_1, \dots,\gamma_n)=K(1, \gamma_1, \gamma_1\gamma_2,\ldots, \gamma_1\dots\gamma_n)$. 

\end{example}

\subsection{Kernels of positive type and of conditionally negative type}
\label{subsec:kernel-positive-conditionegative}

We recall some facts about kernels of positive type and of conditionally negative type (see for example \cite[Chapter 3]{Berg-Christensen-Ressel_Harmonic-analysis-semigroups_1984} with slightly different terminology) , since we wish to emphasize some results about linear independence that we will need in the next sections.

\begin{definition}[positive and conditionally negative type]
\label{def:kernel-positive-cond-negative-type}
A kernel $K\colon X\times X\to \R$
\begin{itemize}[align=left]
	\item[has \emph{positive type}] when for all $x \in X^n$ we have:
	\begin{equation}
		\textstyle
		\label{eq:kernel-strictly-positive-type}
		\forall \lambda \in \R^n\setminus\{0\}
		\colon \sum_{i,j}\lambda_i\lambda_jK(x_i,x_j)\ge 0
	\end{equation}
	and \emph{strictly positive type} when such inequalities are strict for distinct $x_k$.
	
	\item[has \emph{conditionally negative type}] when it has diagonal $0$, and for all $x \in X^n$ we have:
	\begin{equation}
		\textstyle
		\label{eq:kernel-conditionally-negative-type}
		\forall \lambda \in \R^n\setminus\{0\}
		\colon 
		\sum_{k}\lambda_k=0 \implies \sum_{i,j}\lambda_i\lambda_jK(x_i,x_j)\le 0.
	\end{equation}
	and \emph{strictly conditionally negative type} when such inequalities are strict for distinct $x_k$.
\end{itemize}
\end{definition}

For a cardinal $\kappa$ we denote the Hilbert space of that dimension by $\R^\kappa$. 
Recall that a subset $Y$ of a Hilbert is \emph{total} when there is no closed linear subspace containing $Y$ and it is \emph{affinely total} when there is no closed affine subspace containing $Y$.

\begin{theorem}[linear and affine GNS constructions]
\label{thm:Hilbert-GNS}
Consider a kernel $K\colon X\times X \to \R$.

If $K$ has positive type then there is a cardinal $\kappa$ and a continuous $f\colon X\to\R^\kappa$ with total image such that $\forall x,y\in X\colon K(x,y)=\langle f(x), f(y)\rangle$. 
Moreover it is unique: for any two such $f_i\colon X\to\R^{\kappa_i}$, there is a linear isometry $\varphi\colon \R^{\kappa_1}\to\R^{\kappa_2}$ such that $f_2=\varphi\circ f_1$.

If $K$ has conditionally negative type then there is a cardinal $\kappa$ and a continuous $f\colon X\to\R^\kappa$ with affinely total image such that $\forall x,y\in X \colon K(x,y)=\|f(x)-f(y)\|^2$.
Moreover it is unique: for any two such $f_i\colon X\to\R^{\kappa_i}$, there is 
an affine isometry $\varphi\colon \R^{\kappa_1}\to\R^{\kappa_2}$ such that $f_2=\varphi\circ f_1$.
\end{theorem}
\begin{proof}
For both (positive and conditionally negative) types, the existence and unicity of $f$ are explained constructively in \cite[Chapter 3]{Berg-Christensen-Ressel_Harmonic-analysis-semigroups_1984}.
They also follow from the Cayley-Merger theorem and the work of Blumenthal \cite[§42]{Blumenthal_distance_geometry_1950}, as we will explain in more detail for kernels of hyperbolic type in the proof of \Cref{prop:hyperbolic-GNS}.
\end{proof}

\begin{remark}[separation means injectivity]
\label{rem:separation=injectivity}
We say that the kernel $K\colon X\times X \to \R$ of positive or conditionally negative type \emph{separates points} (or is \emph{separating}) when, for all distinct $x,y\in X$ we have:
\begin{itemize}[noitemsep, align=left]
	\item[$\langle\ ,\ \rangle$:] when $K$ is of positive type that $K(x,y)<\max\{K(x,x),K(y,y)\}$ 
	\item[$\|\ \|^2$:] when $K$ is of conditionally negative type that $K(x,y)\neq0$. 
\end{itemize}
The fact that $K$ separates points is equivalent to the injectivity of the functions $f$ in \Cref{thm:Hilbert-GNS}. 
The strictness of $K$ implies that $K$ separates points.
\end{remark}

The strict inequalities respectively characterize linear and affine independence of the images under the associated map $f\colon X\to\R^\kappa$ to a Hilbert space. 

\begin{lemma}[strict means independence]
\label{lem:strct-means-indep_positive-conditionegative}
Let $K\colon X\times X\to \R$ be a kernel.

If $K$ has strictly positive type, then for any continuous $f\colon X\to\R^\kappa$ such that $K(x,y)=\langle \varphi(x),\varphi(y)\rangle$, the images of  points in $X$ are linearly independent in $\R^\kappa$, in particular $f$ is an embedding.

If $K$ has strictly conditionally negative type, then for any continuous $f\colon X\to\R^\kappa$ such that $K(x,y)=\| f(x)-f(y)\|^2$, the images of points in $X$ are affinely independent in $\R^\kappa$, in particular $f$ is an embedding.    
\end{lemma}

\begin{proof}
Consider a map $f\colon X\to\R^\kappa$ such that for all $x,y\in X$ we have either $K(x,y)=\langle f(x),f(y)\rangle$ or $K(x,y)=\| f(x)-f(y)\|^2$ depending on the type. 

For $n\in \N$, consider distinct points $x_1,\dots,x_n \in X$ and $\lambda\in \R^n\setminus\{0\}$.

We have $\left\|\sum_i\lambda_i f(x_i)\right\|^2=\sum_{i,j}\lambda_i\lambda_j\langle f(x_i), f(x_j) \rangle$, which must be $>0$ when $K$ is strictly of positive type, hence $f(X)$ is linearly independent and $f$ is injective.

Assuming $\sum_{k}\lambda_k=0$ we have $\|\sum_i\lambda_i f(x_i)\|^2=-\tfrac{1}{2}\sum_{i,j}\lambda_i\lambda_j \|f(x_i)-f(x_j)\|^2$, which must be $<0$ when $K$ is strictly of conditionally negative type, hence $f(X)$ is affinely independent and $f$ is injective.
\end{proof}

The following will serve to prove \Cref{prop:strictly_of_hyperbolic_type}.

\begin{lemma}[power series of kernels of positive type are strict]
\label{lem:powerseries-of-positive-type}
Consider a separating kernel $K\colon X\times X \to \R$ of positive type such that $\forall x\in X\colon K(x,x) < 1$ and injective associated map $f\colon X\to\R^\kappa$.

For a power series $Q(z)=\sum_{m\geq1}q_mz^m$ with coefficients $q_m\ge 0$ and convergence radius $1$, the kernel $Q\circ K$ is strictly of positive type.
\end{lemma}

\begin{proof}
Let $x\in X^n$ and denote $K_{i,j}=K(x_i,x_j)$.
Let $\lambda\in \R^n$ such that $\sum_{i,j}\lambda_i\lambda_jQ(K_{i,j})=0$, namely $\sum_m q_m\left(\sum_{ij}\lambda_i\lambda_j K^m_{i,j}\right)=0$.

By Schur theorem \cite[Chapter 3, Theorem 1.12]{Berg-Christensen-Ressel_Harmonic-analysis-semigroups_1984}, for all $m\geq1$ the kernel $K^m$ is of positive type, implying that $\left(\sum_{ij}\lambda_i\lambda_j K^m_{i,j}\right)\ge 0$.
Since the convergence radius of $Q$ is $1$ the inequality $q_m\ge 0$ is strict for infinitely many $m$. 
Thus for such $m\geq1$ we have 
$\sum_{i,j}\lambda_i\lambda_jK^m_{i,j}=0$.

Since the kernel $K$ is of positive type, it satisfies the Cauchy-Schwartz inequalities, in particular $\forall i,j\colon \lvert K_{i,j}\rvert^2\leq K_{i,i}K_{j,j}$. Hence for $1\le i\neq j\le n$, by \Cref{rem:separation=injectivity}, we have $\lvert K_{i,j}\rvert<\mu:=\max\left\{K_{i,i}\right\}$.

We deduce from all the above that for infinitely many $m\in \N$ we have $\sum_{i,j}\lambda_i\lambda_j\left(K_{i,j}/\mu\right)^m=0$ and taking the limit on that diverging subsequence of $m$ to $+\infty$ reveals
$\sum_{K_{i,i}=\mu} \lambda_i^2=0$,
namely $K_{i,i}=\mu \implies \lambda_i=0$. 
Repeating the argument with $\max\{K_{j,j} \colon K_{j,j}<\mu\}$ eventually leads to showing that all $\lambda_i=0$.
This proves that $Q\circ K$ is strictly of positive type.
\end{proof}

The following will serve to prove \Cref{prop:independence-exotic-embeddings}
\begin{lemma}[powers of kernels of conditionally negative type are strict]
\label{lem:power_kernel_negative_type}
Consider a kernel $K\colon X\times X\to \R$ of conditionally negative type vanishing only on the diagonal.

For all $t\in (0,1)$ the kernel $K^t$ has strictly conditionally negative type.
\end{lemma}

\begin{proof}
Fix $t\in(0,1)$.
Let $x_1,\dots,x_n$ be distinct points in $X$ and $\lambda\in \R^n\setminus\{0\}$ with $\sum_i \lambda_i=0$. 
Assume by contradiction that $\sum_{i,j}\lambda_i\lambda_jK(x_i,x_j)^t=0$.
By the inverse Laplace transform, we have for any $k\geq0$ that
\begin{equation}
	\label{eq:laplace_transform}
	k^t=\frac{t}{\Upgamma(1-t)}\int_{0}^{+\infty}(1-e^{-uk})\frac{\mathrm{d}u}{u^{t+1}}
\end{equation} 
Applying this integral representation to the values of $K^t$ in the finite sum $0=\sum_{i,j}\lambda_i\lambda_jK^t(x_i,x_j)$ and using $\sum_{ij} \lambda_i\lambda_j=(\sum_k \lambda_k)^2=0$ we find $0=-\int_0^{+\infty} \left[\sum_{ij}\lambda_i\lambda_j \exp(-uK(x_i,x_j)) \right] \tfrac{du}{u^{1+t}}$.
By Shoenberg's theorem \cite[Chapter 3, Theorem 2.2]{Berg-Christensen-Ressel_Harmonic-analysis-semigroups_1984}, $\forall u\in (0,+\infty)$ the kernel $\exp(-uK)$ is of positive type, so the integrand $\sum_{i,j}\lambda_i\lambda_j \exp(-uK(x_i,x_j))$ is $\ge 0$, hence for almost all $u\in (0,\infty)$ it is $=0$. By continuity it vanishes for all $u\in (0,\infty)$. 
Taking the limit as $u\to+\infty$ reveals $\sum_{K(x_i,x_j)\ne 0}\lambda_i\lambda_j=0$.
Since $K$ only vanishes on the diagonal, this says $\sum_{i\ne j}\lambda_i\lambda_j=0$.
Substracting this from $0=(\sum_k \lambda_k)^2$ yields $\sum_k \lambda_k^2=0$ hence $\lambda=0$ which is a contradiction.
Thus $\sum_{i,j}\lambda_i\lambda_jK^t(x_i,x_j)\ne 0$ proving that $K^t$ is strictly of conditionally negative type.
\end{proof}


\subsection{Kernels of hyperbolic type and embeddings in \texorpdfstring{$\Hyp^\kappa$}{Hkappa}}\label{subsec:kernel_hyperbolic_type}

We recall the definition of kernels and functions of hyperbolic type and the main results about them, most of which can be found in \cite{Monod-Py_self-representations-Mobius_2019}.

\begin{definition}[kernel of hyperbolic type]
\label{def:kernel-hyperbolic-type}
A kernel $K\colon X\times X\to \R$ is of \emph{hyperbolic type} when it is nonnegative with diagonal $1$, and for any $x_0\in X$ the \emph{visual kernel} $K_0 \colon X\times X \to \R_{\ge0}$ defined by $K_0\colon (x_i,x_j)\mapsto K(x_0, x_i)K(x_0,x_j)-K(x_i,x_j)$ is positive type, namely we have the following \emph{hyperbolic Cauchy-Schwartz inequalities}:
\begin{equation}
	\label{eq:kernel-hyperbolic-CS}
	\tag{HCS}
	\textstyle 
	\forall n\in \N,\;
	\forall x\in X^n,\;
	\forall \lambda \in \R^n \colon 
	\quad
	\sum_{i,j=1}^n\lambda_i\lambda_jK(x_i,x_j)\leq \left( \sum_{k=1}^n \lambda_k K(x_0,x_k)\right)^2.
\end{equation}
(In particular with $n=1$, a kernel of hyperbolic type must take values in $[1,\infty)$.)
\end{definition}
\begin{proof}
The condition \eqref{eq:kernel-hyperbolic-CS} rewrites as $\forall \lambda_1,\dots,\lambda_n\in \R \colon \sum_{i,j} \lambda_i\lambda_j K_0(x_i,x_j)\ge 0$, meaning that $K_0$ yields a nonnegative quadratic form on the $\R$-vector space with basis $\{x_1,\dots,x_n\}$.
\end{proof}

Let us relate this to the equivalent characterization in terms of Gram-matrix determinants (see \cite[7.2.4]{Ratcliffe_Foundations-Hyperbolic-Manifolds_2019}) also known as Cayley-Merger determinants introduced by Blumenthal (see \cite[§42, §46]{Blumenthal_distance_geometry_1950}).

\begin{proposition}[Cayley-Menger determinant]
\label{def:hyperbolic_Cayley_Menger}

A kernel $K\colon X\times X\to \R$ with diagonal $1$ is of hyperbolic type if and only if all its \emph{hyperbolic Cayley Menger determinants} are non-positive:
\begin{equation}
	\label{eq:kernel-hyperbolic-CM}
	\tag{HCM}
	\forall n\in \N,\;
	\forall x\in X^{1+n} 
	\colon \quad 
	\det \left(-K(x_i,x_j)\right)_{0\le i,j\le n} \le 0
\end{equation}
\end{proposition}
\begin{proof}
Recall Sylvester's law expressing the signature of a real quadratic form in terms of the signs of its diagonal-minors: for any $x_0\in X$, the visual kernel $K_0$ yields a nonnegative quadratic form on the $\R$-vector space with basis $X$ if and only if $\forall x \in X^n \colon \det \left(K_0(x_i, x_j)\right)_{1\le i,j\le n} \ge 0$.

Besides $\forall x \in X^{1+n}$ the determinant of $\left( - K(x_i,x_j) \right)_{0\le i,j \le k}$ can be rewritten (by subtracting to each column $j\ge 0$ the column $0$ times $K(x_0,x_j)$ and expanding with respect to the line $0$) to equal minus the determinant of $\left(K_0(x_i,x_j)\right)_{1\le i,j\le n}$.
Hence \eqref{eq:kernel-hyperbolic-CS} and \eqref{eq:kernel-hyperbolic-CM} are equivalent.
\end{proof}

The notion of kernel of hyperbolic type is for algebraic Minkowski spaces analogous to the notion of kernel of positive type for Hilbert spaces. Its interest lies in \cite[Proposition 3.3]{Monod-Py_self-representations-Mobius_2019}.

\begin{proposition}[hyperbolic GNS]
\label{prop:hyperbolic-GNS}
For a subset $X\subset \Hyp^\kappa$, the function $K(x,y)=\cosh \dH(x,y)$ is a kernel of hyperbolic type. 

Conversely, for every kernel of hyperbolic type $K\colon X\times X\to \R$, there is a cardinal $\kappa$ and a continuous function $f \colon X\to \Hyp^\kappa$ such that $K(x,y)= \cosh \dH(f(x),f(y))$.
Moreover for any two such $f_i\colon X\to\Hyp^{\kappa_i}$ with total images, there is an isometry $\varphi\colon \Hyp^{\kappa_1}\to\Hyp^{\kappa_2}$ such that $f_2=\varphi\circ f_1$.

Additionally when $\kappa$ is countable, it is the supremum of $n\in \N$ such that there exists a non-vanishing hyperbolic Cayley-Menger determinant of size $1+n$. 
\end{proposition}

\begin{proof}
We sketch proofs to explain some ideas and relate them to past works on the topic.

Let us show that the kernel $(\cosh \circ \dH)$ on $\Hyp^\kappa$ satisfies \Cref{def:kernel-hyperbolic-type}.
It is symmetric and normalized to $1$ on the diagonal, and we now show the Hyperbolic Cauchy-Schwartz inequalities.
It lifts in $\R^{1+\kappa}$ to the Minkowski form $\langle , \rangle$ on the upper-half sheet of the unit hyperboloid $\mathbf{X}_{+1}$. The Minkowski form also restricts to a non-degenerate symmetric bilinear form that is negative definite on the tangent space to $\mathbf{X}_{+1}$ at $x_0$.
The orthogonal projections of $x_i\in \mathbf{X}_{+1}$ into that tangent space are $v_i = x_i-\langle x_0, x_i,\rangle x_0$ so the sum $v=\sum \lambda_i v_i$ has Minkowski norm $\langle v, v \rangle\ge 0$. 
This implies that the visual kernel $K_0(x_i, x_j)=\langle x_0, x_i \rangle \langle x_0, x_j \rangle-\langle x_i,x_j\rangle$ is non-negative.

Conversely, if a space $X$ with a nonnegative kernel $d$ normalized to $0$ on the diagonal is such that the kernel $K= \cosh \circ d$ has non-positive Hyperbolic Cayley-Menger determinants, then it embeds isometrically in $\Hyp^\kappa$.
One may refer to \cite[§42]{Blumenthal_distance_geometry_1950} for a thorough analysis of the Euclidean analogue, and to \cite[§106]{Blumenthal_distance_geometry_1950} for its extension to the hyperbolic case. 
Moreover we may characterize when the image lies on some finite dimensional subspace of $\Hyp^\kappa$ and its dimension is the minimal $n$ such that all hyperbolic Cayley-Menger determinants with more than $n+1$ points vanish. Indeed for any $x_0,\dots,x_n\in\Hyp^n$, the determinant of $\left(-\cosh \dH(x_i,x_j)\right)$ vanishes if and only if $x_0,\dots,x_n$ are included in a proper totally geodesic subspace \cite[Theorem 106.1] {Blumenthal_distance_geometry_1950} and $(X,d)$ embeds in $\Hyp^n$ if and only if it is true for any of its $(n+3)$-tuple \cite[Theorem 111.2]{Blumenthal_distance_geometry_1950}.
\end{proof}

\begin{remark}[from scalar product to chord-length-square]
\label{rem:chord-length-kernel}
Consider a kernel $C \colon X \times X \to \R$ with diagonal $1$.
One may define a related kernel given by the \emph{chord length-squared}, defined by $S = (\sinh \tfrac{1}{2}\dH)^2 = (\sinh \tfrac{1}{2}\cosh^{-1}(C))^2$.
It is shown in \cite[Theorem 3]{Gomez-Memoli_Ptoleme-tropical_2024} that $\sign \det \left(S(x_i,x_j)\right)_{0\le i,j\le n} = -\sign \det\left(C(x_i,x_j)\right)_{1\le i,j\le n}$. 
Thus, one may also rephrase the condition \Cref{eq:kernel-hyperbolic-CM} characterizing the fact that $X$ embeds into some $\Hyp^\kappa$ by saying that $\forall x\in X^{1+n}$ $\det \left(-S(x_i,x_j)\right)\ge 0$.

By \cite[Corollary 4]{Gomez-Memoli_Ptoleme-tropical_2024}, these sign conditions on the $4\times 4$ determinants on $S$ recover the Ptolemaic inequality for hyperbolic spaces, which we
will 
use in \Cref{sec:hyperbolic_kernels_to_strong_hyperbolicity}:
\begin{equation}
	\left(\sinh \tfrac{1}{2}\delta_{13}\right) \left(\sinh \tfrac{1}{2}\delta_{24}\right)
	\le 
	\left(\sinh \tfrac{1}{2}\delta_{12}\right)
	\left(\sinh \tfrac{1}{2}\delta_{34}\right)
	+
	\left(\sinh \tfrac{1}{2}\delta_{14}\right)
	\left(\sinh \tfrac{1}{2}\delta_{23}\right)
\end{equation}
Hence the sign conditions for the determinants of the chord lengths squared matrices of larger sizes can correspond to generalizations of this Ptolemaic inequality.
\end{remark}

\begin{definition}[strictly of hyperbolic type]
We say that a kernel of hyperbolic type $K\colon X\times X\to \R_{\ge 1}$ is \emph{strictly of hyperbolic type} when for any $x_0\in X$, the visual kernel $K_0\colon (x,y)\mapsto K(x,x_0)K(x_0,y)-K(x,y)$ is strictly of positive type. 
\end{definition}

This strictness condition is again an independence property as explained in the following \Cref{prop:strct-means-indep_hyperbolic}.

This independence property appears for the embeddings associated to real trees in \cite[Theorem 13.1.1(iv)]{Das-Simmons-Urbanski_gromov-hyperbolic-spaces_2017} and is stated in \cite[Remark 4.3]{Courtois-Guilloux_Hausdorff-infinite-dim-hyp_2024} for the exotic deformations of hyperbolic spaces $\Hyp^n$ that we will see in \Cref{subsec:exotic-deformation}.

\begin{proposition}[strict means independence]
\label{prop:strct-means-indep_hyperbolic}
A kernel of hyperbolic type $K\colon X\times X\to \R$ is strictly of hyperbolic type if and only if every function $f\colon X\to\Hyp^\kappa$ satisfying $K(x,y)=\cosh(\dH(f(x),f(y))$ has the property that for every finite $F\subset X\colon \dim \Span f(F)=\Card(F)$.
\end{proposition}

\begin{proof}
Let $F$ be a finite subset and fix $x_0\in F$. The kernel $K_0\colon (x,y)\mapsto K(x,x_0)K(x_0,y)-K(x,y)$ is strictly of positive type. So, by the first part of \Cref{lem:strct-means-indep_positive-conditionegative}, the image of $F$ in $\R^{1+\kappa}$ is linearly independent which proves the result.
\end{proof}

\begin{definition}[spaces of kernels of hyperbolic type]
\label{def:space-of-kernels}
On a space $X$, we denote by the $\mathcal{K}(X)$ set of kernels of hyperbolic type endowed with the pointwise convergence topology.

It is a closed subspace of $\prod_{X^2} \R_{\ge 1}$ by \Cref{lem:pointwise-closed}.
\end{definition}

For $K\in \mathcal{K}(X)$, we denote by $\dim(K)\in\N\cup\{\infty\}$ the dimension $\kappa$ of the hyperbolic space given by \Cref{prop:hyperbolic-GNS}.
When $X$ is finite any kernel $K$ has $\dim(K)\leq\Card(X)-1$.

\begin{proposition}[dimension is lower semi-continuous]
\label{prop:semicontinuity_kernel-dimension}
The function $K\mapsto \dim(K)$ is lower semi-continuous.
\end{proposition}
\begin{proof}
For $l\in \N$, a kernel of hyperbolic type $K_0$ has $\dim(K_0)\geq l$ if and only if there exists $x_0,\dots,x_l$ such that the Cayley-Menger determinant $\det(-K_{0}(x_i,x_j))_{0\leq i,j\leq l}\neq0$.
Hence the set of kernels of dimension $\ge l$ is $\cup_{x\in X^l}\{K\in \mathcal{K}(X) \colon \det(-K(x_i,x_j))_{0\leq i,j\leq l}\neq0\}$, which is a union of open sets. This gives the lower semi-continuity of the dimension.
\end{proof}

\subsection{Functions of hyperbolic type and actions on \texorpdfstring{$\Hyp^\kappa$}{Hkappa}}

\begin{definition}[function of hyperbolic type]
On a topological group $\Gamma$, a function $F\colon \Gamma\to\R$ is of \emph{hyperbolic type} when the kernel $\Gamma\times\Gamma \ni (\alpha,\beta) \mapsto F(\alpha^{-1}\beta) \in \R$ is of hyperbolic type.

We denote by $\mathcal{F}(\Gamma)$ the space of functions of hyperbolic type on $\Gamma$ endowed with the pointwise convergence topology.
It is a closed subspace of $(\R_{\ge 1})^\Gamma$ by \Cref{lem:pointwise-closed}.
\end{definition}

\begin{example}
Given a continuous representation $\rho\colon \Gamma \to \Isom(\Hyp^\kappa)$, the choice of a point $o\in\Hyp^\kappa$ yields a function of hyperbolic type via the formula $F(g)=\cosh \dH(o,\rho(\gamma)o)$.

Another choice $o'\in\Hyp^\kappa$ yields another function of hyperbolic type $F'$ and several applications of the triangle inequality for $d$ show that \(\lvert \cosh^{-1} F-\cosh^{-1} F'\rvert \leq 2 \dH(o,o')\).
\end{example}

The following summarizes \cite[Theorem 3.4 and Corollary 3.5]{Monod-Py_self-representations-Mobius_2019}.

\begin{proposition}[representing functions of hyperbolic type]
\label{prop:functions-of-hyperbolic-type-to-representations}
On a group $\Gamma$, for every function of hyperbolic type $F\colon \Gamma \times \Gamma \to \R$, there is a hyperbolic space $\Hyp^\kappa$, a point $o\in\Hyp^\kappa$ and a representation $\rho\colon \Gamma\to \Isom(\Hyp^\kappa)$ such that $F(\gamma)=\cosh \dH(o,\rho(\gamma)o)$ and the orbit of $o$ is total.

If $\rho\colon \Gamma \to \Isom(\Hyp^\kappa)$ and $\rho'\colon \Gamma\to \Isom(\Hyp^{\kappa'})$ are two such representations, then there is a unique isometry $\varphi \colon \Hyp^\kappa\to\Hyp^{\kappa'}$ satisfying $\varphi(\rho(\gamma)x)=\rho'(\gamma)\varphi(x)$ for all $x\in \Hyp^\kappa$ and $\gamma\in \Gamma$. 
\end{proposition}

The second statement in \Cref{prop:functions-of-hyperbolic-type-to-representations} implies that one may attach $F\in \mathcal{F}(\Gamma)$ any object or attribute associated to the corresponding class of actions of $\Gamma$ on $\Hyp^\kappa$.

\begin{definition}[length function]
To a function of hyperbolic type $F\colon \Gamma\to \R$ on a group, we associate a translation \emph{length function} $\ell_F \colon \Gamma\to \R_+$ by choosing any associated irreducible representation $\rho \colon \Gamma \to \Isom(\Hyp^\kappa)$ and defining $\ell_F(\gamma)=\ell(\rho(\gamma))$.

We say that $F$ is \emph{neutral} when $\ell_F=0$ and we denote $\mathcal{F}_{nn}(\Gamma)\subset\mathcal{F}(\Gamma)$ the open subspace of non-neutral functions of hyperbolic type.
\end{definition}

The following formula is from \cite[Proposition 4.4]{Monod-Py_self-representations-Mobius_2019}.

\begin{lemma}[$F$ recovers $\ell$]
\label{lem:ellF-from-F}
For $F\in \mathcal{F}(\Gamma)$, its translation length function $\ell_F \colon \Gamma\to \R_{\ge0}$ is given for all $\gamma\in \Gamma$ by:
\[\ell_F(\gamma)=\lim_{n\to \infty} \tfrac{1}{n}\cosh^{-1} \left(F(\gamma^n)\right).\]


Thus, if functions of hyperbolic type $F_i\colon \Gamma \to \R_{\ge 1}$ satisfy $\|\cosh^{-1}(F_1)-\cosh^{-1}(F_2)\| < \infty$ or equivalently $\|\log(F_1/F_2)\| < \infty$, then their length functions $\ell_{F_i}\colon \Gamma \to \R_{\ge 0}$ are equal $\ell_{F_1}=\ell_{F_2}$.
\end{lemma}

\begin{remark}[converse to Proposition \ref{prop:functions-of-hyperbolic-type-to-representations}]
\label{rem:same-length-spectrum-means-conjugate}
Proposition \ref{prop:same-length-spectrum-means-conjugate} will imply that $F_1,F_2\in \mathcal{F}_{nn}(\Gamma)$ satisfying $\|\log(F_1/F_2)\|<\infty$ (equivalently $\|\cosh^{-1}(F_1)-\cosh^{-1}(F_2)\|<\infty$) 
must arise from a same representation $\rho\colon \Gamma \to \Isom(\Hyp^\kappa)$.
\end{remark}

\begin{corollary}[neutral and elementary]
\label{cor:neutral-elementary}
For two functions of hyperbolic type $F_i\colon \Gamma\to \R_{\ge 1}$ such that satisfy $\|\cosh^{-1}(F_1)-\cosh^{-1}(F_2)\| < \infty$ or equivalently  $\|\log(F_1/F_2)\| < \infty$, we have:
\begin{itemize}[noitemsep, align=left]
	\item[Neutral:] $F_1$ is neutral if and only if $F_2$ is neutral.
	\item[Elementary:] $F_1$ is elementary if and only if $F_2$ is elementary.
\end{itemize}
\end{corollary}

\begin{proof}
By Lemma \ref{lem:ellF-from-F}, we know that $\ell_{F_1}=\ell_{F_2}$. 
The result follows since $\ell_F$ determines the neutrality property (by definition), and the elementary property (by Lemmata \ref{lem:Tits-alternative} \& \ref{lem:equal-fixed-points-from-length}).
\end{proof}

Here is another consequence of Lemma~\ref{lem:ellF-from-F}, which is very close to \cite[Theorem 1.9]{Reyes_isometries-Gromov-hyperbolic_2018}.

\begin{lemma}[continuity of the length spectrum]
\label{lem:F-to-ellF-continuity}
The function $F\mapsto\ell_F$ is continuous for the pointwise convergence topology on the real vector space of functions $\Gamma\to\R$.
\end{lemma}
\begin{proof}
For fixed $n\in \N$ and $\gamma\in \Gamma$, the function $F\mapsto \tfrac{1}{n}\cosh^{-1} \left(F(\gamma^n)\right)$ is continuous.
We show that the limit is both upper and lower semi-continuous.

On the one hand, by sub-additivity of $n\mapsto\cosh^{-1} \left(F(\gamma^n)\right)$, the limit is actually an infimum, hence $F\mapsto \ell_F(\gamma)$ is upper semi-continuous. 
On the other hand, applying \cite[Theorem 1.1]{Reyes_isometries-Gromov-hyperbolic_2018} to the action of $\Gamma$ on the Gromov-$\delta$-hyperbolic space $X$ shows that for every $x\in X$ and $\gamma\in \Isom(X)$ we have 
\(d(\gamma^2x,x)\leq d(\gamma x,x)+\ell(\gamma)+2\delta\).
Thus
\[\ell_F(\gamma)\geq\sup\{\tfrac{1}{n}\left(\cosh^{-1} (F(\gamma^{2n}))-\cosh^{-1} \left(F(\gamma^{n})\right)-2\delta\right) \colon n\in\N\}\]
and the reverse inequality follows from the fact that $\ell(\gamma)$ is actually the limit of the right-hand side, which proves lower semi-continuity. 
\end{proof}

\begin{lemma}[$F$ recovers classification]
\label{lem:F-recovers-type-eph}
Consider a function of hyperbolic type $F\colon \Gamma \to \R$ with translation length function $\ell_F \colon \Gamma\to \R_{\ge0}$. 

For a representation $\rho\colon \Gamma \to \Isom(\Hyp^\kappa)$ realizing $F$, a subgroup $\Gamma_0 \subset \Gamma$ is such that:
\begin{itemize}[noitemsep, align=left]
	\item[$\rho(\Gamma_0)$ is elliptic] if and only if $F(\Gamma_0)$ is bounded
	\item[$\rho(\Gamma_0)$ is parabolic] if and only if $F(\Gamma_0)$ is unbounded but $\ell_F(\Gamma_0)$ is bounded
	\item[$\rho(\Gamma_0)$ is hyperbolic] if and only if $\ell_F(\Gamma_0)$ is unbounded
\end{itemize}
This applies in particular to the full group $\Gamma_0=\Gamma$ or a cyclic subgroup $\Gamma_0=\{\gamma^n\colon n\in \Z\}$.
\end{lemma}

\begin{proof}
The first and third points are clear having in mind that $\ell_F(\Gamma_0)$ is unbounded is equivalent to $\ell_F(\Gamma_0)=\{0\}$. For the second point, this follows from the partitioning of elementary representations described in \Cref{rem:elementary}.
\end{proof}

\subsection{The exotic deformation from powers of kernels}\label{subsec:exotic-deformation}

We first recall \cite[Theorem I\textsuperscript{bis}]{Monod-Py_self-representations-Mobius_2019}.

\begin{theorem}[exotic deformation]
\label{thm:exotic-deformation-Hyp}
For $t\in (0,1]$, the function $(\cosh \circ \dH)^t\colon \Hyp^\kappa \times \Hyp^\kappa \to \R$ is a kernel of hyperbolic type.
This defines the continuous \emph{exotic deformation} $\exoH_t \colon \Hyp^\kappa \to \Hyp^{\omega+\kappa}$, hence the associated (conjugacy class of) \emph{exotic representation}(s) $\chi_t \colon \Isom(\Hyp^\kappa)\to \Isom(\Hyp^{\omega+\kappa})$. 

We still denote $\exoH_t\colon\Hyp^\kappa\sqcup \partial \Hyp^\kappa\to \Hyp^{\infty+\kappa}\sqcup \partial \Hyp^{\infty+\kappa}$ its extension to the boundary.
\end{theorem}

This leads to the following \cite[Theorem 3.10]{Monod-Py_self-representations-Mobius_2019}.

\begin{corollary}[powers of kernels and functions of hyperbolic type]
If $K\colon X\times X\to \R$ is a kernel of hyperbolic type, then for all $0\leq t\leq1$, $K^t$ is a kernel of hyperbolic type. 

Similarly, if $F\colon G\to \R$ is a function of hyperbolic type on a group, then for all $0\leq t\leq1$, $F^t$ is a kernel of hyperbolic type. 
Moreover, the following properties hold: 
\begin{enumerate}[noitemsep]
	\item When $t\neq 0$ we have that $F^t$ is non-elementary if and only if $F$ is so.
	\item The translation length functions associated with the $F^t$ are homothetic: $\ell_t=t\ell_1$.
\end{enumerate}
\end{corollary}

\begin{proof}
The statement about elementarity follows from the fact that the equivariant map $\exoH_t\colon \Hyp^\omega\to\Hyp^\omega$ is a $(t,C)$-quasi-isometry, hence induces an $\Isom(\Hyp^\omega)$-equivariant map $\exoH_t\colon \partial\Hyp^\kappa\to\partial\Hyp^\kappa$. 
The statement about proportionality of length functions is proved in \cite[§ 4.2]{Monod-Py_self-representations-Mobius_2019}.
\end{proof}

The following proposition is a strengthening of the previous relying on  the proof of \cite[Theorem 1.4]{Monod_notes-functions-hyperbolic-groups_2020} where powers of kernels of hyperbolic type are proved directly to be of hyperbolic type.

\begin{proposition}[strict powers are strict]
\label{prop:strictly_of_hyperbolic_type}
Let $K$ be a kernel of hyperbolic type on $X$ that separates points. For every $t\in(0,1)$, the power $K^t$ is strictly of hyperbolic type. 
\end{proposition}

\begin{proof}
We use the notations in the proof of \cite[Theorem 1.12]{Monod_notes-functions-hyperbolic-groups_2020} for $K=\beta$ and $Q=q$. 
Monod proves that $\beta^t$ is of hyperbolic type by proving an inequality. 
We do not repeat the proof and only explain how to get a strict inequality using his notations (with kernels $N,M',M''$). 
The goal is to prove that $N$ is strictly of positive type which is equivalent to $M''=q(M')$ is strictly of positive type. This is the case since $M'$ is a separating kernel of positive type and $Q$ satisfies the hypotheses of \Cref{lem:powerseries-of-positive-type}.
\end{proof}

Taking powers of functions of hyperbolic type affects the algebraic cross-ratio accordingly.

\begin{lemma}[powers of cross-ratio]\label{lem:powers-cross-ratio}
For $t\in (0,1]$, and for any $x^-,x^+,y^-,y^+\in \partial^4{\Hyp^\omega}$, 
\[\Cr{\exoH_t(x^-), \exoH_t(x^+), \exoH_t(y^-), \exoH_t(y^+)}=\Cr{x^-, x^+, y^-, y^+}^t.\]    
\end{lemma}

\begin{proof}
Recall Equations \eqref{eq:Cr=0} and \eqref{eq:Cr=infty} characterizing when the cross-ratio equals $0$ and $\infty$.
If $\Cr{x^-,x^+,y^-,y^+}=0$ then $x^-=x^+$ or $y^-=y^+$ so $\exoH_t(x^-)=\exoH_t(x^+)$ or $\exoH_t(y^-)=\exoH_t(y^+)$ hence $\Cr{\exoH_t(x^-), \exoH_t(x^+), \exoH_t(y^-), \exoH_t(y^+)}=0$ as desired.
If $\Cr{x^-,x^+,y^-,y^+}=\infty$ then $x^-=y^+$ or $y^-=x^+$ so $\exoH_t(x^-)=\exoH_t(y^+)$ or $\exoH_t(y^-)=\exoH_t(x^+)$ hence $\Cr{\exoH_t(x^-), \exoH_t(x^+), \exoH_t(y^-), \exoH_t(y^+)}=\infty$ as desired.

Now, assume that $\Cr{x^-,x^+,y^-,y^+}\in (0,+\infty)$. There exists a hyperbolic isometry $\gamma$ with repulsive point $x^+$, attractive point $y^+$ and such that $\gamma(y^-)=y^+$.
Since $\exoH_t$ is $\chi_t$-equivariant, the isometry $\chi_t(\gamma)$ is a hyperbolic translation with axis $(\chi_t(\gamma)^-,\chi_t(\gamma)^+)=(\exoH_t(x^+),\exoH_t(y^+))$ such that $\chi_t(\gamma)(\exoH_t(y^-))=\exoH_t(x^-)$. 
By construction, $\chi_t$ has translation length $\ell(\chi_t(\gamma))=t\ell(\gamma)$.
Thus, by Equation \eqref{eq:translation-length_Cr-boundary_log}: $\Cr{\exoH_t(x^-), \exoH_t(x^+), \exoH_t(y^-), \exoH_t(y^+)}=\Cr{\chi_t(\gamma)\exoH_t(y^-), \chi_t(\gamma)^-, \exoH_t(y^-), \chi_t(\gamma)^+}$ is equal to $\exp\left(-\ell(\chi_t(\gamma))\right)=\exp\left(-t\ell(\gamma)\right)$ hence to $\Cr{x^-, x^+, y^-, y^+}^t$.
\end{proof}

\begin{lemma}
\label{lem:exotic-exp(td)}
Let $(X,d)$ be a strongly hyperbolic space and $t\in \R$ such that $(x,y)\mapsto\exp(td(x,y))$ is a kernel of hyperbolic type.
Let $f_t \colon X\to\Hyp^\kappa$ be a map such that  $\cosh(\dH(f_t(x),f_t(y))=\exp(td(x,y))$ and $\chi_t\colon\Isom(X)\to\Isom(\Hyp^\kappa)$ be the associated representation. 

For all $\gamma\in\Isom(X)$
we have \(\ell(\chi_t(\gamma))=t\ell(\gamma)\).

For all $x^-, x^+, y^-, y^+ \in (\partial X)^{(4)}$  we have
\(\Cr{f_t(x^-), f_t(x^+), f_t(y^-), f_t(y^+)}=\Cr{x^-, x^+, y^-, y^+}^t\).
\end{lemma}

\begin{proof}
By taking the logarithm of the relation $\cosh(\dH(f_t(x),f_t(y))=\exp(t d(x,y))$, we find that 
\(\lvert \dH(f_t(x),f_t(y))-td(x,y)\rvert\leq\ln(2)\) so as $d(x,y)\to\infty$ we have \(\dH(f_t(x),f_t(y))-t d(x,y)\to\ln(2)\) so the results follow by definition of the translation length and continuity of the cross-ratio.
\end{proof}

\begin{example}
On the algebraic hyperbolic space $(\Hyp^\kappa, \dH)$, if $t\in(0,1]$ then the function $\exp(t\dH)$ is a kernel of hyperbolic type (\cite[Theorem 8.5]{Monod_notes-functions-hyperbolic-groups_2020}).

On a real tree $(T,d)$ as in \Cref{def:real-tree}, for every $t\in \R$ the function $\exp(td)$ is a kernel of hyperbolic type (\cite[Proposition 1.5]{Monod_notes-functions-hyperbolic-groups_2020}).
\end{example}

\section{Abstract and algebraic cross-ratios: rigidity properties}

In this section we consider a topological space $Y$ (often a set with the discrete topology).
Let us recall that by $Y^{(4)}\subset Y^4$ we mean the space of quadruples of points of which at least three are distinct.

\subsection{Strongly hyperbolic spaces with unique cross-ratio up to powers}
\label{subsec:unicity_cross_ratio}

The goal of this subsection is to show that for a strongly hyperbolic space $(X,d)$, the cross-ratio defined by $d$ from \Cref{eq:def-cross-ratio-X} and its powers are the unique continuous functions $(\partial X)^{(4)}\to \R$ satisfying the usual identities (\eqref{eq:Cr-Invariance}, \eqref{eq:Cr-Inversion}, \eqref{eq:Cr-Cocycle}, \eqref{eq:Cr=0}, \eqref{eq:Cr=infty}) and that is invariant under any subgroup $\Gamma\leq\Isom(X)$ acting sufficiently transitively on $\partial X$.

This generalizes the similar statement for algebraic hyperbolic spaces in  \Cref{lem:unique_cross-ratio_invariant}.

\begin{definition}[abstract cross-ratio]
\label{def:abstract-cross-ratio}
On a topological space $Y$, an \emph{abstract cross-ratio} is a continuous function $B\colon Y^{(4)}\to [0,\infty]$ satisfying the \eqref{eq:Cr-Invariance}, \eqref{eq:Cr-Inversion} and \eqref{eq:Cr-Cocycle} relations as well as the characterizations \eqref{eq:Cr=0} and \eqref{eq:Cr=infty} of the fibers above $0$ and $\infty$.

We say that an abstract cross-ratio to \emph{generate the topology} when for very triple of distinct points $x,y,z \in Y$, the collection of subsets $\{w\in Y\colon B(w,x,y,z)<\varepsilon\}$ for $\varepsilon>0$ is a basis of neighborhoods of $x$.
\end{definition}

\begin{remark}[$B(y,x,y,z)=1=B(x,y,z,y)$]
\label{rem:Cr=1}
For an abstract cross-ratio $B$ on $Y$, if $(x,x',y,y')\in Y^{(4)}$ satisfies $x=y$ or $x'=y'$ then by \Cref{eq:Cr-Inversion} we have $B(x,x'y,y')=1$.
However the converse may fail drastically, as it is the case for the usual cross-ratio on $\Hyp^3$.
\end{remark}

\begin{remark}[compact metrizable]
If $Y$ is a compact metrizable space, then any abstract cross-ratio necessarily generates the topology. Indeed, if for fixed distinct $x,y,z\in Y$ and a sequence $(w_n)\in Y^\N$ we have  $B(w_n,x,y,z)\to0$, then $(w_n)$ has an extraction converging to some $x'$ which satisfies $B(x',x,y,z)=0$ and as $y\neq z$, we deduce from \Cref{eq:Cr=0} that $x=x'$. 
\end{remark}

\begin{example}[abstract cross-ratios from pairings] \label{eg:abstract-Cr-from-It}
If a continuous function $\It \colon Y\times Y \to \R$ is symmetric ($\It(x,y)=\It(y,x)$) and separates points ($\It(x,y)=0\implies x=y$), then it yields an abstract cross-ratio by the formula \(B(u,v,x,y)=\tfrac{\It(u,v)\It(x,y)}{\It(u,y)\It(x,v)}\).
Indeed, the \eqref{eq:Cr-Invariance}, \eqref{eq:Cr-Invariance} and \eqref{eq:Cr-Cocycle} relations follow from the symmetry of $\It$ and the formula for $B$; while the equations \eqref{eq:Cr=0} and \eqref{eq:Cr=infty} follow from the separation of $\It$.
Moreover, if $\It$ generates the topology of $Y$ (in the sense that for all $x\in Y$ the collection of subsets $\{x\in Y \colon \It(x,y)<\varepsilon\}$ for $\varepsilon>0$ is a basis of neighborhoods of $x$), then $B$ generates the topology of $Y$.
\end{example}

\begin{example}[strongly hyperbolic]
If a metric space $(X,d)$ is strongly hyperbolic, then the cross-ratio introduced in \Cref{subsec:Hn-cross-ratio} defines an abstract cross-ratio on $\partial X$ that generates the topology induced by any visual distance $d_o$.
\end{example}

Recall that a group $\Gamma$ acts $3$-transitively on a space $X$ when for any two triples of distinct points $(x_i)$ and $(y_i)$, there is $\gamma\in\Gamma$ such that $\gamma(x_i)=y_i$ for all $i\in \{1,2,3\}$.

\begin{proposition}[abstract cross-ratios on $3$-full strongly hyperbolic spaces]
\label{prop:power_cross_ratio}
Consider a strongly hyperbolic space $(X,d)$ with $\Card(\partial X)\geq3$ and a subgroup $\Gamma\subset \Isom(X)$ acting $3$-transitively on $\partial X$.
If an abstract cross-ratio $B$ on $\partial X$ generating the topology is $\Gamma$-invariant, then there exists $t>0$ such that for all $(w,x,y,z) \in (\partial X)^{(4)}$ we have  $B(w,x,y,z)=\Cr{w,x;y,z}^t$.
\end{proposition}

\begin{proof}
For distinct points $0,\infty \in \partial X$, consider the subgroup $H \subset \Gamma$ of elements fixing both of them.
Choose $1 \in\partial X\setminus\{0,\infty\}$, and define continuous functions $\Vec{\ell}_B, \Vec{\ell}_C \colon H \to (\R,+)$ by $\Vec{\ell}_B \colon \mapsto \log B(\gamma 1,0,1,\infty)$ and $\Vec{\ell}_C \colon \gamma \mapsto \log \Cr{\gamma 1,0;1,\infty}$.
They are group homomorphisms by the \eqref{eq:Cr-Cocycle} relation.
Moreover every $\gamma \in H$ has translation length $\ell(\gamma)=\lvert \Vec{\ell}_C(\gamma)\rvert$ by \Cref{eq:translation-length_Cr-boundary_log}. 
Observe that for any abstract $\Gamma$-invariant cross-ratio $B$, for any $\xi \in \partial X\setminus\{0,\infty\}$ and $\gamma\in H$, applying the \eqref{eq:Cr-Cocycle} relation, the $H$-invariance of $B$, and the \eqref{eq:Cr-Inversion} relation:
\begin{align*}
	B(\gamma \xi,0,\xi,\infty)&=B(\gamma \xi,0,\gamma 1,\infty)B(\gamma1,0,1,\infty)B(1,0,\xi,\infty)\\
	&=B(\gamma 1,0,1,\infty)
\end{align*}
This implies that the homomorphisms $\Vec{\ell}_B,\Vec{\ell}_C$ do not depend on the choice of $1\in \partial X\setminus\{0,\infty\}$. 
We aim to prove that these homomorphisms are positively proportional.

If $\gamma \in H$ satisfies $\Vec{\ell}_C(\gamma)>0$ then $\gamma^n1 \to\infty$ hence by continuity of $B$ we also have $\log B(\gamma^n1,0,1,\infty)\to+\infty$ thus $\log B(\gamma 1, 0,1,\infty)>0$: this proves that $\ker \Vec{\ell}_B \subset \ker \Vec{\ell}_C$.
A similar argument using the fact that $B$ generates the topology implies the converse inclusion $\Vec{\ell}_B \supset \ker \Vec{\ell}_C$. 
Hence $\Vec{\ell}_B = \ker \Vec{\ell}_C$ consist of all non-loxodromic elements in $H$, and for all $\gamma \in H$ there exists $t_\gamma>0$ such that $\Vec{\ell}_B(\gamma)=t_\gamma \Vec{\ell}_C(\gamma)$, which is unique when $\gamma$ is loxodromic. 

Consider loxodromic $\alpha,\beta \in H$.
We may find $m_k,n_k\to\infty$ such $\Vec{\ell}_C(\alpha^{m_k})-\Vec{\ell}_C(\beta^{n_k})\to0$, hence $\alpha^{m_k}\beta^{-n_k}1$ does not converge to $0$ or $\infty$.
However if $t_\alpha>t_\beta$, then $\Vec{\ell}_B(\alpha^{m_k}\beta^{-n_k})\to \infty$ thus $\alpha^{m_k}\beta^{-n_k}1$ diverges to the attractive fixed point of $\alpha$ in $\{0,\infty\}$ which is a contradiction.
By symmetry we deduce $t_\alpha=t_\beta$.
This proves that there is $t>0$ such that $\Vec{\ell}_B=t\Vec{\ell}_C$.

To conclude the proof, it remains to observe that for any distinct $w,x,y,z\in\partial X$ , there is $\gamma\in \Gamma$ such that $\gamma(x,y,z)=(0,1,\infty)$, so $B(w,x,y,z)=B(\xi,0,1,\infty)$ for some $\xi \in\partial X$ and by triple transitivity, there is $\gamma\in \Gamma$ fixing $0,\infty$ such that $\xi=\gamma(1)$, thus $B(\xi,0,1,\infty)=\Cr{\xi,0,1,\infty}^t$ and finally $B(w,x,y,z)=\Cr{w,x,y,z}^t$. 
The cases where two points are equal follow from \Cref{eq:Cr=0} or \Cref{eq:Cr=infty}, or from \eqref{eq:Cr-Inversion} that for distinct $x,y,z\in \partial X$ we have $B(y,x,y,z)=B(x,y,z,y)=1$.
\end{proof}

\begin{corollary}[cross-ratios preserved by $3$-equivariant boundary maps]
\label{cor:power_cross_ratio}
Consider for $i\in \{1,2\}$ strongly hyperbolic spaces $(X_i,d_i)$ and an embedding $f\colon \partial X_1\to \partial X_2$ that is a homeomorphism on its image.
Assume that there is a subgroup $\Gamma\subset \Isom(X_1)$ acting $3$-transitively on $\partial X_1$ and that $f$ is $\Gamma$-equivariant.

There is $t>0$ such that for all $(w,x,y,z)\in (\partial X_1)^{(4)}$  we have
\begin{equation*}
	\Cr{f(w),f(x);f(y),f(z)}_{\partial{X_2}}=\Cr{w,x;y,z}_{\partial X_1}^t
\end{equation*}
\end{corollary}

\begin{proof}
Define $B\colon (\partial X_1)^{(4)}\to [0,\infty]$ by $B(w,x,y,z)=\Cr{f(w),f(x);f(y),f(z)}_{\partial X_2}$. Since $f$ is a homeomorphism onto its image, the abstract cross-ratio $B$ is covered by \Cref{prop:power_cross_ratio}.
\end{proof}


\subsection{From kernels of hyperbolic type to strong hyperbolicity}\label{sec:hyperbolic_kernels_to_strong_hyperbolicity}

For trees and exotic deformations of hyperbolic spaces, we have metric spaces $(X,d)$ admitting $\varepsilon>0$ such that kernel $\exp(\varepsilon d)$ of hyperbolic type.
This raises the following.
\begin{question}
For which metric spaces $(X,d)$ is there $\varepsilon>0$ such that the kernel $\exp(\varepsilon d)$ is of hyperbolic type?
\end{question}

Recall that \eqref{eq:Ptolemy_strong-hyperbolic} characterizes strong-$\varepsilon$-hyperbolicty by the Ptolemy inequality for $\exp(\tfrac{1}{2}\varepsilon d)$.
The following is a coarse asymptotic weakening of this notion.

\begin{definition}[coarse strong hyperbolicity]
\label{def:coarsely-strongly-hyperbolic} 
Let $\varepsilon>0$ and $C\in \R$.
A metric space $(X,d)$ is called \emph{$C$-coarsely strongly-$\varepsilon$-hyperbolic} when for all $x_1,x_2,x_3,x_4\in X$, their distances $d_{ij}=d(x_i,x_j)$ and their maximum $M=\max\{d_{ij}\colon 1\le i,j\le 4\}$ satisfy: 
\begin{equation}
	\label{eq:coarsely-strongly-hyperbolic}
	\exp\tfrac{\varepsilon}{2}(d_{13}+d_{24})\leq \exp\tfrac{\varepsilon}{2}(d_{12}+d_{34})+\exp\tfrac{\varepsilon}{2}(d_{14}+d_{23})+C \exp(\tfrac{\varepsilon}{2}M)
\end{equation}
\end{definition}

\begin{remark}[terminology]
The spaces defined in \ref{def:coarsely-strongly-hyperbolic} were introduced by
Miao--Schroeder in \cite[Definition 1]{Miao-Schroeder_Hyperbolic-Ptolemey-Mobius_2012} under the name of \emph{asymptotically $\operatorname{PT}_{-1}$ spaces}. 
Their preprint posted on arXiv in 2012 seems to remain unpublished.
Since then, the notion of strong hyperbolicity has been defined by Nica--Spakula in \cite{Nica-Spakula_strong-hyperbolicity_2016} and widely used since its introduction. 
This is why we prefer to rename $\operatorname{PT}_{-1}$ spaces as asymptotically strongly hyperbolic spaces. 
\end{remark}

The following \cite[Theorem 1.2]{Miao-Schroeder_Hyperbolic-Ptolemey-Mobius_2012} shows that coarse strong hyperbolicity retains many of the properties of strong hyperbolicity.
\begin{theorem}[coarsely strongly-hyperbolic implies regularly Gromov-hyperbolic]
\label{thm:Miao-Shroeder-coarse-strong-hyper}
If a metric space $(X,d)$ is coarsely strongly hyperbolic, then $X$ is Gromov-hyperbolic, the Gromov product extends continuously to the boundary, and all visual metametrics are metrics on $\partial X$ and they are all M\"obius equivalent (in the sense that they induce the same metric cross-ratio).
\end{theorem}

In algebraic hyperbolic spaces $(\Hyp^\kappa,\dH)$ we have the following Ptolemy inequality for $\sinh(\tfrac{1}{2}\dH)$, which is \cite[Corollary 4]{Gomez-Memoli_Ptoleme-tropical_2024} recovering \cite[Corollary of Theorem 3.2]{Valentine_Ptoleme-hyperbolic_1970}.
\begin{lemma}[Ptolemy inequality in $\Hyp^\kappa$]
\label{lem:Ptoleme_Hyp}
For all $x_1,x_2,x_3,x_4 \in \Hyp^\kappa$ with mutual distances denoted $\dH(x_i,x_j)=\delta_{ij}$, we have the Ptolemaic inequality:
\begin{equation}
	\label{eq:Ptolemy-sinh-Hyp}
	\left(\sinh \tfrac{1}{2}\delta_{13}\right) \left(\sinh \tfrac{1}{2}\delta_{24}\right)
	\le 
	\left(\sinh \tfrac{1}{2}\delta_{12}\right)
	\left(\sinh \tfrac{1}{2}\delta_{34}\right)
	+
	\left(\sinh \tfrac{1}{2}\delta_{14}\right)
	\left(\sinh \tfrac{1}{2}\delta_{23}\right)
\end{equation}
which may be written in terms of the kernel of hyperbolic type using $\sqrt{2}\sinh(\tfrac{1}{2}\delta)=\sqrt{\cosh(\delta)-1}$.
\end{lemma}

\begin{proposition}[kernel $\exp(d)$ of hyperbolic type implies coarse strong-hyperbolicity]
\label{prop:coarsely-strongly-hyperbolic}
If a metric space $(X,d)$ admits an $\varepsilon>0$ such that $\exp(\varepsilon d)$ is a kernel of hyperbolic type then $(X,d)$ is $3$-coarsely strongly-$\varepsilon$-hyperbolic.
In particular it satisfies the conclusions of \cref{thm:Miao-Shroeder-coarse-strong-hyper}.
\end{proposition}

\begin{proof}
After replacing $(X,d)$ by $(X,\varepsilon d)$, it suffices to prove it for $\varepsilon=1$.

Suppose that $\exp(d)$ is a kernel of hyperbolic type.
There is a cardinal $\kappa$ and a continuous function $f\colon X\to\Hyp^\kappa$ such that $\exp d(x,y)=\cosh \dH(f(x),f(y))$.
Let $x_1,\dots,x_4\in X$. Denote $d_{ij}=d(x_i,x_j)$ and $\delta_{ij}=\dH(f(x_i),f(x_j))$ so that $\exp d_{ij} =\cosh \delta_{ij}$, and $M=\max_{ij} d_{ij}$. 
Developing the Ptolemy inequality for $\sqrt{2}\sinh(\delta_{ij}/2)=\sqrt{\cosh(\delta_{ij})-1} = \sqrt{\exp(d_{ij})-1}$ and using $\sqrt{x}-1\le \sqrt{x-1}\le \sqrt{x}$ yields:
\begin{equation*}
	\exp \tfrac{1}{2}(d_{13}+d_{24}) \le
	\exp \tfrac{1}{2}(d_{12}+d_{34})
	+
	\exp \tfrac{1}{2}(d_{14}+d_{23})
	+3\exp(M/2)
\end{equation*}
proving the desired coarse strong-$1$-hyperbolicity with $C=3$ .
\end{proof}

\begin{remark}[not sufficient]
The necessary condition of \Cref{prop:coarsely-strongly-hyperbolic} is not sufficient.
Indeed, a quaternionic hyperbolic space $\Hyp_{\mathbb{H}}^n$ is $\operatorname{CAT}(-1)$ thus (coarsely) strongly hyperbolic; however for $\varepsilon>0$ the kernel $\exp(\varepsilon d)$ cannot be of hyperbolic type, since all kernels of hyperbolic type on $\Hyp_{\mathbb{H}}^n$ that are $\Isom(\Hyp_{\mathbb{H}}^n)$-invariant must be bounded as we now explain. 
On a real hyperbolic space $\Hyp_{\mathbb{R}}^\kappa$ the kernel $\log \cosh d$ is of conditionally negative type by~\cite[Corollary 8.2]{Faraut-Harzallah_distance-hilbertiennes-invariantes_1974}.
However a kernel on $\Hyp_{\mathbb{H}}^n$ that is $\Isom(\Hyp_{\mathbb{H}}^n)$ invariant and of conditionally negative type must be bounded, since $\Isom(\Hyp_{\mathbb{H}}^n)$ has Property $\operatorname{FH}$ by~\cite[\S 3.3]{Bekka-Valette-Harpe_Kazhdan-property_2008}.
\end{remark}

\subsection{Algebraic cross-ratios}\label{sec:algebraic_cross_ratio}

In this section we propose a GNS construction from infinity for real hyperbolic spaces. 

\begin{definition}[algebraic cross-ratio]
\label{def:algebraic-cross-ratio}
Let $Y$ be a topological space. We define an \emph{algebraic cross-ratio} on $Y$ to be an abstract cross-ratio
such that the function on $(Y\setminus \{\infty\})^2 \to \R_{\ge 0}$ defined on the diagonal by $(y,y)\mapsto 0$ and outside the diagonal by 
\begin{equation} 
	\label{eq:algebraic-cross-ratio}
	(x,y) \mapsto \left(R(x,y,0,\infty)R(y,0,1,\infty)\right)^2
\end{equation}
is a kernel conditionally of negative type on $Y\setminus \{\infty\}$.
\end{definition}

\begin{remark}[indeterminate value]
Note that when $y=0\ne x$, we can determine the value of $R(x,y,0,\infty)R(y,0,1,\infty)=\infty \times 0$ using the cocycle relation to find $R(x,0,1,\infty)\in(0,\infty)$.
\end{remark}

\begin{example}[algebraic model]\label{eg:cross_ratio_algebraic_cross_ratio}
Let us explain why, for an algebraic hyperbolic space $\Hyp^\kappa$, the cross-ratio on $\partial\Hyp^\kappa$ is an algebraic cross-ratio. 
In the upper half space $\Hyp^\kappa=\R^{\kappa-1} \times(0,+\infty)$, the boundary at infinity is identified with  $\R^{\kappa-1} \cup\{\infty\}$ and for any $x,y,z \in \R^{\kappa-1}$, we have
\begin{equation}
	\Cr{x,y;z, \infty}=\tfrac{\|y-x\|}{\|y-z\|}
\end{equation}
because any quadruple is contained in an isometric copy of $\Hyp^3$ so we may use Corollary \ref{cor:cross-ratio-H3}.

Thus for any fixed distinct $0,1\in\R^{\kappa-1}$, \[\Cr{x,y;0,\infty}\Cr{y,0;1,\infty}=\tfrac{\|y-x\|}{\|y-0\|}\cdot \tfrac{\|0-y\|}{\|0-1\|}=\tfrac{\|y-x\|}{\|0-1\|}\] which is the square root of a kernel of conditionally negative type. 
\end{example}

\begin{theorem}[GNS for algebraic cross-ratios]
\label{thm:algebraic-cross-ratio}
Consider a space $Y$ having at least four points with an algebraic cross-ratio $R\colon Y^{(4)} \to [0,+\infty)$.

There is a cardinal $\kappa$ and a continuous map $f\colon Y \to \partial \Hyp^{\kappa}$ with total image such that for all $x^-,x^+,y^-, y^+ \in Y^{(4)}$ we have $\Cr{f(x^-),f(x^+),f(y^-),f(y^+)}=R(x^-, x^+, y^-, y^+)$.

Moreover, for any two such maps $f_1\colon Y \to \partial \Hyp^{\kappa_1}$ and $f_2\colon Y \to\partial \Hyp^{\kappa_2}$, there exists a unique isometry $\varphi\colon\Hyp^{\kappa_1}\to\Hyp^{\kappa_2}$ such that $f_2=\varphi \circ f_1$, in particular $\kappa_1 = \kappa_2$.

Consequently, if a group $\Gamma$ acts on $Y$ preserving $R$ then there is a representation $\rho\colon G\to \Isom(\Hyp^\kappa)$ such that $f$ is $\rho$-equivariant.
\end{theorem}

\begin{proof}
Choose a point $\infty \in \partial X$. 
Fix two distinct points $0,1\in Y \setminus\{\infty\}$, and perform the GNS construction \ref{thm:Hilbert-GNS} to obtain a Hilbert space $\R^{\kappa-1}$ and a map $f \colon Y \setminus\{\infty\}\to\R^{\kappa-1}$ with total image such that \[\|f(y)-f(x)\|=R(x,y,0,\infty)R(y,0,1,\infty).\] 

Consider the upper half-space model $\Hyp^\kappa\simeq\R^{\kappa-1}\times(0,+\infty)$ and extend $f \colon Y \to \partial \Hyp^\kappa$ by sending $\infty$ to $\infty$.
For all
$x,y,z\in Y\setminus\{\infty\}$ we have, using \eqref{eq:Cr-Inversion} and \eqref{eq:Cr-Cocycle} for $R$, that:
\begin{equation*}
	\Cr{f(x),f(y);f(z),f(\infty)}=\tfrac{\|f(y)-f(x)\|}{\|f(y)-f(z)\|}
	=\tfrac{R(x,y,0,\infty)}{R(z,y,0,\infty)}
	=R(x,y,z,\infty)
\end{equation*}
Using again \eqref{eq:Cr-Cocycle}, \eqref{eq:Cr-Invariance} and \eqref{eq:Cr-Inversion} we get for all $(x^-, x^+, y^-, y^+)\in Y^{(4)}$ that:
\begin{align*}
	[f(x^-),f(x^+);f(y^-),f(y^+)]&=[f(x^-),f(x^+);f(\infty),f(y^+)][f(\infty),f(x^+);f(y^-),f(y^+)]\\
	&=[f(x^-),f(x^+);f(y^+),f(\infty)]^{-1}[f(y^-),f(y^+);f(x^+),f(\infty)]^{-1}\\
	&=R(x^-,x^+,y^+,\infty)^{-1}R(y^-,y^+,x^+,\infty)^{-1}
	=R(x^-,x^+,y^-, y^+).
\end{align*}
Thus $f \colon Y \to \partial \Hyp^{\kappa}$ preserves the cross-ratio.

Now, consider two such functions $f_1$ and $f_2$. Choose $\infty \in A$ and let $\infty_i=f_i(\infty)$. Consider the upper half-space model of $\Hyp^{\kappa_i}\simeq\R^{\kappa_i-1}\times(0,+\infty)$ so that $\infty_i$ is the point at infinity in $\partial\Hyp^{\kappa_i}$ and $||\ ||_i$ is the Hilbert norm on $\R^{\kappa_i-1}$.
By construction, there exists a unique $\lambda>0$ such that all $x,y\in Y\setminus \{\infty\}$ satisfy \[\|f_2(x)-f_2(y)\|_2=\lambda\|f_1(x)-f_1(y)\|_1.\] 
Since the images of $f_i$ in $\R^{\kappa_i-1}$ are total, there is a unique isometry $\varphi\colon\R^{\kappa_1-1}\to\R^{\kappa_2-1}$ such that $f_2= \lambda.\varphi \circ f_1$ on $Y\setminus\{\infty\}$.
The function $f'\colon \Hyp^{\kappa_1}\to \Hyp^{\kappa_2}$ defined by $f'(x,t)=(\lambda f(x),\lambda t)$ yields an isometry with the desired properties, and its unicity follows from that of $f$.
\end{proof}

\begin{proposition}[independence of exotic embeddings]
\label{prop:independence-exotic-embeddings}
For $t\in (0,1)$, the exotic deformation $\exoH_t \colon \partial \Hyp^\kappa \to \partial\Hyp^{\kappa+\infty}$ has the property that if $F$ is a finite subset of $ \partial \mathbf{\Hyp^\kappa}$ has $\Card(F)\ge 2$, then the minimal totally geodesic subspace of $\Hyp^{\kappa+\infty}$ whose boundary contains $\exoH_t(F)$ has dimension $\Card(F)-1$.
\end{proposition}

\begin{proof}
By \Cref{lem:powers-cross-ratio} on powers of cross-ratios, for any quadruple of distinct $w,x,y,z\in\partial \Hyp^\kappa$, we have $\Cr{\exoH_t(w),\exoH_t(x),\exoH_t(y),\exoH_t(z)}=\Cr{w,x,y,z}^t$.

Fix a point $\infty\in\partial\Hyp^\kappa$. 
Using the upper half-space model described in \Cref{eg:cross_ratio_algebraic_cross_ratio}, the set $\partial \Hyp^\kappa\setminus\{\infty\}$ has a structure of Hilbert space, and we have a kernel $K=R^2$ of conditionally negative type on $\partial \Hyp^\kappa\setminus\{\infty\}$ defined by $R(x,y)=\|x-y\| =\frac{\Cr{x,y,0,\infty}}{\Cr{y,0,1,\infty}}$. 
The same is true for $\partial \Hyp^{\kappa+\infty}\setminus\{\exoH_t(\infty)\}$ and we have $R^t(x,y)=\frac{\Cr{\exoH_t(x),\exoH_t(y),\exoH_t(0),\exoH_t(\infty)}}{\Cr{\exoH_t(y),\exoH_t(0),\exoH_t(1),\exoH_t(\infty)}}=\frac{\|\exoH_t(x)-\exoH_t(y)\|}{\|\exoH_t(0)-\exoH_t(1)\|}$ . 

The affine subspaces of dimension $n$ in the Hilbert space $\partial \Hyp^{\kappa+\infty}\setminus\{\exoH_t(\infty)\}$ correspond to the boundaries of totally geodesic subspaces of $\Hyp^{\kappa+\infty}$ of dimension $n+1$ containing $\exoH_t(\infty)$ in their boundary.
Since $K=R^2$ is conditionally of negative type, by \Cref{lem:power_kernel_negative_type}, $K^{t}=R^{2t}$ is strictly of conditionally negative type, thus $\exoH_t$ has affinely independent image on finite subsets of $\partial\Hyp^\kappa\setminus\{\infty\}$.
\end{proof}

\subsection{Rigidity results from algebraic cross-ratios}

We now apply \Cref{thm:algebraic-cross-ratio} to obtain rigidity results for actions on algebraic hyperbolic spaces.

\begin{corollary}[Poincaré extension]
\label{cor:Poincare-extension} 
Consider, for $i=1,2$, total subsets $L_i\subset \partial\Hyp^{\kappa_i}$.

For every bijection $f \colon L_1\to L_2$ preserving the cross-ratio, there exists a unique isometry $\varphi \colon \Hyp^{\kappa_1}\to\Hyp^{\kappa_2}$ such that $\varphi\colon \partial\Hyp^{\kappa_1}\to\partial\Hyp^{\kappa_2}$ coincides with $f$ on $L_1$.

The unicity implies in particular that if $f$ is equivariant for some group action then so is $\varphi$.
\end{corollary}

\begin{proof}
When $n=\Card L_i<4$, this follows from the fact $\Isom(\Hyp^{1+n})$ acts transitively on $n$-tuples of distinct points in $\partial \Hyp^{1+n}$.
When $\Card L_i\ge 4$, we apply \Cref{thm:algebraic-cross-ratio} with $Y=L_1$ to the identity $\mathbf{1} \colon L_1 \hookrightarrow \Hyp^{\kappa_1}$ and $f\colon L_1 \to L_2\subset  \Hyp^{\kappa_2}$.
There exists a unique isometry $\varphi \colon \Hyp^{\kappa_1}\to\Hyp^{\kappa_2}$ whose restriction to $L_1$ is $f$.
\end{proof}

Let us illustrate the usefulness of algebraic cross-ratios to simplify one step in the proof of \cite[Theorem A]{Kim_marked-length-rigidity-symmetric-space_2001} about the rigidity of marked length spectrum, at least in the real case.

\begin{theorem}[rigidity of marked-length spectrum]
Let $\Gamma$ be a group and $\rho_i\colon \Gamma\to\Isom(\Hyp^{\kappa_i})$ be two Zariski dense hyperbolic representations with finite dimensions $\kappa_i\in \N$ for $i=1,2$.
If the marked length spectrum of the two representations are proportional then the representations are conjugated. 
In particular, $\kappa_1=\kappa_2$ and the marked length spectrum are equal.
\end{theorem}

\begin{proof}
Let $\ell_i$ be the length function associated with $\rho_i$. Up to permuting the two representations, there is $0<t\leq1$ such that $\ell_1=t\ell_2$. 
Let $\rho_2'=\chi_t\circ\rho_2\colon\Gamma\to\Isom(\Hyp^{\kappa})$ where $\kappa=\kappa_2$ if $t=1$ and $\kappa=\infty$ if $t<1$. The length function $\ell_2'$ associated with $\rho_2'$ satisfies $\ell_2'=\ell_1$. 
Hence denoting by $L_i$ the set of attractive points at infinity of loxodromic elements of $\rho_1$ and $\rho_2'$, it follows from \cite[Lemma 3]{Kim_marked-length-rigidity-symmetric-space_2001} that there is a bijection $f\colon L_1\to L_2$ that preserves the cross-ratio. 
By Corollary~\ref{cor:Poincare-extension}, there is an equivariant isometry $\varphi\colon\Hyp^{\kappa_1}\to\Hyp^{\kappa}$. In particular, $\kappa=\kappa_1$ and thus $t=1$ which implies that $\kappa_1=\kappa_2$.
\end{proof}

The following generalizes \cite[Lemma 3]{Kim_marked-length-rigidity-symmetric-space_2001} to infinite dimensions.

\begin{lemma}[translation length recovers cross-ratio on limit set]
\label{lem:equivariant_cross_ratio_preserving_map}
Consider for $i\in\{1,2\}$, representations $\rho_i\colon \Gamma \to \Isom(\Hyp^{\kappa_i})$ 
whose translation length functions are non-zero and equal.
There exists an equivariant map between their limit sets $\varphi \colon \Lambda_1\to \Lambda_2$ preserving the cross-ratio such that every $\rho_i$-loxodromic $\gamma\in \Gamma$ has its fixed points $\rho_i(\gamma)^\pm$ satisfying $\varphi(\rho_1(\gamma)^\pm)=\rho_2(\gamma)^\pm$.
\end{lemma}

\begin{proof}
Since $\ell_i$ is non-zero, the subset of fixed points of its loxodromic elements $L_i\subset \partial \Hyp^{\kappa_i}$ has $\Card L_i\in \{2,\infty\}$, and it is dense by \cite[Propositions 7.4.6]{Das-Simmons-Urbanski_gromov-hyperbolic-spaces_2017} in its limit set $\Lambda_{i}\subset \Hyp^{\kappa_i}$.

Since the translation length functions coincide, the subsets of elements in $\Gamma$ that act loxodromically with respect to $\rho_1$ and $\rho_2$ coincide. 
For two such loxodromic elements $\alpha,\beta \in \Gamma$, it follows from \Cref{lem:equal-fixed-points-from-length} that $\rho_1(\alpha)^+=\rho_1(\beta)^+ \iff \rho_2(\alpha)^+=\rho_2(\beta)^+$.
Thus we have a bijection $\varphi \colon L_1\to L_2$ defined by the formula $\rho_1(\gamma)^+\mapsto \rho_2(\gamma)^+$. 
Let us show that $\varphi$ preserves the cross-ratio. Actually, for any elements $\alpha,\beta\in\Gamma$ that are hyperbolic for $\rho_1,\rho_2$, by \Cref{eq:translation-length_Cr-boundary_exp_1} we have \[\Cr{\rho_1(\alpha)^+,\rho_1(\alpha)^-,\rho_1(\beta)^+,\rho_1(\beta)^-}=\Cr{\rho_1(\alpha)^+,\rho_1(\alpha)^-,\rho_1(\beta)^+,\rho_1(\beta)^-}.\]

Moreover for any group elements $\alpha,\beta,\gamma,\delta$ that are hyperbolic for $\rho_i$, with $\rho_i(\alpha)^+\neq\rho_i(\beta)^+$ and $\rho_i(\gamma)^+\neq\rho_i(\delta)^+$, we have $\rho_i(\alpha^n\beta^{-n})^+\to\rho_i(\alpha)^+$, $\rho_i(\alpha^n\beta^{-n})^-\to\rho_i(\beta)^+$, $\rho_i(\gamma^n\delta^{-n})^+\to\rho_i(\gamma)^+$ and $\rho_i(\gamma^n\delta^{-n})^-\to\rho_i(\delta)^+$. So

\begin{align*}
	\Cr{\rho_1(\alpha)^+,\rho_1(\beta)^+;\rho_1(\gamma)^+,\rho_1(\delta)^+}
	=&\lim_{n}\Cr{\rho_i(\alpha^n\beta^{-n})^+,\rho_i(\alpha^n\beta^{-n})^-;\rho_i(\gamma^n\delta^{-n})^+,\rho_i(\gamma^n\delta^{-n})^-}\\
	=&\Cr{\rho_2(\alpha)^+,\rho_2(\beta)^+;\rho_2(\gamma)^+,\rho_2(\delta)^+}
\end{align*}

which proves that $\varphi\colon L_1\to L_2$ preserves the cross-ratio. It extends by continuity to a cross-ratio preserving map between the limit sets $\Lambda_1\to \Lambda_2$.

Observe that if $\Card L_i=2$ then the condition about the cross-ratio is empty.
\end{proof}

Recall from \Cref{rem:irred=>eleminimal} that an irreducible representation is a minimal representation without fixed point at infinity.
The following proposition shows that functions of hyperbolic type  at uniformly bounded distance from one another correspond to irreducible  representations up to conjugacy.

\begin{proposition}[rigidity of marked length spectrum]
\label{prop:same-length-spectrum-means-conjugate}
Consider a group $\Gamma$ with minimal representations $\rho_i\colon G\to \Isom(\Hyp^{\kappa_i})$ for $i=1,2$. 
If $\rho_1,\rho_2$ have non-trivial and equal length functions $\ell_{\rho_1}=\ell_{\rho_2}$, then $\kappa_1=\kappa_2$ and $\rho_1,\rho_2$ are conjugated by an isometry $\Hyp^{\kappa_1}\to \Hyp^{\kappa_2}$.
\end{proposition}

\begin{proof}
The Lemma \ref{lem:equivariant_cross_ratio_preserving_map} provides a $\rho_i$-equivariant cross-ratio preserving map $\varphi \colon \Lambda_1 \to \Lambda_2$, and we have $\Card \Lambda_i$ is $2$ or uncountable.
If $\Card L_i=\infty$ is uncountable, by minimality of $\rho_i$, the set $\Lambda_i$ is total in $\partial \Hyp^{\kappa_i}$, so the Corollary \ref{cor:Poincare-extension} yields a $
\rho_i$-equivariant isometry $\Hyp^{\kappa_1}\to\Hyp^{\kappa_2}$.
If $\Card L_i=2$, then we have two isometric actions of $\Gamma$ on the real line, and it is an exercise to show that two such actions with the same (non-zero) marked length spectrum are conjugated.
\end{proof}


\begin{corollary}[length equivalent]
\label{cor:equal_length_spectrum}
For $i\in\{1,2\}$, consider non-neutral $F_i\in \mathcal{F}_{nn}(\Gamma)$ yielding non-zero length functions $\ell_i \colon \Gamma\to \R_{\ge 0}$.
For any minimal representations $\rho_i\colon \Gamma \to \Isom(\Hyp^{\kappa_i})$ realizing $F_i$, the following are equivalent:
\begin{itemize}[noitemsep]
	\item[$F$:] The functions $\log(F_1/F_2)$ or equivalently $\cosh^{-1} F_1 - \cosh^{-1} F_2$ are bounded.
	\item[$\ell_F$:] The length functions $\ell_1$ and $\ell_2$ are equal.
	\item[$\rho_F$:] The representations $\rho_1$ and $\rho_2$ are conjugate.
\end{itemize}
\end{corollary}

\begin{proof}
If $\log(F_1/F_2)$ is bounded then the Lemma \ref{lem:ellF-from-F} recovering $\ell_{F}$ from $F$ implies $\ell_1 = \ell_2$.

In that case, the representations $\rho_1, \rho_2$ realizing $F_1$ and $F_2$ are non-neutral and minimal so by Proposition \ref{prop:same-length-spectrum-means-conjugate} they are conjugated by an isometry $\varphi\colon \Hyp^{\kappa_1} \to \Hyp^{\kappa_2}$.

In that case, choosing $o_i\in \Hyp^{\kappa_i}$ with $o_2=\varphi(o_1)$, we have $\dH(o_1,\rho_1(\gamma) o_1)=\dH(o_2, \rho_2(\gamma) o_2)$ so $\| \cosh^{-1}F_1(\gamma) - \cosh^{-1}F_{2}(\gamma)\|$ is bounded, or equivalently $\log(F_1/F_2)$ is bounded.
\end{proof}

\begin{remark}[neutral exceptions]
\Cref{cor:equal_length_spectrum} does not generalize to the neutral case.
For example, recall from  \Cref{subsec:Isom(Hyp)} that elliptic or parabolic representations have zero length functions, but by Lemma \ref{lem:F-recovers-type-eph} their associated functions of hyperbolic type are respectively bounded and unbounded.
\end{remark}

\begin{corollary}[homothetic marked length spectra]
\label{cor:homothetic_length_spectrum}
Consider functions of hyperbolic type $F_i\colon \Gamma\to \R_{\ge 1}$ yielding non-trivial length functions $\ell_i \colon \Gamma\to \R_{\ge 0}$, and any minimal representations $\rho_i\colon \Gamma \to \Isom(\Hyp^{\kappa_i})$ realizing $F_i$. 
For every $t\in (0,1]$ the following are equivalent:
\begin{itemize}[noitemsep]
	\item[$F$:] The $F_i$ are \emph{$t$-equivalent}, namely $\log(F_1/F_2^t)$ is bounded.
	\item[$\ell_F$:] The $\ell_i$ are \emph{$t$-homothetic}, namely $\ell_1=t\ell_2$.
	\item[$\rho_F$:] The representation $\rho_1$ is conjugate to the $t$-deformed representation $\chi_t \circ \rho_2$.
\end{itemize}
In that case the $F_i$ share the same objects and attributes associated to the homothety class of $\ell_i$, such as elementarity and neutrality (yielding a $t$-variant of \Cref{cor:neutral-elementary}).
\end{corollary}

\begin{proof}
If $\log(F_1/F_2^t)$ is bounded then the Lemma \ref{lem:ellF-from-F} recovering $\ell_{F}$ from $F$ implies $\ell_1 = t\ell_2$.

In that case, the representations $\rho_1, \chi_t\circ \rho_2$ realizing $F_1$ and $F_2^t$ are non-neutral and minimal so by Proposition \ref{prop:same-length-spectrum-means-conjugate} they are conjugated by $\varphi\colon \Hyp^{\kappa_1} \to \Hyp^{\kappa_2}$.

In that case, choosing $o_i\in \Hyp^{\kappa_i}$ with $o_2=\varphi(o_1)$, we have $\dH(o_1,\rho_1(\gamma) o_1)=\dH(o_2, (\rho_t\circ \rho_2)(\gamma) o_2)$ so $\| \cosh^{-1}F_1(\gamma) - \cosh^{-1}F_{2}^t(\gamma)\|$ is bounded, or equivalently $\log(F_1/F_2^t)$ is bounded.
\end{proof}

\section{Deformations and degenerations to actions on trees}
\label{sec:degenerations-actions-trees}

\subsection{Critical exponent, dimension, representative}
\label{subsec:critical-exponent}
\begin{definition}[critical exponent of a kernel of hyperbolic type]
\label{def:critical-exponent}
For a kernel of hyperbolic type $K\colon X\times X\to \R_{\ge 1}$ we define its \emph{critical exponent} $t_K\in [1,\infty]$ as the supremum of $t\in [1,\infty)$ such that $K^t$ is a kernel of hyperbolic type.

When $t_K<\infty$ we call $K^{t_K}$ the \emph{critical representative} of $K$, and thus \Cref{prop:hyperbolic-GNS} defines the associated \emph{critical dimension} $\kappa_K$ for which there exists an essentially unique \emph{critical embedding} $f_{K^{t_K}}\colon X\to \Hyp^{\kappa_K}$ with total image satisfying $K^{t_K}=\cosh \circ \dH$.
\end{definition}

\begin{definition}[arboreal]
\label{def:arboreal}
A kernel of hyperbolic type $K\in \mathcal{K}(X)$ is \emph{arboreal} when there is a real tree $(T,d)$, a function $f\colon X\to T$ such that $K(x,y)=\exp d(f(x),f(y))$.
\end{definition}

\begin{proposition}[arboreal = infinite critical exponent]
\label{prop:characterisation_kernel_hyperbolic_type_tree}
A kernel of hyperbolic type $K\in \mathcal{K}(X)$ is arboreal if and only if it has critical exponent $t_K=\infty$.
\end{proposition}
\begin{proof}
Assume that $X$ is arboreal, and consider a function to a real tree $f\colon X\to T$ such that $K(x,y)=\exp(d(f(x),f(y)))$.
For all $t>0$, we have $\exp(td)\in \mathcal{K}(T)$ hence composing with $f$ shows that $K^t\in \mathcal{K}(X)$.

Conversely assume that $t_K=+\infty$, namely for all $t>0$ the power $K^t\colon X\times X\to \R_{> 0}$ is a kernel of hyperbolic type.
By \Cref{prop:hyperbolic-GNS}, there is an embedding $f_t\colon X\to\Hyp^{\kappa_t}$ with total image such that $K^t(x,y)=\cosh(\dH(f_t(x),f_t(y)))$. 
Thus for all $t>1$, we have:
\begin{equation*}
	\log K(x,y) = \tfrac{1}{t} \log \cosh \dH(f_t(x),f_t(y)).
\end{equation*}
This implies that $\dH(f_t(x),f_t(y))\to +\infty$ as $t\to \infty$, and since $\log \cosh(u)\sim u$ as $u\to+\infty$ we have:
\begin{equation*}
	\log K(x,y)\sim\tfrac{1}{t}\dH(f_t(x),f_t(y)).
\end{equation*}
By taking the limit we find that $\log K$ satisfies the triangle inequality, and as it is symmetric it defines a semi-metric on $X$.
Hence on $X$ the relation $\log K(x,x')=0$ is an equivalence relation and the quotient $X\bmod{K=1}$ is endowed with metric $\log K$.
Moreover we deduce from the asymptotic for $\log K$ and the fact $(\Hyp^{\kappa_t}, \dH)$ is Gromov-$\log(2)$-hyperbolic, that the metric space $(X\bmod{K=1}, \log K)$ is $0$-hyperbolic.
The result now follows from \Cref{prop:embedding_0_hyperbolic_space_tree}.
\end{proof}

\begin{remark}[infinitely divisible]
There is a similar notion to being arboreal for kernels of positive type: to be infinitely divisible. 
A kernel of positive type $K$ that takes nonnegative values is \emph{infinitely divisible} when for all $t>0$ the kernel $K^t$ is of positive type (see \cite[Definition 2.6 and Proposition 2.7]{Berg-Christensen-Ressel_Harmonic-analysis-semigroups_1984}).

This is more than an analogy: if $K$ is a kernel hyperbolic type that is arboreal then $K^{-1}$ is a kernel of positive type that is infinitely divisible. Indeed if $K$ is arboreal, then there is $f \colon X\to T$ and $d$ as in \Cref{def:arboreal}, such that $K=\exp(D_T)$ where $D_T \colon (x,y) \mapsto d(f(x),f(y))$.
The kernel $D_T$ is conditionally of negative type (for trees, this can be found in \cite{Julg-Valette_K-theory-amenablity_1984}; for real trees, it follows by approximating the convex hull of finitely many points by a finite tree with arbitrarily small edges). 
Hence the kernel $K^{-1}=\exp(-D_T)$ takes nonnegative values and satisfies condition (ii) in \cite[Proposition 2.7]{Berg-Christensen-Ressel_Harmonic-analysis-semigroups_1984}, hence it is infinitely divisible. 

The converse does not hold (namely $K^{-1}$ could be infinitely divisible and $K$ not arboreal), see for example \cite[Exercise 2.14]{Berg-Christensen-Ressel_Harmonic-analysis-semigroups_1984}.
\end{remark}

\begin{proposition}[critical exponent is upper-semi-continuous]
\label{prop:semicontinuity_critical_exponent}
The function sending a kernel of hyperbolic type $K\in \mathcal{K}(X)$ to its critical exponent $t_K\in \R_{\ge 1}\sqcup\{\infty\}$ is upper semi-continuous.
Moreover if $t_K<\infty$ then $K^{t_K} \in \mathcal{K}(X)$.
\end{proposition}
\begin{proof}
For the upper semi-continuity, we will show that for $t\in \R_{\ge 1}\sqcup\{\infty\}$ the preimage of $[1,t)$, namely the subset $\mathcal{V}^t$ of $K\in\mathcal{K}(X)$ such that $K^t\notin\mathcal{K}(X)$, is open.
For fixed $x=(x_0,\dots,x_n)\in X^{1+n}$, the subset $\mathcal{V}^t_{x}$ of $K\in \mathcal{K}(X)$ such that $K^t$ violates the hyperbolic Cayler-Menger determinant inequality \eqref{eq:kernel-hyperbolic-CM} at $x$, namely such that $\det \left(-K^t(x_i,x_j)\right)_{0\le i,j\le n} > 0$, is open for the pointwise convergence.
By definition $\mathcal{V}^t$ is the union of $\mathcal{V}^t_x$ over all $n\in \N$ and $x\in X^{1+n}$, so it is open.

Now assuming that $t_K<\infty$, suppose by contradiction that $K^{t_K}\notin \mathcal{K}$. 
By definition, there exists $n\in \N$ and $x\in X^{1+n}$ such that $\det \left(-K^{t_K}(x_i,x_j)\right)_{0\le i,j\le n}>0$.
Since for every $x$ the function $t \mapsto \det \left(-K^t(x_i,x_j)\right)_{0\le i,j\le n}$ is continuous, there exists $t_x<t_K$ such that $K\in \mathcal{V}^t_x$, which would imply $t_K \le t_x$, a contradiction.
Hence $K^{t_K}\in \mathcal{K}(X)$ as desired.
\end{proof}

Applying upper-semi-continuity of the critical exponent at $t_K=\infty$ yields the following.

\begin{corollary}[arboreal are closed]
\label{cor:arobreal-closed}
In the space of kernels of hyperbolic type, those which are arboreal form a closed set.
\end{corollary}

\begin{definition}[critical dimension and representative]
\label{def:critical_representative}
Let $K\in \mathcal{K}(X)$. 

When $t_K<\infty$, we define the \emph{critical representative} of $K$ as $K^{t_K} \in \mathcal{K}(X)$, and \Cref{prop:hyperbolic-GNS} defines the associated \emph{critical dimension} $\kappa_K$ for which there exists an essentially unique \emph{critical map} $f_{K^{t_K}}\colon X\to \Hyp^{\kappa_K}$ with total image satisfying $K^{t_K}=\cosh\circ \dH$. 

When $K$ is arboreal, we define the \emph{critical dimension} $\kappa$ such that for all $t>0$, we have a map $f_{K^t} \colon X\to \Hyp^\kappa$ with total image satisfying $K^{t}=\cosh \circ\dH$ as in \Cref{prop:hyperbolic-GNS}.
\end{definition}

\begin{proposition}
\label{prop:semicontinuity_critical_dimension}
For $K_0\in \mathcal{K}(X)$ with $\kappa_{K_0}<\infty$, if $t_{K_0}<\infty$ then the critical dimension function $K\mapsto \kappa_K$ is lower semi-continuous at $K_0$.
\end{proposition}

\begin{remark}[variant assumption]
\label{rem:semicontinuity_critical_dimension}
When $X$ is infinite, if $\kappa_{K_0}<\infty$ and $K_0$ separates points (namely $K_0(x,y)=1\implies x=y$) then $t_{K_0}<\infty$ (since otherwise we would have $\dim(K_0^{t_{K_0}})=\Card(X)-1=\infty$ by \Cref{prop:strictly_of_hyperbolic_type} contradicting $\kappa_0<\infty$), so \Cref{prop:semicontinuity_critical_dimension} applies.
\end{remark}

\begin{proof}
Choose any ultrafilter $\mathcal{F}$ on $\mathcal{K}(X)$ that converges to $K_0$. 

By \Cref{prop:semicontinuity_critical_exponent}, the limit $t_\mathcal{F}=\lim_{\mathcal{F}}t_K$ satisfies $t_\mathcal{F}\le t_{K_0} <\infty$. 
Moreover, we have $\lim_{\mathcal{F}}K^{t_K}=K_0^{t_{\mathcal{F}}}$ so by \Cref{prop:semicontinuity_kernel-dimension} on lower semi-continuity of the dimension we have
\(\lim_{\mathcal{F}}\kappa_K\geq \dim(K_0^{t_{\mathcal{F}}})\).

Finally, the linear independence proved in \Cref{prop:strictly_of_hyperbolic_type} shows that for any $K\in \mathcal{K}(X)$ and all $t\le 1$, we have $\dim(K^t)\geq\dim(K)$, so $\dim(K_0^{t_{\mathcal{F}}})\geq \kappa_{K_0}$.

Chaining up the inequalities yields
\(\lim_{\mathcal{F}}\kappa_K\geq \dim(K_0^{t_{\mathcal{F}}})\ge \kappa_{K_0}\), proving the desired lower semi-continuity of $K\mapsto\kappa_K$ at $K_0$.

Let us note that when $X$ is countable, the space of kernels of hyperbolic type is metrizable, so one may simplify this proof replacing the ultrafilter by a converging sequence up to extracting two successive subsequences.
\end{proof}

\begin{remark}[\Cref{prop:semicontinuity_critical_dimension} may fail at arboreal]
When $X$ has at least four elements $o,x,y,z$, if $K_0$ is arboreal then \Cref{prop:semicontinuity_critical_dimension} can fail as we now show.

In $\Hyp^2$ consider for every $n\in \N_{\ge1}$ an equilateral triangle with vertices $x_n,y_n,z_n$ and incenter $o_n$ such that $d(o_n,x_n)=n$. 
The map $f_n\colon X\to \Hyp^2$ sending $(x,y,z)\mapsto (x_n,y_n,z_n)$ and sending all other points to $o_n$ yields a kernel of hyperbolic type on $X$ defined by $K_n(u,v) = \cosh \dH(f_n(u),f_n(v))$; satisfying in particular $K_n(o,x)=K(o,y)=K(o,z)=\cosh(n)$.
The sequence of kernels $K_n^{1/n}$ converges to the arboreal kernel of hyperbolic type $K_0$ associated to the star with three edges of length $e$.
For all $n$, $\kappa_{K_n^{1/n}}=2$ but $\kappa_{K_0}=3$, so this contradicts lower semi-continuity. The issue is that $t_{K_n^{1/n}}=n\to\infty$. 

Such counter examples cannot occur when $X$ is infinite and $K_0$ with $\kappa_0<\infty$ separates points (as mentioned in \Cref{rem:semicontinuity_critical_dimension}).
\end{remark}

\subsection{Equivariant embeddings of real trees into \texorpdfstring{$\Hyp^\omega$}{Homega}}

After \cite{Morgan-Shalen_valuations-trees-degernations-hyperbolic_1984, Bestvina_degenerations-hyperbolic-space_1988, Paulin_1988}, the spaces of actions of groups containing non-abelian free group on a finite dimensional hyperbolic space have been compactified by spaces of actions on real trees.

In fact, the action of a group $\Gamma$ on a real tree $T$ also yields an action on $\Hyp^\omega$ via the following embedding theorem, proved for simplicial trees in \cite{Burger-Iozzi-Monod_embedding-trees-hyperbolic-spaces_2005}, and extended to real trees in \cite[Theorem 13.1.1]{Das-Simmons-Urbanski_gromov-hyperbolic-spaces_2017} and \cite[Proposition 1.5]{Monod_notes-functions-hyperbolic-groups_2020}.

\begin{theorem}[equivariant isometric embedding of trees into hyperbolic space]
\label{thm:tree-embedding-Hyp-infty} 
Let $T$ be a real tree. For any $\lambda\in \R_{>1}$ there exists a cardinal $\kappa$ and a representation $\pi_\lambda\colon \Isom(T)\to\Isom(\Hyp^\kappa)$ and an $\Isom(T)$-equivariant continuous injective function $\exoT_\lambda\colon T\to\Hyp^\kappa$ such that 
\begin{equation}
	\label{eq:tree_distance}
	\lambda^{d(x,y)}=\cosh \dH(\exoT_\lambda(x),\exoT_\lambda(y)).
\end{equation}
Moreover, the embedding $\exoT_\lambda$ extends to an $\Isom(T)$-equivariant embedding $\partial \exoT_\lambda\colon\partial T\to\partial\Hyp^\kappa$ which is a homeomorphism on its image.
\end{theorem}

\begin{remark}[unicity]
\label{rmk:unicity_pt} It follows from the unicity statement for kernels of hyperbolic type (\Cref{prop:hyperbolic-GNS}) if any function $f\colon T\to\Hyp^{\kappa'}$ with total image satisfies \Cref{eq:tree_distance}, then $\kappa=\kappa'$ and there exists $\varphi \in\Isom(\Hyp^\kappa)$ such that $f=\varphi \circ\exoT_\lambda$.
\end{remark}

\begin{remark}[maybe not irreducible]
The representation constructed in \Cref{thm:tree-embedding-Hyp-infty} is often not irreducible, even when the action on $T$ is minimal (namely has no invariant convex subspace. For example, it $T$ is a simplicial tree, the closed span of the images of the vertices will be an invariant subspace which does not contain any midpoint of an edge. Actually, to see that, one can consider the barycentric subdivision of the tree and apply the explicit construction in \cite[\S8]{Burger-Iozzi-Monod_embedding-trees-hyperbolic-spaces_2005} to see that the image of the vertices lie in the Hilbert space spanned by the Dirac functions $\delta_v$ where $v$ is a vertex and for the midpoint of an edge $m$, its image has non zero component with respect to $\delta_m$ which is a unit vector orthogonal to all $\delta_v$'s.
\end{remark}

\begin{remark}[homothetic]
\label{rem:homotheties}
It follows that for a tree $T$ with a fixed base point $o\in T$, for $\lambda_0,\lambda_1>1$ the functions of hyperbolic type $P_{i}\colon \gamma \in \Isom(T) \mapsto \cosh \dH(\exoT_{\lambda_i}(o),\pi_{\lambda_i}(\gamma)\cdot \exoT_{\lambda_i}(o))\in \R_{\ge 1}$ on $\Isom(T)$ associated with $\exoT_{\lambda_i}(o),\pi_{\lambda_i}$ satisfy $P_{\lambda_1}=P_{\lambda_0}^t$ where $t=\tfrac{\log(\lambda_1)}{\log(\lambda_0)}$, hence the $P_i$ are homothetic.
\end{remark}

\begin{remark}[homothety at large scale but non-rectifiable at small scale]
The embedding $\exoT_\lambda$ in Theorem \ref{thm:tree-embedding-Hyp-infty} scales the metric according to different regimes.

At large scale $\exoT_\lambda$ behaves like a homothety of ratio $\log(\lambda)$ since $\dH(\exoT_\lambda(x),\exoT_\lambda(y))$ is asymptotic to $\log(\lambda)\cdot d(x,y)$ when $d(x,y)\to \infty$ (see Lemma~\ref{tree-length-cross-ratio}). 
At infinitesimal scale, the embedding is much more complicated. For example, the image of a geodesic in $T$ is not a rectifiable curve since by Equation~\eqref{eq:tree_distance}, we have $\dH(\exoT_\lambda(x),\exoT_\lambda(y))=H(d(x,y))$ where $H\colon \R_{\ge 0} \to \R_{\ge 0}$ is a continuous function that is differentiable on $\R_{>0}$ satisfying \(\sinh(H(d))H'(d)=\log(\lambda)\lambda^d\).
Thus $\lim_{d\to0}H'(d)=+\infty$ which implies that the image of any geodesic has infinite length.
\end{remark}

We may now characterise when a representation $\rho\colon \Gamma\to\Isom(\Hyp^\kappa)$ comes from an action on a real tree as in \Cref{thm:tree-embedding-Hyp-infty}.

\begin{corollary}[arboreal]
\label{cor:arboreal-functions-hyperbolic-type}
For $F\in \mathcal{F}(\Gamma)$, the following are equivalent:
\begin{itemize}[noitemsep]
	\item The kernel $K\in \mathcal{K}(\Gamma)$ defined by $K(\alpha,\beta)=F(\alpha^{-1}\beta)$ is arboreal (\Cref{def:arboreal}).
	\item For all $t\in \R_{>0}$, the function $F^t\colon \Gamma \to \R_{\ge 1}$ is of hyperbolic type.
	\item For any corresponding representation $\rho\colon\Gamma\to\Isom(\Hyp^\kappa)$, there is a real tree $(T,d)$ and a representation $\rho_T \colon \Gamma\to\Isom(T)$ such that $\rho=\pi_e \circ\rho_T$ where $\pi_e$ is given by \Cref{thm:tree-embedding-Hyp-infty}.
\end{itemize}
\end{corollary}
\begin{proof}
The proof follows from \Cref{prop:characterisation_kernel_hyperbolic_type_tree} using the equivariance of the kernels and embeddings in \Cref{thm:tree-embedding-Hyp-infty}.
\end{proof}

\begin{remark}[bounded from arboreal does not imply arboreal]
On a group $\Gamma$, there may exist functions of hyperbolic type at bounded distance from an arboreal one but that are not arboreal.

For instance with $\Gamma=\Z$, consider a representation $\rho \colon \Gamma \to \Hyp^2$ sending the generator to a hyperbolic translation, and choose any point $o$ off its axis to obtain such a function of hyperbolic type $F=F_{\rho,o}$.

More generally, consider any minimal action of a group $\Gamma$ on an infinite real tree $(T,d)$.
We obtain from the embedding $\exoT_\lambda \colon T \to \Hyp^\kappa$ an irreducible representation $\rho \colon \Gamma \to \Isom(\Hyp^\kappa)$.
Choosing a base point $o\in \Hyp^\kappa$ yields a function of hyperbolic type $F_{\rho,o}$: it is arboreal if and only if $o\in \exoT_\lambda(T)$.
Those are all at bounded distance from one another. 
\end{remark}

We may apply Lemma \ref{lem:exotic-exp(td)} in this context to find how the representation $\pi_\lambda$ modifies the length function on $\Isom(T)$ and the cross-ratio on $\partial T$. The statement about length function already appears in \cite[Remark 13.1.6]{Das-Simmons-Urbanski_gromov-hyperbolic-spaces_2017}.

\begin{lemma}
\label{tree-length-cross-ratio}
Let $T$ be a real tree and consider an $\Isom(T)$-equivariant embedding $\pi_\lambda\colon \Isom(T)\to\Isom(\Hyp^\omega)$ satisfying Equation~\eqref{eq:tree_distance} for some $\lambda\in \R_{\ge 1}$.

For $\gamma\in\Isom(T)$ we have \(\ell_{\Hyp^\omega}(\pi_\lambda(\gamma))=\log(\lambda)\ell_T(\gamma)\).

For $x^-,x^+,y^-,y^+\in\partial T$, we have
\(\Cr{\exoT_\lambda(x^-), \exoT_\lambda(x^+); \exoT_\lambda(y^-), \exoT_\lambda(y^+)}=\Cr{x^-,x^+,y^-,y^+}^{\log(\lambda)}\).
\end{lemma}

\begin{remark}[deforming $\exoH_t$ versus $\exoT_\lambda$]
For the equivariant embeddings of hyperbolic spaces $\exoH_t \colon \Hyp^n \to \Hyp^\omega$ given by the exotic deformation \Cref{thm:exotic-deformation-Hyp}, the length spectrum can only be compressed by $t'\leq1$.
For the equivariant embeddings of real trees $\exoT_\lambda \colon T\to \Hyp^\omega$ in \Cref{thm:tree-embedding-Hyp-infty} the length spectrum can be compressed or dilated depending on whether $\lambda$ is smaller or larger than $\exp(1)$.
\end{remark}

\begin{remark}[characterising $\exoT_\lambda$ from the cross-ratio on dense subsets of $\partial T$]
Consider a real tree $(T,d)$ and a subgroup $\Gamma\subset \Aut(T)$.
Assume that $\Gamma$ acts minimally on $T$ (so that it has no strict non-empty closed convex $\Gamma$-invariant subset). 
In that case $\exoT_\lambda$ can be characterized as the unique irreducible representation $\rho\colon\Gamma\to\Isom(\Hyp^\kappa)$ such that every $\gamma\in\Gamma$ satisfies $\ell_\rho(\gamma)=\lambda\ell_T(\gamma)$.
Equivalently, this is the unique irreducible representation $\rho\colon\Gamma\to\Isom(\Hyp^\kappa)$ such that the image of any loxodromic element is loxodromic, and for any loxodromic $\alpha,\beta \in\Gamma$ we have $[\rho(\alpha)^-,\rho(\alpha)^+;\rho(\beta)^-,\rho(\beta)^+]=[\alpha^-,\alpha^+;\beta^-,\beta^+]^\lambda$. This follows from \Cref{prop:same-length-spectrum-means-conjugate} and \Cref{cor:Poincare-extension}.
\end{remark}

\subsection{Geometric limits of actions on \texorpdfstring{$\Hyp^\kappa$}{Hkappa} to actions on a real tree \texorpdfstring{$T$}{T}}

We endow $\Isom(\Hyp^\omega)$ with its Polish topology (induced by the pseudo-metrics $(g,h)\mapsto d(gx,hx)$ for $x\in\Hyp^\omega$, see \cite{Duchesne_Polish-topology-isom-H-infinity_2023}).
For a countable group $\Gamma$, the product topology on $\Isom(\Hyp^\omega)^\Gamma$ is Polish as well, so this restricts to a Polish topology on the closed subset of representations $\Hom(\Gamma,\Isom(\Hyp^\omega))$.
More concretely, a sequence of representations $\rho_m$ converges to a representation $\rho$ if and only if for all $(\gamma,x)\in \Gamma \times \Hyp^\omega$ we have convergence $\rho_m(\gamma)(x)\to\rho(\gamma)(x)$ (it suffices to consider $x$ in a total subset of $\Hyp^\omega$).

\begin{definition}[infimum displacement length]
For a finite subset $S \subset \Isom(\Hyp^{\kappa})$, we define the infimum $S$-displacement length as \[\Delta(S)=\inf \{ \max\{\dH(\gamma o,o)\colon \gamma\in S\} \colon o\in\Hyp^\kappa\}.\]
\end{definition}

\begin{lemma}[infimum displacement achieved]
\label{lem:min_displacement}
If a finite set $S$ generates a subgroup $\Isom(\Hyp^\kappa)$ with no fixed points in $\partial\Hyp^\kappa$, then the infimum defining $\Delta(S)$ is achieved at some point $o\in \Hyp^\kappa$.
\end{lemma}

\begin{proof} 
The function $f\colon\Hyp^\kappa\to\R^+$ defined by $f(x)=\max\{\dH(\gamma x,x)\colon \gamma\in S\}$ is a maximum of continuous and convex functions, hence it is continuous and convex. For $l\in (\Delta(S), +\infty)$, the set $C_l=f^{-1}([\Delta(S),l])$ is a closed convex subset of $\Hyp^\kappa$, and we obtain decreasing intersections $C_{\cap}= \bigcap_l C_l$ and $\overline{C}_{\cap} = \bigcap_l (C_l\cup \partial C_l)$.
By \cite[Proposition 4.4]{Duchesne_inf-dim-sym-spaces_2013}, either $C_{\cap}\ne \emptyset$, in which case $f$ has a minimum in $\Hyp^\kappa$ as desired; or else $\overline{C}_{\cap}$ consists of single point $\xi\in \partial \Hyp^\kappa$, which must be fixed by $S$ (since every element in $S$ displaces $C_l$ by at most $l$), a contradiction.
\end{proof}

\begin{theorem}[geometric convergence to an action on a tree]
\label{thm:geometric-convergence-tree}
Consider a group $\Gamma$ with a finite generating set $S$, a sequence of cardinals $\kappa_m\in \N\cup\{\omega\}$ and a sequence of representations $\rho_m\colon \Gamma\to\Isom(\Hyp^{\kappa_m})$.

If no $\rho_m$ has a fixed point at infinity, but the infimum displacement length $\Delta (\rho_m(S))$ is unbounded as $m\to \infty$, then up to extracting a subsequence of $(\rho_m)$ there are:
\begin{enumerate}[noitemsep]
	\item a real tree $(T,d)$ and an isometric action $\rho_T\colon\Gamma\to\Isom(T)$
	\item an embedding $\exoT_\lambda\colon T\to\Hyp^\omega$ equivariant for $\pi_\lambda\colon \Isom(T)\to\Isom(\Hyp^\omega)$ with $\lambda=e$,
	\item an isometric action $\rho_\omega \colon \Gamma \to \Isom(\Hyp^\omega)$ preserving $\exoT_\lambda(T)$,
	\item a sequence of representations $\rho_m'\colon\Gamma\to\Isom(\Hyp^\omega)$ such that $F_{\rho_m'}\sim F_{\rho_m}$ and
	\item a sequence of $(\rho_m,\rho_m')$-equivariant exotic embeddings $\exoH_{t_m}\colon\Hyp^{\kappa_m}\to\Hyp^\omega$.
\end{enumerate}
such that $(\rho'_m)_m \colon \Gamma \to \Isom(\Hyp^\omega)$ converges to $\rho_\omega \colon \Gamma \to \Isom(\Hyp^\omega)$ and $\rho_\omega=\pi_\lambda\circ\rho_T$.

Moreover, there exist points $o_m\in \operatorname{Core}(\rho_m)$ and $o\in T$ such that for every finite set $K\subset \Gamma$, we have Hausdorff convergence of $\exoH_{t_m}(\Conv(\rho_m(K) \cdot o_m))$ to $\exoT_\lambda(\Conv(\rho_\omega(K) \cdot o))$.
\end{theorem}

\begin{remark}[recovering previous results]
\label{rem:past-tree-compactifications}
Let us first emphasize that one of our new contributions lies in the freedom to consider actions $\rho_m$ on hyperbolic spaces whose dimensions $\kappa_m\in \N\cup\{\omega\}$ form a non constant sequence of countable cardinals, whereas all past works are in the setting where $\kappa_m$ is bounded by or equal to a fixed finite integer $n$.

If we specialize our Theorem \ref{thm:geometric-convergence-tree} to the case of a group $\Gamma$ which is not virtually Abelian and of faithful discrete representations $\rho_m\colon \Gamma\to \Isom(\Hyp^{n})$ with a constant sequence of finite dimensions $\kappa_m=n$, then the tree that we obtain is isometric to one of the trees obtained by Morgan \cite[Theorem 3.6]{Morgan_actions-trees-compactification-SOn1-representation_1986} (see also \cite{Otal_compact-repres_2015} for a survey of similar results by Culler, Morgan, Shalen and Otal) and later recovered independently by Bestvina \cite[Theorem A]{Bestvina_degenerations-hyperbolic-space_1988} and Paulin \cite{Paulin_1988}.
Under these additional assumptions, they deduce that the limiting action of $\Gamma$ on $T$ has small edge stabilizers. 
Our proof does not use their results and reconstructs the tree by different methods, but still relies on the idea to renormalize the metric.

The actions on $\Hyp^{\kappa_m}$ once embedded inside $\Hyp^\omega$ converge in a precise geometric sense, as explained by the Hausdorff convergence of the convex hull of finite portion of orbits.
The proof in \emph{Step 7} is to be compared with those in \cite[§1.3 and Proposition 1.9]{Courtois-Guilloux_Hausdorff-infinite-dim-hyp_2024} which deals with the case where the sequence of dimensions $\kappa_m$ is a constant integer $\kappa_m=n\in \N$.
\end{remark}

\begin{proof}
Assume that the finite generating set $S$ is symmetric and contains $1$, so that the ball of radius $k\in \N$ is $S^k \subset \Gamma$.

\emph{\underline{Step 1:} Position the convex cores $\operatorname{Core}(\rho_m)$ around a fixed point $o$.}

Since $\rho_m$ has no fixed points at infinity, we may find by \Cref{lem:min_displacement} a base point $o_m\in \Hyp^{\kappa_m}$ minimizing the $\rho_m(S)$-displacement function $o\in \Hyp^\omega \mapsto \max\{d(\rho_m(\gamma)(o),o)\colon \gamma\in S\}$. Up to embedding isometrically each $\Hyp^{\kappa_m}$ in $\Hyp^\omega$ extending trivially $\rho_m$ on the orthogonal of the span of the image of $\Hyp^{\kappa_m}$, we may and shall assume that $\kappa_m=\omega$ for all $m\in\N$.
After conjugating each $\rho_m$, we may assume that each $o_m$ coincide with a single base point $o\in\Hyp^\omega$.

Since $\rho_m$ has no fixed points at infinity there is a smallest non-empty closed convex invariant set $\operatorname{Core}(\rho_m)\subset \Hyp^{\omega}$.
Since projections on closed convex subsets contract distances, we must have $o\in\operatorname{Core}(\rho_m)$. 
Let $C_m^k\subset \Hyp^{\omega}$ be the closed convex hull of $\rho_m(S^k)(o)$. 
Observe that $\overline{\cup_{k\in\N}C_m^k}$ is a non-empty invariant closed convex subset of $\operatorname{Core}(\rho_m)$, hence these subsets coincide.

\emph{\underline{Step 2:} Rescale by exotic deformations $\exoH_{t_m}$ to focus on the convex cores.}

Let $d_m=\max\{\dH(\rho_m(\gamma)(o),o)\colon \gamma\in S\}$ and $t_m=d_m^{-1}$. 
Consider the equivariant exotic deformation $\exoH_{t_m} \colon \Hyp^{\omega}\to\Hyp^\omega$ from Theorem~\ref{thm:exotic-deformation-Hyp}, so that:
\begin{equation}
	\label{eq:equivariant}
	\forall x, y \in \Hyp^{\omega} \colon \quad
	\cosh \dH(\exoH_{t_m}(x),\exoH_{t_m}(y)) = \left(\cosh \dH(x,y) \right)^{t_m}.
\end{equation}
and assume (after composing $\exoH_{t_m}$ with an isometry) that $\exoH_{t_m}(o)=o$.
Let $\rho_m'=\chi_{t_m}\circ \rho_m \colon \Gamma \to \Isom(\Hyp^\omega)$.
Triangle inequalities show that every $\gamma\in S^k$ satisfies $\dH(\rho_m(\gamma)(o),o)\leq k d_m$, thus taking $\cosh^{-1}\log$ of \eqref{eq:equivariant} and using for $x\in \R_{\ge 0}$ the inequalities $\cosh^{-1}\exp(x)\le x+\log( 2)$ and $\log \cosh(x)\le x$ yields:
\begin{equation*}
	\dH(\rho_m'(\gamma)(o),o)\leq \cosh^{-1}\exp(t_m \log \cosh(kd_m)) \le (t_m \log \cosh(kd_m))+\log(2) \le k+\log(2)
\end{equation*}
hence $\exoH_{t_m}(C_m^k)\subset B(o, k+\log( 2))$.

\emph{\underline{Step 3:} Construct dyadic subdivisions $X_m^{k,i}$ of $C_m^k$ and
	reposition $Y_m^{k,i}=\exoH_{t_m}(X_m^{k,i})$ in $ \Hyp^\omega$.}

For $\gamma_0,\gamma_1\in \Gamma$ and $\delta\in [0,1]$, let $x_m(\gamma_0,\gamma_1;\delta)$ be the point on $[\rho_m(\gamma_0)(o),\rho_m(\gamma_1)(o)]$ at distance $\delta \times \dH(\rho_m(\gamma_0)(o),\rho_m(\gamma_0)(o))$ from $\rho_m(\gamma_0)(o)$, and let $y_m(\gamma_0,\gamma_1;\delta) = \exoH_{t_m}(x_m(\gamma_0,\gamma_1;\delta))$.

Denote, for $i\in \N$ the depth-$i$ dyadic cube $D_i=[0,1]\cap \tfrac{1}{2^i}\Z$, and $D_\infty = \bigcup_{i} D_i = [0,1]\cap \Z[1/2]$.
Define the depth-$i$ dyadic subdivision \(X_m^{k,i} = \{x_m(\gamma_0,\gamma_1;\delta) \colon \gamma_0,\gamma_1 \in S^k,\, \delta \in D_i\}\) of the $1$-skeleton of $C_m^k$, and its exotic rescale $Y_m^{k,i} = \exoH_{t_m}(X_m^{k,i}) \subset \exoH_{t_m}(C_m^{k})$.
Note that if $k\le k'$ and $i\le i'$ then $Y_m^{k,i}\subset Y_m^{k',i'}$.

The subset $Y_m^{k,i}\subset \Hyp^\omega$ has cardinal at most $1+\kappa_{k,i} = (1+2^i)\times \Card(S^k)^2$ so it lies in a totally geodesic subspace of dimension $\kappa_{k,i}$ containing $o$.
Since the stabilizer of $o$ in $\Isom(\Hyp^\omega)$ acts transitively on flags of totally geodesic subspaces containing $o$ of dimension sequence $(\kappa_{n,n})$, we may assume (up to conjugating $\rho'_m$) that there exists such a flag of subspaces $(\Hyp^{\kappa_{n}})$ of dimensions $\kappa_{n,n}$ in $\Hyp^\omega$ such that for all $m,k,i$ we have $Y_m^{k,i} \subset \Hyp^{\kappa_{n,n}}$ where $n=\max\{k,i\}$.

\emph{\underline{Step 4:} Extracting a convergent subsequence of $Y_m^{k,i}$ as $m,k,i \to \infty$ to find $Y_\infty^{\infty,\infty} \subset \Hyp^\omega$.}

For all $m$ we have $Y_m^{k,i}\subset B(o,k+\log(2))\cap \Hyp^{\kappa_{n,n}}$ where $n=\max\{k,i\}$. So by compactness of closed balls of finite dimensional hyperbolic spaces, we may assume (up to extracting a subsequence) that each sequence $y_m(\gamma_0,\gamma_1;\delta)\in Y_m^{k,i}$ converges to some point in $\Hyp^\omega$. 
Taking a diagonal extraction over $m$, we may assume that for all $\gamma_0,\gamma_1 \in \Gamma$ and $\delta \in D_\infty$ the sequence $y_m(\gamma_0,\gamma_1;\delta)$ converges to some $y(\gamma_0,\gamma_1;\delta)\in \Hyp^\omega$: let $Y_\infty^{\infty,\infty}=\{y(\gamma_0,\gamma_1;\delta) \colon \gamma_0, \gamma_1\in \Gamma,\, \delta \in D_\infty\}$. 
Denote $T\subset \Hyp^\omega$ the topological closure of $Y_\infty^{\infty,\infty}$ in $\Hyp^\omega$.

\emph{\underline{Step 5}: Defining a distance $d$ on $T$ which is $0$-hyperbolic.}

For $y,y'\in T$, define $d(y,y')=\log\cosh \dH(y,y')$. 
We compute for $y=\lim y_m(\gamma_0,\gamma_1;\delta)$ and $y'=\lim y_m(\gamma_0',\gamma_1';\delta')$, using \eqref{eq:equivariant} and the asymptotic $(\log \cosh x)/x \to 1$ as $x\to \infty$, that
\begin{align*}
	d(y,y')
	&=\lim_{m\to\infty}\log\left(\cosh\left(\dH(y_m(\gamma_0,\gamma_1;\delta),y_m(\gamma_0',\gamma_1';\delta'))\right)\right)\\
	&=\lim_{m\to\infty}t_m\log\left(\cosh\left(\dH(x_m(\gamma_0,\gamma_1;\delta),x_m(\gamma_0',\gamma_1';\delta'))\right)\right)\\
	&=\lim_{m\to\infty}t_m\dH(x_m(\gamma_0,\gamma_1;\delta),x_m(\gamma_0',\gamma_1';\delta')).     
\end{align*}
Since the hyperbolic distance on $\Hyp^{\omega}$ is Gromov-hyperbolic and $t_m\to0$, we have that $(Y_\infty^{\infty,\infty},d)$ is a $0$-hyperbolic metric space. Moreover a Cauchy sequence in $(Y_\infty^{\infty,\infty},d)$ must be a Cauchy sequence in $(\Hyp^\omega, \dH)$ which is complete, thus converges to a point in its closure $T$. 
Hence $(T,d)$ coincides with the completion of $(Y_\infty^{\infty,\infty},d)$. 
Moreover, since $Y^{\infty,\infty}_\infty$ admits approximate midpoints (for any $\varepsilon>0$ and $x,y\in Y^{\infty,\infty}_\infty$, there is $m\in Y^{\infty,\infty}_\infty$ such that $d(x,m),d(m,y)\leq d(x,y)/2+\varepsilon$), its completion $(T,d)$ is a real tree by \Cref{def:real-tree}.
By construction:
\begin{equation*}
	\forall y,y'\in T \colon \quad \cosh(\dH(y,y'))=e^{d(y,y')}
\end{equation*}
hence the inclusion $T\hookrightarrow\Hyp^\omega$ corresponds to the embedding $\exoT_\lambda$ of Theorem \ref{thm:tree-embedding-Hyp-infty} for $\lambda=e$.

\emph{\underline{Step 6}: The non-isometric embedding $T\hookrightarrow \Hyp^\omega$ is $\Gamma$-equivariant and $\Gamma$ acts isometrically.}

For $\gamma_0,\gamma_1\in \Gamma$ and $\delta\in D_\infty$ and $\alpha\in\Gamma$, we have $y_m(\alpha\gamma_0,\alpha\gamma_1;\delta)=\rho_m(\alpha)y_m(\gamma_0,\gamma_1;\delta)$.
Thus we obtain an action $\rho_\omega$ of $\Gamma$ on $Y_\infty^{\infty,\infty}$ defined by $\rho(\alpha)(y(\gamma_0,\gamma_1;\delta))=y(\alpha\gamma_0,\alpha\gamma_1;\delta)$, which is isometric since for all $\gamma_0,\gamma_1, \gamma_0',\gamma_1'\in\Gamma$ and $\delta,\delta'\in D_\infty$ and $\alpha\in \Gamma$, the limit as $m\to \infty$ of
\(\dH(\rho_m(\alpha)y_m(\gamma_0,\gamma_1;\delta),\rho_m(\alpha)y_m(\gamma_0',\gamma_1';\delta')) = 
\dH(y_m(\alpha\gamma_0,\alpha\gamma_1;\delta),y_m(\alpha\gamma_0',\alpha\gamma_1';\delta')) 
\)
yields 
\[\dH\rho(\alpha)y(\gamma_0,\gamma_1;\delta),\rho(\alpha)y(\gamma_0',\gamma_1';\delta'))= \dH(y(\gamma_0,\gamma_1;\delta),y(\gamma_0',\gamma_1';\delta')).\]

The isometric action of $\Gamma$ on $Y_\infty^{\infty,\infty}$ extends uniquely to an isometric action of $\Gamma$ on its topological closure $(T,d)$ and to the smallest closed totally geodesic subspace containing them which we may still denote $(\Hyp^\omega, \dH)$.
The embedding $\exoT_e \colon (T,d) \hookrightarrow \Hyp^\omega$ is $\Gamma$-equivariant.

\emph{\underline{Step 7:} Hausdorff convergence $\exoH_{t_m}(\Conv(\rho_m(S^k) \cdot o_m)) \xrightarrow[]{m\to \infty} \exoT_\lambda(\Conv(\pi_\lambda(S^k) \cdot o))$ in $\Hyp^\omega$.}

Let us begin by clarifying the objective and notations.
Fix $k\in \N$.
Recall that the convex hull $C_m^k$ of $X_m^{k,0}=\rho_m(S^k) \cdot o_m$ is also the convex hull of its $1$-skeleton, or of its depth-$i$ dyadic subdivision $X_m^{k,i}$.
Note that for finite $i$, the set $\exoH_{t_m}(X_m^{k,i})=Y_m^{k,i}$ pointwise converges as $m\to\infty$ to $Y_\infty^{k,i}$.
Moreover, in the real tree $T$ (the closure of $Y_\infty^{\infty,\infty}$), the convex hull of the finitely many points $\pi_\lambda(S^k) \cdot o =Y_\infty^{k,0}$ is the union of the geodesic segments between them, namely the closure of $Y_\infty^{k,\infty}=\{y(\gamma_0,\gamma_1;\delta),\ \gamma_0,\gamma_1\in S^k,\delta\in D_\infty\}$.
That being said, we aim to prove that $\exoH_{t_m}(C_m^k)\subset \Hyp^\omega$ Hausdorff converges to the closure of $Y_\infty^{k,\infty}$. 

Let us now outline the steps of the strategy.
By triangle inequalities, for all $m,i$ we have:
\begin{equation*}
	\dH_H(\overline{Y_\infty^{k,\infty}},\exoH_{t_m}(C^k_m))
	\leq\dH_H(\overline{Y_\infty^{k,\infty}},Y_\infty^{k,\infty})
	+\dH_H(Y_\infty^{k,\infty},Y_\infty^{k,i})
	+\dH_H(Y_\infty^{k,i},Y_m^{k,i})
	+\dH_H(Y_m^{k,i},\exoH_{t_m}(C^k_m))
\end{equation*}
The first term vanishes.
For all finite $i\in \N$, the finite set $Y_m^{k,i}$ converges (pointwise hence Hausdorff) to $Y_\infty^{k,i}$ as $m\to \infty$ so this deals with the third term.
We must therefore bound the second and fourth terms.

Fix $\varepsilon>0$.
Let $M^k=\sup \{t_m\operatorname{diam}(C_m^k) \colon m\in\N\}\leq t_m 2 k d_m=2k$.
Suppose that $i\in \N$ is large enough so that $\exp(M^k/2^{i}) < \cosh(\varepsilon)$.
For any $\gamma_0,\gamma_1,\gamma_0',\gamma'_1\in\Gamma$ and $\delta,\delta'\in[0,1]$ and any $m\in\N$, applying equation \eqref{eq:equivariant} and for $t\in \R_{\ge 1}$ the inequality $(\cosh d)^t\le \exp(td)$, we have:
\begin{equation*}
	\cosh \dH(y_m(\gamma_0,\gamma_1;\delta),y_m(\gamma_0',\gamma_1';\delta')) 
	\le \exp \left( t_m \dH(x_m(\gamma_0,\gamma_1;\delta),x_m(\gamma_0',\gamma_1';\delta'))\right).
\end{equation*}
For $\delta\in [0,1]$ we may chose $\delta'\in D_i$ such that $\lvert \delta-\delta'\rvert<1/2^i$, thus for all $\gamma_0,\gamma_1\in S^k$ we have: 
\begin{equation*}
	\exp(t_m\dH(x_m(\gamma_0,\gamma_1;\delta),x_m(\gamma_0,\gamma_1;\delta')))
	\leq \exp(M^k/2^i)
\end{equation*}
hence $\dH(y_m(\gamma_0,\gamma_1;\delta),y_m(\gamma_0,\gamma_1;\delta')) \le \cosh^{-1} \exp(M^k/2^i) < \varepsilon$.
Thus $\dH_H(Y_m^{k,\infty}, Y_m^{k,i})\le \varepsilon$.

Next, assume that $i$ is at least as large as above.
By Corollary \ref{prop:Hausdist-polytope-1-skeleton}, for any $m$ the Hausdorff distance between the polytope $C_m^{k} \subset \Hyp^{\omega}$ and its $1$-skeleton is at most $\sinh^{-1}(1)$. 
Thus for $i$ large enough we have $\dH_H(C_m^k,X_m^{k,i})<2\sinh^{-1}(1)$.
Now if $m$ is large enough so that $(\cosh2\sinh^{-1}(1))^{t_m}\leq\exp(M^k/2^i)$, their images by $\exoH_{t_m}$ satisfy $\dH_H(\exoH_{t_m}(C^k_m), Y_{m}^{k,i})<\varepsilon$.

Finally, observe that one may rewrite Step 7 replacing $S^k$ by any finite subset $K\subset \Gamma$ to prove the claim as anounced in the Theorem.
\end{proof}

\subsection{Ultraproducts of hyperbolic spaces and actions on real trees}\label{sec:ultraproducts}

The goal of this subsection is to give a slightly different point of view on the construction of the real $T$ inside $\Hyp^\omega$ obtained in \Cref{thm:geometric-convergence-tree} using ultraproducts.
For background on ultrafilters and ultraproducts (or ultralimits), we refer to \cite[Chapter I.5]{Bridson-Haefliger_metric-non-positive-curvature_1999}.
One key point is that an ultraproduct of algebraic hyperbolic spaces is still a algebraic hyperbolic space.
Let $\beta\N$ denote the Stone--\v{C}ech compactification of $\N$, that is the space of ultrafilters on $\N$.

\begin{lemma}[ultraproduct of hyperbolic spaces]
\label{lem:ultraproduct-Hyp}
For a sequence of cardinals $(\kappa_m)$ indexed by $\N$ and an ultrafilter $\ufi\in\beta\N$, there exists a cardinal $\kappa$ such that the ultraproduct $X=\prod_\ufi\Hyp^{\kappa_m}$ with distance $d_\ufi(x,y)=\lim_{\ufi}d(x_,y_m)$ is isometric to $\Hyp^\mathbf{\kappa}$.

When each $\kappa_m$ is countable and the sequence is not bounded by a finite cardinal, the cardinal $\kappa$ is equal to that $\mathfrak{c}=\Card(\R)$ of the continuum.
\end{lemma}

\begin{proof}
Let $\dH_m$ be the distance function on $\Hyp^{\kappa_m}$. The function $\cosh\circ \dH_m$ is a kernel of hyperbolic type. 
Hence the distance $d_\ufi$ on $X$ yields a kernel of hyperbolic type by $\cosh\circ d_\ufi = \lim \cosh \circ \dH_m$, so there exists a minimal cardinal $\kappa$ such that $X$ is isometric to a subset of $\Hyp^\kappa$.
Since each $\Hyp^{\kappa_m}$ is geodesic and each geodesic segment lies in a bi-infinite geodesic, $X$ enjoys the same property. 
Moreover $X$ is complete (this holds for all ultraproducts over non-principal filters $\ufi$).
Consequently, $X$ is complete and totally geodesic, so it is isometric to $\Hyp^\kappa$.
\end{proof}

Let us now recover \Cref{thm:geometric-convergence-tree} using ultralimits.

Consider a group $\Gamma$ generated by a finite $S$, and a sequence of cardinals $\kappa_m$ and of representations $\rho \colon \Gamma \to \Isom(\Hyp^{\kappa_m})$ with no fixed points at infinity.

By \Cref{lem:min_displacement} we may find, a base point $o_m\in \Hyp^{\kappa_m}$ realising the minimum of the $\rho_m(S)$-displacement function $o\in \Hyp^\omega \mapsto \max\{\dH(\rho_m(\gamma)(o),o)\colon \gamma\in S\}$, which we abbreviate as $\Delta_m = \Delta(\rho_m(S))$.

Assume that $\Delta_m$ is unbounded, and consider an ultrafilter $\ufi$ on $\N$ such that $\lim_\ufi \Delta_m =+\infty$.

The ultraproduct $\prod_\ufi\left(\Hyp^{\kappa_m}, \dH_m/\Delta_m,o_m\right)$ is a pointed real tree $(T_\ufi,d, o_\ufi)$ because it is a complete metric space which is path connected and $0$-hyperbolic (since $(\Hyp^{\kappa_m},d_m/\Delta_m)$ is Gromov $\log(2)/\Delta_m$-hyperbolic and $\lim_\ufi \Delta_m = +\infty$).
It is endowed with the isometric $\Gamma$-action given by $\rho_T=\prod_\ufi \rho_m$.

Moreover, applying \Cref{thm:exotic-deformation-Hyp} with $t_m=1/\Delta_m$ we have exotic embeddings $\exoH_{t_m} \colon \Hyp^{\kappa_m} \to \Hyp^{\kappa_m'}$ yielding exotically-rescaled representations $\rho_m'=\chi_{t_m} \circ \rho_m\colon\Gamma\to\Isom(\Hyp^{\kappa'_m})$.

The ultraproduct $\prod_\ufi\left(\Hyp^{\kappa_m'}, \dH_m', \exoH_{t_m}(o_m)\right)$ is a pointed algebraic hyperbolic space $(\Hyp^{\kappa_\ufi}, \dH_\ufi, \exoH_\ufi(o_\ufi))$ By \Cref{lem:ultraproduct-Hyp}.
It is endowed with the isometric diagonal $\Gamma$-action $\rho'_\ufi:=\prod_\ufi \rho_m'$.

The ultraproduct map $\exoH_\ufi$ induced by the product map $\prod_\ufi \exoH_{t_m}$ is an embedding of the pointed real tree $(T_\ufi, d, o_\ufi)$ into the pointed algebraic hyperbolic space $(\Hyp^{\kappa_\ufi}, \dH_\ufi, \exoH_\ufi(o_T))$ which is equivariant with respect to the isometric $\Gamma$-actions $\prod_\ufi \rho_m$ and $\prod_\ufi \rho_m'$. 

Moreover distances between points in $T_\ufi$ and their images in $\Hyp^{\kappa_\ufi}$ satisfy the following relation: If $x,y\in T_\ufi$ are given by $x=\lim_\ufi x_m$ and $y=\lim_\ufi y_m$ then $d(x,y)=\lim_\ufi t_m\dH_m(x_m,y_m)$ This implies that $d(x_m,y_m)\to\infty$. 
Since, for each $m\in\N$, we have the equality $\cosh(\dH(\exoH_{t_m}(x_m),\exoH_{t_m}(y_m))=\cosh(\dH_m(x_m,y_m))^{t_m}$ and because of $\lim_\ufi t_m = 0$ we have 
\begin{equation*}
\cosh\left(\dH_\ufi(\exoH_\ufi(x),\exoH_\ufi(y))\right)=e^{d(x,y)}.
\end{equation*}

\begin{theorem}[geometric and ultrafliter limits coincide] 
Consider a group $\Gamma$ generated by a finite $S$, and a sequence of cardinals $\kappa_m$ and of representations $\rho \colon \Gamma \to \Isom(\Hyp^{\kappa_i})$ with no fixed points at infinity such that $\Delta(\rho_m(S))\to +\infty$.
By \Cref{thm:geometric-convergence-tree}, one can extract from $(\rho_m)$ a limiting action of $\Gamma$ by isometries on a real tree $(T,d)$ with an equivariant embedding in $\Hyp^\omega$ which is $\Gamma$-equivariant for the isometric actions $\rho_T, \rho_\omega$.

There is an ultrafilter $\ufi\in\beta\N$ such that, up to passing to the unique minimal invariant real subtree, the real trees $T_\ufi$ and $T$ are isometric in a $\Gamma$-equivariant way and up to passing to the unique minimal totally geodesic invariant subspace that contain the embeddings of $T$ and $T_\ufi$, the actions $\rho'_\ufi$ and $\rho_\omega$ of $\Gamma$ on $\Hyp^\omega$ are conjugated. 
\end{theorem}

\begin{proof}We explain the isometry between $T_\ufi$ and $T$.
For $\gamma_0, \gamma_1\in\Gamma$ and $\delta\in[0,1]$, we denote by $z_\ufi(\gamma_0,\gamma_1,\delta)\in T_\ufi$ the unique point on the geodesic segment $[\gamma_0o_\ufi,\gamma_1o_\ufi]$ at distance $\delta d_\ufi(\gamma_0o_\ufi,\gamma_1o_\ufi)$ from $\gamma_0o_\ufi$.

We first perform all the successive extractions done in Step 1 of the proof of \Cref{thm:geometric-convergence-tree} and choose an ultrafilter $\ufi$ on $\N$ that contains all the images of these extractions. 
Such $\ufi$ exists by compactness of $\beta\N$. In particular, with the notations of \Cref{thm:geometric-convergence-tree} and $\gamma_0,\gamma_1,\gamma_0',\gamma_1'\in\Gamma$ and $\delta,\delta'$ dyadic numbers in $[0,1]$, for $y=\lim y_m(\gamma_0,\gamma_1,\delta), y'=\lim y_m(\gamma_0',\gamma_1',\delta')$ and $z=z_\ufi(\gamma_0,\gamma_1,\delta), z'=z_\ufi(\gamma_0',\gamma_1',\delta')$ we have

\begin{align*}
	d_T(y,y')&
	=\lim_{m\to\infty}\log\cosh \dH(y_m(\gamma_0,\gamma_1,\delta),y_m(\gamma_0',\gamma_1',\delta'))\\
	&=\lim_{m\to\infty}t_m\dH(x_m(\gamma_0,\gamma_1,\delta),x_m(\gamma_0',\gamma_1',\delta'))\\
	&
	=\lim_{\ufi}t_m\dH(x_m(\gamma_0,\gamma_1,\delta),x_m(\gamma_0',\gamma_1',\delta'))\\
	&=d_{T_\ufi}(z,z').
\end{align*}

Thus the map $f_\ufi$ from $T$ to the closed convex hull of the orbit of $o_\ufi$ in $T_\ufi$ which sends $y$ to $z$ is an equivariant isometry. 
The real tree $T$ is minimal, since it is equal to the closure of the union of the axes of elements in $\Gamma$. Up to replacing $T_\ufi$ by the image of $f_\ufi$, $T$ and $T_\ufi$ are isometric in a $\Gamma$-equivariant way.

Moreover, both trees $T$ and $T_\ufi$ are embedded  via maps $\exoH\colon T\to\Hyp^\omega$ and $\exoH_\ufi\colon T_\ufi\to\Hyp^\omega$ in $\Hyp^\omega$ according to \Cref{thm:tree-embedding-Hyp-infty} for the same parameter $\lambda=e$.

Now, \Cref{rmk:unicity_pt} shows that the two representations  $\rho_\omega$ and $\rho'_\ufi$ are conjugated.
\end{proof}


\section{Equivalence classes of functions of hyperbolic type}

Given a topological group $\Gamma$ and a cardinal $\kappa\le \Card(\Gamma)$, we wish to consider the space of continuous representation $\Gamma \to \Isom(\Hyp^\kappa)$ up to the action of $\Isom(\Hyp^\kappa)$ at the target, namely inner automorphisms and exotic representations $\rho_t\colon \Isom(\Hyp^\kappa)\to \Isom(\Hyp^\kappa)$.

For this purpose, we will consider functions of hyperbolic type up to bounded pertubations (which arise by conjugacy) and powers (which arise from the exotic deformation).

Recall that on a (topological) group $\Gamma$, the space of functions of hyperbolic type $\mathcal{F}(\Gamma)$ is endowed with the pointwise convergence topology.

\begin{definition}[comparability and homothety]
On a group $\Gamma$, we say that $F_1,F_2\in \mathcal{F}(\Gamma)$ are 
\begin{itemize}[noitemsep, align=left]
	\item[\emph{comparable}] when $\log(F_1/F_2)$ or equivalently $\cosh^{-1}F_1-\cosh^{-1}F_2$ is bounded.
	\item[\emph{homothetic}] when there exists $t_1,t_2\in (0,1]$ such that $F_1^{t_1}$ and $F_2^{t_2}$ are comparable.
\end{itemize}
These define equivalence relations on $\mathcal{F}(\Gamma)$, the latter coarser than the former, and endow the quotient spaces with the quotient topologies so as to obtain successive continuous quotient maps denoted $\mathcal{F}(\Gamma) \to \mathcal{C}(\Gamma) \to \Proj\mathcal{C}(\Gamma)$ sending $F \mapsto \Bar{F} \mapsto \class{F}$.
\end{definition}

\begin{proof}
The comparability relation is of course reflexive and symmetric, and its transitivity follows from the triangular inequality on sup-norms.
The homothety relation coarsens comparability; its reflexivity follows by setting $t_1=1=t_2$, symmetry follows from that for the comparability relation by inversion $t\mapsto 1/t$, and transitivity follows from that for comparability and the existence quantifier on $t$.
\end{proof}

In $\mathcal{F}(\Gamma)$, the comparability class of the constant function equal to $1$ is the set of all bounded functions. 
The set of unbounded comparability classes is denoted $\mathcal{C}^*(\Gamma)=\mathcal{C}(\Gamma)\setminus \bar{\mathbf{1}}$ and the corresponding homothety classes is denoted $\Proj \mathcal{C}^*(\Gamma) = \Proj \mathcal{C}(\Gamma)\setminus \class{\mathbf{1}}$. 

The space $\Proj \mathcal{C}(\Gamma)$ is not Hausdorff because the class $\class{\mathbf{1}}$ is in the closure of any other class: if $F$ is a function of hyperbolic type then $F^t\to \mathbf{1}$ as $t\to0$.

\subsection{Topology of \texorpdfstring{$\Proj\mathcal{C}(\Gamma) \supset \Proj\mathcal{C}_{nn}(\Gamma) \supset \Proj\mathcal{C}_{ne}(\Gamma)$}{PC-PCnn-PCne} and length functions}\label{subsec:topology_PC}

The notions of elementarity and neutrality are invariant by comparability (\Cref{cor:neutral-elementary}) and homothety (\Cref{cor:homothetic_length_spectrum}). 
We denote the corresponding spaces of comparable classes by $\mathcal{C}_{ne}(\Gamma)\subset \mathcal{C}_{nn}(\Gamma) \subset \mathcal{C}^*(\Gamma)$ and of homothety classes by $\Proj\mathcal{C}_{ne}(\Gamma)\subset\Proj\mathcal{C}_{nn}(\Gamma)$ of $\Proj\mathcal{C}^*(\Gamma)$.
The aim of this subsection is to show that the topology on these subspaces is quite well behaved.

Recall that in restriction to the subset of non-neutral functions $\mathcal{F}_{nn}(\Gamma)$, the comparability and homothety relations are characterized by \Cref{cor:equal_length_spectrum} and \Cref{cor:homothetic_length_spectrum} in terms of the length spectrum and the associated minimal representations.

Let $\R^\Gamma$ be the product vector space with the product topology, and $\Proj\R^\Gamma$ be the associated projective space with the quotient topology.
After removing the trivial element $\{0\}$ which is its own class, we obtain a space $\Proj\R^\Gamma\setminus\{0\}$ that is metrizable when $\Gamma$ is countable.

\begin{theorem}[length function on $\Proj\mathcal{C}_{nn}(\Gamma)\subset \Proj\mathcal{C}_{ne}(\Gamma)$]
\label{thm:length-non-neutral-classes}
Let $\Gamma$ be a group. 

The length function $\ell \colon \mathcal{F}(\Gamma) \to \R^\Gamma$ defined by $F \mapsto \left(\ell_F(\gamma)\right)_{\gamma\in\Gamma}$ quotients to a function $\ell \colon \Proj\mathcal{C}(\Gamma)\to\Proj\R^\Gamma$ which is continuous and injective in restriction to $\Proj\mathcal{C}_{nn}(\Gamma)$; in particular the sets $\Proj\mathcal{C}_{nn}(\Gamma)$ and $\Proj\mathcal{C}_{ne}(\Gamma)$ are Hausdorff and open in $\Proj\mathcal{C}(\Gamma)$.
\end{theorem}

\begin{remark}[]
This can be compared with representations of discrete subgroups in Lie groups of rank $1$: the conjugacy classes of irreducible discrete faithful representations some  countable groups $\Gamma$ form a nice subspace of in their variety of representations into $\Isom(\Hyp^n)$, parametrized by length functions. See for example \cite{Otal_spectre-marque_1990} for an original result and  \cite{Kim_marked-length-rigidity-symmetric-space_2001} for a more comprehensive one.
\end{remark}

\begin{proof}
The function $\ell \colon \mathcal{F}(\Gamma) \to \R^\Gamma$ defined by $F \mapsto \left(\ell_F(\gamma)\right)_{\gamma\in\Gamma}$ is continuous by Lemma \ref{lem:F-to-ellF-continuity}.
By Corollary~\ref{cor:homothetic_length_spectrum}, this descends to a function $\Proj\ell \colon \Proj\mathcal{C}(\Gamma)\to \Proj\R^\Gamma$ defined by $\class{F}\mapsto [(\ell_F(\gamma))_\gamma]$, which is continuous and injective.
Since $\Proj\R^\Gamma\setminus\{0\}$ is open and Hausdorff, so is its preimage $\Proj\mathcal{C}_{nn}(\Gamma)$.
Let us show that $\Proj\mathcal{C}_{ne}(\Gamma)$ is open in $\Proj\mathcal{C}(\Gamma)$ by providing an open cover. 
For $\alpha,\beta\in \Gamma$, let $\mathcal{U}(\alpha,\beta)$ be the subset of $\class{F}\in \Proj\mathcal{C}(\Gamma)$ such that 
$\ell_{F}(\alpha),\ell_{F}(\beta),\ell_{F}(\alpha\beta)>0$ and $\ell_F(\alpha\beta)\ne \lvert \ell_F(\alpha) \pm \ell_F(\beta)\rvert$.
This well-defined subset of $\Proj\mathcal{C}(\Gamma)$ is open.
If $\class{F}\in \mathcal{U}(\alpha,\beta)$ then $\ell_F$ is not the absolute value of a homomorphism on the subgroup generated by $\alpha$ and $\beta$, so by \Cref{cor:elem-Hyp_length-homomorphism}, $\class{F}\in \Proj\mathcal{C}_{ne}(\Gamma)$.
Conversely if $\class{F}\in \Proj\mathcal{C}_{ne}(\Gamma)$ then $F$ it is not the absolute value of a homomorphism there exist $\alpha,\beta\in \Gamma$ with $\ell_F(\alpha),\ell_F(\beta), \ell_F(\alpha\beta)>0$ and $\ell_F(\alpha\beta)\ne \lvert \ell_F(\alpha)+\ell_F(\beta)\rvert$ hence $\class{F}\in \mathcal{U}(\alpha,\beta)$.
\end{proof}

For the rest of the section we assume that $\Gamma$ is finitely generated: the aim is to prove the next \Cref{thm:topo-PC-PCnn-PCne} giving more information about the topology of $\Proj\mathcal{C}(\Gamma) \supset \Proj\mathcal{C}_{nn}(\Gamma) \supset \Proj\mathcal{C}_{ne}(\Gamma)$.

\begin{theorem}[topology of $\Proj\mathcal{C}(\Gamma)$ for finitely generated $\Gamma$]
\label{thm:topo-PC-PCnn-PCne}
Let $\Gamma$ be a finitely generated group. 
\begin{enumerate}[noitemsep]
	\item \label{thm:compact-PC(Gamma)} 
	The space $\Proj\mathcal{C}(\Gamma)$ is compact.
	
	\item \label{thm:topo-PCnn(Gamma)} 
	The subspace $\Proj\mathcal{C}_{nn}(\Gamma)$ is a filtering union of compact metrizable subspaces $\mathcal{A}_{nn}(S;2,\varepsilon)$ so the restriction $\Proj\ell \colon \mathcal{A}_{nn}(S; 2,\varepsilon)\to\Proj\R^\Gamma\setminus\{0\}$ is a homeomorphism onto its image.
	
	\item \label{thm:topo-PCne(Gamma)}
	The subspace $\Proj\mathcal{C}_{ne}(\Gamma)$ is a filtering union of compact metrizable subspaces $\mathcal{A}_{ne}(S; 2,\varepsilon)$ so the restriction $\Proj\ell \colon \mathcal{A}_{ne}(S; 2,\varepsilon)\to\Proj\R^\Gamma\setminus\{0\}$ is a homeomorphism onto its image.
\end{enumerate}
\end{theorem}

\begin{proof}
\Cref{lem:compact-Bmu(Gamma)} proves \Cref{thm:compact-PC(Gamma)}.
\Cref{lem:PCnn_union-of-compact} and \Cref{lem:PCne_union-of-compact} prove that $\Proj\mathcal{C}_{nn}$ and $\Proj\mathcal{C}_{ne}$ are filtering unions of compact subsets.
Let us explain how to deduce the metrizability and the homeomorphism properties. 
The restriction of $\Proj\ell\colon \Proj\mathcal{C}_{nn}(\Gamma) \to \Proj\R^\Gamma\setminus \{0\}$ is continuous injective by \cref{thm:length-non-neutral-classes}, hence in restriction to any compact subset, it is a homeomorphism onto its image.
The space $\Proj\R^\Gamma\setminus\{0\}$ has a countable basis for its topology (since $\Gamma$ is countable), and it is Hausdorff separated, so its compact subsets are metrizable.
\end{proof}


\begin{lemma}[compactness of $\Proj\mathcal{C}(\Gamma)$]
\label{lem:compact-Bmu(Gamma)}

For a finite generating set $S$ of $\Gamma$ and $\mu \in \R_{\ge 0}$, we define the subset $\mathcal{B}(S;\mu)=\{F\in \mathcal{F}(\Gamma) \colon \max F(S) \le \cosh(\mu) \}$ of $\mathcal{F}(\Gamma)$.

For generating $S\subset \Gamma$ and $\mu>0$, the subset $\mathcal{B}(S;\mu)$ is compact, and the continuous quotient map $\mathcal{F}(\Gamma)\to \Proj\mathcal{C}(\Gamma)$ is surjective in restriction to $\mathcal{B}(S;\mu)$.
In particular $\Proj\mathcal{C}(\Gamma)$ is compact.
\end{lemma}

\begin{proof}
Assume for notational simplicity that the generating set $S$ is symmetric $S=S^{\pm1}$ and contains the identity element $1$, so that if $\lvert \cdot \rvert_S \colon  \Gamma \to \N$ denotes its associated word-length function, for $n\in \N$ we have $S^n=\{\gamma \in \Gamma \colon \lvert \gamma\rvert_S \le n \}$.

We will show that the growth of a function of hyperbolic type is bounded above by some absolute function depending only on its values on a generating set.
Fix $F\in \mathcal{F}(\Gamma)$ and for $n\in \N$, let $\cosh \mu_n = \max F(S^n)$.
For $\gamma\in S^{n+1}$ written as $\gamma = \alpha\beta$ with $\alpha\in S^1$ and $\beta\in S^n$:
\begin{align*}
	\cosh^{-1}(F(\gamma))
	&\leq \dH(\alpha\beta o,\alpha o)+\dH(\alpha o,o)+\dH(o,\beta o)+\dH(\beta o,o)\\
	&\leq \dH(\beta o, o)+\dH(\alpha o,o)+\dH(o,\beta o)+\dH(\beta o,o)\\
	&\leq \mu_1+3 \mu_n.
\end{align*}
Hence $\mu_{n+1}\le \mu_1+3 \mu_n$, so $\mu_n\le 4^n\mu_1$ so every $\gamma\in\Gamma$ satisfies $1\le F(\gamma)\leq \cosh(4^{\lvert \gamma\rvert_S}\mu_1)$.

Thus for $\mu>0$, the subset of $F\in \mathcal{F}$ with $\max F(S)\le \cosh \mu$ is closed and contained in the compact subset $\prod_{\gamma\in \Gamma} [1, \cosh(4^{\lvert \gamma\rvert_S}\mu)]$, so it is a compact subset of $\mathcal{F}(\Gamma)$.

Let $\mu>0$.
For $\class{F}\in \Proj\mathcal{C}(\Gamma)$ one may choose any representative $F$ and find $t\in (0,1)$ such that $(\max F(S))^t\le \cosh(\mu)$, thus $F^t\in \mathcal{B}(S;\mu)$.
This proves surjectivity of $\mathcal{B}(S;\mu) \to \Proj \mathcal{C}(\Gamma)$.
\end{proof}

\begin{lemma}[filtering union  of compacta to $\Proj\mathcal{C}_{nn}$]
\label{lem:PCnn_union-of-compact}
For a finite generating $S\subset \Gamma$ and $\mu \in \R_{\ge 0}$ and $\varepsilon \in \R_{>0}$, we define the subset $\mathcal{A}_{nn}(S;\mu, \varepsilon)$ of $F\in \mathcal{B}(S; \mu)$ such that $\max \ell_F(S)\ge \varepsilon$.

It is compact and contained in $\mathcal{F}_{nn}(\Gamma)$, so it quotients to a compact $\Proj\Bar{\mathcal{A}}_{nn}(S;\mu,\varepsilon) \subset \Proj \mathcal{C}_{nn}(\Gamma)$.
Moreover for all $\mu>0$, the space $\Proj \mathcal{C}_{nn}(\Gamma)$ is the filtering union of the $\Proj\Bar{\mathcal{A}}_{nn}(S;\mu,\varepsilon)$ over increasing finite generating sets $S\subset \Gamma$ and decreasing $\varepsilon>0$.
\end{lemma}

\begin{proof}
The subset $\mathcal{A}_{nn}(S;\mu, \varepsilon)$ of $\mathcal{B}(S;\mu)$ is closed hence compact by \Cref{lem:compact-Bmu(Gamma)}.
Moreover by definition $\mathcal{A}_{nn}(S;\mu, \varepsilon)$ is contained in $\mathcal{F}_{nn}(\Gamma)$, hence it quotients by continuity to a compact $\Proj\Bar{\mathcal{A}}_{nn}(S;\mu,\varepsilon) \subset \Proj \mathcal{C}_{nn}(\Gamma)$.

Let $\mu\in \R_{\ge0}$.
Let $\class{F} \in \Proj\mathcal{C}_{nn}(\Gamma)$, and choose any representative $F\in \mathcal{F}_{nn}(\Gamma)$.
Since $F$ is non-neutral we may find $\alpha\in \Gamma$ such that $\ell_F(\alpha)>0$.
Choose any finite generating set $S$ containing $\alpha$ (and one may enlarge to so that it is symmetric and contains the identity). 
Let $t\in (0,1]$ such that $ \max F^t(S) \le \cosh(\mu)$ so that $F^t\in \mathcal{B}(S;\mu)$.
Choose $\varepsilon\in \R_{>0}$ such that $\varepsilon \le \max \ell_{F^t}(S)$.
We found $S,\varepsilon$ as desired such that $\class{F}$ admits a representative $F^t$ in $ \mathcal{A}_{nn}(S;\mu, \varepsilon)$.
\end{proof}

\begin{lemma}[filtering union of compacta to $\Proj\mathcal{C}_{ne}$]
\label{lem:PCne_union-of-compact}
For a finite generating $S\subset \Gamma$ and $\mu \in \R_{\ge 0}$ and $\varepsilon \in \R_{>0}$, we define the subset $\mathcal{A}_{ne}(S; \mu, \varepsilon)$ of $F\in \mathcal{B}(S; \mu)$ for which there exists $\alpha,\beta \in S$ satisfying $\ell_F(\alpha),\ell_F(\beta) \ge \varepsilon$ and
\(\ell_F(\alpha\beta)\ge \ell_F(\alpha) + \ell_F(\beta)+\varepsilon\).

It is compact and contained in $\mathcal{F}_{ne}(\Gamma)$, so quotients to a compact subset $\Proj \Bar{\mathcal{A}}_{ne}(S; \mu, \varepsilon)\subset \Proj \mathcal{C}_{ne}(\Gamma)$.
Moreover for all $\mu>0$, the space $\Proj \mathcal{C}_{ne}(\Gamma)$ is the filtering union of the $\Proj \Bar{\mathcal{A}}_{ne}(S; \mu, \varepsilon)$ over increasing finite generating sets $S\subset \Gamma$ and decreasing $\varepsilon>0$.
\end{lemma}

\begin{proof}
The subset $\mathcal{A}_{ne}(S; \mu, \varepsilon)$ of $\mathcal{B}(S; \mu)$ is closed, hence compact by \Cref{lem:compact-Bmu(Gamma)}. 
Moreover $\mathcal{A}_{ne}(S; \mu, \varepsilon)$ is by definition  and \Cref{cor:elem-Hyp_length-homomorphism} contained in $\mathcal{F}_{ne}(\Gamma)$, hence it quotients by continuity to a compact subset $\Proj\Bar{\mathcal{A}}_{ne}(S;\mu,\varepsilon) \subset \Proj \mathcal{C}_{nn}(\Gamma)$.

Let $\mu \in \R_{\ge 0}$.
Fix $\class{F} \in \Proj\mathcal{C}_{nn}(\Gamma)$, and choose any representative $F\in \mathcal{F}_{ne}(\Gamma)$, with an associated representation $\rho\colon \Gamma \to \Isom(\Hyp^\kappa)$.

Since $\rho$ is non-elementary, we may find $\alpha, \beta\in \Gamma$ with $\ell_F(\alpha),\ell_F(\beta)\ge 1$ whose $\rho$-fixed points have $L:=2\log \Cr{\rho(\beta)^-, \rho(\alpha)^+; \rho(\alpha)^{-}, \rho(\beta)^{+}}>0$. 
Note that the quantity $L$ is invariant under replacing $\alpha, \beta$ by any of their positive powers, and by \Cref{lem:cross-ratio-from-translation-length} we have as $m,n\to \infty$ the limit $\ell_F(\alpha^{m}\beta^{n}) -\ell_F(\alpha^m)-\ell_F(\beta^n) \to L$.
Thus after replacing $\alpha,\beta$ by powers, we may assume that $\ell_F(\alpha),\ell_F(\beta)\ge 1$ and for all $m,n\in \N_{\ge 1}$ we have $\ell_F(\alpha^{m}\beta^{n}) -\ell_F(\alpha^m)-\ell_F(\beta^n) \ge  L/2$.


Choose any finite generating $S\subset \Gamma$ containing $\alpha,\beta$.
%
Let $t\in (0,1]$ such that $F^t\in \mathcal{B}(S;\mu)$. 
We have $\ell_{F^t}(\alpha), \ell_{F^t}(\beta)\ge t$ and
$\ell_{F^t}(\alpha\beta) -\ell_{F^t}(\alpha)-\ell_{F^t}(\beta) \ge tL/2$ so for $\varepsilon= t\min\{1, L/2\}\in \R_{>0}$ we have $F^t\in \mathcal{A}_{ne}(S;\mu, \varepsilon)$ as desired.
\end{proof}

\begin{remark}[not closed]
Let us compare the compactness of $\Proj\mathcal{C}(\Gamma)$ with classical results in finite dimension.

In finite dimension $n\in\N$, for a finitely generated non-virtually abelian group $\Gamma$, the space of discrete and faithful representations $\rho\colon \Gamma\to\Isom(\Hyp^n)$ is closed by \cite[Theorem 3]{Wielenberg_Discrete-Mobius-groups_1977}. One reason is the Margulis lemma which tells that there is $\varepsilon_n>0$ such for any such $\rho$, the subgroup generated by $\{\gamma\in\Gamma \colon d(\rho(\gamma)x,x)<
\varepsilon \}$ is virtually Abelian. 

In infinite dimension, the situation is different. There is no Margulis constant for representations $\Gamma\to\Isom(\Hyp^\kappa)$ when $\kappa$ is infinite. Indeed \cite[Observation 11.1.4]{Das-Simmons-Urbanski_gromov-hyperbolic-spaces_2017} shows that for all $\varepsilon>0$, one can find $\alpha,\beta \in \Isom(\Hyp^\kappa)$ with $\Delta(\{\alpha, \beta\})<\varepsilon$ that freely generate a rank-$2$ free subgroup $\Gamma \subset \Isom(\Hyp^\kappa)$ (with a fixed point at infinity) that is strongly discrete (for every bounded subset $B\subset\Hyp^\kappa$ the set $\{\gamma\in \Gamma \colon \gamma B\cap B\neq\emptyset\}$ is finite, see \cite[Definition 5.2.1]{Das-Simmons-Urbanski_gromov-hyperbolic-spaces_2017}).
\end{remark}

\subsection{Background on the equivariant Gromov-Hausdorff topology}

The goal for the rest of this section is to compare our compact space $\Proj\mathcal{C}(\Gamma)$ of classes of function of hyperbolic type on a finitely generated group $\Gamma$ with previous compactifications of character varieties with respect to the \emph{equivariant Gromov-Hausdorff topology} formalized by Paulin in \cite{Paulin_these_1987, Paulin_1988, Paulin_Gromov-topogy-R-trees_1989}, which we first recall.

Fix a group $\Gamma$, and consider a subset $\tilde{\mathcal{X}}$ of its isometric actions: an element of $\tilde{\mathcal{X}}$ consists of a metric space $(X,d)$ together with a homomorphism $\rho\colon \Gamma \to \Isom(X,d)$.

A basis of neighborhoods of $((X,d),\rho)\in \tilde{\mathcal{X}}$ is defined by the following subsets $V(P,K,\varepsilon)\subset \mathcal{X}$, indexed by finite $P\subset \Gamma$ and compact $K\subset X$ and $\varepsilon\in \R_{>0}$.
The subset $V(P,K,\varepsilon)$ consists of all $((X',d'),\rho')\in \tilde{\mathcal{X}}$ such that there exists a compact $K'\subset X'$ and a closed $P$-equivariant $\varepsilon$-approximation between $K$ and $K'$. 
An \emph{$\varepsilon$-approximation} between $K$ and $K'$ is a closed subset $R\subset K\times K'$ which projects surjectively onto $K$ and $K'$, such that for all $(x,x'), (y,y')\in R$, we have $\lvert d(x,y)-d'(x',y')\rvert<\varepsilon$.
It is $P$-equivariant when for all $\gamma\in P$ and $(z,z')\in R$ we have $\rho(\gamma) z\in K\implies (\rho(\gamma) z,\rho'(\gamma) z')\in R$.

This defines the \emph{equivariant Gromov-Hausdorff topology} on $\tilde{\mathcal{X}}$. 
For $i\in \{1,2\}$, the elements $((X_i, d_i),\rho_i)$ of $\tilde{\mathcal{X}}$ \emph{belong to the same class} when there is a $\Gamma$-equivariant isometry $X_1\to X_2$: we denote the topological quotient space $\mathcal{X}$; such a space is called a \emph{Gromov-Hausdorff space}.

\begin{remark}[Gromov-Hausdorff inseparability]
\label{rem:Gromov-Hausdorff inseparability}
In a topological space, two elements are called \emph{separable} when they admit disjoint open neighborhoods (called distinguishable in \cite{Paulin_1988}).

Let us give examples of non-isometric actions that are inseparable in any Gromov-Hausdorff space that contains them.

\emph{Fixed points:}
If a $\Gamma$-action by isometries has a fixed point then it is inseparable from the trivial action on a singleton.
In general, a $\Gamma$-isometric action $((X,d),\rho)$ is inseparable from the trivial action on a singleton if and only if for all finite $P\subset\Gamma$ and $\varepsilon>0$, there exists a compact $K \subset X$ of diameter $\max\{d(x,y)\colon x,y\in K\}<\varepsilon$ which is $P$-invariant. In particular, if an action of $\Gamma$ on $\Hyp^\kappa$ is parabolic, then it is inseparable from the trivial action on a singleton.

\emph{Invariant subspaces:}
A $\Gamma$-action by isometries is Gromov-Hausdorff inseparable from its restriction to any $\Gamma$-invariant subspace.
In particular, for two $\Gamma$-actions by isometries $((X_i, d_i),\rho_i)$, the Cartesian product $\Gamma$-action $\prod_i \rho_i$ of $\Gamma$ on $\prod_i X_i$ preserves any product metric (such as the $L^p$ combination of metrics with $p\in [1,\infty]$).
If $\rho_0$ has a fixed point then these product actions are all Gromov-Hausdorff inseparable from $\rho_1$.
\end{remark}

For a Gromov-Hausdorff space $\mathcal{X}(\Gamma)$, we denote $\Proj\mathcal{X}(\Gamma)$ its quotient under the following equivalence relation $((X,d_X),\rho_X)\sim ((Y,d_Y),\rho_Y)$ when there is a $\Gamma$-equivariant bijection $f\colon X\to Y$ and $\lambda>0$ such that for all $x,x'\in X$, $d_Y(f(x),f(x'))=\lambda d_X(x,x')$.

\subsection{Embedding Gromov-Hausdorff spaces in \texorpdfstring{$\Proj\mathcal{C}(\Gamma)$}{PC(gamma)}}

Consider the set $\mathcal{H}(\Gamma)$ of all conjugacy classes of $\Gamma$-isometric actions on some algebraic hyperbolic space $\Hyp^\kappa$ for some cardinal $\kappa$ with the equivariant Gromov-Hausdorff topology, and its subset $\mathcal{H}_{mnn}(\Gamma)$ of minimal non-neutral actions.
We may wish restrict to actions on an algebraic hyperbolic space of a fixed dimension $\kappa$, and we denote the corresponding space with $\kappa$ in exponent $\mathcal{H}^\kappa(\Gamma)$.

\begin{remark}[convex cores]
\label{rem:convex-cores-Hypn}
A representation $\rho \colon \Gamma \to \Isom(\Hyp^n)$ with no fixed points at infinity admits a unique minimal closed convex $\Gamma$-invariant subspace $X\subset\Hyp^n$ (called its convex core). In this case, any Gromov-Hausdorff neighborhood of the action of $\Gamma\to \Isom(X,d)$ contains the full action $\Gamma\to \Isom(\Hyp^n,d)$, hence these two actions are inseparable from one another.

When $\Gamma$ is countable, it follows from \cite[Chapitre 4, Théorème 2.6]{Paulin_these_1987} that for $i\in \{1,2\}$ two representations $\rho_i\in \mathcal{H}^{\kappa_i}(\Gamma)$ without fixed points at infinity are Gromov-Hausdorff inseparable if and only if the actions on their associated convex cores are conjugated.
In particular, if $\rho_i$ are irreducible and $\kappa_i\geq 2$, then these actions are Gromov-Hausdorff inseparable if and only if $\kappa_1=\kappa_2$ and the representations are conjugated. 
Indeed, in the linear models $\R^{1+\kappa_i}$, the linear span of the convex cores equal the whole space, and the actions are linear.
\end{remark}

\begin{proposition}
\label{lem:continuity_hyp_rep}
For a group $\Gamma$, the function 
$\mathcal{H}(\Gamma) \to \Proj\mathcal{C}(\Gamma)$ defined by $\class{\rho} \mapsto \class{F_\rho}$ is continuous and its restriction to $\mathcal{H}_{mnn}(\Gamma)$ is injective.
\end{proposition}

\begin{proof}
Let us first prove continuity of $\mathcal{H}(\Gamma)\to \Proj\mathcal{C}(\Gamma)$ at $\rho\in \mathcal{H}^\kappa(\Gamma)$.
Fix $o\in \Hyp^\kappa$.
For all $\eta>0$ and $F\subset \Gamma$ finite, we must construct a finite $P\subset \Gamma$, compact $K\subset \Hyp^\kappa$ and real $\varepsilon>0$ such that if $\rho'\in V_{\rho}(P,K,\varepsilon)$ then there exists some cardinal $\kappa'$ and $o'\in \Hyp^{\kappa'}$ such that $\forall \gamma\in F \colon \lvert F_{\rho,o}(\gamma)-F_{\rho',o'}(\gamma)\rvert <\eta$.
Set $P=F$, $K=\rho(P) \cdot o$, and $\varepsilon = \eta / \sinh(M)$ where $M=\operatorname{diam}(K)$.
It follows from the definition of the Gromov-Hausdorff neighborhoods that if $\rho' \in V_\rho(P,K,\varepsilon)$ then there exists $o'\in \Hyp^{\kappa'}$ such that $\forall \gamma\in P \colon \lvert \dH(\rho(\gamma)o,o)-\dH(\rho'(\gamma)o',o')\rvert<\varepsilon$.
This implies, using the fact that an $\varepsilon$-modulus of continuity for $\cosh$ on $[0,M]$ is given by $\varepsilon \sinh(M)$, we have that $\forall \gamma\in F\colon \lvert F_{\rho,o}(\gamma)-F_{\rho',o'}(\gamma)\rvert < \varepsilon\times \sinh(M)=\eta$, as desired.

Now let us show that the restriction $\mathcal{H}_{mnn}(\Gamma)\to \Proj\mathcal{C}(\Gamma)$ is injective.
Consider for $i=\{1,2\}$ minimal non-neutral representations $\rho_i\colon\Gamma\to\Isom(\Hyp^{\kappa_i})$ whose functions of hyperbolic type are homothetic, by \Cref{cor:homothetic_length_spectrum} we have $\kappa_1=\kappa_2$, and there is a parameter $t\in (0,1]$ for the exotic deformation  $\chi_t$ from \Cref{thm:exotic-deformation-Hyp} such that $\rho_1$ conjugated to $\chi_t\circ \rho_2$ or $\rho_2$ conjugated to $\chi_t\circ \rho_1$, hence $\rho_1, \rho_2$ define the same point in $\mathcal{H}_{mnn}(\Gamma)$.
\end{proof}

Consider the set $\mathcal{T}(\Gamma)$ of all conjugacy classes of $\Gamma$-isometric actions on complete real trees. An isometric action on a real tree is called \emph{minimal} when there is no closed invariant real sub tree (that is a closed invariant convex subspace in this situation).
The set $\mathcal{T}(\Gamma)$ contains the subset $\mathcal{T}_{m}(\Gamma)$ of minimal actions, which contains the subsets $\mathcal{T}_{mnn}(\Gamma)$ of those which are, moreover, non-neutral and $\mathcal{T}_{mne}(\Gamma)$ of those which are minimal and non-elementary.
We endow all these spaces with the equivariant Gromov-Hausdorff topology.

\begin{remark}[minimal representative in Gromov-Hausdorff neighbourhood]
Any action on a real tree has a minimal invariant closed subtree, unless there is a unique fixed point at infinity and all elements are elliptic. 
In particular every non-neutral action on a real tree is Gromov-Hausdorff-inseparable from a minimal non-neutral action on a complete real tree.

All neutral actions on a real tree are Gromov-Hausdorff-inseparable from the trivial action on a point.
\end{remark}

Let $\Proj\mathcal{T}(\Gamma) \supset \Proj\mathcal{T}_{m}(\Gamma), \Proj\mathcal{T}_{mnn}(\Gamma),\Proj\mathcal{T}_{mne}(\Gamma)$ be their quotient under the action of $\R_{>0}$ by homotheties where $t\in \R_{>0}$ acts on the class of $((T,d),\rho)\in \Proj\mathcal{T}(\Gamma)$ by $\lambda\cdot(T,d)=(T, td)$.

\begin{proposition}[compact metrizable space of actions on trees]
\label{prop:compactness_PT_mnn}
If $\Gamma$ is finitely generated, the Gromov-Hausdorff space $\Proj\mathcal{T}_{mnn}(\Gamma)$ is compact metrizable. 
\end{proposition}

\begin{proof}
Under the assumption that $\Gamma$ contains a free subgroup on $2$ generators, this result was proved for the subspace $\Proj\mathcal{TS}(\Gamma)$ of classes of isometric actions on real trees with small edge stabilizers by Paulin in \cite[Chapitre 2, Théorème 6.2]{Paulin_these_1987}.
Since $\Gamma$ is assumed to be finitely generated, the topology on $\Proj\mathcal{T}_{mnn}(\Gamma)$ is second countable, so it suffices to prove that it is sequentially compact, and we explain below how to adapt Paulin's proof to our situation.

Let $(T_n,\rho_n)$ be a sequence of minimal non-neutral $\Gamma$-actions on the complete real tree $T_n$.
Since there is no fixed point in $T_n$, there are at most $2$ fixed point at infinity. Up to extraction, we may assume that for all $n$, there is the same number $\in \{0,1,2\}$ of fixed points in $\partial T_n$, and we now distinguish those three cases. 

If every $\rho_n$ has no fixed points, then the actions are minimal and irreducible: Paulin's proof shows that one can extract a limit and that the limit representation has a loxodromic element (the assumption about small stabilizers is only used to ensure that the limiting representation has no fixed point at infinity and to prove that the limit representation has small stabilizers has well. Thus, it is unnecessary for us).

If every $\rho_n$ has a unique fixed point, then for a finite set of generators, the intersection of their minimal subsets (axis or fixed point set) is non-empty, so we may choose a base point $o_n$ in this intersection and apply Paulin's construction to obtain a limit representation with one generator that is hyperbolic, hence non-neutral.

If every $\rho_n$ has a $2$ fixed points, then the minimality assumption implies that the tree $T_n$ is the real line and the image of $\rho_n$ consists only of translations (possibly trivial), thus choosing any base point $o_n\in T_n\simeq\R$, rescaling by the maximum of translations lengths of the generators, and extracting a subsequence such that for any $n$ this maximum is realized by the same generator, we obtain a limit representation on $\R$ with a generator acting by translation by $\pm1$.
\end{proof}

\begin{remark}[minimal representatives]
\label{rmk:compactness_Hausdorff_quotient}
Consider a finitely generated group $\Gamma$.
The space $\mathcal{T}(\Gamma)$ of isometric actions on separable real trees has a maximal Hausdorff separated quotient $\overline{\mathcal{T}}(\Gamma)$: denote $\Proj\overline{\mathcal{T}}(\Gamma)$ its set of homothety classes with the quotient topology. \Cref{prop:compactness_PT_mnn} shows that $\Proj\overline{\mathcal{T}}(\Gamma)$ is compact metrizable homeomorphic to $\Proj\mathcal{T}_{mnn}(\Gamma)\sqcup\{\class{\Bar{1}}\}$ where the isolated point $\class{\Bar{1}}$ corresponds to the class of the trivial action.
This can be compared with \Cref{rem:convex-cores-Hypn}.
\end{remark}

\begin{proposition}[comparing topologies for actions on trees]
\label{lem:continuity_trees}
Let $\Gamma$ be a group.

The map $\Proj\mathcal{T}(\Gamma)\to\Proj\mathcal{C}(\Gamma)$ associating to the class of $((T,d),\rho)$ the class of functions of hyperbolic type $\gamma\mapsto \exp{d(o,\rho(\gamma) o)}$ for $o\in T$ is continuous.

Consequently, if $\Gamma$ is finitely generated the restriction $\Proj\mathcal{T}_{mnn}(\Gamma) \to \Proj\mathcal{C}_{nn}(\Gamma)$ is a homeomorphism on its image.
\end{proposition}

\begin{proof}
The map which to an isometric $\Gamma$-action on a tree $((T,d),\rho)\in \tilde{\mathcal{T}}^\cdot(\Gamma)$ and a point $o\in T$ associates the class of the function of hyperbolic type $\gamma\mapsto F(\gamma)=\exp {d(o,\rho(\gamma) o)}$ is independent of $o$ and invariant by homothety.
Moreover it is constant across the Gromov-Hausdorff class of $((T,d),\rho)$. The proof of continuity is very similar to the one in the proof of \Cref{lem:continuity_hyp_rep}

Fix some $((T,d),\rho)$ and $o\in T$. For a finite set $P\subset \Gamma$, we set $K=\rho(P) \cdot o$, and $\varepsilon = \eta e^{-M}$ where $M=\operatorname{diam}(K)$.
It follows from the definition of the Gromov-Hausdorff neighborhoods that if $((T',d'),\rho') \in V_\rho(P,K,\varepsilon)$ then there exists $o'\in T'$ such that $\lvert d(\rho(\gamma)o,o)-d'(\rho'(\gamma)o',o')\rvert<\varepsilon$ for all $\gamma\in P$.
This implies, using the fact that an $\varepsilon$-modulus of continuity for the exponential map on $[0,M]$ is given by $\varepsilon e^{-M}$, that $\lvert F_{\rho,o}(\gamma)-F_{\rho',o'}(\gamma)\rvert < \varepsilon\times e^M=\eta$  for all $\gamma\in P$ as desired.

Hence the map $((T,d),\rho)\mapsto F_\rho$ descends to a map $\mathcal{T}(\Gamma)\to \Proj\mathcal{C}(\Gamma)$. 
This map is constant across the Gromov-Hausdorff class of $((T,d),\rho)$ and invariant by homothety, so it descends to a continuous map $\Proj\mathcal{T}\to \Proj\mathcal{C}(\Gamma)$.

Now consider its restriction $\Proj\mathcal{T}_{mnn}\to \Proj\mathcal{C}(\Gamma)$. 
To show that this map is injective, we show how a limiting isometric action on a tree $T$ can be recovered from the cross-ratio at infinity. 
By minimality, the tree $T$ must be the union of axes of loxodromic elements for $\rho(\Gamma)$ hence the function of hyperbolic type determines the cross-ratio on $\partial T$. Consequently, the reconstruction procedure from $\partial T$ to $T$ via the cross-ratio in \Cref{prop:reconstruction_tree} below gives the injectivity.

The map restricted to $\Proj\mathcal{T}_{mnn}(\Gamma)$, which is compact and metrizable by \Cref{prop:compactness_PT_mnn}, is continuous and injective so it is a homeomorphism on its image. 
\end{proof}

\begin{remark}
The subspace of $\Proj\mathcal{C}(\Gamma)$ corresponding to the space of functions of hyperbolic type which are arboreal actions is closed by \Cref{cor:arobreal-closed}.
For finitely generated $\Gamma$, the \Cref{lem:continuity_trees} and \Cref{rmk:compactness_Hausdorff_quotient} shows that it corresponds to the compact metrizable space $\Proj\mathcal{C}(\Gamma)$ with the class of the trivial function being isolated. 

Moreover, it is known by \cite[Théorèmes 5.1 \& 5.3]{Paulin_these_1987}, that the topology on $\Proj\mathcal{T}_{mne}(\Gamma)$ is the one given by the projectivized marked length spectrum. 
\end{remark}

\begin{proposition}[reconstring tree from cross-ratio]
\label{prop:reconstruction_tree}
Consider for $i\in\{1,2\}$ two (complete) real trees $(T_i,d_i)$ such that the convex hull of $\partial T_i$ is $T_i$. If $\varphi\colon \partial T_1\to\partial T_2$ is a bijection preserving the boundary cross-ratio then there is an isometry $f\colon T_1\to T_2$ whose boundary extension coincides with $\varphi \colon \partial T_1 \to T_2$.

Moreover when $\Gamma$ acts by isometries on the $T_i$, if $\varphi$ is $\Gamma$-equivariant then $f$ is $\Gamma$-equivariant.  
\end{proposition}

\begin{proof} 
First, observe that if $\partial T_i$ consists of less than 4 points, then the trees $T_i$ are both a line or an infinite tripod, and the result follows. Now assume $\partial T_i$ has at least $4$ points.

In \cite[Theorem 1.1]{Davies}, the author considers the case of affine buildings of rank $1$ (namely simplicial trees with no leaves). Her reconstruction procedure gives the tree $T$ as a quotient of the tangent  space $\ufi(T)$ defined as the set of triples $((\xi_1,\xi_2,\xi_3),t)\in (\partial T)^{(3)}\times \R$ of at infinity together with a real parameter $t$. The point $((\xi_1,\xi_2,\xi_3),t)$ corresponds to the point on the oriented geodesics $(\xi_1,\xi_2)$ with parameter $t$ for the arc-length parametrization with origin at the projection of $\xi_3$ on $(\xi_1,\xi_2)$. The same construction works for a real tree $T$ that is the union of bi-infinite geodesics, or equivalently that is the convex hull of its boundary $\partial T$.

The equivariant statement follows from the construction.
\end{proof}

\subsection{Recognition of previous compactifications}

For a finitely generated torsion-free group $\Gamma$ containing a free group of rank $2$ and an for integer  $n\geq2$,
following \cite{Paulin_1988}, let $\mathcal{H}_{fd}^n(\Gamma)$ be the subspace of $\mathcal{H}^n(\Gamma)$ corresponding to faithful and discrete representations $\Gamma\to \Isom(\Hyp^n)$: it is endowed the equivariant Gromov-Hausdorff topology, and let $\Proj\mathcal{H}_{fd}^n(\Gamma)$ be its quotient by the homothety relation with the quotient topology.
Note that for $n_1,n_2\in \N$, if $n_1<n_2$ then $\mathcal{H}^{n_1}_{fd} \subset \mathcal{H}^{n_2}_{fd}$. 
The quotient map $\mathcal{H}_{fd}^n(\Gamma)\to \Proj\mathcal{H}_{fd}^n(\Gamma)$ is an homeomorphism. 

Recall that the Gromov-Hausdorff space of minimal non-elementary actions on trees $\Proj\mathcal{T}_{mne}(\Gamma)$ contains the subspace $\Proj\mathcal{TS}(\Gamma)$ of those having small stabilizers.
Note that $\Proj\mathcal{TS}(\Gamma)$ is compact.
Paulin proved in \cite[Théorème 6.4]{Paulin_1988} that the space $\Proj\mathcal{H}_{fd}^n(\Gamma)\sqcup \Proj\mathcal{TS}(\Gamma)$ is compact.
Thus $\Proj\mathcal{TS}(\Gamma)$ contains the \emph{equivariant Gromov-Hausdorff compactification of the space of faithful and discrete representations of $\Gamma$ on $\Hyp^n$}.

We will show that this compactification also holds in the corresponding spaces of homothety classes of functions of hyperbolic type.
For this, denote $\Proj\mathcal{C}_{fd}^n(\Gamma)$ the space of homothety classes of functions of hyperbolic type associated to faithful and discrete actions of $\Gamma$ on $\Hyp^n$, namely the image of $\mathcal{H}^n_{fd}(\Gamma)\subset \mathcal{H}_{mnn}(\Gamma)$ by the (injective and continuous) function $\mathcal{H}_{mnn}(\Gamma)\to \Proj \mathcal{C}(\Gamma)$ defined by $\rho \mapsto \class{F_{\rho}}$ obtained in \Cref{lem:continuity_hyp_rep}.

\begin{theorem}[recovering compactification of faithful and discrete]
\label{thm:identification_compactification} 
Let $\Gamma$ be a finitely generated torsion-free group containing a free group of rank $2$. 
The Gromov-Hausdorff space $\Proj\mathcal{H}_{fd}^n(\Gamma)\sqcup \Proj\mathcal{TS}(\Gamma)$ is homeomorphic to its image $\Proj\mathcal{C}_{fd}^n(\Gamma)\sqcup \Proj\mathcal{TS}(\Gamma)$ in $\Proj\mathcal{C}(\Gamma)$. 
\end{theorem}

\begin{proof} 
By \Cref{lem:continuity_hyp_rep} we have a continuous injective function $\mathcal{H}_{mnn}^n(\Gamma) \to \Proj\mathcal{C}(\Gamma)$ given by $\rho \mapsto \class{F_{\rho}}$.
Its restriction to $\mathcal{H}_{fd}^n(\Gamma)$ and then quotients to a function $\Proj \mathcal{H}_{mnn}^n(\Gamma) \to \Proj\mathcal{C}_{fd}^n(\Gamma)$ which by \Cref{cor:homothetic_length_spectrum} is still injective.
Its image denoted $\Proj\mathcal{C}_{fd}^n(\Gamma)$ consists of the homothety classes of functions of hyperbolic type associated to faithful and discrete actions of $\Gamma$ on $\Hyp^n$.
We thus obtain a continuous and bijective function $\Proj \mathcal{H}_{fd}^n(\Gamma) \to \Proj\mathcal{C}_{fd}^n(\Gamma)$.

By Lemmas \ref{lem:continuity_hyp_rep} and \ref{lem:continuity_trees}, the well-defined map $\Psi \colon \Proj \mathcal{H}^n(\Gamma) \sqcup \Proj \mathcal{T}(\Gamma) \to \Proj\mathcal{C}^n(\Gamma)$ is continuous in restriction to $\Proj\mathcal{H}_{fd}^n(\Gamma)$ and to $\Proj\mathcal{TS}(\Gamma)$, and we wish to show that $\Psi$ is continuous in restriction to $\Proj\mathcal{H}_{fd}^n(\Gamma)\sqcup\Proj\mathcal{TS}(\Gamma)$. Since $\Proj\mathcal{TS}(\Gamma)$ is compact and $\Proj\mathcal{H}_{fd}^n(\Gamma)\cup \Proj\mathcal{TS}(\Gamma)$ is metrizable, it remains to show that if a sequence of $\alpha_m \in \mathcal{H}^n_{fd}(\Gamma)$ has homothety classes converging to some $T_1 \in \Proj\mathcal{TS}(\Gamma)$, then the same is true for the associated equivalence classes of functions of hyperbolic type. But this follows from \Cref{thm:tree-embedding-Hyp-infty}: it actually constructs a real tree $T_2$ with a point, and a sequence of functions of hyperbolic type $F_{\rho_m}$ equivalent those associated to $\alpha_m$ such that $F_{\rho_m}$ converges to the function of hyperbolic associated to the action on $T_2$. 
By continuity of the length spectrum for the equivariant Gromov-Hausdorff topology \cite[Theorem 4.2]{Paulin_Gromov-topogy-R-trees_1989}, the homothety classes of the length spectrum for $T_1$ and $T_2$ are the same and thus by \Cref{cor:homothetic_length_spectrum} the classes of functions of hyperbolic type are the same. So the map $\Psi$ is continuous. It is injective by rigidity of the length spectrum for actions on $\Hyp^n$ and on real trees \cite[Remark 4.6]{Paulin_Gromov-topogy-R-trees_1989}. Since $\Proj\mathcal{H}_{fd}^n(\Gamma)\sqcup \Proj\mathcal{TS}(\Gamma)$ is compact and $\Proj\mathcal{C}_{nn}(\Gamma)$ is Hausdorff, $\Psi$ is a homeomorphism on its image.
\end{proof}

\begin{remark}[recovering previous compactifications]
It follows from \Cref{thm:identification_compactification}, that the compactification of the character variety $\mathcal{H}_{fd}^n(\Gamma)$ in $\Proj\mathcal{C}(\Gamma)$ is homeomorphic to the one obtained by Culler, Morgan, Otal and Shalen in various works \cite{Morgan-Shalen_valuations-trees-degernations-hyperbolic_1984, Morgan_actions-trees-compactification-SOn1-representation_1986, Culler-Morgan_group-actions-R-trees_1987, Morgan_Otal_relative-growth-geodesics_1993}.
Hence, our compactification encompasses these compactifications for all values of $n$.
\end{remark}

\begin{remark}[fundamental groups of hyperbolic manifolds]

When $\Gamma$ is the fundamental group of a closed surface of genus $\ge 2$, it follows from \cite[Theorem 3.2]{Skora_splittings_1996} that for all $n\ge 2$, the closure of $\mathcal{H}^n_{fd}(\Gamma)$ in this compactification is all of $\Proj\mathcal{TS}(\Gamma)$, and this coincides with Thurston's compactification of the Teichm\"uller space by measured laminations.
\end{remark}

\begin{example}[small actions trees not in the closure]
\label{eg:strict-inclusions-closure-PCn-fd}
There are finitely generated groups $\Gamma$ for which the inclusion $\bigcup_n \overline{\Proj\mathcal{C}^n_{fd}(\Gamma)} \subset \bigcup_n \Proj\mathcal{C}^n_{fd}(\Gamma)\sqcup \Proj\mathcal{TS}(\Gamma)$ is strict.

In fact, we can find such groups for which the set of faithful (maybe non-discrete) classes of actions $\Proj\mathcal{C}^n_{fd}(\Gamma)$ is empty whereas $\Proj \mathcal{TS}(\Gamma)$ is not. 

For this, consider two finitely generated groups $A,B$ such that for all $n\in \N$ they do not admit a faithful action on $\Hyp^n$ (for instance any non-$\R$-linear group, or any group having property-$T$ such as $\SL_k(\Z)$ for $k\ge 3$). 
The amalgam $\Gamma=A\star B$ also has $\bigcup_n \Proj\mathcal{C}^n_{fd}(\Gamma)=\emptyset$, but it acts on its Bass-Serre tree with trivial (hence small) edge stabilizers so $\Proj\mathcal{TS}(\Gamma)\ne \emptyset$.
\end{example}

\begin{remark}[when $\Proj \mathcal{TS} \subsetneq \Proj \mathcal{T}$]
\label{eg:strict-inclusion-PST-PT}
Note that $\Proj \mathcal{TS}(\Gamma)$ is closed in $\Proj \mathcal{T}(\Gamma$, so any point in the difference $\Proj \mathcal{T}(\Gamma)\setminus \Proj \mathcal{TS}(\Gamma)$ is isolated from the compactification $\bigcup_n \Proj\mathcal{C}^n_{fd}(\Gamma) \sqcup \Proj\mathcal{TS}_n(\Gamma)$.

Examples of $\Gamma$ for which $\Proj \mathcal{T}(\Gamma)\setminus \Proj \mathcal{TS}(\Gamma)\ne \emptyset$ include free groups on $\ge 2$ generators and automorphism groups of simplicial trees that are $k$-regular with valency $k\ge 3$.
\end{remark}

The following example suggests that $\Proj\mathcal{C}(\Gamma)$ can be much larger than $\bigcup_n \Proj \mathcal{C}^n(\Gamma) \sqcup \Proj \mathcal{T}(\Gamma)$.

\begin{example}[bendings]
\label{eg:bendings}
Let $\Gamma=\pi_1(\Sigma)$ be the fundamental group of a closed orientable surface of genus $\ge 2$. In \cite[Theorem 1.3]{Xu_convex-cocmpact-rep-infinite-dim-hyp_2024}, Xu exhibits a deformation space of representations of $\Gamma\to \Isom(\Hyp^\omega)$ of infinite dimension (parametrized by some infinite dimensional Lie group). 
These representations are not arboreal and we expect that most of them have critical exponent $1$ (namely they are not exotic deformations of representations into some $\Isom(\Hyp^n)$ for finite $n\in \N$, nor limits of such (since this would yield an arboreal one)).

In short, these representations are obtained from an initial representation $\rho$ by the following bending procedure. 

First, consider a simple closed curve $\gamma\subset \Sigma$ separating the surface $\Sigma$ to obtain (by the Van Kampen theorem) a presentation of its fundamental group as an amalgam $\Gamma = A*_\Z B$ where $A,B$ are the fundamental groups of the holed surfaces with genus $\ge 1$ on each side of $\gamma$.
Choose a representation $\rho\colon\Gamma\to \Isom(\Hyp^\kappa)$ which is an exotic deformation of a representation $\Gamma\to\Isom(\Hyp^2)$ obtained from a hyperbolic metric on $\Sigma$. 
\cite[Proposition 4.11]{Xu_convex-cocmpact-rep-infinite-dim-hyp_2024} shows that $\rho(A)$ is irreducible.

Next, for $g\in \Isom(\Hyp^\kappa)$ centralizing $\rho(\Z)$, define the bending $\rho^g_{A}\colon \Gamma \to \Isom(\Hyp^\kappa)$ by $g$-conjugating the values on $\beta\in B$ as $\rho^g_{A}(\alpha)=g\rho(\beta)g^{-1}$ while unchanging the values of $\alpha\in A$ as $\rho^g_{A}(\alpha)=\rho(\alpha)$.
For all $g$ the representation $\rho^g_A$ remains non-elementary, and choosing a path of such elements from $g$ to the identity (for instance varying in a one-parameter subgroup) we obtain a continuous path of non-elementary representations from $\rho$ to $\rho^g_A$.

Finally, if $\ell(g_1)\ne \ell(g_2)$ then the bendings $\rho_{g_1}$ and $\rho_{g_2}$ yield distinct points in $\Proj\mathcal{C}(\Gamma)$, as we now explain.
On the one hand, by \cite[Proposition 4.19]{Xu_convex-cocmpact-rep-infinite-dim-hyp_2024} $\rho_A^{g_1}$ and $\rho_A^{g_2}$ cannot be conjugated since $g_1\ne g_2$.
On the other hand, the $\rho_A^{g_i}$ are non-elementary and coincide on $A$ which contains a loxodromic element, so if two such representations yield  functions of hyperbolic type that are homothetic, then the homothety factor $t>0$ must be $1$, hence both representations must have the same the length spectrum, so by Proposition \ref{prop:same-length-spectrum-means-conjugate} they are conjugated, contradicting the previous observation.
\end{example}

\section{Global rigidity for certain group actions on \texorpdfstring{$\Hyp^\kappa$}{Hinfty} }\label{sec:rigidity}

In this section, we apply the unicity of abstract cross-ratio (up to powers) obtained in \Cref{subsec:unicity_cross_ratio} to prove that for some groups, the space of homothety classes of non-elementary functions of hyperbolic type $\Proj\mathcal{C}_{ne}(\Gamma)$ is reduced to a point, namely that there is a unique non-elementary representation in some $\Isom(\Hyp^\kappa)$ up to conjugacy and exotic deformations.

\subsection{Boundedness properties of groups}\label{subsec:boundedness}

To deal with the groups we have in mind, we will first need to extend our scope beyond locally compact groups and rely on the coarse geometry of topological groups developed by Rosendal \cite{Rosendal_coarse-geometry_2021}. 
We briefly recall the few notions we use and refer to the previous reference for a general treatment or to \cite{Rosendal_geometrisation_2023} for a shorter introduction.

The main idea is to have a well-defined notion of boundedness to extend geometric group theory outside the class of locally compact groups. Following Roe, Rosendal introduces coarse structures $\mathcal{E}$ on groups in \cite[Definition 2.2]{Rosendal_coarse-geometry_2021}. 
For a topological group $G$, we will only use two such coarse structures: the left coarse structure $\EE_L$, and the coarse structure  $\EE_d$ coming from a continuous isometric group action of $G$ on $(X,d)$. Actually, we will only need and recall the notions of bounded subsets relative to these coarse structures.  

\begin{definition}[$\EE_d$-bounded subsets] 
Let $G$ be topological group.
Consider a continuous isometric action on a metric space $(X,d)$. 
A subset $B\subset G$ is $\EE_d$-\emph{bounded} when for any $x\in X$ we have $\sup \{ d(gx,hx) \colon g,h\in B\} <\infty$.
\end{definition}
\begin{proof}
The triangle inequality shows that the condition is independent of $x$ (justifying the use of ``any''), and amounts to saying that $B$ has bounded image under any orbit map $g\mapsto gx$. 
\end{proof}

\begin{remark}[inverse and products]
The collection of $\EE_d$-bounded subsets in $G$ is stable under inverse and finite products: if $A,B\subset G$ are $\EE_d$-bounded then $A^{-1}$ and $AB$ are $\EE_d$-bounded.
\end{remark}

\begin{definition}[bornologous and coarsely proper actions]
Let $G$ be a subgroup of some $\Isom(X,d)$ where $(X,d)$ is metric space. An action of $G$ on another metric space $(X',d')$ by continuous isometries is:
\begin{itemize}[align=left, noitemsep]
	\item[$\EE_d$-\emph{bornologuous}] when every $\EE_d$-bounded subset is $\EE_{d'}$-bounded.
	\item[$\EE_d$-\emph{coarsely proper}] when every $\EE_{d'}$-bounded subset is $\EE_d$-bounded.
\end{itemize}
\end{definition}

\begin{definition}[coarsely bounded subset]
In a topological group $G$, a subset $B$ is \emph{coarsely bounded} when it is $\EE_d$-bounded for every continuous isometric action of $G$ on a metric space $(X,d)$.
In particular, the topological group $G$ itself is coarsely bounded when every continuous isometric action on a metric space has bounded orbits. 
\end{definition}

\begin{remark}[groups with an action witnessing coarse boundedness]
For some topological groups $G$, there is a faithful isometric action on a metric space $(X,d)$ so that coarsely bounded subsets coincide with $\EE_d$-bounded subsets. 
In this case, by definition of coarse boundedness, every continuous isometric action of $G$ on a metric space will be $\EE_d$-bornologous.
This holds for the action of $\Isom(\Hyp^\omega)$ on its symmetric space $\Hyp^\omega$ as we will see in \Cref{prop:proper_H^k}.
\end{remark}

\begin{example}[second countable locally compact groups]
A second countable locally compact group has a left invariant proper metric $d$, so its coarsely bounded subsets are precisely its $\EE_d$-bounded subsets and correspond to its relatively compact subsets.
\end{example}

\begin{definition}[locally bounded group]
A topological group $G$ is \emph{locally bounded} when it has a neighborhood of the identity that is coarsely bounded.
This encompasses local compactness.
\end{definition}

\begin{definition}[cobounded action]
For a metric space $(X,d)$ and a subgroup $G\subset \Isom(X,d)$, the action of $G$ on $X$ is \emph{cobounded} when there exists $o\in X$ and $M\geq0$ such that for every $x\in X$ we have $d(x,Go)\leq M$. 
\end{definition}
\begin{remark}[quasi-isometric]
Let us note that when an action of $G$ on a geodesic space $X$ is both coarsely proper and cobounded, the spaces $G$ and $X$ are \emph{quasi-isometric} in the terminology of Rosendal (see \cite[Theorem 2.57]{Rosendal_coarse-geometry_2021}).
\end{remark}

Our first example for this coarse geometry is given by $\Isom(\Hyp^\kappa)$ and its action on $\Hyp^\kappa$. The following solves in particular \cite[Problem A.11]{Rosendal_coarse-geometry_2021}.

\begin{proposition}[bounded properties of $\Isom(\Hyp^\kappa)$]
\label{prop:proper_H^k}
For a countable cardinal $\kappa$, the Polish group $\Isom(\Hyp^\kappa)$ is generated by a coarsely bounded neighborhood of the identity and it acts on $\Hyp^\kappa$ coarsely properly coboundedly.
In particular, a subset of $\Isom(\Hyp^\kappa)$ is coarsely bounded if and only if it is $\EE_{\dH}$-bounded for its action on $(\Hyp^\kappa,\dH)$.
\end{proposition}
\begin{proof}
Recall that $G=\Isom(\Hyp^\kappa)=\operatorname{PO}(1,\kappa)$ contains the subgroup $K=\operatorname{PO}(\kappa)$, and acts by left multiplication on the homogeneous space $X=\Hyp^\kappa = G/K$: the stabilizer of $o=\operatorname{1}K$ is the group $K$ which is coarsely bounded in $G$ (when $\kappa\in \N$ it is compact and when $\kappa=\infty$ is countably infinite one may find this in \cite[§6]{Rosendal_2009} or \cite[§3]{Duchesne_Polish-topology-isom-H-infinity_2023} for details).

The Cartan decomposition $G=KA(\R)K$ (where $A \colon \R \to G$ is a one parameter hyperbolic group of transformations $A(t)$ acting by translations of $\lvert \ell(A(t))\rvert=t$ along the common axis) shows that $G$ is generated by $KA([0,1])K$.
The subset $A([0,1])$ is coarsely bounded since it is compact.
This proves that $G$ is generated by a coarsely bounded neighborhood of the identity $K\cdot A([0,1]) \cdot K$, in particular it is locally bounded.
Moreover this shows that a subset $B\subset G$ is coarsely bounded if and only if there is $R\ge 0$ such that $B\subset K\cdot A([0,R])\cdot K$.

The action of $G$ on $X=G/K$ is transitive (hence cobounded) and the orbit map of any point is isometrically conjugate to that of the point $o$.
Since $K$ is coarsely bounded, a subset of $X=G/K$ is bounded if and only if its preimage in $G$ by the orbit map of $o$ is coarsely bounded, hence the action is coarsely proper.
\end{proof}

\subsection{Rigidity of groups acting fully on \texorpdfstring{$\operatorname{CAT}(-1)$}{CAT-1} spaces}
The main \Cref{thm:point_character_variety} in this section is a general rigidity result: certain groups can admit at most one continuous isometric action on $\Hyp^\kappa$ up to conjugacy and exotic deformations. 

\begin{definition}[$3$-full actions on $\operatorname{CAT}(-1)$ spaces]
\label{def:full-action-CAT-1}
Consider a topological group $G$ acting on a $\operatorname{CAT}(-1)$-space $(X,d)$ continuously by isometries.

We say that the action of $G$ on $(X,d)$ is \emph{$3$-full} when it is non-elementary and such that:
\begin{itemize}[noitemsep, align=left]
	\item[\emph{Boundary action is $3$-transitive}:] \label{item:full-CAT-3transitive} the action of $G$ on $\partial X$ is $3$-transitive,
	\item[\emph{Boundary stabilizers are amenable}:] \label{item:full-CAT-amenable} for every $\xi\in\partial X$, its stabilizer $G_\xi$ is amenable,
	\item[\emph{Coarse Cartan decompositions of rank-$1$}:] \label{item:full-CAT-BAB} for every loxodromic $a\in G$, there is a $\EE_d$-bounded subset $B\subset G$ such that $G=BAB$ where $A$ is the subgroup generated by $a$.
\end{itemize}
\end{definition}
\begin{remark}
The transitivity of the $G$ action on $\partial X$ implies that for all $\xi,\xi'\in\partial X$ the groups $G_\xi$ and $G_{\xi'}$ are conjugated in $G$, thus every $G_\xi$ contains an element acting loxodromically on $X$ (since otherwise none of them would, contradicting non-elementarity of the action).

Moreover, the $3$-transitivity of the $G$ action on $\partial X$ implies that for all $\xi\in\partial X$ the action of $G_\xi$ on $\partial X\setminus\{\xi\}$ is transitive, so for any triple of distinct points $\xi_1,\xi_2,\xi_3\in \partial X$ we have $G=G_{\xi_1}G_{\xi_2}\cup G_{\xi_3}G_{\xi_2}$.
\end{remark}

\begin{remark}[geometry of $\operatorname{CAT}(-1)$-spaces]
Let us briefly recall some properties that we will use about a $\operatorname{CAT}(-1)$-space $(X,d)$.

As a strongly-$1$-hyperbolic space, its geometry satisfies all results covered in the previous sections.
In particular, its cross-ratio generates the topology in the sense of \Cref{def:abstract-cross-ratio}.

For every pair of distinct boundary points $\xi,\xi'\in \partial X$, there is a unique geodesic connecting them $(\xi,\xi')=\{p\in X\colon (\xi,\xi')_p=0\}$, and any two points $\eta,\eta'\in X\sqcup \partial X$ admit a unique projections $p, p' \in (\xi',\xi)$, and those satisfy $\lvert \log \Cr{\xi,\eta,\xi',\eta'} \rvert = d(p,p') $ (see \cite[§1]{Bourdon_cross-ratio-CAT(-1)_1996}).
\end{remark}

\begin{theorem}[rigidity of groups acting fully on $\operatorname{CAT}(-1)$ spaces]
\label{thm:point_character_variety}
Consider a topological group $G$ acting $3$-fully on a $\operatorname{CAT}(-1)$-space $(X,d)$.

If every non-elementary continuous action of $G$ on $\Hyp^\kappa$ is $\EE_d$-bornologous, then $\Proj\mathcal{C}_{ne}(G)$ has at most one point.  
\end{theorem}

The strategy of the proof is to find a $G$-equivariant boundary map $\partial X\to\partial \Hyp^\kappa$ that is a homeomorphism on its image and  pull back the cross-ratio from $\partial \Hyp^\kappa$ to $\partial X$. 
We may then conclude using the unicity (up to taking powers) of the cross-ratio proved in \Cref{cor:power_cross_ratio}, and the fact that the cross-ratio determines the length spectrum of the representation \Cref{lem:translation-length-from-cross-ratio}, hence the conjugacy class of the representation by \Cref{cor:homothetic_length_spectrum}.

\begin{lemma}[boundary map]
\label{lem:boundary_map}
Consider a topological group $G$ acting $3$-fully on a $\operatorname{CAT}(-1)$-space $(X,d)$ as in \Cref{def:full-action-CAT-1}.

Consider a non-elementary $\EE_d$-bornologous continuous representation $\rho\colon G\to\Isom(\Hyp^{\kappa})$.
For every $\xi\in\partial X$, the group $\rho(G_\xi)$ fixes a unique point $\varphi(\xi)\in\partial\Hyp^\kappa$, and for every loxodromic $a\in G$ with $a^+=\xi$ its image $\rho(a)$ is hyperbolic with $\rho(a)^+=\varphi(\xi)$.

Moreover, the action of $G$ on $\Hyp^\kappa$ is $\EE_d$-coarsely proper. 
\end{lemma}

\begin{proof}
The continuous actions of the amenable groups $\rho(G_\xi)$ (which are all conjugate) on $\Hyp^\kappa$ must, by \cite[Proposition 5.1]{Duchesne_Polish-topology-isom-H-infinity_2023}, simultaneously fall into (at least) one of the following cases:
\begin{enumerate}[noitemsep]
	\item $\forall \xi \in \partial X$, the group $\rho(G_\xi)$ fixes a point in $\Hyp^\kappa$,
	\item $\forall \xi \in \partial X$, the group $\rho(G_\xi)$ stabilizes a geodesic line,
	\item $\forall \xi \in \partial X$, the group $\rho(G_\xi)$ fixes a single point in $\partial\Hyp^\kappa$
\end{enumerate}
and we will show that neither of the cases 1 nor 2 can occur, so that only case 3 occurs.

Suppose that case 1 occurs: for any triple of distinct points $\xi_1,\xi_2,\xi_3\in \partial X$ the group $\rho(G_{\xi_i})$ has a fixed point in $\Hyp^\kappa$, so $G=G_{\xi_1}G_{\xi_2}\cup G_{\xi_3}G_{\xi_2}$ has bounded orbits (by the triangle inequality), hence a fixed point, contradicting non-elementarity.

Suppose now that only case 2 occurs: for every $\xi\in \partial X$ the action of $\rho(G_\xi)$ does not fix a point in $\Hyp^\kappa$ but stabilizes a geodesic line $L$. 
We may choose $a\in G_\xi$ which is loxodromic for the action on $X$ with $a^+=\xi$, and denote $\xi'=a^-$. Since $\rho(a)$ preserves $L$, it is either loxodromic or elliptic. If it were elliptic, then by the property that $G=BAB$ for some $\EE_d$-bounded subset $B\subset G$ and $A$ the group generated by $a$, the image $\rho(G)$ would have a bounded orbit, contradicting non-elementarity of its action on $\Hyp^\kappa$.
Thus $\rho(a)$ is loxodromic and since $a\in G_\xi\cap G_{\xi'}$, the geodesic line $L$ invariant by $\rho(G_{\xi'})$ is the axis of $\rho(a)$. Finally, since $G_{\xi}$ acts transitively on $\partial X\setminus\{\xi\}$ while fixing $L$, for every $\eta\in \partial X$ the group $\rho(G_\eta)$ stabilizes $L$, thus $L$ is $G$-invariant, contradicting non-elementarity.

Now we know that every $\rho(G_\xi)$ fix a single point in $\partial \Hyp^\kappa$. 

Consider a loxodromic $a\in G_\xi$ with $a^+=\xi$, and denote $a^- = \xi' \in \partial X$ its repulsive fixed point. The same reasoning as before shows that $\rho(a)$ cannot be elliptic. If $\rho(a)$ is parabolic, then it has unique fixed point $\zeta\in\partial \Hyp^\kappa$, hence $G_{\xi'}$ also fixes $\zeta$, and as before the transitivity of $G_{\xi}$ on $\partial X\setminus\{\xi\}$ shows that for $\eta\in\partial X$ the group $\rho(G_\eta)$ fixes $\zeta$, hence $\zeta$ is fixed by $\rho(G)$, contradicting non-elementarity. Thus $\rho(a)$ must be loxodromic.
Let us now show that the fixed point of $\rho(G_\xi)$ is $\rho(a)^+$ and not $\rho(a)^-$. 
Choose a point $x\in X$ on the axis of $a$ and $b\in G_\xi\setminus G_{\xi'}$.
A comparison inequality in the triangle $(x,bx,\xi)$ shows that for all $n\in \N$ we have $d(a^{-n}ba^nx,x)=d(b(a^nx),a^nx)\leq d(bx,x)$. In particular, there is a $\EE_d$-bounded subset $B\subset G$ such that for all $n\in \N$ we have $a^{-n}ba^{n}\in B$. However $d(a^{n}ba^{-n}x,x)=d(b(a^{-n}x),a^{-n}x)\to\infty$ since $\xi'\ne b\xi'$. The same inequalities hold for $\rho(a)$ and $\rho(b)$ and their axes and fixed points: since the representation is bornologous, we deduce that $\rho(a)^+$ is the fixed point of $\rho(G_\xi)$.

To show that the action is coarsely proper, consider a loxodromic $a$, and let $G=BAB$ the coarse Cartan decomposition given by \Cref{item:full-CAT-BAB}.
We know from the previous that $\rho(a)$ is loxodromic: choose a point $y\in\Hyp^\kappa$ on its axis. The supremum  $M=\sup\{\dH(\rho(b)y,y)\colon b\in B\}$ is finite since $B$ is $\EE_d$-bounded and the action of $G$ on $\Hyp^\kappa$ is $\EE_d$-bornologous. For an element $g=b_1a^nb_2\in BAB$, two triangle inequalities show that $\dH(\rho(g)y,y)\geq \lvert n\rvert \ell(\rho(a))-2M$.
Two more triangular inequalities show that  for every bounded set $Y$, a similar inequality holds for all $y\in Y$, hence the action of $G$ on $\Hyp^\kappa$ is $\EE_d$-coarsely proper.
\end{proof}

The following lemma will serve to deduce bicontinuity of the boundary map $\varphi$ in \Cref{lem:boundary_map} from its (un)boundedness.

\begin{lemma}
\label{lem:proper}
Consider a topological group $G$ acting continuously by isometries on a $\operatorname{CAT}(-1)$ space $(Y,d)$, triply transitively on $\partial Y$. Let $a\in G$ be a loxodromic element. Consider a sequence of boundary points $\xi_n\in \partial Y \setminus \{ a^\pm\}$.
Denote $A_n(\xi) = \{g\in G_{a^-}\cap G_{a^+} \colon  g(\xi_0)=\xi_n\}$.

We have  $\xi_n \to a^+$ if and only if both of the following hold:
\begin{itemize}[noitemsep]
	\item the sets $A_n(\xi)$ leave any $\EE_{d}$-bounded set in $G$ as $n\to \pm \infty$
	\item there is an $\EE_{d}$-bounded set $B\subset G$  such that for $\forall n\in \N,\, \exists k_n\in \N \colon a^{-k_n}A_n(\xi)\subset B$.
\end{itemize} 
\end{lemma}

\begin{proof}
Recall the definition and properties of the unique projection $p_n$ of $\xi_n$ on $(a^-,a^+)$. We have $\xi_n\to a_+ \iff p_n\to a_+$. Moreover every $g \in A_n(\xi)$ satisfies $g(p_0)=p_n$. 

If $\xi_n\to a^+$ then for any $g_n\in A_n(\xi)$ we have $d(g_n p_0,p_0)\to\infty$ so for any $\EE_{d}$-bounded subset $B\subset G$ and for $n$ large enough, $A_n(\xi)\cap B=\emptyset$. Moreover, for any $n$, we may find $k_n\in\N$ such that $a^{-k_n}p_n\in B(p_0,\ell(a))$ and thus $a^{-k_n}A_n(\xi)$ is $\EE_{d}$-bounded.

Conversely if $A_n(\xi)$ escapes any $\EE_{d}$-bounded subset of $G$ then $d(p_0,p_n)\to\infty$ so up to extraction $p_n$ converges to $a^\pm$, but since the distance from $\{a^{-k_n}p_n\colon k\in \N \}$ to $p_0$ is bounded the limit must actually be $a^+$.
\end{proof}

\begin{proof}[Proof of \Cref{thm:point_character_variety}]
Consider a non-elementary continuous representation $\rho\colon G\to\Isom(\Hyp^{\kappa})$. 
By \Cref{lem:boundary_map}, we have a boundary function $\varphi\colon\partial X\to\partial \Hyp^\kappa$ sending $\xi\in\partial X$ to the unique fixed point of $\rho(G_\xi)$.
Let us show that $\varphi$ is a homeomorphism on its image.

Since the topology on $\partial X$ and $\partial \Hyp^\kappa$ are metrizable, we will prove that a sequence $\xi_n\in \partial X$ converges if and only if $\varphi(\xi_n)\in \partial \Hyp^\kappa$ converges.
Fix distinct $\xi^\pm\in \partial X$ and choose a loxodromic $a\in G$ such that $a^\pm=\xi^\pm$. 
We now apply \Cref{lem:proper} with $Y=X$: a sequence $(\xi_n)$ in $\partial X\setminus\{\xi^+\}$ converges to $\xi^+$ if and only the subsets $A_n(\xi)=\{g\in G_{a^+}\cap G_{a^-} \colon g(\xi_0)=g(\xi_n)\}$ leave any $\EE_d$-bounded subset of $G$ and there is a $\EE_d$-bounded subset $B\subset G$ such that $\forall n\in \N, \exists k_n\in\N \colon a^{-k_n} A_n(\xi) \subset B$. 
Since the action of $G$ on $\Hyp^\kappa$ is $\EE_d$-coarsely proper (by \Cref{lem:boundary_map}) and $\EE_d$-bornologous, this is equivalent to saying that the image under $\rho$ of
\begin{equation*}
	A_n'(\xi) = \{g \in G \colon \rho(g)(\varphi(\xi^\pm))=\varphi(\xi^\pm) \textrm{ and } \rho(g)(\varphi(\xi_0))=\rho(g)(\varphi(\xi_n))\}
\end{equation*}
escapes as $n\to +\infty$ from any $\EE_{\dH}$-bounded subset of $\Isom(\Hyp^\kappa)$ and there is a $\EE_{\dH}$-bounded subset $B\subset \Isom(\Hyp^\kappa)$ such that for all $n\in \N$ there is $k_n\in \N$ such that $\rho(a)^{-k_n} A_n'(\xi) \subset B$.

By \Cref{lem:proper} for $Y=\Hyp^\kappa$ this is equivalent to the convergence of $\varphi(\xi_n)$ to $\varphi(\xi^+)$.
The proves that $\varphi \colon \partial X \to \partial \Hyp^\kappa$ is a homeomorphism on its image.

To complete the proof consider for $i\in\{1,2\}$ two non-elementary representations $\rho_i\colon G\to\Isom(\Hyp^\kappa)$ with boundary maps $\varphi_i\colon\partial X\to\partial \Hyp^\kappa$. By \Cref{cor:power_cross_ratio} there is $t>0$ such that for all $\xi_1,\xi_2,\xi_3,\xi_4\in(\partial X)^{(4)}$, 
\begin{equation*}
	\Cr{\varphi_1(\xi_1),\varphi_1(\xi_2);\varphi_1(\xi_3),\varphi_1(\xi_4)}
	=
	\Cr{\varphi_2(\xi_1),\varphi_2(\xi_2);\varphi_2(\xi_3),\varphi_2(\xi_4)}^t
\end{equation*}
In particular, up to composing $\rho_1$ with the exotic representation $\rho_t$ or $\rho_2$ with $\rho_{1/t}$ (depending on  whether $t\leq1$ or $t>1$), we may reduce to the case $t=1$ and deduce by \Cref{cor:Poincare-extension} that the restriction of the $\rho_i$ to their minimal invariant totally geodesic subspaces are conjugated, hence their associated functions of hyperbolic type belong to the same class.
\end{proof}

\subsection{Rigidity for \texorpdfstring{$\Isom(\Hyp^\kappa), \Aut(T_\kappa), \PGL_2(\K)$}{IsomHk-AutTk-PGL2K)}}

\begin{corollary}[rigidity of classical groups]
\label{cor:point_character_variety}
If $G$ is one the following topological groups:
\begin{enumerate}
	\item the isometry group $\Isom(\Hyp^\kappa)$ of the real hyperbolic space $\Hyp^\kappa$ of countable dimension $\kappa$,       
	\item the isometry group $\Aut(T_\kappa)$ of the $\kappa$-regular simplicial tree $T_\kappa$ where $\kappa\ge 3$ is countable,
	\item the group $\PGL_2(\K)$ where $\K$ a complete non-Archimedean valued field
\end{enumerate}
then the space $\Proj\mathcal{C}_{ne}(G)$ has exactly one point. 
\end{corollary}

\begin{remark}[recovering previous results]
\Cref{cor:point_character_variety} recovers (with a different proof) the following cases.
The case of $\Isom(\Hyp^\kappa)$ for $\kappa\in \N$ (in particular of $\PGL_2(\K)$ for $\K\in \{\R,\C\}$) and of $\Isom(\Hyp^\omega)$ are treated respectively in \cite{Monod-Py_exotic-deformation-hyperbolic_2014} and \cite{Monod-Py_self-representations-Mobius_2019}.
The case of $\PGL_2(\K)$ where $\K$ is a non-Archimedean local field is treated in \cite[Theorem C]{Burger-Iozzi-Monod_embedding-trees-hyperbolic-spaces_2005}.
\end{remark}

\begin{remark}[new examples for $\PGL_2(\K)$]
In 
\Cref{cor:point_character_variety}, the case of $\PGL_2(\K)$ is new when the field $\K$ is nonlocal, so it provides in particular the following sources of new examples.

We emphasize that when $\K$ is Polish, the group $\PGL_2(\K)$ is Polish for the topology induced by the uniform topology (as we will see in \Cref{cor:case_PGL2K}). 

\begin{itemize}
	\item $\K=\C_p$, the completion of the algebraic closure of a $p$-adic field $\Q_p$, which is Polish.
	\item $\K=\Kres((t))$, the field of Laurent series over any field $\Kres$, which is Polish if $\Kres$ is countable.
	\item $\K=\Kres((t^\Lambda))$, the field of Hahn series over a field $\Kres$ with exponents in a subgroup $\Lambda\subset \R$ (the previous case being $\Lambda=\Z$), which is Polish if $\Kres$ is countable and $\Lambda$ is discrete. 
\end{itemize}
\end{remark}

\begin{remark}[from $\Aut(T_k)$ to Burger-Mozes groups]
\label{rem:Burger-Mozes}
The case of the automorphism group $\Aut(T_k)$ of the regular tree with infinite countable valency should be extendable to the following.

Consider the infinite regular tree $T_\N$ whose vertex links are in bijection with $\N$.
The group of all permutations of a countable set $\mathfrak{S}(\N)$ is a non-Archimedean Polish group.
To a closed subgroup $\Gamma \subset \mathfrak{S}(\N)$ is associated a universal Burger-Mozes groups $U(\Gamma)$ consisting of those automorphisms of $T_\N$ whose local actions on vertex links belongs to $\Gamma$.

If the action of $\Gamma$ on $\N$ is $3$-transitive then the same is true for the action of $U(\Gamma)$ on $\partial T_\N$.
If the action of $\Gamma$ on $\N$ is oligomorphic then, by the argument in the proof of \cite[Proposition 5.2]{Duchesne-Tarocchi_homeo-fractals_2025}, the action $U(\Gamma)$ on $T_\N$ has coarsely bounded vertex stabilizers. The coarse Cartan decompositions in \Cref{item:full-CAT-BAB} follows from this coarse boundedness of vertex stabilizers and the fact that the triple transitivity at infinity implies that the stabilizer of a vertex acts transitively on $\partial T_\omega$. 
Under these assumptions, the only remaining condition to check in \Cref{item:full-CAT-amenable} is the amenability of stabilizers of points at infinity: it is less straightforward, but may follow from additional properties for the action of $\Gamma$ on $\N$ such as amenability of its point stabilizers. 
When it holds, we may apply our \Cref{thm:point_character_variety} to encompass such new examples, including the homeomorphism groups of the Basilica Julia set and other rabbits Julia sets \cite{Duchesne-Tarocchi_homeo-fractals_2025}.
\end{remark}

\begin{remark}[dependence on the topology]
\Cref{cor:point_character_variety} is stated for topological groups. 

In some cases, the continuity of the representation is not a restriction. For example, the group $\Isom(\Hyp^\omega)$ has the automatic continuity property, so that when $\kappa$ is countable any group representation in $\Isom(\Hyp^\kappa)$ is automatically continuous (see \cite[Theorem III]{Monod-Py_self-representations-Mobius_2019} for these specific representations and \cite[Theorem 1.4]{Duchesne_Polish-topology-isom-H-infinity_2023} for the general automatic continuity property).

For other groups, the topology is a key ingredient. 
The field $\C$ and $\C_p$ are algebraically isomorphic (being the unique algebraically closed field with characteristic zero and continuum cardinality) but not homeomorphic since $\C$ is locally compact and Archimedean whereas $\C_p$ is not locally compact and non-Archimedean.
Hence the groups $\PGL_2(\C)$ and $\PGL_2(\C_p)$ are algebraically isomorphic but not homeomorphic. 

\Cref{cor:point_character_variety} shows that both topological groups have a unique homothety class of non-elementary functions of hyperbolic type. However their associated representations are different. 
Indeed for $\PGL_2(\C)$ the restrictions of these representations to $\PGL_2(\Z)$ have unbounded orbits whereas for $\PGL_2(\C_p)$ the restrictions to $\PGL_2(\Z)$ are bounded.
\end{remark}

\begin{remark}[Gelfand pairs]
The proof of \cite[Theorem C]{Burger-Iozzi-Monod_embedding-trees-hyperbolic-spaces_2005} relies on Gelfand pairs $(G,K)$, which are well defined only when $G$ is locally compact (so that it has a Haar measure). 

The notion of Gelfand pair has been extended to topological groups acting by isometries by Rosendal in \cite{Rosendal_equivariant-geometry-Banach_2014}. One can check that the hypotheses and in particular the triple transitivity at infinity in \Cref{thm:point_character_variety} implies that the action of $G$ on $X$ yields a geometric Gelfand pair according to the terminology of Rosendal. Our proof does not use this notion.
\end{remark}

In the remaining subsections, we prove \Cref{cor:point_character_variety} case by case. Observe that the non-emptiness of $\Proj\mathcal{C}_{ne}(G)$ follows, in each case, from the existence of a non-elementary action on some $\Hyp^\kappa$ or on some real tree.

\subsubsection{The case of $\Isom(\Hyp^\kappa)$}

\begin{proof}[Proof of \Cref{cor:point_character_variety} for $\Isom(\Hyp^\kappa)$ with its Polish topology]
We check the conditions in \Cref{def:full-action-CAT-1} and \Cref{thm:point_character_variety} for $G=\Isom(\Hyp^\kappa)$ and $X=\Hyp^\kappa$ with $\kappa$ countable. 

We saw that $G$ acts triply transitively on $\partial X$ in \ref{subsec:Hn-models}.
%
%
The stabilizer of a point at infinity is amenable by \cite[Lemma 5.4]{Duchesne_Polish-topology-isom-H-infinity_2023}. 
The coarse Cartan decompositions of rank-$1$ in \Cref{item:full-CAT-BAB} follow from the Cartan decompositions of $\operatorname{PO}(1,\infty)$ recalled in \ref{subsec:Hn-models} and the transitivity of the action of $\operatorname{PO}(1,\infty)$ on $\Hyp^\omega$.

Finally, the action of $G$ on $(X,d)$ is coarsely proper by \Cref{prop:proper_H^k}, so any $\EE_d$-bounded subset is coarsely proper, hence any continuous isometric action is $\EE_d$-bornologous. 
\end{proof}

\subsubsection{The case of $\Aut(T_\kappa)$}

Let $T_\kappa$ be the simplicial $\kappa$-regular tree with countable valency $\kappa\ge 3$.
Its automorphism group $G=\Aut(T_\kappa)$ endowed with the pointwise convergence on the set of vertices of $T_\kappa$ is a non-Archimedean Polish group. 

For a finite integer $\kappa\ge 3$, the group $\Aut(T_\kappa)$ is locally compact (see for example \cite{Garrido-Glasner-Tornier_2018}) and it is well known that its action on $T_\kappa$ satisfies \Cref{def:full-action-CAT-1}. 
Hence we will focus on the case where $\kappa=\omega$ is the infinite countable degree, even though all our arguments can be adapted for finite cardinals.

\begin{proof}[Proof of \Cref{cor:point_character_variety} for $\Aut(T_\omega)$]

We first check the conditions in \Cref{def:full-action-CAT-1}.
The fact that the action on $G$ on $\partial T_\omega$ is $3$-transitive follows from a standard back and forth argument.
We will prove that points in $\partial T_\omega$ have amenable stabilizers in \Cref{prop:prop:amenabe-stab-point-infinity} below.
Let us show now the coarse Cartan decomposition.
First let us note that the stabilizer of a point is coarsely bounded by \cite[Example 3.6]{Rosendal_geometrisation_2023}, hence for a segment $[x,x']\subset T_\omega$ the subset $B(x,x')=\{g\in \Aut(T_\omega)\colon gx\in [x,x']\}$ is a finite union of translates of a coarsely bounded subset, hence coarsely bounded. 
Now consider a loxodromic $a\in G$ with translation length $\ell(a)=l\in \N_{>0}$. Choose a point $x$ on its axis $L$ and let $B=\{g\in G\colon gx \in [x,ax]\}$.
For $g\in G$, there is $s\in \Stab(x)$ such that $s^{-1}gx\in L$, hence $m\in\Z$ such that $a^ms^{-1}g(x)\in [x,x']$ hence $a^ms^{-1}g x\in B$, that is $g\in BAB$.

Finally, the action of $G$ on $T_\omega$ is coarsely proper by \cite[Example 3.6]
{Rosendal_geometrisation_2023}, in particular $\EE_d$-bounded subsets coincide with coarsely bounded subsets so every continuous isometric action is $\EE_d$-bornologous.
\end{proof}

For $\xi\in\partial T_\omega$, its stabilizer $G_\xi$ is endowed with its Busemann character $\Bus_\xi \colon G_\xi \to \Z$, whose kernel $G_\xi^0 \subset G_\xi$ consists of elements which stabilize each horosphere centred at $\xi$.
Note that by regularity, the Busemann character is surjective so we have an extension $G_\xi^0 \to G_\xi \to \Z$.

\begin{proposition}[amenability of $G_\xi$]
\label{prop:prop:amenabe-stab-point-infinity}
There is an increasing sequence of finite subgroups $G_\xi^0(n)\subset G_\xi^0$ indexed by $n\in \N$ such that for every $g\in G_\xi^0$ and every finite subset $F$ of vertices in $T_\omega$, there exists $n\in\N$ and $h\in G_\xi^0(n)$ such that $g$ and $h$ coincide on $F$.

Consequently, $G_\xi^0$ is amenable, hence $G_\xi$ is amenable.
\end{proposition}

\begin{proof}
Fix a base vertex $w_0$ and denote $(w_i)_{i\in\N}$ its path to $\xi$.
Note that for $g\in G_\xi$, there is $m\in \N$ such that $[w_0,\xi) \cap [gw_0,\xi) = [w_m,\xi)$ and if $g\in G_\xi^0$ then $i\geq m \implies g(w_i)=w_i$.
Moreover for every finite subset $F$ of vertices, we may consider the subtree $\Conv(w_m,F)$ induced by the vertices in $F$ and $w_m$.
We will see how to construct a finite order element $h\in \Aut(T_\omega)$ whose action in restriction to $\Conv(w_m,F)$ coincides with that of $f$, and, moreover, how to exhaust all such finite order elements by a sequence of finite groups.

For each vertex $v$ of $T_\omega$, we denote by $E_v$ be the set of edges incident to $v$ with $e_{v,\xi}\in E_v$ the unique one on the path from $v$ to $\xi$, and choose a bijection (called a colouring) $C_v\colon E_v\setminus\{e_{v,\xi}\}\to\N$ with the condition that for all $i\in\N$ we have $C_{w_{i+1}}(w_{i+1}w_i)=0$, as in the following Figure. 
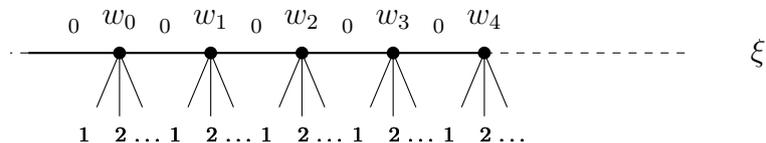
\begin{center}
	\begin{tikzpicture}[scale=1.2]
		
		\draw[thick] (-1,0) -- (4,0);
		\draw[dashed] (-1,0) -- (-1.2,0);
		
		\foreach \x/\lab in {0/$w_0$,1/$w_1$,2/$w_2$,3/$w_3$,4/$w_4$} {
			\fill (\x,0) circle (2pt);
			\node[above] at (\x,0.15) {\lab};
		}
		
		\foreach \x in {-0.5,0.5,1.5,2.5,3.5} {
			\node[above,font=\scriptsize] at (\x,0.1) {0};
		}
		
		\foreach \x in {0,1,2,3,4} {
			\draw (\x,0) -- ++(-0.25,-0.6);
			\draw (\x,0) -- ++(0,-0.7);
			\draw (\x,0) -- ++(0.25,-0.6);
		}
		
		\foreach \x in {0,1,2,3,4} {
			\node[font=\scriptsize] at (\x,-0.9) {$1\;\;\;\;2\,\ldots$};
			\node[font=\scriptsize] at (\x,-0.9) {$1\;\;\;\;2\,\ldots$};
			\node[font=\scriptsize] at (\x,-0.9) {$1\;\;\;\;2\,\ldots$};
			\node[font=\scriptsize] at (\x,-0.9) {$1\;\;\;\;2\,\ldots$};
		}
		\draw[dashed] (4,0) -- (6.2,0);
		
		\node[right] at (6.8,0) {$\xi$};
		
	\end{tikzpicture}
\end{center}

For $g\in G_\xi$, its local action at a vertex $v$ is the bijection $E_v\setminus\{e_{v,\xi}\} \to E_{gv}\setminus\{e_{gv,\xi}\}$ encoded as the permutation $\sigma(g,v)\colon \N\to\N$ defined by $C_{gv}\circ g_{\mid E_v\setminus\{e_{v,\xi}\}}\circ C_v^{-1}$. 
The local actions satisfy the cocycle relation $\sigma(gh,v)=\sigma(g,hv)\circ\sigma(h,v)$.

For a vertex $v$ and $k,n\in \N_{>0}$, we define the finite subtree $S_{k,n}(v)\subset T_\omega$ which is the union of all paths $(v=v_0,\dots,v_n)$ of length at most $n$ with outgoing edges having labels $C_{v_i}((v_i,v_{i+1}))\in\N_{<k}$.
Thus $S_{k,n}(v)$ is a $k$-regular rooted tree at $v$ of depth $n$.

Now let $G_{k,n}(v)$ be the subset of $g\in G_\xi$ such that for very vertex $u\in S_{k, n}(v)$ that is not a leaf the permutation $\sigma(g,u)$ fixes pointwise $\N_{\geq k}$, and for every other vertex $u\in T_\omega$ we have $\sigma(g,u)=\mathbf{1}$. 
Thus $G_{k,n}(v)$ is the subgroup of $G_\xi$ preserving $S_{k,n}(v)$ and acting rigidly on all edges outside $S_{k,n}(v)$ (namely preserving colours of edges not in $E(S_{k,n}(v)$). 
In particular an element of $G_{k,n}(v)$ is completely determined (from the initial colouring) by its restriction on $S_{k,n}(v)$. 
Hence the group $G_{k,n}(v)$ is isomorphic to the automorphism group of the $k$-regular rooted tree of depth $n$, which is finite. 

We note that $S_{n,2n}(w_n) \subset S_{n+1,2(n+1)}(w_{n+1})$, and we claim that $\bigcup_{n\in\N_{\ge 1}}S_{n,2n}(w_n) =T_\omega$.
Indeed, a vertex $v$ belongs to $S_{n,2n}(w_n)$ as soon as $n$ is larger than $d(v,w_0)$ and all colours of edges on the path from $w_0$ to $v$ are smaller than $n$.

For $n\in \N_{\geq2}$, denote $G^0_\xi(n) = G_{n,2n}(w_n)$.
Since $S_{n,2n}(w_n)\subset S_{n+1,2(n+1)}(w_{n+1})$ we have $G_\xi^0(n)\subset G_\xi^0(n+1)$ and we finally show that $\bigcup G_\xi^0(n)$ is dense in $G_\xi^0$.

Let $g\in G_\xi^0$ and $m\in \N$ such that $i\geq m \implies g(w_i)=w_i$.
For a finite subset $F$ of vertices, there is $n\geq m$ large enough such that $F\cup g(F)\subset S_{n,2n}(w_n)$ (it suffices to consider the subtree $\Conv(w_m,F)$ and its edge colours). Since any partial automorphism of $S_{n,2n}(w_n)$ (that is an isomorphism between subtrees containing the root $w_n$) can be extended to an automorphism of $S_{n,2n}(w_n)$, there is an element of $G_\xi^0(n)$ coinciding with $g$ on $F$.

This proves that $G_\xi^0$ contains a countable dense union of finite subgroups.
Therefore $G_\xi^0$ is amenable, so its $\Z$-extension $G_\xi$ is amenable.
\end{proof}

\subsubsection{The case of $\PGL_2(\K)$.}
\label{cor:case_PGL2K}
Fix a (commutative) field $\K$ with a non-trivial complete non-Archimedean norm $\lvert \cdot \rvert$.
The valuation ring $\OO=\{x\in \K \colon  \lvert x\rvert \leq1\}$ has maximal ideal $\mathfrak{m} = \{x\in \K \colon \lvert x\rvert <1\}$ and residue field $\Kres=\OO/\mathfrak{m}$.
Note that the residue map $\OO \to \Kres$ is continuous for the discrete topology on $\Kres$ (which is induced by the trivial norm), hence for any other topology (including those induced by non-trivial norms).

More generally for $r\in \R_{>0}$ denote $\OO_r=\{x\in \K \colon  \lvert x\rvert \leq r\} \supset \mathfrak{m}_r = \{x\in \K \colon \lvert x\rvert < r\}$.
These are both clopen subsets of $\K$ and yield as $r \to 0$ a basis of neighbourhoods of $0 \in \K$.
For all $r\in \R_{>0}$, the subsets $\mathfrak{m}_r\subset \OO_r$ are fractional ideals of $\K$, and their quotient $\OO_r/\mathfrak{m}_r$ is isomorphic to $\Kres$ or $0$ according to whether $\K$ has or has no elements of norm $r$ (if $\lambda \in \mathcal{O}_r\setminus \mathfrak{m}_r$ then multiplication by $\lambda$ sends the pair of ideals $\mathcal{O}\supset \mathfrak{m}$ to the pair of ideals $\mathcal{O}_r\supset \mathfrak{m}_r$).

For $n\in \N_{\ge 1}$, the framed $\K$-vector space $\K^n$ is endowed with the norm extending $\lvert \rvert$ on $\K$, defined using its canonical framing on any $v=(v_1,\dots,v_n) \in V$ by $\lvert v\rvert =\max\{ \lvert v_i \rvert \colon 1\le i\le n\}$, and yielding an ultrametric on $\K^n$.
Hence the $\K$-algebra of matrices $\operatorname{M}_n(\K)$, is endowed with the \emph{operator norm} defined by
\(
\|g\|:= \sup \left\{\lvert gv\rvert / \lvert v\rvert \colon v\in \K^n\setminus\{0\} \right\}
\)
which induces the topology given by uniform convergence for its action on $\K^n$.
It can be computed, using the ultrametric inequality for $\lvert \cdot \rvert$ as $\|g\| = \max\{\lvert g_{i,j}\rvert \colon 1\le i,j\le n\}$, so it also coincides with the canonical norm on the $\K$-vector space $\operatorname{M}_n(\K)$ and induces the product topology of its components. 

The group of invertibles $\GL_n(\K)\subset \operatorname{M}_n(\K)$ with the induced topology is a topological group, and its projectivization $\PGL_n(\K)= \GL_n(\K)/\K^\times\mathbf{1}$ is endowed with the quotient topology.
If $\K$ is separable then $\K$ it is a Polish field, and $\GL_n(\K), \PGL_n(\K)$ are Polish groups.

The group $\GL_n(\K)$ acts on its Bruhat-Tits affine building and its essentialization: we only recall the definition of the building as metric space, and of the isometric action of $\GL_n(\K)$ but for more details we refer to \cite[§6.2]{Rousseau_2023} and \cite{Parreau_2023} (or its French original \cite{Parreau_2000}). 

On a $\K$-vector space $V$ of dimension $n\in \N$, a \emph{norm} is a function $\eta\colon V\to\R_+$ such that 
\begin{itemize}[noitemsep]
\item $\eta(v)=0\iff v=0$,
\item for all $v\in V$ and $\lambda\in \K$ we have $\eta(\lambda v)=\lvert \lambda\rvert \eta(v)$
\item for all $u,v\in V$ we have $\eta(u+v)\leq\max(\eta(u),\eta(v))$
\end{itemize}
The distance $d_\infty$ between two norms $\eta_1,\eta_2$ on $V$ is given by 
\begin{equation*}
d_\infty(\eta_1,\eta_2)=\sup\left\{\left|\log\tfrac{\eta_1(v)}{\eta_2(v)}\right| \colon v\in V\setminus\{0\}\right\}
\end{equation*}

Given a basis $(e_1,\dots,e_n)$ of $V$, a norm $\eta$ on $V$ is called \emph{split} when for all $\lambda_1,\dots, \lambda_n\in\K$,
\begin{equation*}\textstyle
\eta\left(\sum_i \lambda_i e_i\right)=\max\{\lvert a_i\rvert\eta(e_i)  \colon 1\le i \le n\}.
\end{equation*}
A norm is \emph{splittable} when there is a basis in which it is split. 
The \emph{affine building} $\mathcal{I}(\GL(V))$ is the set of all splittable norms with the distance $d_\infty$. The action of the group $\GL(V)$ by precomposition is continuous and isometric. 
Its essentialization, denoted $\mathcal{I}^e(\GL(V))$, is the topological space of homothety classes of splittable norms, and the distance between two homothety classes is the infinimum of the distances betwen their representatives.
The action of $\GL(V)$ on $\mathcal{I}(\GL(V))$ by isometries quotients to an action of $\PGL(V)$ by isometries on $\mathcal{I}^e(\GL(V))$.

Once a determinant has been fixed on $V^n$, we may realize $\mathcal{I}^e(\GL(V))$ in $\mathcal{I}(\GL(V))$ as the subspace of splittable norms of volume $1$, hence it is endowed with the induced metric $d_\infty$.
The isometric action of $M\in \PGL(V)$ on $\eta\in \mathcal{I}^e(\GL(V))$ is given for any representative $M\in \GL(V)$ by $(\eta\circ M)/\sqrt[n]{\lvert \det(M)\rvert}$.

We now focus on the case $n=2$ and choose a basis $e=(e_1,e_2)$ to identify $V=\K^2$ and $\GL(V)=\GL_2(\K)$.
We also use the determinant associated to that basis to identify the Bruhat-Tits building $\mathcal{I}^e(\GL_2(\K))$ as above: it is a real tree \cite[§6.2.20]{Rousseau_2023} which from now on we denote by $T$.
When the valuation associated to the norm is  discrete, the construction of the tree can be obtained by using lattices (see \cite[§6.2.23]{Rousseau_2023}), but otherwise the real tree is not the geometric realization of a simplicial tree.

We define the \emph{base point} $\eta_e \in T$ as the homothety class of the standard norm $\eta(a_1e_1+a_2e_2)=\max\{\lvert a_1 \rvert,\lvert a_2\rvert)$; its stabilizer is $\PGL_2(\OO)$.
The orbit under $\PGL_2(\K)$ of $\eta_e$ consists of all splittable norms whose image is contained in $\lvert \K \rvert\subset \R$.
The points of $T$ in this orbit are called \emph{special} \cite[§3B1]{Parreau_2023}; their stabilizers are all conjugate to $\PGL_2(\OO)$.

The projective line $\K\Proj^1$ identifies naturally with the boundary $\partial T$ (see \cite[§ 6.2.20]{Rousseau_2023}), and the projective line $\Kres\Proj^1$ identifies with the set of connected components of the complement $T\setminus \eta_e$ (this follows from \cite[Theorem 6.2.16.1(b)]{Rousseau_2023}).
In particular very special point $x\in T$ is a branch point in the sense that its complement $T\setminus x$ has at $\ge 3$ components (which are open since $T$ is locally connected).

We now show $3$-transitivity of the action on $\partial T$ and describe stabilizers of boundary points.

\begin{lemma}[action on boundary]
\label{lem:Action-PGL2K-boundary-Tree}
The action of the topological group $\PGL_2(\K)$ by isometries of its real tree $(T,d_\infty)$ satisfies the first two conditions in \Cref{def:full-action-CAT-1}:
\begin{itemize}[noitemsep]
	\item the action on $\partial T$ is $3$-transitive
	\item every $\xi\in \partial T$ its stabilizer $G_\xi\subset \PGL_2(\K)$ is amenable.
\end{itemize}
\end{lemma}

\begin{proof} 
The boundary $\partial T$ identifies naturally with the projective line $\K\Proj^1$, so that the action of $\PGL_2(\K)$ is $3$-transitive (see \cite[§ 6.2.20]{Rousseau_2023}).

For a point in $\xi\in \partial T$, its stabilizer $G_\xi\subset \PGL_2(\K)$ is conjugated to the projectivized group of upper-triangular matrices, which is solvable hence amenable.
\end{proof}

In the tree $T$, the set of homothety classes of norms split in the basis $e$ forms a subset $T_A$ which is isometric to a real line (this is an apartment in the language of buildings). 
In $\PGL_2(\K)$, the classes of diagonal matrices form a subset $A$ uniquely parametrized by $\lambda\in \K^\times$ according to $
A^\lambda=\begin{bsmallmatrix}\lambda &0\\0&1\end{bsmallmatrix}$. 
Such matrix $A^\lambda$ acts on $T$ as a loxodromic element, with translation axis $T_A$ and translation length $\left|\log \lvert \lambda\rvert \right|$ (this computation follows from the definitions of $\|\cdot\|$ and $d_\infty$, and we refer to \cite[§3B2]{Parreau_2023} for details).

Since the action of $\PGL_2(\K)$ on $\partial T$ is $3$-transitive, for all distinct $\xi_1,\xi_2 \in \partial T$, there exists $g\in \PGL_2(\K)$ sending the geodesic $(\xi_1,\xi_2)$ to $T_A$, in particular the geodesic $(\xi_1,\xi_2)$ contains special points.
In fact, $\PGL_2(\K)$ acts transitively on pairs consisting of a special point included in an apartment (this follows from the fact that $\K$ is complete, implying that the system of apartments is maximal \cite[Remark 6.2.17.3 (c)]{Rousseau_2023} and \cite[Proposition 2.18.]{Parreau_2023}).

We will not need the following \cref{rem:topology_PGL2K} but it gives more consistency to this whole discussion, in particular the next \Cref{lem:Action-PGL2K-boundary-Tree_BAB}.

\begin{remark}[topologies on $\PGL_2(\K)$] 
\label{rem:topology_PGL2K}
It is explained in \cite[\S 6.2.10.1.(2)]{Rousseau_2023} that the topology induced by the operator norm (called the strong topology in this reference) coincides with the so-called building topology (the coarsest topology for which pointwise stabilizers of bounded subsets are open). 

Since $T$ has a non-discrete metric topology, one may also consider the topology of pointwise convergence on $\PGL_2(\K)$, that is the coarsest topology for which orbit maps are continuous. By construction of the distance on $T$, the action of $\PGL_2(\K)$ endowed with the topology induced by the operator norm is continuous, namely the operator topology is finer than the topology of pointwise convergence. 

Let us sketch a proof of the converse. 
Consider any point $x\in T$.
First, we claim that, if a sequence of elements of $g_n\in \PGL_2(\K)$ converges to the identity for the pointwise convergence topology then for $n\in \N$ large enough we have $g_n x=x$. 

First assume that $x\in \PGL_2(\K)\cdot e$ is a special point (that it is a branch point) and choose three points $x_1,x_2,x_3$ in three distinct components of $T\setminus \{x\}$ so that $x$ is the only intersection $\cap_{i\neq j}[x_i,x_j]$.
For $n$ large enough, the element $g_nx_i$ lies in the same (open) connected component of $T\setminus\{x\}$ as $x_i$, hence $g_nx=x$. 

Consequently, for all positive value of the norm $M\in \lvert \K^\times \rvert \subset \R_{>0}$, there exists $n(M)\in \N$ such that if $n>n(M)$ then $g_n$ fixes the base point $\eta\in T$ and the three special points obtained by intersecting the sphere of radius $M$ around $\eta$ with the lines $(\eta,0),(\eta,1),(\eta,\infty)$, hence $g_n \equiv \mathbf{1} \bmod{O_M}$ namely $\|g_n-\mathbf{1}\| \le 1/M$.
This proves that pointwise convergence implies uniform convergence.
\end{remark}

We now check that the action of $\PGL_2(\K)$ on $T$ admits coarse Cartan decompositions.

\begin{lemma}[coarse Cartan decompositions of rank-$1$]
\label{lem:Action-PGL2K-boundary-Tree_BAB}
The action of the topological group $G=\PGL_2(\K)$ by isometries of its real tree $(T,d_\infty)$ admits coarse Cartan decompositions of rank-$1$ as in \Cref{def:full-action-CAT-1}:
for every loxodromic $a\in G$ generating a subgroup $A$, there is a $\EE_{d_\infty}$-bounded subset $B$ such that $G=BAB$.
\end{lemma}

\begin{proof}
Consider a loxodromic element $a'\in \PGL_2(\K)$, generating a subgroup denoted $A'$ and with translation axis denoted $T_{A'}$. 
Choose a point $x\in T_{A'}$ which is special (in the $\PGL_2(\K)$ orbit of the base point).
The set $B=\{g\in\PGL_2(\K) \colon d_\infty(gx,x)\leq\ell(a')\}$ is $\EE_{d_\infty}$-bounded by definition, and note that it contains $\Stab(x)$.

Choose $g\in G$. Since $G$ acts transitively on pairs consisting of an apartment with a special point in an apartment, there is $k\in \Stab(x)$ such that $k g(x)\in T_{A'}$. 
Hence there is $a''\in A'$ such that $d(a''kgx,x)\leq\ell(a')$. Thus $a''kg$ belongs to $B$ and $G=BA'B$.
\end{proof}

At this point, we have verified all conditions in \Cref{def:full-action-CAT-1}.
We are left to show the assumptions in \Cref{thm:point_character_variety}.
This will be done \Cref{prop:PGL2K_bornlogous_coarsely-proper} which requires some lemmata.

For $\lambda,\mu,\nu \in \K$ with $\lambda \ne 0$, denote the following matrices of $\SL_2(\K)$:
\begin{equation}
\label{eq:matrices-DLRS}
D^\lambda = \begin{bsmallmatrix}
	\lambda &0\\0&1/\lambda
\end{bsmallmatrix}
\quad
L^\mu=\begin{bsmallmatrix}
	1 & 0 \\ \mu & 1
\end{bsmallmatrix}
\quad
R^\nu=\begin{bsmallmatrix}
	1 & \nu \\ 0 & 1
\end{bsmallmatrix}
\qquad
S=\begin{bsmallmatrix}
	0 & -1 \\ 1 & 0
\end{bsmallmatrix}
\end{equation}
which satisfy the "generalized braid" identities:
\(
L^{-\lambda} R^{1/\lambda} L^{-\lambda} = 
R^{1/\lambda} L^{-\lambda} R^{1/\lambda} = 
- S D^\lambda
\).

\begin{lemma}[unipotent factorisation]
\label{lem:unipotent-factorisation}
For all $g\in \SL_2(\K)$ there exist $i,j\in \{0,1\}$ and $\mu,\nu \in \OO$ and $\lambda\in \K^*$ with $\lvert \lambda \rvert = \|g\|$ such that 
\begin{equation}
	\label{eq:LRS-factorisation}
	g = S^i \cdot L^\mu  \cdot D^\lambda \cdot R^\nu \cdot S^j
	= (S^i \cdot L^\mu S)  \cdot ( R^{1/\lambda} L^{-\lambda} R^{1/\lambda}) \cdot (R^\nu \cdot S^j).
\end{equation}
\end{lemma}
\begin{proof}
Let 
\(g=\begin{bsmallmatrix}
	a&b\\c&d
\end{bsmallmatrix}\in \GL_2(\K)\).
After replacing $g$ by $S^{-i}gS^{-j}$ for some $i,j\in \{0,1\}$ we may assume that $\lvert a\rvert = \|g\|$, and Gauss-reduction yields
\(g = L^{\mu} \operatorname{diag}(a,d') R^{\nu}\)
where $\mu=c/a, \nu = b/a \in \OO$ and $d'=\det(g)/a \in \OO$.
Thus if $g\in \SL_2(\K)$ then $d'=1/a$ so $\operatorname{diag}(a,d')=D^\lambda$ with $\lambda=a$ hence $g=L^\mu D^\lambda R^\nu$ and writing $D^\lambda = SR^{1/\lambda} L^{-\lambda} R^{1/\lambda}$ yields the result.
\end{proof}

The following lemma describes the coarsely bounded subsets of $\SL_2(\K)$ in terms of its action on its building $T$, or of the operator norm $\|\cdot \|$.
A subset $B\subset \operatorname{M}_n(\K)$ is called \emph{$\|\cdot\|$-bounded} when there is $M\in\R_{\geq0}$ such that every $g\in B$ satisfies $\|g\|\leq R$. 

\begin{lemma}[coarsely bounded subsets of $\SL_2(\K)$]
\label{lem:bounded_PSL2K} 
For a subset $B\subset \SL_2(\K)$. The following are equivalent: 
\begin{enumerate}[noitemsep]
	\item $B$ is coarsely bounded
	\item $B$ is $\EE_{d_\infty}$-bounded
	\item $B$ is $\|\cdot \|$-bounded.
\end{enumerate}
\end{lemma}

\begin{proof}
\emph{(1$\implies$2)}
If $B$ is coarsely bounded then by definition it is in particular $\EE_{d_\infty}$-bounded. 

\emph{(2$\implies$3)}
Assume that $B$ is $\EE_{d_\infty}$-bounded. 
Let $\eta\in\mathcal{T}$ be the canonical norm in the basis $e$ of $\K^2$ given by $\eta(\sum_ia_ie_i)=\max_i\lvert a_i\rvert$. 
There is $M>0$ such for all $g\in B$ we have $d_\infty(g\eta,\eta)=d_\infty(\eta,g^{-1}\eta)\leq M$. In particular, for all $j=\{1,2\}$ we have $\max_i\left|\log \lvert g_{ij}\rvert \right|=\left|\log\left(\frac{\eta(g(e_j)}{\eta(e_j)}\right)\right|\leq M$, which implies that $B$ is $\| \cdot \|$-bounded.

\emph{(3$\implies$1)}
Assume that $B$ is $\|\cdot\|$-bounded by some real $M\in \R_{\ge 1}$.
Consider the identity neighbourhoods in $\SL_2(\K)$ defined for $\varepsilon>0$ by $U_\varepsilon=1+\varepsilon \{g\in \SL_2(\K)\colon \| g\| < 1\}$.
Note that these form a basis of neighbourhoods of the identity as $\varepsilon$ decreases to $0$.

Let us first show that for $\varepsilon\in (0,1]$, there is a finite $F_\varepsilon\subset \SL_2(\K)$ such that $B\subset (F_\varepsilon U_\varepsilon)^{16}$.
Let $\lambda\in \K$ such that $\lvert \lambda \rvert \geq\sqrt{M/\varepsilon}\ge 1$ and \(F_\varepsilon=\left\{\begin{bsmallmatrix} \lambda^{s}&0\\0&\lambda^{-s}
\end{bsmallmatrix}\colon s\in \{-1,0,+1\}\right\}\).
The set $(F_\varepsilon U_\varepsilon)^2$ contains $F_\varepsilon U_\varepsilon F_\varepsilon \supset \{L^\mu,R^\nu\colon \mu,\nu \in \OO_M\}$ and contains $U_\varepsilon F_\varepsilon U_\varepsilon\supset \{S^i\colon i\in \Z/4\}$.
The \Cref{eq:LRS-factorisation} shows that $B\subset (F_\varepsilon U_\varepsilon)^{2\times 8}$ as desired. 

Let us finally deduce that $B$ is coarsely bounded.
Consider an action of $\SL_2(\K)$ on a metric space $(X,d)$ by continuous isometries.
Fix $x\in X$ and consider the open neighbourhood of the identity $U=\{g\in\SL_2(\K)\colon d(gx,x)<1\}$.
There is $\varepsilon\in (0,1]$ such that $U_\varepsilon\subset U$, and we may thus find a finite $F_\varepsilon \subset\SL_2(\K)$ as above so that $B\subset(F_\varepsilon U_\varepsilon)^{16}$. Let $M=\max\{d(gx,x)\colon g\in F_\varepsilon\}$. For all $f\in F_\varepsilon$ and $u\in U_\varepsilon$ we have $d(fux,x)\leq d(fux,fx)+d(fx,x)$ so for all $g\in F_\varepsilon U_\varepsilon$ we have $d(gx,x)\leq 1+M$ and $16-1$ similar triangle inequalities show that every $g\in(F_\varepsilon U_\varepsilon)^{16}$ satisfies $d(gx,x)\leq 16(M+1)$, hence the image of $B$ under the orbit map $g\mapsto gx$ is bounded in $X$.
\end{proof}

\begin{remark}[locally bounded]
The topological groups $\SL_2(\K)$, $\GL_2(\K)$ and $\PGL_2(\K)$ are locally bounded.

Indeed, the subgroup $\SL_2(\OO)$ is open and coarsely bounded in both $\SL_2(\K)$ and $\GL_2(\K)$, so its image in $\PGL_2(\K)$ is a coarsely bounded neighbourhood of the identity.
\end{remark}

\begin{example}[field Laurent series is not locally coarsely bounded]
\label{eg:Laurent-R-not-lcb}
Consider a field $\Kres$ with a non-trivial norm $\lvert \cdot \rvert'$, and denote $\OO'=\{c\in \Kres\colon \lvert c\rvert'=1\}$ its valuation ring. 
(Let us recall that a field admits no non-trivial norms if and only if it is an algebraic extension of the finite field $\mathbb{F}_p$ of characteristic $p\ne 0$, see \cite[Chapitre 6]{Bourbaki_Alg-Comm-1-7_1981}). 

Consider the field $\K=\Kres((x))$ of Laurent series over $\Kres$ consisting of formal series $f(x)=\sum_{n\in\Z}c_nx^n$ with only finitely many non-zero coefficients $c_n\in \Kres$ of negative index. 
It is endowed with the discrete valuation $v(f)=\inf\{n \in \Z\colon c_n\neq0\}$ yielding the (non-trivial and complete non-Archimedean) norm $\lvert f \rvert=e^{-v(f)}$.
The valuation ring $\mathcal{O}=\Kres[[x]]$ consisting of regular formal power series $\sum_{n\in\N} c_nx^n$ has maximal ideal $\mathfrak{m}=x\Kres[[x]]$ consisting of those with $c_0=0$, and the residue map $\OO_\K \to \Kres$ is given by $f\mapsto f(0) = c_0$ 

This residue map extends to an additive group homomorphism $(\K, \lvert \cdot \rvert )\to (\Kres, \lvert \cdot \rvert')$ which is continuous and surjective, thus it has unbounded image (since the norm $\lvert \cdot \rvert'$ is non-trivial).
More generally, the $n$-th component map $D^n\colon \Kres[[x]]\to \Kres$ defined by $D^n\colon \sum_{n\in\Z}c_nx^n \mapsto c_n$ is locally constant hence continuous, and it a (split) group homomorphism which is surjective in restriction to $x^n\mathcal{O}$.

Identifying $\Kres$ with the translation subgroup of $\Isom(\Kres)=\Kres \rtimes \OO'$, this yields for very $n\in \N$ a continuous representation of $\K \to \Isom(\Kres)$ such that $x^n \mathcal{O}_\K$ has unbounded image.
Since the sets $x^n\OO$ as $n\in \N$ form a basis of neighbourhoods of the origin in $\K$, we deduce that $\K$ is not locally coarsely bounded.

One may generalize this idea to show that other valued fields $(\K,\lvert \rvert)$ are not locally coarsely bounded, but a general characterization seems out of reach in model-theoretic terms \cite{Kartas_valued-field-total-residue_2024}. Moreover, for every field of positive characteristic $\Kres$ (of which many are infinite, such as $\mathbb{F}_p(T)$), there is a discretely valued ring of zero characteristic $\OO$ of which it is the residue field (if the field $\Kres$ is perfect, then $\K$ can be constructed from Witt vectors, otherwise a result of Cohen ensures existence of such a field $\K$, se \cite[Chapitre 9]{Bourbaki_Alg-Comm-8-9_1983}
Now, its field of fractions $\K=\operatorname{Frac}(\OO)$ of characteristic $0$ has an additive group that is divisible, so it cannot act isometrically on $\Kres$ which has characteristic $p$ (for instance $1/p$ would act trivially).
\end{example}

\begin{remark}[unipotent locally coarsely bounded only inside $\SL_2(\K)$]
The group of unipotent matrices parametrized by $\mu \in\K$ according to $R^\mu = \begin{bsmallmatrix}
	1&\mu \\0&1
\end{bsmallmatrix}$ is isomorphic (as topological group) to $(\K,+)$.
Thus for many fields (such as those in \cref{eg:Laurent-R-not-lcb}), that unipotent is not locally coarsely bounded.
Still we deduce from \Cref{lem:bounded_PSL2K}, that inside $\SL_2(\K)$ this unipotent subgroup is locally coarsely bounded.
\end{remark}

\begin{question}[$\EE_d$-bounded implies coarsely bounded]
The \Cref{lem:bounded_PSL2K} shows that for a representation of $\PSL_2(\K)$, does $\EE_{d_\infty}$-bounded imply coarsely bounded.

Does the same hold for representations of $\PGL_2(\K)$?
\end{question}

From now on $G$ denotes $\PGL_2(\K)$. 
An element of $\PGL_2(\K)$ is called loxodromic when its action on $T$ is  loxodromic.

\begin{proposition}[bornologous and coarsely proper]
\label{prop:PGL2K_bornlogous_coarsely-proper}
Consider a continuous representation $\rho\colon \PGL_2(\K) \to\Isom(\Hyp^\kappa)$ that is non-elementary.

There is a $\PGL_2(\K)$-equivariant boundary map $\varphi \colon \partial T\to\partial \Hyp^\kappa$ which is injective.

Moreover, the action of $\PGL_2(\K)$ on $\Hyp^\kappa$ is $\EE_{d_\infty}$-bornologous.
\end{proposition}

\begin{proof}
Let $\xi \in \partial T=\K\Proj^1$ and consider its stabilizer $G_\xi$ in $\PGL_2(\K)$.
We saw in \Cref{lem:Action-PGL2K-boundary-Tree} that $\PGL_2(\K)$ acts transitively on $\K\Proj^1$, so the group $G_\xi$ is conjugate to the standard parabolic subgroup $G_\infty$ generated by the homotheties $\{A^\lambda \colon \lambda\in \K^\times\}$ and transvections $\{R^\mu \colon \mu \in \K\}$, in particular it is amenable.
Thus by continuity of $\rho$, the subgroup $\rho(G_\xi)\subset \Isom(\Hyp^\kappa)$ is amenable.
Hence (as in the proof of \Cref{lem:boundary_map}) it must satisfy (at least) one of the following cases:
\begin{enumerate}[noitemsep]
	\item $\forall \xi \in \partial T$, the group $\rho(G_\xi)$ fixes a point in $\Hyp^\kappa$,
	\item $\forall \xi \in \partial T$, the group $\rho(G_\xi)$ stabilizes a geodesic line,
	\item $\forall \xi \in \partial T$, the group $\rho(G_\xi)$ fixes a single point in $\partial\Hyp^\kappa$
\end{enumerate}
and we will show (for $\xi=\infty$) that neither of the cases 1 nor 2 can occur, so that only case 3 occurs.

The first case would imply, as in \Cref{lem:boundary_map}, that $\PGL_2(\K)$ fixes a point in $\Hyp^\kappa$. Observe that if $\rho(G_\infty)$ has either $0$ or $\ge 3$ fixed points in $\partial \Hyp^\kappa$, then it must have a fixed point in $\Hyp^\kappa$. Thus it suffices to show that the fixed point set of $\rho(G_\infty)$ is not of cardinal $2$. 

Assume by contradiction that $\rho(G_\infty)$ has exactly two fixed points $\xi_1,\xi_2\in\partial \Hyp^\kappa$ and no fixed point in $\Hyp^\kappa$. In particular all $A^\lambda$ are elliptic or loxodromic.
Fix $\mu \in\K^\times$, and consider sequence of $\lambda\in \K$ with $\lvert \lambda \rvert \to\infty$.
We have $A^{\lambda^{-1}} R^\mu A^{\lambda}=\begin{bsmallmatrix}1&\mu/\lambda\\0&1\end{bsmallmatrix} \to \mathbf{1}$, hence by continuity of $\rho$ we have $\rho(A^{\lambda^{-1}})\rho(R^\mu)\rho(A^{\lambda})\to \mathbf{1}$. 
Since the length function is invariant by conjugacy and continuous, we have $\ell (\rho(R^\mu))=0$ so $\rho(R^\mu)$ is elliptic or parabolic and this type is independent of $\mu$. Since $\rho(G_\infty)$ has exactly two fixed points in $\partial \Hyp^\kappa$, the only possibility is that $R^\mu$ is elliptic, hence fixes every point on the geodesic $(\xi_1,\xi_2)$. If all $A^\lambda$ are elliptic then the geodesic $(\xi_1,\xi_2)$ is pointwise fixed by $\rho(G_\infty)$, a contradiction. Hence there is $\lambda\in \K^\times$ such that $\rho(A^\lambda)$ is loxodromic, hence has $\{\xi_1,\xi_2\}$ as pair of fixed points. Since $A^\lambda\in G_0$ as well, the points $\xi_1,\xi_2$ are also fixed by $\rho(G_0)$ in $\partial\Hyp^\kappa$. Finally, since $G_\infty$ acts transitively on $\K\Proj^1\setminus\{\infty\}$, the points $\xi_1,\xi_2$ are fixed by $\PGL_2(\K)$. This contradicts the non-elementarity of the action of $\PGL_2(\K)$ on $\Hyp^\kappa$.

This proves that for all $\xi\in \partial T$, the group $\rho(G_\xi)$ has no fixed points in $\Hyp^\kappa$ and a unique fixed point denoted $\varphi(\xi) \in\partial\Hyp^\kappa$. 
We have thus defined a boundary function $\varphi \colon \xi \in \partial T \mapsto \varphi(\xi) \in \partial \Hyp^\kappa$ which is $\rho$-equivariant.
Since $G$ acts $2$-transitively on $\partial T$, the group $\rho(G)$ acts $2$-transitively on the image $\varphi(\partial T)$, so the map $\varphi$ must be either injective or constant, but if it were constant then $\rho$ would have a global fixed point at infinity contradicting non-elementarity, so $\varphi$ is injective. 

The subgroup of upper triangular matrices $G_\infty$ and the subgroup of lower triangular matrices $G_0$ intersect in the group of diagonal matrices $\{A^\lambda\colon \lambda\in \K^\times\}$, thus $\rho(A^\lambda)$ stabilizes the geodesic line $(\varphi(\infty),\varphi(0))$, and it consists of either elliptic or loxodromic elements. 
Define the function $\Xi \colon \K^\times \to \R$ by $\lambda\mapsto s \cdot \ell( \rho(A^\lambda))$ where the sign $s\in \{+1,-1\}$ only needs to be defined when $\rho(A^\lambda)$ is loxodromic, as $s=+1$ if $\rho(A^\lambda)^+=\varphi(\infty)$ and $s=-1$ if $\rho(A^\lambda)^+=\varphi(0)$.
This function $\Xi$ is a homomorphism (as the composition of the homomorphisms $\lambda \in \K^\times \mapsto A^\lambda \in A$ and $s\ell \colon \rho(A) \to \R$).
Let us show that its kernel is equal to the subgroup of units $\OO^\times \subset \K^\times$.

Fix $\mu\in \K\setminus\{0\}$ and $x \in (\varphi(0),\varphi(\infty))$. 
By continuity of $\rho$, for all $\varepsilon>0$, there is $M>0$ such that if $\lambda\in \K^\times$ has $\lvert \lambda \rvert >M$ then $\dH(\rho(A^{\lambda^{-1}} R^\mu A^\lambda)x,x)<\varepsilon$. 
The isometry $\rho(R^\mu)$ sends the geodesic line $(\varphi(\infty),\varphi(0))$ to the geodesic line $(\varphi(\infty),\varphi(\mu))$. 
In particular for fixed $\mu \in \K\setminus\{0\}$, the function of $y\in (\varphi(\infty),\varphi(0))$ defined by $y\mapsto \dH(\rho(R^\mu)y,y)$ is monotonous, with limits $\dH(\rho(R^\mu)y,y)\to0$ as $y\to\varphi(\infty)$ and $\dH(\rho(R^\mu)y,y)\to+\infty$ when $y\to\varphi(0)$. 
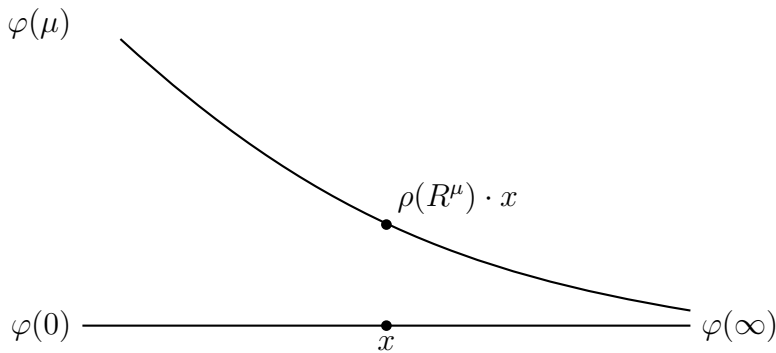
\begin{center} 
	\begin{tikzpicture}
		
		\draw[thick] (0,0) -- (8,0);
		\node[left] at (0,0) {$\varphi(0)$};
		\node[right] at (8,0) {$\varphi(\infty)$};
		
		\node[left] at (0,4) {$\varphi(\mu)$};
		
		\draw[thick]
		(0.5,3.8)
		.. controls (3,1.5) and (5,.7) ..
		(8,0.2);
		
		\fill (4,1.34) circle (2pt);
		\node[above right] at (4,1.34) {$\rho(R^\mu)\cdot x$};
		
		\fill (4,0) circle (2pt);
		\node[below] at (4,0) {$x$};
		
	\end{tikzpicture}
\end{center}
Let $\lambda\in \K^\times \setminus \OO^\times$. Up to inversion we may assume $\lvert \lambda \rvert >1$, hence as $n\in \N$ goes to $+\infty$ we have $\lvert \mu / \lambda^n\rvert$ goes to $0$.
We have $\dH(\rho(R^{\mu/\lambda^{-n}})x,x)=\dH(\rho(A^{\lambda^{-n}}R^\mu A^{\lambda^n})x,x)=\dH(\rho(R^\mu)\rho(A^{\lambda^n})x,\rho(A^{\lambda^n})x)$ and by continuity of $\rho$ the left-hand side tends to $\dH(\mathbf{1}\cdot x, x)=0$.
Thus applying the previous observation to $y_n=\rho(A^{\lambda^n})x$, we deduce that $\rho(A^{\lambda^n})x$ tends to $\varphi(\infty)$ as $n\in \N$ tends to $+\infty$, thus $\rho(A^\lambda)$ is loxodromic with (attractive, repulsive) fixed points $(\varphi(\infty), \varphi(0))$.

Let $\lambda\in \OO^\times$.
Suppose by contradiction that $\ell(\rho(A^\lambda))>0$, and to inverting $\lambda$ we may therefore also assume that as $n\in \N$ goes to $+\infty$ we have $\rho(A^{\lambda})^n x\to \varphi(\infty)$ thus by the previous observation we have $\dH(\rho(R^\mu)\rho(A^\lambda)^nx,\rho(A_\lambda)^nx)\to 0$.
However $\dH(\rho(R^\mu)\rho( A_\lambda)^nx, \rho(A_\lambda)^nx)=\dH(\rho(R^{\mu/\lambda^n})x,x)$ and as $\lvert \mu/\lambda^n\rvert = \lvert \mu \rvert$, we have that $\{\rho(R^{\lambda^n\mu})\colon n\in\N\}$ is coarsely bounded by \Cref{lem:bounded_PSL2K} so $\dH(\rho(R^{\mu/\lambda^n})x,x)$ is bounded as $n\to\infty$.
This yields a contradiction.

So far we have shown that $\Xi$ factors through the quotient $\K^\times \bmod{\OO^\times} \to \R$ to a function that is proper in the sense that $\lvert \lambda\rvert \to \infty \implies \lvert \ell( \rho(A^\lambda))\rvert \to \infty$.
Moreover we saw that if $\lvert \lambda\rvert >1$ then $\rho(A^\lambda)^+=\varphi(\infty)$ and if $\lvert \lambda\rvert <1$ then $\rho(A^\lambda)^+=\varphi(0)$.
Hence there is a proper increasing function $f \colon \R_{>0} \to \R$ such that $\Xi(\rho(A^{\lambda}))=f(\lvert \lambda\rvert)$, so $f \colon (\R_{>0},\times) \to \to (\R,+)$ which is a group homomorphism.
Such an increasing homomorphism $f$ must be continuous, so there is $c>0$ such that $f=c\log$. 
In particular for all $M>0$ the set $A_M=\{A^{\lambda} \colon \lambda,\lambda^{-1} \in \OO_M\}$ has bounded images by $\rho$ in $\Isom(\Hyp^\kappa)$. 

Finally, the Cartan decomposition for $\PGL_2(\K)=KAK$ with $K=\PSL_2(\OO)$.
Since $K$ is coarsely bounded by \cref{lem:bounded_PSL2K}, we deduce from with previous paragraph that a subset $B\subset PGL_2(\K)$ is $\mathcal{E}_{d_\infty}$-bounded (namely any point $x\in T$ has bounded displacement under $B$) if and only if there is $M>0$ such that $B\subset KA_MK$. 
This proves that the $\rho$-action of $\PGL_2(\K)$ on $\Hyp^\kappa$ is $\EE_{d_\infty}$-bornologous.
\end{proof}

\begin{proof}[Proof of \Cref{cor:point_character_variety} for $\PGL_2(\K)$]
\Cref{lem:Action-PGL2K-boundary-Tree} and \Cref{lem:Action-PGL2K-boundary-Tree_BAB} show that the action of $\PGL_2(\K)$ on $(T,d_\infty)$ satisfies \Cref{def:full-action-CAT-1}, and \Cref{prop:PGL2K_bornlogous_coarsely-proper} shows that an non-elementary representation $\PGL_2(\K)\to \Isom(\Hyp^\kappa)$ is $\EE_{d_\infty}$-bornologous, thus \Cref{thm:point_character_variety} applies.
\end{proof}

\bibliographystyle{alpha} 
\bibliography{biblio.bib}

\end{document}